\documentclass[10pt]{article}%

\usepackage{marvosym}

\usepackage{booktabs}
\usepackage{multirow}
\usepackage{diagbox}
\usepackage{colortbl}
\renewcommand{\arraystretch}{1.3}
\usepackage{csquotes}

\usepackage{caption}
\usepackage{subcaption}

\usepackage{fancyvrb}

\usepackage[show]{ed}

\usepackage{tikz}
\usepackage{pgf}
\usetikzlibrary{arrows}
\usetikzlibrary{calc}
\usetikzlibrary{decorations.pathmorphing}

\usepackage{soul}
\setstcolor{red}
\setulcolor{red}
\usepackage{bbm}

\usepackage{rotating} 

\usepackage{mathrsfs}
\DeclareMathAlphabet{\mathpzc}{OT1}{pzc}{m}{it}

\usepackage{color}
\usepackage{amsmath}
\usepackage{graphicx}
\usepackage{amsfonts}
\usepackage{amssymb}%
\usepackage{latexsym}
\usepackage{psfrag}
\usepackage{accents}

\newtheorem{theorem}{Theorem}[section]
\newtheorem{corollary}[theorem]{Corollary}
\newtheorem{conjecture}[theorem]{Conjecture}
\newtheorem{definition}[theorem]{Definition}
\newenvironment{proof}[1][Proof]{\noindent \emph{#1.} }
{\hfill \ \rule{0.5em}{0.5em}}
\newtheorem{lemma}[theorem]{Lemma}
\newtheorem{proposition}[theorem]{Proposition}

\newtheorem{assumption}[theorem]{Assumption}

\numberwithin{equation}{section}
\numberwithin{table}{section}
\numberwithin{figure}{section}

\usepackage[a4paper]{geometry}
\newtheorem{remark}[theorem]{Remark}
\newtheorem{example}[theorem]{Example}

\newcommand{\nc}{\newcommand}

\makeatletter
\newcommand\RedeclareMathOperator{%
  \@ifstar{\def\rmo@s{m}\rmo@redeclare}{\def\rmo@s{o}\rmo@redeclare}%
}
\newcommand\rmo@redeclare[2]{%
  \begingroup \escapechar\m@ne\xdef\@gtempa{{\string#1}}\endgroup
  \expandafter\@ifundefined\@gtempa
     {\@latex@error{\noexpand#1undefined}\@ehc}%
     \relax
  \expandafter\rmo@declmathop\rmo@s{#1}{#2}}
\newcommand\rmo@declmathop[3]{%
  \DeclareRobustCommand{#2}{\qopname\newmcodes@#1{#3}}%
}
\@onlypreamble\RedeclareMathOperator
\makeatother

\nc{\bfx}{\mathbf{x}} 
\nc{\bfy}{\mathbf{y}} 
\nc{\bfz}{\mathbf{z}} 
\nc{\bfu}{\mathbf{u}} 
\nc{\bfv}{\mathbf{v}} 
\nc{\bfw}{\mathbf{w}} 
\nc{\bft}{\mathbf{t}} 
\nc{\bfb}{\mathbf{b}} 
\nc{\bfn}{\mathbf{n}} 
\nc{\bfr}{\mathbf{r}} 
\nc{\bfc}{\mathbf{c}} 

\nc{\bfA}{\mathbf{A}} 
\nc{\bfB}{\mathbf{B}} 
\nc{\bfC}{\mathbf{C}} 
\nc{\bfR}{\mathbf{R}} 
\nc{\bfD}{\mathbf{D}} 
\nc{\bfI}{\mathbf{I}} 
\nc{\bfM}{\mathbf{M}} 
\nc{\bfK}{\mathbf{K}} 
\nc{\bfP}{\mathbf{P}} 
\nc{\bfU}{\mathbf{U}} 
\nc{\bfZ}{\mathbf{Z}} 
\nc{\bfV}{\mathbf{V}} 
\nc{\bfE}{\mathbf{E}} %
\nc{\bfH}{\mathbf{H}} %
\nc{\bfX}{\mathbf{X}} %
\nc{\bfId}{\mathbf{I}_{\mathrm{d}}} 

\nc{\bbN}{\mathbb{N}} 
\nc{\bbR}{\mathbb{R}} 
\nc{\bbP}{\mathbb{P}} 
\nc{\bbC}{\mathbb{C}} 
\nc{\wH}{\widetilde{H}}
\nc{\calS}{\mathcal{S}} %
\nc{\calD}{\mathcal{D}} %
\nc{\calN}{\mathcal{N}} %
\nc{\calT}{\mathcal{T}} %
\nc{\calR}{\mathcal{R}} %
\nc{\calP}{\mathcal{P}} %
\nc{\calV}{\mathcal{V}} %
\nc{\calW}{\mathcal{W}} %
\nc{\calJ}{\mathcal{J}} %
\nc{\calE}{\mathcal{E}} %
\nc{\calA}{\mathcal{A}} %
\nc{\calU}{\mathcal{U}} %
\nc{\calK}{\mathcal{K}} %
\nc{\calL}{\mathcal{L}} %

\nc{\rouge}{\color{red}}
\nc{\bleu}{\color{blue}}
\nc{\cyan}{\color{cyan}}
\nc{\noir}{\color{black}\rm}

\nc\dif{\mathop{}\!\mathrm{d}}   
\DeclareMathOperator{\supp}{supp}   
\DeclareMathOperator{\diam}{diam} 

\newcommand{\vertiii}[1]{{\left\vert\kern-0.25ex\left\vert\kern-0.25ex\left\vert #1
    \right\vert\kern-0.25ex\right\vert\kern-0.25ex\right\vert}} 
\RedeclareMathOperator{\Re}{Re} 
\RedeclareMathOperator{\Im}{Im} 

\newcommand{\cA}{{\cal A}}

\newcommand{\cF}{{\cal F}}
\newcommand{\cH}{{\cal H}}
\newcommand{\cI}{{\cal I}}

\newcommand{\cZ}{{\cal Z}}
\newcommand{\bx}{x}
\newcommand{\by}{y}
\newcommand{\ba}{\hat{a}}

\newcommand{\re}{{\rm e}}
\newcommand{\ri}{{\rm i}}
\newcommand{\rd}{d}

\newcommand{\e}{\epsilon}

\newcommand{\beq}{\begin{equation}}
\newcommand{\eeq}{\end{equation}}
\newcommand{\beqs}{\begin{equation*}}
\newcommand{\eeqs}{\end{equation*}}
\newcommand{\bit}{\begin{itemize}}
\newcommand{\eit}{\end{itemize}}
\newcommand{\ben}{\begin{enumerate}}
\newcommand{\een}{\end{enumerate}}
\newcommand{\bal}{\begin{align}}
\newcommand{\eal}{\end{align}}
\newcommand{\bals}{\begin{align*}}
\newcommand{\eals}{\end{align*}}
\newcommand{\bse}{\begin{subequations}}
\newcommand{\ese}{\end{subequations}}
\newcommand{\bpr}{\begin{proposition}}
\newcommand{\epr}{\end{proposition}}
\newcommand{\bre}{\begin{remark}}
\newcommand{\ere}{\end{remark}}
\newcommand{\bpf}{\begin{proof}}
\newcommand{\epf}{\end{proof}}
\newcommand{\ble}{\begin{lemma}}
\newcommand{\ele}{\end{lemma}}
\newcommand{\bco}{\begin{corollary}}
\newcommand{\eco}{\end{corollary}}
\newcommand{\bex}{\begin{example}}
\newcommand{\eex}{\end{example}}
\newcommand{\bth}{\begin{theorem}}
\newcommand{\enth}{\end{theorem}}

\newcommand{\Rea}{\mathbb{R}}
\newcommand{\Com}{\mathbb{C}}

\newcommand{\Oi}{{\Omega^-}}

\newcommand{\Oe}{{\Omega^+}}

\newcommand{\eps}{\varepsilon}

\newcommand{\pdiff}[2]{\frac{\partial #1}{\partial #2}}

\newcommand{\half}{\frac{1}{2}}

\newcommand{\LtG}{{L^2(\bound)}}

\newcommand{\LtGt}{{\LtG\rightarrow \LtG}}

\newcommand{\HoG}{H^1(\Gamma)}

\newcommand{\tendi}{\rightarrow \infty}

\newcommand{\mac}[1]{\mathsf{H}_{\hsc}^{#1}}
\newcommand{\macl}[1]{\mathsf{H}_{\hsc,\loc}^{#1}}

\def\XXint#1#2#3{{\setbox0=\hbox{$#1{#2#3}{\int}$}
     \vcenter{\hbox{$#2#3$}}\kern-.5\wd0}}

\newcommand*{\N}[1]{\left\|#1\right\|}

\providecommand{\blue}{\textcolor{black}}
\allowdisplaybreaks[4]

\usepackage{hyperref}
\definecolor{myblue}{rgb}{0,0,0.6}
\definecolor{lightGray}{RGB}{198,198,198}
\hypersetup{colorlinks=true, linkcolor=myblue,citecolor=myblue,filecolor=myblue,urlcolor=myblue}

\newcommand{\tfa}{\text{ for all }}
\newcommand{\tfor}{\text{ for }}

\newcommand{\tor}{\text{ or }}
\newcommand{\tif}{\text{ if }}

\newcommand{\ton}{\text{ on }}

\newcommand{\tand}{\text{ and }}
\newcommand{\tst}{\text{ such that }}

\newcommand{\FL}{F_{\mathscr{L}}}
\newcommand{\FH}{F_{\mathscr{H}}}
\newcommand{\FLm}{\mathcal{F}_{\mathscr{L}}}
\newcommand{\FHm}{\mathcal{F}_{\mathscr{H}}}

\newcommand{\bound}{\Gamma}

\usepackage{soul}

\definecolor{jwcol}{RGB}{27, 137, 18}  

\definecolor{dalcol}{rgb}{0.8,0,0}

\definecolor{jeffColor}{RGB}{102, 0, 204}

\definecolor{escol}{rgb}{0,0,0.8}
\definecolor{estcol}{rgb}{0,0.5,0}
\definecolor{esnewcol}{rgb}{0,0.5,0}

\newcommand{\Creg}{{C_{\rm osc}}}

\newcommand{\mythmname}[1]{\emph{\textbf{(#1)}}}

\newcommand{\noi}{\noindent}

\newcommand{\hsc}{\hbar}
\newcommand{\Cqo}{C_{\rm qo}}

\newcommand{\Op}{{\rm Op}}
\newcommand{\WF}{{\rm WF}}

\newcommand{\Reg}{S_{\ri k}}

\newcommand{\Breg}{ B_{k, {\rm reg}}}
\newcommand{\Bregp}{ B'_{k, {\rm reg}}}

\newcommand{\DL}{K}
\newcommand{\operator}{\cA}

\newcommand{\pert}{L}
\newcommand{\hDarg}{(|k^{-1} D'|_g^2)}
\newcommand{\Hsh}{H^s_k(\Gamma)}
\newcommand{\Hrh}{H^r_k(\Gamma)}
\newcommand{\Hshs}{H^s_k}
\newcommand{\Hsk}{{H^s_k(\Gamma)}}
\newcommand{\Hsht}{H^s_k(\Gamma)\to H^s_k(\Gamma)}
\newcommand{\Hshts}{H^s_k\to H^s_k}

\newcommand{\ind}{\mathbbm{1}}
\newcommand{\LN}{\mathfrak{L}}
\newcommand{\TN}{\mathfrak{T}}
\newcommand{\oldK}{n}
\newcommand{\loc}{\operatorname{loc}}
\newcommand{\comp}{\operatorname{comp}}
\newcommand{\mc}[1]{\mathcal{#1}}
\newcommand{\Ell}{\operatorname{Ell}}
\newcommand{\Ang}{\operatorname{Tan}}
\usepackage{titlesec}
\newcommand{\cJ}{\mathcal{J}}
\titleformat{\section}
  {\centering\Large\rmfamily}  
  {\thesection}                
  {1em}                        
  {} 
  
  \titleformat{\subsection}
  {\normalsize\bfseries}  
  {\thesubsection}                
  {1em}                        
  {} 

\newtheorem{question}[theorem]{Question}
  \newcommand{\trig}{\mathscr{T}}

\newcommand{\Psit}{\Psi_{\mathsf{T}}}
\newcommand{\St}{S_{\mathsf{T}}}
\newcommand{\nch}{N_{\textrm{pan}}}
\newcommand{\ngk}{p}

\VerbatimFootnotes

\title{
Helmholtz boundary integral methods and the pollution effect}

\author{
    J. Galkowski\thanks{Department of Mathematics, University College London, 25 Gordon Street, London, WC1H 0AY, UK,
    \tt J.Galkowski@ucl.ac.uk}
    \and
           M.~Rachh\thanks{Department of Mathematics, Indian Institute of Technology Bombay, Powai, Mumbai 400076, India
    \tt mrachh@iitb.ac.in}
    \and
    E.~A.~Spence\thanks{Department of Mathematical Sciences, University of Bath, Bath, BA2 7AY, UK,
      \tt E.A.Spence@bath.ac.uk}}
\date{
    \today
}

\begin{document}

\maketitle

\begin{abstract}

This paper is concerned with solving the Helmholtz exterior Dirichlet and Neumann problems with large wavenumber $k$ and smooth obstacles
using the standard second-kind boundary integral equations.
We consider Galerkin and collocation methods -- with  subspaces consisting of \emph{either} piecewise polynomials (in 2-d for collocation, in any dimension for Galerkin) \emph{or} trigonometric polynomials (in 2-d) -- as well as a fully discrete quadrature (Nystr\"om) method based on trigonometric polynomials (in 2-d).

For each of these methods, we 
prove -- in many cases for the first time -- rigorous results about 
 the fundamental question:~how quickly must 
 the number of degrees of freedom (the dimension of the approximation space) grow with $k$ to maintain accuracy of the computed solution? 
Importantly, we determine which of these methods suffer from \emph{the pollution effect}. That is, we address the question:~must 
 the number of points per wavelength $\to \infty$  to maintain accuracy as $k\to\infty$?
 

\paragraph{AMS:}
65N35, 65N38, 65R20, 35J05

\end{abstract}

\setcounter{tocdepth}{1}
\tableofcontents

%
%
%
%
%
%
%

\section{Introduction}

\subsection{Informal discussion of the main results, their context, and their novelty}

\paragraph{Boundary integral equations for scattering problems.}

The fundamental solution of a linear partial differential equation (PDE) can be used to convert 
boundary value problems for that PDE to integral equations
on the boundary of the domain; i.e., a boundary integral equation (BIE).
If the
original domain is the exterior of a bounded obstacle, then the problem is converted
from one on a $d$-dimensional infinite domain, to one on a $(d-1)$-dimensional finite
domain.
In particular, radiation conditions at infinity (ensuring that the PDE problem is well-posed) are automatically imposed via an appropriate choice of fundamental solution in the boundary integral operators. 
 Together with the fact that the fundamental solutions of the corresponding PDEs are easy to evaluate, these features make BIEs 
 a popular way to compute approximations to acoustic, electromagnetic, and elastic scattering problems with constant wave speed and bounded scatterer. 

In this paper, we consider solving the Helmholtz equation $\Delta u+k^2u=0$ posed in the exterior of a smooth obstacle, with either Dirichlet or Neumann boundary conditions, using the standard second-kind BIEs, i.e., those of the form ``identity plus compact". The main attraction of these second-kind equations is that, 
with $N$ the number of degrees of freedom (i.e., the dimension of the approximation space), in the limit $N\to\infty$ with $k$ fixed the condition number is bounded independent of $N$  \cite[\S3.6]{At:97} (which is not the case for first-kind equations \cite[\S4.5]{SaSc:11}).

We study Galerkin and collocation methods -- with  subspaces consisting of \emph{either} piecewise polynomials (in 2-d for collocation, in any dimension for Galerkin) \emph{or} trigonometric polynomials (in 2-d) -- as well as a fully discrete quadrature (a.k.a., Nystr\"om) method based on trigonometric polynomials (in 2-d). 
The reasons we study these three methods together, rather than separately, are that (i) the Galerkin and collocation methods both fall into an abstract projection-method framework, so that proving results about these two methods together involves essentially no extra work to considering them separately, (ii) the results about the Nystr\"om method use in an essential way the collocation results, and (iii) there is large interest from the numerical-analysis and engineering communities in each of these three methods.

\paragraph{The question studied in this paper.}
For the second-kind boundary integral equations described above, we study the fundamental question:~how quickly must $N$ (the dimension of the approximation space) grow with $k$ to maintain accuracy of the computed solution as $k\tendi$?
The error analysis of these methods in the limit $N\to\infty$ with $k$ fixed is well understood (see, e.g., the books 
\cite{At:97, SaVa:02, HsWe:08, St:08, SaSc:11}). However, there are relatively few existing error analyses valid as $k\to \infty$ \cite{BaSa:07, LoMe:11, GrLoMeSp:15, SpKaSm:15, GaMuSp:19, GaSp:22}. This is in contrast to the large literature on the analogous question for finite element methods, starting from the late 1980s \cite{AzKeSt:88} and 1990s \cite{IhBa:95a, IhBa:97, Me:95} to today \cite{BeChMe:25, GS3, LiWu:25} (see, e.g., \cite[\S3.1--3.2]{GSAN} for a more-detailed review).
Furthermore, there are no $k$-explicit error analyses for either collocation or Nystr\"om boundary integral methods, and none for the Galerkin method applied to the standard second-kind Neumann BIEs. 

\paragraph{The results of this paper.}

The first main result of this paper gives sufficient conditions on $N$ for $k$-uniform accuracy of boundary integral methods as $k\to\infty$ for each of the Galerkin, collocation, and Nystr\"om methods.
These conditions are summarised in Table \ref{f:summary} below, and the bounds control 
both the relative error and the error in terms of the best approximation error.

A key question is whether these BIE methods suffer from \emph{the pollution effect}; i.e., 
must the number of points per wavelength $\to \infty$ to maintain accuracy as $k\to\infty$, or, equivalently, 
must $N\gg k^{d-1}$ to maintain accuracy as $k\to\infty$? (To see where the $k^{d-1}$ threshold comes from, recall that, by the Nyquist--Shannon--Whittaker sampling theorem, approximating an arbitrary function oscillating with frequency $\lesssim k$ in $d-1$ dimensions requires $k^{d-1}$ degrees of freedom.)

The sufficient conditions on $N$ summarised in Table \ref{f:summary} show that for all the methods involving trigonometric polynomials, given $\e>0$, at most $\sim k^{d-1+\e}$ degrees of freedom are required to maintain accuracy, even for trapping obstacles. Moreover, in some cases only $\sim k^{d-1}$ degrees of freedom are required; i.e., there is no pollution.

The most commonly-used approximation spaces in both the finite-element and boundary-element method are piecewise polynomials of fixed degree,
and finite-element methods with these spaces
famously suffer from the pollution effect; i.e. $\gg k^d$ degrees of freedom are required to maintain accuracy~\cite{IhBa:95a, IhBa:97, BaSa:00} (where now the exponent is $d$ because finite-element methods compute the solution in a subset of $\Rea^d$). 

There is a common belief in both engineering and numerical analysis that boundary element methods \blue{for scattering problems} do \emph{not} suffer from the pollution effect
~\cite{Ma:02, BaSa:07, LoMe:11,GrLoMeSp:15, GaSp:22}. Indeed, it is standard in engineering to ask whether a specific number of points per wavelength (e.g., six~\cite{Ma:02}) suffices to obtain accurate solutions.  
However, the sufficient conditions on $N$ summarised in Table \ref{f:summary} for piecewise-polynomial spaces leave open the possibility that boundary integral methods using these spaces suffer from the pollution effect. 

The second main result of the paper is that, for obstacles that trap geometric-optic rays, the Galerkin method with piecewise polynomials of fixed degree applied to the standard second-kind Dirichlet BIEs \emph{suffers from the pollution effect}; i.e., accuracy cannot be maintained as $k\to\infty$ for \emph{any} fixed number of points per wavelength. This result is illustrated by numerical experiments demonstrating pollution for two trapping situations; see Figures \ref{f:bReallyCoolPicture}--\ref{f:cavityPollution} below. 
Furthermore, we prove that the Galerkin method with piecewise polynomials of fixed degree applied to Neumann BIEs suffers from the pollution effect \emph{even for nontrapping obstacles}. These two results establishing pollution show that the sufficient conditions for control of the Galerkin error summarised in Table \ref{f:summary} are indeed necessary. 
We note that our results show that, just as for the finite-element method, increasing the polynomial degree substantially reduces the amount of pollution.

We highlight that our proof that the Galerkin method suffers from the pollution effect is the first rigorous proof of pollution for a numerical method applied to a Helmholtz problem with non-empty scatterer (recall that the FEM results in  \cite{BaSa:00} consider the Helmholtz equation with constant coefficients posed on a 2-d infinite domain with no boundaries, and hence no scattering). Furthermore, our results about the Nystr\"om method are, to our knowledge, the first $k$-explicit convergence results about a fully-discrete boundary integral method applied to the Helmholtz equation.

\paragraph{Overview of the ingredients in the proofs of these results.}
The sufficient conditions on $N$ for $k$-uniform accuracy are obtained by combining the following:
\bit
\item[(i)] Results from the first author's PhD thesis \cite{Ga:19} showing that the  high-frequency components of boundary integral operators are ``well behaved" in the sense that they are semiclassical pseudodifferential operators (i.e., pseudodifferential operators with a large parameter $k$ built in, tailoring the calculus to study functions oscillating at frequency $k$). 
\footnote{Recall that the integral operators' structure as homogeneous -- as opposed to semiclassical -- pseudodifferential operators is classically used to obtain results for fixed $k$ as $N\to \infty$; see, e.g., the books \cite{SaVa:02, HsWe:08, GwSt:18} and the references therein.}
\item[(ii)] A new abstract projection-method argument for second-kind equations; the main novelties in this argument are that it carefully analyses the high-to-high, high-to-low, low-to-high, and low-to-low frequency norms of the best-approximation-to-error map and, furthermore, uses a norm that weights high and low frequencies differently, with the weighting tailored to the dimension of the approximation space and $k$. 
\eit
The results proving that pollution occurs for the Galerkin method with piecewise polynomials are obtained by combining Point (i) above with the following: 
\bit
\item[(a)] The fact that the second-kind structure of the integral operators allows one to write the Galerkin error concretely in terms of the discrete inverse and the best approximation error 
(see Lemma \ref{lem:QOabs} below); this is used to show that 
demonstrating pollution reduces to approximately solving $(I-P_N)\widetilde{f}=f$, where $P_N$ is the Galerkin projection, 
for particular BIE data $f$ such that the (continuous) BIE inverse applied to $f$ is large (along with some additional properties). 
\item[(b)] Recent lower bounds on piecewise-polynomial approximation of oscillatory functions from \cite{Ga:25}; these are used to approximately solve the problem $(I-P_N)\widetilde{f}=f$.
\item[(c)] The relationship between the inverses of the integral operators and interior and exterior Helmholtz solution operators  \cite[Theorem 2.33]{ChGrLaSp:12}, \cite[Lemma 7.4]{GaMaSp:21N}; this is used to reduce the problem of finding the data $f$ (on the boundary) to finding particular Helmholtz solutions in the exterior domain.
\item[(d)] PDE results about how the high-frequency behaviour of Helmholtz solutions is dictated by the geometric-optic rays (usually known as ``propagation of singularities" results 
\cite{MeSj:78,MeSj:82,Ho:07}); these are used to construct the Helmholtz solutions required in (c).
\eit
These ideas behind the proofs are discussed in more detail in \S\ref{s:discussion}.

\subsection{
Second-kind boundary integral equations for the Helmholtz equation.}
\label{sec:BIEs}

Let $\Oi \subset \Rea^d$, $d\geq 2$ be a bounded open set such that its open complement $\Oe :=\Rea^d \setminus \overline{\Oi }$ is connected. Let $\Gamma:= \partial \Oi $ and assume that $\Gamma$ is $C^\infty$. The second-kind BIE formulations reformulate solving the scattering problem as: given $f\in L^2(\Gamma)$, find $v\in L^2(\Gamma)$ such that
\begin{equation}
\label{e:basicForm}
\operator v=f,\qquad \operator := c_0 (I+\pert),\qquad c_0\in\mathbb{C}\setminus \{0\},
\end{equation}
where the operator $\pert:\LtGt$ is compact and depends on a parameter $k$. 
Theorem \ref{thm:BIEs} below recaps the standard result that the Helmholtz exterior Dirichlet and Neumann problems can be reformulated as integral equations of the form~\eqref{e:basicForm}
where $f$ is given in terms of the known Dirichlet/Neumann boundary data and $\operator$ is one of the boundary integral operators (BIOs)
\beq\label{e:DBIEs}
A_k':= \half I + \DL'_k - \ri\eta_D S_k
    \quad\tand\quad
A_k:= \half I + \DL_k - \ri\eta_D S_k
       \eeq
for the Dirichlet problem and
\beq\label{e:NBIEs}
    \Breg := \ri\eta_N \left(\dfrac{1}{2}I-\DL_k\right) +  \Reg H_k
    \quad\tand\quad
    \Breg' := \ri\eta_N \left(\dfrac{1}{2}I-\DL_k'\right) +  H_k\Reg
       \eeq
       for the Neumann problem. The operators $S_k, \DL_k, \DL'_k$, and $H_k$ are the single-, double-, adjoint-double-layer and hypersingular operators defined by \eqref{e:SD'} and \eqref{e:DH}. Standard mapping properties of these operators (see, e.g., \cite[Theorems 2.17 and 2.18]{ChGrLaSp:12}) imply that $A'_k, A_k, \Breg, \Bregp$ are bounded on $\LtG$. Furthermore, if $\eta_D,\eta_N\in \Rea\setminus\{0\}$, then each of  $A'_k, A_k, \Breg, \Bregp$ is invertible on $\LtG$; moreover they are each equal to a multiple of the identity plus a compact operator on $\LtG$ (see, e.g., \cite[\S2.6]{ChGrLaSp:12} for $A'_k, A_k$ and \cite[Theorem 2.2]{GaMaSp:21N} for $\Breg, \Bregp$). \footnote{In the real-valued $L^2(\Gamma)$ inner product, $A'_k$ and $A_k$ are each other's adjoints (hence the $'$ notation); similarly for $\Breg, \Bregp$.}    
       
\bre[The history of the BIEs \eqref{e:DBIEs} and \eqref{e:NBIEs}]
The BIEs involving the operators $A_k'$ and $A_k$ were introduced in \cite{BrWe:65, Le:65, Pa:65}. The subscript ``reg" on the Neumann boundary-integral operators indicates that these are not the 
``combined-field'' (or ``combined-potential") Neumann BIEs introduced by \cite{BuMi:71} (denoted by $B'_k$ and $B_k$ in, e.g., \cite[\S2.6]{ChGrLaSp:12}); the  BIOs introduced in \cite{BuMi:71} are given by \eqref{e:NBIEs} with $S_{\ri k}$ removed, and thus are not bounded on $\LtG$ because $H_k: \LtG\to H^{-1}(\Gamma)$. The idea of preconditioning $H_k$ with an order $-1$ operator goes back to \cite{Bu:76} (see, e.g., the discussion in \cite{AmHa:90}),
with the use of $S_{\ri k}$ proposed in \cite{BrElTu:12}, and then advocated for in
\cite{BoTu:13,ViGrGi:14} (for more details, see the discussion in, e.g., \cite[\S2.1.1]{GaMaSp:21N}).
In fact, the results of the present paper hold for a wider class of regularising operators, of which $S_{\ri k}$ is the prototypical example; see \cite[Assumption 1.1]{GaMaSp:21N}. For simplicity, however, here we only consider $\Breg$ and $\Bregp$ defined by \eqref{e:NBIEs} involving $S_{\ri k}$.
\ere

   \begin{assumption}[The parameters $\eta_D$ and $\eta_N$]
   \label{ass:parameters}
  There exists $C>0$ such that $C^{-1}\leq  |\eta_N|\leq C$, $C^{-1}\leq k^{-1}|\eta_D|\leq C$, where $\eta_N,\eta_D\in \mathbb{R}$ are the parameters in 
  $A_k$, $A_k'$ \eqref{e:DBIEs}, $\Breg$ and $\Bregp$ \eqref{e:NBIEs}.
   \end{assumption}
   
     \begin{remark}[Choosing the parameters $\eta_D$ and $\eta_N$]
     The question of how to choose the parameters $\eta_D$ and $\eta_N$ has been the subject of much research, starting with    \cite{KrSp:83, Kr:85, Am:90} and then continuing with \cite{BaSa:07, ChMo:08, ChGrLaLi:09, BeChGrLaLi:11, BaSpWu:16, GaMaSp:21N}. The choices in Assumption~\ref{ass:parameters} are the most commonly-used and most rigorously-justified, with our results below contributing to this -- Theorem \ref{t:dirichletDisk} below shows that other choices of $\eta_D$ are worse from the point of view of the pollution effect.
     \end{remark}

Since the operator $\operator$ is a compact perturbation of the identity, $\|\operator^{-1}\|_{\LtGt}$ is bounded below by a $k$-independent constant (see Lemma \ref{lem:inversebound} below). We make the following assumption on the growth of $\|\operator^{-1}\|_{\LtGt}$.

\begin{assumption}[Polynomial boundedness of $\operator^{-1}$]\label{ass:polyboundintro}
There exists $P_{\rm inv}\geq 0$ and $\cJ\subset [0,\infty)$ such that, given $k_0>0$, there exists $C>0$ such that
\beq\label{e:rho}
\rho(k):=\|\operator^{-1}\|_{\LtGt}\leq C k^{P_{\rm inv}} \quad\tfa k\geq k_0 \text{ with } k\in \Rea\setminus \cJ.
\eeq
\end{assumption}

\bre[When is Assumption \ref{ass:polyboundintro} satisfied?]
When the obstacle $\Oi$ is nontrapping, Assumption \ref{ass:polyboundintro} holds with $\cJ= \emptyset$, with $P_{\rm inv} =0$ for the Dirichlet BIEs \cite{BaSpWu:16} and $P_{\rm inv} =2/3$ for the Neumann BIEs \cite[Theorem 2.3]{GaMaSp:21N}. 
In contrast, when the obstacle $\Oi$ is strongly trapping, Assumption \ref{ass:polyboundintro} does not hold for $\cJ=\emptyset$. Indeed, the norm $\rho(k)$ grows exponentially through an increasing sequence of $k$s (see \cite[Theorem 2.8]{BeChGrLaLi:11} for the Dirichlet BIEs and \cite[Theorem 2.6]{GaMaSp:21N} for the Neumann BIEs).
However, this assumption is satisfied for any $\Oi$ for ``most'' frequencies by the results of \cite{LSW1}. More precisely, for any smooth (or even Lipschitz) $\Oi$, given $k_0, \delta>0$, there exists $\cJ\subset [k_0,\infty)$ with $|\cJ|\leq \delta$ such that Assumption \ref{ass:polyboundintro} holds for $k\in\Rea\setminus \cJ $ with $P_{\rm inv}$ independent of $\delta$; see \cite[Corollary 1.4]{LSW1} for this stated for the Dirichlet BIEs and \cite[Theorem 2.3(iii)]{GaMaSp:21N} for the Neumann BIEs. Moreover, the results of \cite[Theorem 3.5]{LSW1} given $N>0$ there exists $P_{\rm inv}$, $\cJ$, and $C>0$ such that Assumption \ref{ass:polyboundintro} holds and $|\cJ \cap (k,\infty)|\leq C k^{-N}$.
\ere

\subsection{Solving the second-kind BIEs via projection methods}

We consider solving the second-kind BIEs~\eqref{e:basicForm} via projection methods; i.e., given $s\geq 0$, a finite dimensional space $V\subset L^2(\Gamma)$ with $\dim V=:N$, and projection $P_V:H^s(\Gamma)\to V$, we approximate the solution of the problem~\eqref{e:basicForm} by:~given $f\in H^s(\Gamma)$, find $v_N\in V$ such that 
\begin{equation}
\label{e:projectionForm}
(I+P_V\pert )v_{N}=c_0^{-1}P_Vf.
\end{equation}
In this paper we study two classic projections -- the Galerkin and collocation projections (defined in \S\ref{s:proj} below) -- and two classic choices of $V$:~piecewise polynomials and (in 2-d) trigonometric polynomials.

The practical implementation of \eqref{e:projectionForm} involves a further approximation of $\pert$ acting on elements of $V$; i.e., \eqref{e:projectionForm} is replaced by:~given $f\in L^2(\Gamma)$, find $v_N\in V$ such that 
\begin{equation}
\label{e:nystromForm}
(I+P_V\pert_V)v_{ N}=c_0^{-1}P_Vf,
\end{equation}
where $\pert_V:V\to H^s(\Gamma)$ is a discrete approximation of $\pert.$ When $P_V$ is the collocation projection, \eqref{e:nystromForm} is known as a \emph{Nystr\"om method.}

We 
study the question:
 \begin{question}
 \label{q:main}
\noi 
Given $\pert=\pert(k)$, a family of approximation spaces $V$, and the corresponding projections $P_V$, how quickly must $N$ increase with $k$ for the solution $v_N$ of \eqref{e:projectionForm}/\eqref{e:nystromForm} to accurately approximate the solution $v$ of \eqref{e:basicForm} as $k\tendi$?
\end{question}

When $\pert$ is fixed and $N\tendi$, the 
\blue{convergence of $v_N$ to $v$} is well understood; see, e.g., the books \cite{At:97,Kr:14, SaVa:02}. 
In contrast, the only rigorous answers to Question~\ref{q:main} in the literature have been for the Galerkin method
\eqref{e:projectionForm} with piecewise-polynomials applied to the Dirichlet BIEs; see \cite{BuSa:06, BaSa:07, LoMe:11, Me:12, GrLoMeSp:15, GaMuSp:19, GaSp:22}.


In this paper we give answers to Question \ref{q:main} for 
\bit
\item \eqref{e:projectionForm} with the Galerkin projection and subspaces consisting of either piecewise polynomials or, in 2-d, trigonometric polynomials, 
\item \eqref{e:projectionForm} in 2-d with the collocation projection and subspaces consisting of either piecewise polynomials or trigonometric polynomials,
\item \eqref{e:nystromForm} in 2-d with the collocation projection, trigonometric polynomials, and $L_V$ approximated by a commonly-used quadrature method (so-called ``Kress quadrature" \cite{Ma:63, Ku:69, Kr:91, Kr:95}).
\eit
We note that 
\bit
\item[--] our analysis of \eqref{e:projectionForm} is done in an abstract projection-method framework, covering both Galerkin and collocation projections, as well as the two choices of subspaces, simultaneously, and
\item[--] our results about the Nystr\"om method \eqref{e:nystromForm} build on the analysis of \eqref{e:projectionForm} with the collocation projection and trigonometric polynomials.
\eit

\bre
The present paper focuses exclusively on 
Question \ref{q:main}. 
We highlight that important considerations in the implementation of \eqref{e:projectionForm} and \eqref{e:nystromForm} for the standard second-kind Helmholtz BIEs are that 
the kernels of the integral operators are singular and long-range -- thus requiring high-order quadrature methods for their accurate discretization, and fast algorithms for applying or inverting the resulting dense linear systems. 
Over the last four decades, a variety of high order quadrature methods
(see, e.g.,  
\cite{Kr:91,
erichsen1998quadrature,
ying,
bremer_2013,
klockner2013quadrature,
bruno_garza_2020,
Siegel2018ALT,
wu2023unified,
af2021accurate,stein2022quadrature,
ding2021quadrature})
fast algorithms for applying the discretized matrices 
(see, e.g., 
\cite{
rokhlin1990rapid, greengard-1997,
bruno1, RjSt:07,liubook,
martinsson-book, jiang2025dual})
and fast direct methods for constructing compressed representations of the matrix and/or its inverse
(see, e.g., \cite{
hackbusch-h-1999,
hackbusch-h-2000,
Beb-hlu-2005,
martinsson-2005,
Borm-h2-2010,
gillman-dir_hss-2012,
ho-2012,
coulier-ifmm_prec-2017,
sushnikova-ce-2018,
jiang2022skel})
have been developed, resulting in accurate, robust, and highly performant linear complexity (up to log factors) computational tools for solving time harmonic wave scattering problems.
\ere

\subsection{The pollution effect}

For a scattering problem, the solution $v$ to the BIE \eqref{e:basicForm} 
oscillates at frequency $\lesssim k$ with 
little additional structure.  The Weyl law (see e.g.~\cite[Chapter 14]{Zw:12}) or the Nyquist--Shannon--Whittaker sampling theorem then implies that the dimension of this space is $\sim k^{d-1}$. Thus, it is certainly not possible to achieve accuracy for general scattering data with a space of dimension $\ll k^{d-1}$. Motivated by this, we say that a numerical BIE method \emph{suffers from the pollution effect} if $N\gg k^{d-1}$ is required to maintain accuracy of the computed solutions as $k\tendi.$ We make this precise via the following definition (we work in $\LtG$ for concreteness, but note that the definition for higher-order spaces on $\Gamma$ is analogous).

\begin{definition}[The pollution effect in $\LtG$]\label{d:pollution}
Let $\mathcal{V}$ be a collection of subspaces of $\LtG$. For $V\in \mathcal{V}$ and $k>0$, let $\operator^{-1}_V(k) : \LtG\to V$ an approximation of $\operator^{-1}(k)$. 
The pair $(\mathcal{V}, \{\operator^{-1}_V\}_{V\in \mathcal{V}})$ \emph{does not suffer from the pollution effect in $\LtG$} if 
there exists $k_0, \Lambda,$ and $\Cqo$ such that for all $k\geq k_0$, $V\in \mathcal{V}$ with $\dim V\geq \Lambda k^{d-1},$ and $f\in \LtG$, 
\beq\label{e:Cqo}
 \N{\operator^{-1} f-\operator^{-1}_V f}_{\LtG} \leq \Cqo \min_{w \in V} \N{ \operator^{-1} f-w}_{\LtG}.
\eeq
\end{definition}
That is $(\mathcal{V}, \{\operator^{-1}_V\}_{V\in \mathcal{V}})$ does not suffer from the pollution effect if $k$-uniform quasi-optimality is achieved (for all possible data) with a choice of subspace dimension proportional to $k^{d-1}$.

By taking the negation of \eqref{e:Cqo}, we see that the pair $(\mathcal{V}, \{\operator^{-1}_V\}_{V\in \mathcal{V}})$ suffers from the pollution effect in $\LtG$ if 
\blue{\begin{align}
\label{e:qo_pollution}
&\inf_{\Lambda>0}\limsup_{k\to \infty}\sup_{\substack{V\in \mathcal{V}\\  \dim V \geq \Lambda k^{d-1}}} 
\sup_{\substack{f\in L^2(\Gamma)\\w\in V}}   \frac{\N{\operator^{-1} f-\operator^{-1}_V f}_{\LtG} }{\N{ \operator^{-1} f-w}_{\LtG}}
=\infty.
\end{align} }

The most commonly-used approximation spaces in both the finite-element and boundary-element method are piecewise polynomials of fixed degree. 
Finite element methods with these spaces 
famously suffer from the pollution effect; i.e. $\gg k^d$ degrees of freedom are required to maintain accuracy~\cite{IhBa:95a, IhBa:97, BaSa:00}. 
In earlier work~\cite{GaSp:22}, the first and third authors showed that for Dirichlet BIEs for a nontrapping obstacle 
(in which case $\rho$ defined by \eqref{e:rho} is bounded uniformly in $k$)
the Galerkin method with piecewise polynomials does not suffer from the pollution effect. 
Indeed, there is a common belief in both engineering~\cite{Ma:02,Ma:17,Ma:16a} and numerical analysis~\cite{BaSa:07, LoMe:11,GrLoMeSp:15} that boundary-element methods \blue{with fixed polynomial degree} do not suffer from the pollution effect. In fact, it is standard in engineering to ask whether a specific number of points per wavelength (e.g., six~\cite{Ma:02}) suffices to obtain accurate solutions.  

In this article, we both improve the analysis in \cite{GaSp:22} -- giving less stringent conditions for $k$-uniform quasioptimality -- and show that i) for trapping obstacles the Galerkin method with piecewise polynomials applied to Dirichlet BIEs \emph{suffers from the pollution effect} and, ii) even for nontrapping obstacles, the Galerkin method with piecewise polynomials applied to Neumann BIEs \emph{suffers from the pollution effect}. We also present numerical experiments demonstrating the pollution effect for the Dirichlet BIEs for two trapping situations. 
\blue{The results of this article show that the pollution effect decreases with increasing $p$. We recall that this is also the case for the finite-element method \cite{GS3} and that -- under 
a polynomial bound on the solution operator and Gevrey/analytic regularity assumptions -- if the polynomial degree grows logarithmically with $k$ then both the finite-
\cite{MeSa:10, MeSa:11, LSW3, LSW4, GLSW1, GSM1, BeChMe:25}
and boundary-element methods \cite{LoMe:11} do not suffer from the pollution effect.
}

\blue{Finally, we note that since the preprint of this paper first appeared, some of the results have been presented in Section 8 of the review article \cite{GSAN}.}

\renewcommand{\arraystretch}{1.4}
\begin{table}\begin{tabular}{|c|c|c|c|c|c|}

\hline
\bf Method&\bf Scatterer& \bf Approx. Space&$d$&\bf{$N\gtrsim$}&$C_{\rm qo}\lesssim$\\
\hline
Galerkin&any& poly.~degree $p$&$\geq 2$&$ k^{(d-1)}\rho^{\frac{d-1}{2(p+1)}}$&$ 1+\rho(k/N^{\frac{1}{d-1}})^{p+1}$\\
\hline
Galerkin&any& trig.~poly. &$2$&$ k$&$ 1$\\
\hline
Collocation&any& poly.~degree $p$ &$2$&$ k\rho^{\frac{1}{p+1}}$&$ 1+\rho(k/N)$\\
\hline
Collocation&any& trig.~poly.&$2$&$ k$&$ 1+ \rho (k/N)^{\infty}$\\
\hline
Dirichlet Nystr\"om& convex&trig.~poly.&$2$&$ k$&$ 1$\\
\hline
Dirichlet Nystr\"om& nontrapping&trig.~poly.&$2$&$ k(\log k)^{+0}$&$ 1$\\
\hline
Dirichlet Nystr\"om& any&trig.~poly.&$2$&$ k^{1+0}$&$ 1$\\
\hline
Neumann Nystr\"om& any&trig.~poly.&$2$&$k^{1+0}$&$ 1$\\
\hline
\end{tabular}
\caption{\label{f:summary}In the table above we summarize the results estimating the quasioptimality constant $C_{\rm qo}$ in \eqref{e:Cqo} in various settings (for Galerkin methods \eqref{e:Cqo} holds in $\LtG$, but for collocation and Nystr\"om methods our results establish the analogue of \eqref{e:Cqo} in $H^s_k(\Gamma)$ for certain $s$). In the above expressions $+0$ represents the fact that for any $\epsilon>0$ the statement holds with $+0$ replaced by $+\epsilon$. Similarly, $(\cdot)^\infty$ means that $\infty$ can be replaced by any power $M>0$.
Finally, recall that $\rho$ appearing in the expressions  is the norm of the inverse of the boundary integral operator (see  \eqref{e:rho}).}
 \end{table}

\section{Statement of the main results}\label{sec:main}

We study three methods for solving BIEs numerically:~the Galerkin method, the collocation method, and the Nystr\"om method. Our results typically depend on several things:~the choice of the approximation space, the growth of the solution operator, and, occasionally, the geometry of the scatterer. A summary of our results is given in Table~\ref{f:summary}.


\subsection{Choices of projections}\label{s:proj}
The most natural choice of projection from the Hilbert space structure of $\LtGt$ is the Galerkin projection.
\begin{definition}[Galerkin projection and Galerkin solution]
Given a subspace $V\subset L^2(\Gamma)$, the \emph{Galerkin projection onto $V$}, $P_{V}^G:L^2(\Gamma)\to V$ is the orthogonal projection onto $V$; in particular,
\beq\label{e:Galerkin_def}
\N{(I-P^G_V)v}_{\LtGt}= \min_{w\in V}\N{v-w}_{\LtG}.
\eeq 
We call the solution to~\eqref{e:projectionForm} with $P_V=P_V^G$ the \emph{Galerkin solution of~\eqref{e:basicForm}}.
\end{definition}
Computing the Galerkin projection involves numerically approximating integrals. It is therefore usually less computationally expensive to use a projection based on point values. 

To define this type of projection, we need to use points that determine the finite dimensional space $V$ in an appropriate sense.
\begin{definition}
Let $V\subset C^0(\Gamma)$ and $\{x_j\}_{j=1}^N\subset \Gamma$. We say that $\{x_j\}_{j=1}^N$ are \emph{unisolvent for $V$} if for any $\{a_j\}_{j=1}^N\subset \mathbb{C}$, there is a unique $v_h\in V$ such that $v_h(x_j)=a_j$ for $j=1,\dots, N$.  
\end{definition}
We can now define the collocation projection.
\begin{definition}
\label{d:collocation}
Given a subspace $V\subset C^0(\Gamma)$ and points $\{x_j\}_{j=1}^N\subset \Gamma$ that are unisolvent for $V$, we define the \emph{collocation projection for $\{x_j\}_{j=1}^N$}, $P_V^C:C^0(\Gamma)\to V$, by
$$
(P_V^Cv)(x_j)=v(x_j),\quad j=1,\dots, N.
$$
Note that $P_N^C$ is well defined since $\{x_j\}_{j=1}^N$ are unisolvent for $V$.  We call the solution to~\eqref{e:projectionForm} with $P_V=P_V^C$ the \emph{collocation solution of~\eqref{e:basicForm}}.

\end{definition}

\subsection{Sufficient conditions for quasi-optimality with piecewise-polynomial approximation spaces}

Our results about piecewise polynomial spaces require some further technical assumptions that, roughly speaking, guarantee that the piecewise-polynomial spaces are maximally dense in Sobolev spaces. 

\begin{assumption}[Assumption on piecewise-polynomial $V_h$ for the Galerkin method]\label{ass:ppG}

\

(i) $(\mathcal{T}_h)_{0<h\leq h_0}$ is a sequence of meshes of $\Gamma$ (in the sense of, e.g., \cite[Definition 4.1.2]{SaSc:11}), such that, for all $\tau \in \mathcal{T}_h$ there exists a reference map $\chi_\tau: \widehat{\tau}\to \tau$ where $\widehat{\tau}$ is a reference element.

(ii)
\beqs
V_h:= \Big\{ v : \Gamma \to \mathbb{C} : \tfa \tau \in \mathcal{T}_h, \,v \circ \chi_\tau \text{ is a polynomial of degree $p$} \Big\}.
\eeqs

(iii)
Given $p\geq 0$ there exists $C>0$ such that for all $0<h\leq h_0$
there exists a operator $\cI_h$ such that, for $0\leq t\leq p+1$, $\cI_h:H^t(\Gamma)\to V_N$ with
\beq\label{e:assppG}
\N{v-\cI_h v}_{\LtG} \leq C h^{t} \N{v}_{H^t(\Gamma)}.
\eeq
\end{assumption}

\begin{theorem}[Galerkin method with piecewise polynomials]\label{thm:PG}
Suppose $\operator$ is given by one of~\eqref{e:DBIEs} or~\eqref{e:NBIEs}, Assumptions~\ref{ass:parameters} and~\ref{ass:polyboundintro} hold. Let $p\geq 0$,  $V_h\subset L^2(\Gamma)$ satisfy Assumption~\ref{ass:ppG}, and  $k_0>0$. Then there are $c,C>0$, such that if
\beq\label{e:Pthres}
(hk)^{2p+2}\rho\leq c ,\qquad hk\leq c,
\eeq
and $k\geq k_0$, $k\notin\cJ$ the Galerkin solution, $v_h$, to~\eqref{e:basicForm} exists, is unique, and satisfies the quasi-optimal error bound
\beq\label{e:PPMR1}
\N{v-v_h}_\LtG\leq \Big(1+C(hk)^{p+1}\rho+Chk+Ck^{-1}\Big)
\N{(I-P^G_{V_h})v}_{\LtG}.
\eeq
Moreover, if the right-hand side $f$ corresponds to plane-wave scattering (i.e., is given as in Theorem \ref{thm:BIEs}), then
\beq\label{e:PPMR2}
\frac{\N{v-v_h}_\LtG}{\N{v}_{\LtG}} \leq \Big(1+C(hk)^{p+1}\rho+Chk+Ck^{-1}\Big)C (hk)^{p+1}
\eeq
(i.e., the relative error can be made arbitrarily-small by decreasing $(hk)^{2p+2}\rho$).
\end{theorem}

\bre[Preasymptotic $h$-BEM and $h$-FEM error estimates]
\label{rem:preasymptotic}
Theorem \ref{thm:PG}, which is shown to be sharp by Theorem\blue{s} \ref{t:pollutionIntro}, \blue{\ref{t:diamond}, \ref{t:diamondNew}, and~\ref{t:neumannDisk}} below, is the $h$-BEM analogue of preasymptotic error estimates for the $h$-FEM. Indeed, for the $h$-FEM, if $(hk)^{2p} \mathscr{R}$ is sufficiently small, then the Galerkin solution exists, is unique, and satisfies a quasi-optimal error estimate with constant proportional to $1+ (hk)^p \mathscr{R}$, where $\mathscr{R}$ is the $L^2\to L^2$ norm of the solution operator, and such that $\mathscr{R}\sim k$ for nontrapping problems. Thus, for $k$-oscillating data, the relative error of the $h$-FEM is controllably small when $(hk)^{2p} \mathscr{R}$ is small. This threshold for bounded relative error was famously identified for 1-d problems in \cite{IhBa:95a, IhBa:97}, with the associated bound first proved in \cite{DuWu:15} for constant-coefficient Helmholtz problems in 2-d and 3-d with an impedance boundary condition, and then proved for general Helmholtz problems in \cite{GS3}. Theorem \ref{thm:PG} is the \emph{full} analogue of these results, with $(hk)^{2p}\mathscr{R}$ replaced by $(hk)^{2p+2}\rho$. Moreover, \eqref{e:PPMR1} shows that the 
 in the limit $k\to \infty$ with $h$ chosen so that $(hk)^{p+1}\rho \to 0$, the Galerkin solution is asymptotically \emph{optimal}; i.e., $\|v-v_h\|_{\LtG}\to \|(I-P^G_{V_h})v\|_{\LtG}$.
\ere

For the collocation solution of~\eqref{e:basicForm} with piecewise polynomial spaces, we specialize to $d=2$ (this is for technical reasons related to the failure of the Sobolev embedding $H^1\to L^\infty$ when $d\geq 3$).
We also require an analogue of \blue{Assumption~\ref{ass:ppG}}.

\begin{assumption}[Assumption on piecewise-polynomial $V_h$ for the collocation method]\label{ass:ppC}

$d=2$, Parts (i) and (ii) of Assumption \ref{ass:ppG} hold and
%

(iii)
Given $p\geq 1$ 
and points $\{x_j\}_{j=1}^N\subset \Gamma$ that are unisolvent for $V_h$ 
there exists $C>0$
such that for all $0<h\leq h_0$ 
the collocation projection $P_V^C$ satisfies
for $0\leq q \leq 1$ and  
$1\leq t\leq p+1$,
\beq\label{e:assppC}
\big|(I-P_{V_h}^C)v\big|_{H^q(\Gamma)} \leq C h^{t-q} \N{v}_{H^t(\Gamma)}.
\eeq
\end{assumption}

\blue{The next result, and also later results in the paper,} use the $k$-weighted Sobolev space norms; i.e., the standard Sobolev space $H^s_k(\Gamma)$ but with each derivative of weighted by $k^{-1}$ (see \S\ref{sec:SC1}, \blue{\eqref{e:weightedNormGamma}} below for a precise definition); these weighted Sobolev spaces are the natural spaces in which to study functions oscillating with frequency $\blue{\sim k}$ (such as Helmholtz solutions). Note that $H^0_k(\Gamma)=\LtG$.

\begin{theorem}[Collocation method with piecewise polynomials]\label{thm:PC}
Suppose $\operator$ is given by one of~\eqref{e:DBIEs} or~\eqref{e:NBIEs} and Assumptions~\ref{ass:parameters} and~\ref{ass:polyboundintro} hold. Let 
$p\geq 1,$ 
$V_h\subset L^2(\Gamma)$ satisfy Assumption \ref{ass:ppC}, $P^C_{V_h}$ be the collocation projection for $\{x_j\}_{j=1}^{N_h}$, 
 and $k_0>0$. There exist $c,C>0$ such that if
\beq\label{e:Pthres3}
(hk)^{p+1}
\rho\leq c,\qquad hk\leq c
\eeq
and $k\geq k_0$, $k\notin \cJ$ then the collocation solution, $v_h$, to~\eqref{e:basicForm} exists, is unique, and satisfies the quasi-optimal error bound
\beq\label{e:PPMR4}
\N{v-v_h}_{H^1_k(\Gamma)}\leq \Big(\|I-P_{V_h}^C\|^2_{H^1_k(\Gamma)\to H^1_k(\Gamma)}+Chk\rho+Ck^{-1} \Big) \N{(I-P^C_{V_h})v}_{H^1_k(\Gamma)}.
\eeq
Moreover, if the right-hand side $f$ corresponds to plane-wave scattering (i.e., is given as in Theorem \ref{thm:BIEs}), then 
\beq\label{e:PPMR5}
\frac{\N{v-v_h}_{H^1_k(\Gamma)}}{\N{v}_{H^1_k(\Gamma)}} \leq C\Big(\|I-P_{V_h}^C\|^2_{H^1_k(\Gamma)\to H^1_k(\Gamma)}+Chk\rho+Ck^{-1}\Big)(hk)^p
\eeq
(i.e., the relative error can be made arbitrarily-small by decreasing $(hk)^{p+1}\rho$).
\end{theorem}

\bre[Implementing the operator product for the Neumann BIEs \eqref{e:NBIEs}]
Because of the operator product $S_{\ri k}H_k$ in $\Breg$, in practice, approximations to the solution of $\Breg v= f$ are \blue{computed} by applying the 
projection 
method, not to $\Breg v= f$, 
but to the system
\begin{align}\label{e:Neumann_system}
\left(
\begin{array}{cc}
\ri \eta_N (\half I - K) & S_{\ri k}\\
H_k & - 1
\end{array}
\right)
\left(
\begin{array}{c}
v\\
w
\end{array}
\right)
=
\left(
\begin{array}{c}
f \\
0
\end{array}
\right);
\end{align}
similarly for $\Bregp$.
For simplicity, when studying $\Breg$ and $\Bregp$ in this paper we consider the idealised situation of \eqref{e:projectionForm} with $P_V= P_V^G$ or $P_V = P_V^C$; i.e., we ignore the issue of discretising the operator product in Galerkin or collocation methods.
We emphasise, however, that we \emph{do} analyse discretising the operator product in the Nystr\"om method.
\ere


\paragraph{Comparison with previous results for the Galerkin method.}
Of the investigations \cite{BaSa:07, LoMe:11, GrLoMeSp:15, SpKaSm:15, GaMuSp:19, GaSp:22} into proving quasioptimality of the Galerkin method with piecewise polynomials applied to $A_k'$ and $A_k$, the best results are the following. The result of \cite[Lemma 3.1]{GaSp:22} that if $\Gamma$ is smooth and
$\rho (hk)$ is sufficiently small, then the Galerkin method is quasioptimal with constant proportional to $\rho$.
The  result of \cite[Theorem 3.17]{LoMe:11} that if $\Gamma$ is analytic then the Galerkin method applied to $A_k'$ is quasioptimal (with constant independent of $k$) if
$(hk)^{p+1} k^5 \rho$
is sufficiently small \cite[Equation 3.22]{LoMe:11}; a similar result holds for $A_k$ with $k^5$ replaced by $k^6$ \cite[Equation 3.26]{LoMe:11}. These quasi-optimality constants are larger than that in \eqref{e:PPMR1} and the thresholds for existence are more-restrictive than that in \eqref{e:Pthres}.

When $\Gamma$ is either a circle or sphere, $k$ is sufficiently large, and $\eta_N \geq C k^{1/3}$ with $C$ sufficiently large, then $\Breg$ and $\Bregp$ are coercive on $\LtG$ with constant independent of $k$ by \cite[Theorems 3.6 and 3.9]{BoTu:13}.
C\'ea's lemma then implies that the Galerkin method applied to these operators is quasioptimal, with 
quasioptimality constant $\gtrsim k^{1/3}$. This is in contrast to Theorem \ref{thm:PG} where the quasioptimality constant $\sim 1$, albeit under the thresholds \eqref{e:Pthres}. 
Note that the choice $\eta_N\sim k^{1/3}$ is \emph{not} used in practice, indeed, \cite[\S5]{BoTu:13} and \cite[Equation 23]{BrElTu:12} recommend using $\eta_N\sim 1$, stating that, out of all the possible choices of $\eta_N$, this gives a ``nearly optimal number''/```small number'' of iterations when GMRES is used to solve the Galerkin linear systems. 

\paragraph{Spaces satisfying Assumption~\ref{ass:ppG} and~\ref{ass:ppC}.}

Boundary element spaces satisfying one of Assumptions \ref{ass:ppG} and \ref{ass:ppC} are constructed when $d=3$ in  \cite[Chapter 4]{SaSc:11}.
Indeed, discontinuous boundary element spaces satisfying Assumption \ref{ass:ppG} are constructed in
\cite[Theorem 4.3.19]{SaSc:11}, and continuous boundary element spaces satisfying Assumption 
\ref{ass:ppC} are constructed in
\cite[Theorem 4.3.22(a)]%
{SaSc:11} 
(a subtlety is that \cite{SaSc:11} work in 3-d, but the constructions and arguments there establish \eqref{e:assppC} when $d=2$ in fact for $0\leq q<3/2$, $t>1/2$, and $q\leq t\leq p+1$).

\subsection{Pollution in the Galerkin method with piecewise polynomials}

\subsubsection{Quasimodes imply pollution}

Recall that \emph{quasimodes} for the exterior Dirichlet/Neumann Helmholtz problem are  functions $u$ in $\Omega^+$
such that, informally, $(-k^{-2}\Delta-1)u$ is ``small" and compactly supported. More precisely, 
\beqs
\|(-k^{-2}\Delta -1)u \|_{L^2(\Omega^+)} =o(k^{-1}) \|u\|_{L^2(B(0,R)\cap \Omega^+)}, \,\text{ either }\gamma^+ u=0 \text{ or }\partial_\nu^+ u=0 \ton \partial \Omega^+, 
\eeqs
and $u$ satisfies the Sommerfeld radiation condition \eqref{e:src}. The factor of $k^{-1}$ in the bound is natural, since $(-k^{-2}\Delta -1)(\chi e^{i k x\cdot \xi_0})=O(k^{-1})$ for any $k$-independent $\chi \in C_c^\infty$ and $\xi_0$ with $|\xi_0|=1$; thus functions $u$ such that $(-k^{-2}\Delta-1)u=O(k^{-1})$ can be constructed for any $\Omega^-$.
A quasimode (in the sense above) exists if and only if the obstacle is trapping \cite{BoBuRa:10}, \cite[Theorem 7.1]{DyZw:19}.

In the BIE context, a quasimode is a function $v\in \LtG$ such that $\|\operator v\|_{\LtG} =o(1) \|v\|_{\LtG}$. For the Dirichlet BIEs (with $\eta_D$ satisfying Assumption \ref{ass:parameters}), quasimodes exists if and only if $\Omega^-$ is trapping. However, for the Neumann BIEs, \blue{perhaps surprisingly}, quasimodes exist when $\Omega^-$ is the unit ball, since $\| \operator^{-1}\|_{\LtGt} \sim k^{1/3}$ in this case (see, e.g,. \cite[Theorem 2.3]{GaMaSp:21N}). 

Our results showing the existence of pollution are based on the following theorem, which states, roughly, that the existence of a sufficiently good quasimode for the BIE that oscillates at frequency $\sim k$ guarantees pollution. To state this carefully, we let $\Delta_{\Gamma}$ denote the surface Laplacian on $\Gamma$. 
\begin{theorem}[Quasimodes imply pollution]
\label{t:pollutionIntro}
Suppose that  Assumption \ref{ass:parameters} holds.
Let $\Xi_0>1$, $0<\e<\frac{1}{10}$,  $0<\e_0<\Xi_0$, and $\cJ\subset (0,\infty)$ unbounded such that Assumption~\ref{ass:polyboundintro} holds and $V_h\subset L^2(\Gamma)$ satisfy Assumption~\ref{ass:ppG}, $p\geq 0$. Then there is $k_0>0$ such that  for all $\chi \in C_c^\infty((-(1+\e)^2, (1+\e)^2))$ with $\chi \equiv 1$ on $[-1,1]$ there is $c>0$ such that if there are $k_n\to \infty$, $k_0<k_n\notin \cJ$,  $u_n,f_n\in L^2(\Gamma)$, $\beta_n,\alpha_n\in (0,1]$ such that 
$$
\beta_n\|\operator^{-1}(k_n)\|_{\LtGt}\leq c,
$$
and
\begin{gather}\operator(k_n) u_n=f_n,\qquad \|f_n\|_{\LtG}\leq \alpha_n, \qquad \|u_n\|_{\LtG}=1,\label{e:quasiTime}\\
\| (1-\chi(\Xi_0^{-2}k_n^{-2}\Delta_{\Gamma}))f_n\|_{\LtG}\leq \beta_n\|f_n\|_{\LtG},\qquad \|\chi(\e^{-2}_0k_n^{-2}\Delta_{\Gamma})f_n\|_{\LtG}\leq\beta_n\|f_n\|_{\LtG},\label{e:oscillating}\\
\text{ either } \quad p=0 \quad \text{ or }\quad k^{-p-1}\alpha_n\|1_{(0,k_0^2]}(-\Delta_{\Gamma})(\operator^*)^{-1}u_n\|_{\LtG}\leq c,\label{e:normalOk1}
 \end{gather}
then, for any $h$ and $n$ such that and $(I+P^G_h\pert(k_n))$ has an inverse,
\begin{equation}
\label{e:thePollution}
\|(I+P^G_{V_h}\pert(k_n))^{-1}(I-P^G_{V_h})\|_{\LtG\to \LtG}\geq \frac{1}{2}+ c \begin{cases} (hk_n)^{-p-1}, &  1\leq (hk_n)^{2(p+1)}\alpha_n^{-1}\,\\
(hk_n)^{p+1}\alpha_n^{-1},&\alpha_n^{-1}(hk_n)^{2(p+1)}\leq 1.
\end{cases}
\end{equation}
\end{theorem}

We make the following immediate remarks about Theorem~\ref{t:pollutionIntro}:
\begin{itemize}
\item The assumptions~\eqref{e:quasiTime} and~\eqref{e:oscillating} are the precise versions of, respectively,  the statement that $u_n$ is a sufficiently good quasimode for $\operator$, and that $f_n$ oscillates at frequency $\sim k_n$. 
\item When $p\geq 1$, the assumption~\eqref{e:normalOk1} guarantees that $u_n$ is not close to any solution of $\operator v=\tilde{f}$ where $\tilde{f}$ is in the span of finitely many low frequency eigenfunctions of the surface Laplacian. This finite-dimensional assumption is required for technical reasons when $p\geq 1$ and we expect the theorem to hold without it even in that case.  In the concrete examples below (see Theorems~\ref{t:diamondNew} and~\ref{t:qualitative1}) the second part of~\eqref{e:normalOk1} imposes an extra restriction on growth of the solution operator when $p\geq 1$. Since the space of obstructions is finite-dimensional, this extra assumption could also be avoided by constructing sufficiently many good quasimodes.
\item If $\alpha_n \leq C\rho^{-1}$, i.e. $f_n$ achieves the full growth of the solution operator, then  Theorem~\ref{t:pollutionIntro} implies that
\begin{equation}
\label{e:itIsOptimal}
\|(I+P_{V_h}^G\pert)^{-1}(I-P_{V_h}^G)\|_{\LtG\to \LtG}\geq \frac 12 + c \begin{cases} (hk)^{-(p+1)}, &  1\leq (hk)^{2(p+1)}\rho\,\\
(hk)^{p+1}\rho,&\begin{aligned}
&(hk)^{2(p+1)}\rho\leq c.
\end{aligned}
\end{cases}
\end{equation}
In particular, if~\eqref{e:itIsOptimal} holds, then Theorem~\ref{thm:PG} is optimal since the Galerkin solution, $v_h$ to $\operator v=f$ satisfies
\beqs
v-v_h= 
\blue{(I+P_{V_h}^G \pert)^{-1}(I-P^G_{V_h}) v
=}
(I+P_{V_h}^G \pert)^{-1}(I-P^G_{V_h})(I-P^G_{V_h}) v
\eeqs
\blue{(see Lemma \ref{lem:QOabs} below).}
\item  By Definition~\ref{d:pollution}, \eqref{e:thePollution} with $\alpha=o(1)$ implies that the $h$-BEM suffers from the pollution effect. In particular, if $hk=\e$, then~\eqref{e:thePollution} implies that the quasioptimality constant is \emph{at least} $c\e^{-p-1}$ and hence that decreasing $\e$ actually makes the quasioptimality constant worse until the point that 
$\e^{2(p+1)} \alpha^{-1}\lesssim1$; i.e., $\e \lesssim \alpha^{1/2(p+1)}.$

\item The proof of Theorem \ref{t:pollutionIntro} is sketched in \S\ref{s:discussion}; the key technical ingredient is the lower bound from \cite{Ga:25} on how well 
piecewise polynomials 
approximate a function with frequency $\sim k$.
\end{itemize}


While it is relatively straightforward to show that a strong quasimode for the Helmholtz equation implies the existence of a strong quasimode for the BIEs
\cite{BeChGrLaLi:11}, \cite[\S5.6.2]{ChGrLaSp:12} (for the Dirichlet BIEs) and \cite[Theorem 2.6]{GaMaSp:21N} (for the Neumann BIEs), it is more challenging to determine the properties of the BIE quasimode from the properties of the quasimode of the Helmholtz equation; \S\ref{s:quasimodes} below does this in the Dirichlet case, resulting in concrete examples of pollution for the Dirichlet BIEs (Theorems \ref{t:diamond}, \ref{t:diamondNew}, \ref{t:qualitative1}). We expect that the results of \S\ref{s:quasimodes} can be extended to the Neumann case; we do not pursue this here \blue{because the proof, although phenomenologically very similar to that for the Dirichlet problem, its proof would require slightly different (long) technical lemmas}. \blue{Instead we show, the perhaps even more surprising result}, that pollution occurs for the Neumann BIEs for the unit disk \blue{(i.e. the simplest possible scatterer)}. 

\subsubsection{Pollution for the Dirichlet BIEs}

Throughout this section, $V_h\subset L^2(\Gamma)$ satisfies Assumption~\ref{ass:ppG} for some $p\geq 0$.
We start with an example that shows \emph{quantitative} pollution for the Dirichlet BIEs. 

\begin{definition}[Four-diamond domain]
We say that $\Omega^-$ is a \emph{four-diamond domain} if there exists 
$0<\e<\pi/2$ such that $\Omega^-=\cup_{i=1}^4\Omega_i\Subset \mathbb{R}^2$,
where 
\blue{$\{\Omega_i\}_{i=1}^4$ are disjoint,}  each 
$\Omega_i$ is open, convex and has smooth boundary,
and
\begin{gather*}
\partial \big((-\pi,\pi)\times(-\pi, \pi)\big)\setminus \big(\cup_{a,b\in\pm 1}B((a \pi,b \pi),\e)\big)\subset \Gamma\subset \mathbb{R}^2\setminus \big(\cup_{a,b\in\pm 1}B((a \pi,b \pi),\e/2)\big),\\ 
\Omega^-\cap (-\pi,\pi)\times (-\pi,\pi)=\emptyset.
\end{gather*}
\end{definition}

Figure \ref{f:bReallyCoolPicture} shows an example of a four-diamond domain.

\begin{theorem}[Pollution for the four-diamond domain when $p=0$]
\label{t:diamond}
Let $p=0$, $\Omega^-$ be a four-diamond domain, $\eta_D$ satisfy Assumption~\ref{ass:parameters}, $\cJ$ satisfy Assumption~\ref{ass:polyboundintro} with $\operator=A_{k}$ or $A_k'$, and $k_n:=\sqrt{2}n$. Then there is $c>0$, such that for all $k\in \mathbb{R}\setminus \cJ$ and $n$ such that $k_n^{-1}\leq \delta \leq 1$, $|k-k_n|<\delta^{-1} k_n^{-1}$, and all $h>0$ there is $f\in L^2(\Gamma)$ such that the Galerkin solution, $v_h$, to~\eqref{e:basicForm} with $\operator$ given by either $A_{k}$ or $A_{k}'$, if it exists, satisfies
$$
\frac{\|v_h-v\|_{\LtG}}{\|(I-P_{V_h}^G)v\|_{\LtG}}\geq 
\frac 12 + 
c \begin{cases} (hk_n)^{-1}, &  1\leq (hk_n)^{2} \delta k_n,\\
(hk_n)\delta k_n,&\delta k_n(hk_n)^{2}\leq 1.
\end{cases}
$$
\end{theorem}

\begin{theorem}[Pollution for the four-diamond domain when $p\geq 1$]
\label{t:diamondNew}
Let $\Omega^-$ be a four-diamond domain, $\eta_D$ satisfy Assumption~\ref{ass:parameters},
$\operator=A_{k}$ or $A_k'$, and $k_n:=\sqrt{2}n$. 

If there are $C_1>0, k_0>0$ such that 
\beq\label{e:fourDiamondsBound}
\big\|\operator^{-1}\big\|_{\LtGt} \leq C_1 k^2 \quad \tfa k\geq k_0
\eeq
there are $c,C>0$ such that for all $k\geq k_0$ and $n$ such that $Ck_n^{-1}\leq \delta \leq 1$ and $|k-k_n|<\delta^{-1} k_n^{-1}$ and for all $h>0$ there is $f\in L^2(\Gamma)$ such that the Galerkin solution, $v_h$, to~\eqref{e:basicForm} with $\operator$ given by either $A_{k}$ or $A_{k}'$, if it exists, satisfies
$$
\frac{\|v_h-v\|_{\LtG}}{\|(I-P_{V_h}^G)v\|_{\LtG}}\geq \frac 12 +  c \begin{cases} (hk_n)^{-p-1}, &  1\leq (hk_n)^{2(p+1)} \delta k_n,\\
(hk_n)^{p+1}\delta k_n,&\delta k_n(hk_n)^{2(p+1)}\leq 1.
\end{cases}
$$
\end{theorem}

\begin{remark}[Discussion of Theorems \ref{t:diamond} and \ref{t:diamondNew}]

\

\begin{enumerate}
\item 
When $\Omega^-$ consists of two aligned squares (or rounded squares) -- where the trapping is the same nature (\emph{parabolic}) as for a four-diamond domain (albeit on a smaller set in phase space) -- 
the bound \eqref{e:fourDiamondsBound} is proved in \cite[Corollary 1.14]{ChSpGiSm:20}. Furthermore, from both quasimode considerations and 
the bounds of 
\cite{ChWu:13, Ch:18} in the setting of scattering by metrics of revolution, 
\cite{ChSpGiSm:20}
conjecture that for $\Omega^-$ as in Theorem~\ref{t:diamond}, 
\beq\label{e:thisMustBeTrue}
\big\|A_{k}^{-1}\big\|_{\LtGt}+\big\|(A'_{k})^{-1}\big\|_{\LtGt}\leq Ck.
\eeq
If the bound \eqref{e:thisMustBeTrue} holds, then Theorems~\ref{t:diamond} and \ref{t:diamondNew} show that Theorem~\ref{thm:PG} is optimal in this case and $\cJ=\emptyset$. 

\item 
We emphasise that the phenomenon behind Theorems~\ref{t:diamond} and \ref{t:diamondNew} is \emph{not} sparse in frequency. Indeed, 
both Theorems~\ref{t:diamond} and \ref{t:diamondNew}
hold with $\delta=1$ for a set of $k$ with infinite measure (since $\sum_{n=1}^\infty n^{-1}=\infty$). Furthermore, although 
we prove Theorem~\ref{t:diamond} with $k_n=\sqrt{2}n$, it is easy to generalize our construction in the proof of Theorem~\ref{t:diamond} to take $k_n=\sqrt{m_n^2+\ell_n^2}$ for any $m_n,\ell_n\in\mathbb{Z}$ with $c<|\frac{m_n}{\ell_n}|<C$ and such that, if $m_n/\ell_n=p_n/q_n$ with $p_n$ and $q_n$ relatively prime, then $|q_n|\leq C$. 
\item 
Theorems~\ref{t:diamond} and \ref{t:diamondNew} are
 illustrated by numerical experiments in Figures~\ref{f:bReallyCoolPicture} and~\ref{f:aReallyCoolPicture},
with Figures~\ref{f:cavityReallyCoolPicture} and \ref{f:cavityPollution} illustrating pollution for a cavity domain.
 (The set up for all the numerical experiments in this section is described in \S\ref{sec:numerical}.)
\end{enumerate}
\end{remark}

Next, we show that in a wide variety of trapping problems, one has pollution \emph{at all frequencies}. To avoid technicalities, we informally define the 
trapped set, $K$, as
set of points $x$ and directions $\xi$ such that the billiard trajectory through $(x,\xi)$  in $\Omega^+$ remains in a compact set 
for all time. (We refer the reader to \S\ref{s:dynamics} for the careful definitions). 
\begin{theorem}[Qualitative pollution for a large class of trapping domains]
\label{t:qualitative1}
Suppose that Assumptions~\ref{ass:parameters} and~\ref{ass:polyboundintro} hold, that $K\neq \emptyset$ and that for every point $x$ and direction $\xi$ such that $(x,\xi)\in K$, the line
$$
\mathscr{L}_{(x,\xi)}:=x+\mathbb{R}\xi
$$
is not tangent to $\Gamma$.
Let $C_1,k_0>0$. Suppose that 
\beq\label{e:iWantALabel}
\text{ either } \quad p=0 \quad \text{ or } \quad p\geq 1 \,\, \tand\,\,  \big\| A_{k}^{-1}\big\|_{\LtGt} \leq C_1 k^{p+1} \tfa k \geq k_0.
\eeq
 Then for all $k_n\to \infty$ $k_n\notin \cJ$,
and $h>0$ there are $f_n\in L^2(\Gamma)$ and $\alpha_n\to 0$ such that the Galerkin solution, $v_h$, to~\eqref{e:basicForm} with $\operator$ given by $A_{k_n}$, if it exists, satisfies
\begin{equation}
\label{e:qualPollute}
\frac{\|v_h-v\|_{\LtG}}{\|(I-P_{V_h}^G)v\|_{\LtG}}\geq \frac 12 + c \begin{cases} (hk_n)^{-p-1}, &  1\leq (hk_n)^{2(p+1)}\alpha_n^{-1},\\
(hk_n)^{p+1}\alpha_n^{-1},&\alpha_n^{-1}(hk_n)^{2(p+1)}\leq 1.
\end{cases}
\end{equation}
\end{theorem}
The assumptions in Theorem~\ref{t:qualitative1} are satisfied, for example, by two aligned squares with rounded corners (since \eqref{e:iWantALabel} holds by the bound \eqref{e:fourDiamondsBound} proved in \cite{ChSpGiSm:20}).

There is an analogous result for $A_{k}'$, but it requires slightly stronger hypotheses that are harder state in a non-technical way so we postpone it to Theorem~\ref{t:qualitativeAp}. 

\subsubsection{Pollution for the Dirichlet and Neumann BIEs on the unit disk}

\begin{theorem}[Pollution for the Neumann BIEs on the unit disk]
\label{t:neumannDisk}
Let $\Omega_-= B(0,1)\subset \mathbb{R}^2$.
For all $C>0$, $k_n\to \infty$,  and  $C^{-1}\leq \eta_N\leq C$ there are $f_n\in L^2(\Gamma)$ such that the Galerkin solution, $v_h$, to~\eqref{e:basicForm} with $\operator$ given by $B_{k_n,\rm{reg}}$ or $B_{k_n,\rm{reg}}'$ satisfies
\begin{equation*}
\frac{\|v_h-v\|_{\LtG}}{\|(I-P_{V_h}^G)v\|_{\LtG}}\geq 
\frac 12 + 
c \begin{cases} (hk_n)^{-p-1}, &  1\leq (hk_n)^{2(p+1)}k_n^{\frac{1}{3}},\\
(hk_n)^{p+1}k_n^{\frac{1}{3}},&k_n^{\frac{1}{3}}(hk_n)^{2(p+1)}\leq 1.
\end{cases}
\end{equation*}
In particular, Theorem~\ref{thm:PG} is optimal in this case.
\end{theorem}
\begin{remark}
We expect that the more sophisticated microlocal description of boundary layer operators in~\cite{Ga:19} can be used to extend Theorem~\ref{t:neumannDisk} 
to any domain with smooth boundary in any dimension \blue{but do not pursue this extension here}.
\end{remark}

Next, for the Dirichlet problem, we demonstrate pollution for any $k^{-N}\leq \eta_D=o(k)$. 
\begin{theorem}\mythmname{Pollution for the Dirichlet BIEs on the unit disk with non-standard coupling parameters}
\label{t:dirichletDisk}
Let $\Omega_-= B(0,1)\subset \mathbb{R}^2$. 
There is $k_n\to \infty$, such that for all $N>0$, $C>0$, $k^{-N}\leq \eta_D\leq Ck$  there are $f_n\in L^2(\Gamma)$ such that the Galerkin solution, $v_h$, to~\eqref{e:basicForm} with $\operator$ given by $A_{k_n}$ or $A_{k_n}'$ satisfies
\begin{equation*}
\frac{\|v_h-v\|_{\LtG}}{\|(I-P_{V_h}^G)v\|_{\LtG}}\geq 
\frac 12 + 
c \begin{cases} (hk_n)^{-p-1}, &  1\leq (hk_n)^{2(p+1)}k_n\eta_n^{-1},\\
(hk_n)^{p+1}k_n\eta_n^{-1},&k_n\eta_n^{-1}(hk_n)^{2(p+1)}\leq 1.
\end{cases}
\end{equation*}
In particular, Theorem~\ref{thm:PG} is optimal in this case.
\end{theorem}

\begin{remark}
The specific examples in Theorems~\ref{t:diamond},~\ref{t:qualitative1},~\ref{t:neumannDisk}, and~\ref{t:dirichletDisk} show that Theorem~\ref{t:pollutionIntro} is effective in demonstrating pollution in concrete situations. Moreover, although it is difficult to prove a general statement for all domains simultaneously, given a trapping domain it is usually possible to construct quasimodes that, together with Theorem~\ref{t:pollutionIntro} demonstrate pollution.
\end{remark}

\bre
When $\Gamma$ is the sphere and $\eta = k^{-2/3}$ \cite{BaSa:07} 
conjecture that the Galerkin method applied to $A_k$ and $A_k'$ with piecewise polynomials does not suffer from the pollution effect.
More precisely, \cite[Proposition 3.14]{BaSa:07} show that if a specific bound on a combination of Bessel and Hankel functions holds (\cite[Equation 3.20]{BaSa:07}) then the Galerkin solution is k-uniformly quasioptimal when hk is sufficiently small. 

In contrast, Theorem \ref{t:dirichletDisk} proves that when $\eta=k^{-2/3}$, the Galerkin method applied to $A_k$ and $A_k'$ with piecewise polynomials \emph{does} suffer from the pollution effect. Thus the analogue of the bound
\cite[Equation 3.20]{BaSa:07} for the disk cannot hold.
\ere

\bre
\blue{The recent paper \cite{GiMeAn:25} 
considers Maxwell BIEs on the circular cylinder, which reduce to Helmholtz BIEs on the disk under transverse magnetic (TM)/tranverse electric (TE) polarization. The paper \cite{GiMeAn:25} analyses the spectrum of the piecewise-linear discretisations of these BIOs, using asymptotics of Bessel and Hankel functions. \cite{GiMeAn:25} gives evidence for pollution by finding the asymptotics of the eigenvalues 
\cite[\S IIIB]{GiMeAn:25} and then arguing heuristically about aliasing errors \cite[\S IIIC]{GiMeAn:25}. These upper bounds are then compared to numerical experiments, which indicate pollution \cite[\S VII]{GiMeAn:25}.
}
\ere

\begin{figure}[htbp]
\begin{center}
Piecewise constant elements

\begin{tikzpicture}
\node at(0,0){\includegraphics[width=.8\textwidth]{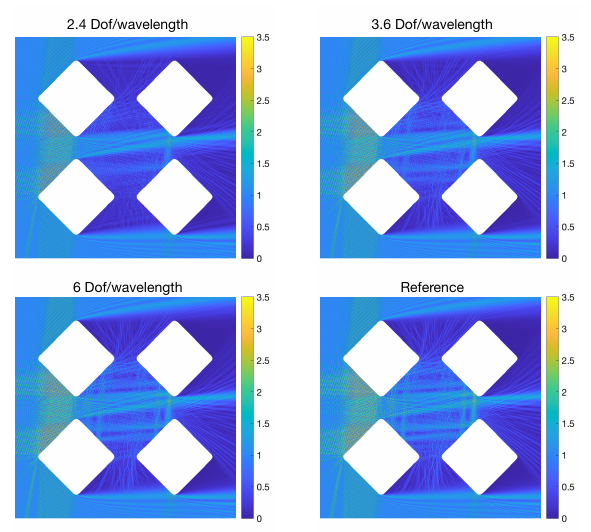}};
\end{tikzpicture}
\end{center}
\caption{\label{f:bReallyCoolPicture} 
The plots depict the absolute value of the total field at $k=40\sqrt{2}$ when the plane wave $e^{ik\langle \omega,x\rangle}$, with $\omega=(\cos (5\pi/180), \sin(5\pi/180))$, is incident on a sound-soft domain consisting of four nearly square obstacles.
The plot shows the Galerkin solutions for the indirect BIE at 2.4 (top left), 3.6 (top right), and 6 (bottom left) points per wavelength and piecewise constant elements (i.e., $p=0$), with the reference solution (bottom right) computed with $p=11$.
These numbers of points per wavelength correspond to points before, on, and after the peak of the quasioptimality constant in Figure~\ref{f:aReallyCoolPicture}. Many of the features of the Galerkin and reference plots are similar. However, in the Galerkin solutions with low numbers of points per wavelength the trapped rays are much less pronounced. 
The plots for piecewise linear elements (i.e., $p=1$) look qualitatively similar.}
\end{figure}

\begin{figure}[htbp]
\begin{center}
\begin{tikzpicture}
\node at(0,0){\includegraphics[width=1\textwidth]{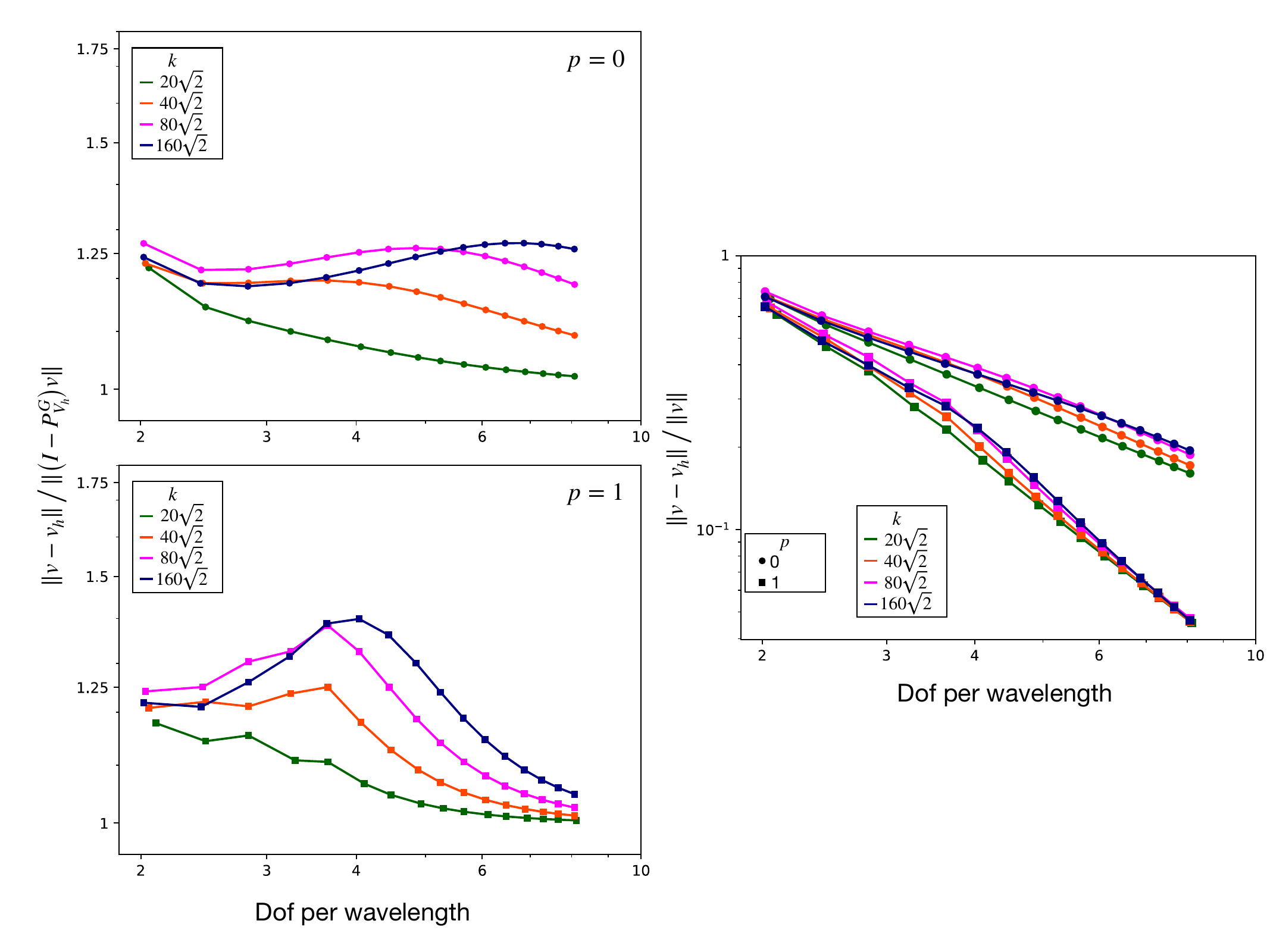}};
\end{tikzpicture}
\end{center}
\caption{\label{f:aReallyCoolPicture} The left plots shows the implied quasioptimality constant for piecewise constant (top) and piecewise linear (bottom) elements when the plane wave $e^{ik\langle \omega,x\rangle}$ with $\omega=(\cos (5\pi/180), \sin(5\pi/180))$  is incident on a sound-soft domain consisting of four nearly square obstacles (see Figure~\ref{f:bReallyCoolPicture}), with this problem solved using the indirect BIE. The right plot shows the corresponding $L^2$ relative error in the density. The reference and computed scattered solutions at $k=40\sqrt{2}$ are shown in Figure~\ref{f:bReallyCoolPicture}. The presence of pollution in this example can be seen from both plots:~in the left hand plot, one observes a peak in the implied quasioptimality constant that shifts up and to the right as $k$ increases. In the right hand plot, the error increases for a fixed number of points per wavelength as $k$ increases. }
\end{figure}

%

\begin{figure}[htbp]
\begin{center}
Piecewise constant elements

\begin{tikzpicture}
\node at(0,0){\includegraphics[width=.8\textwidth]{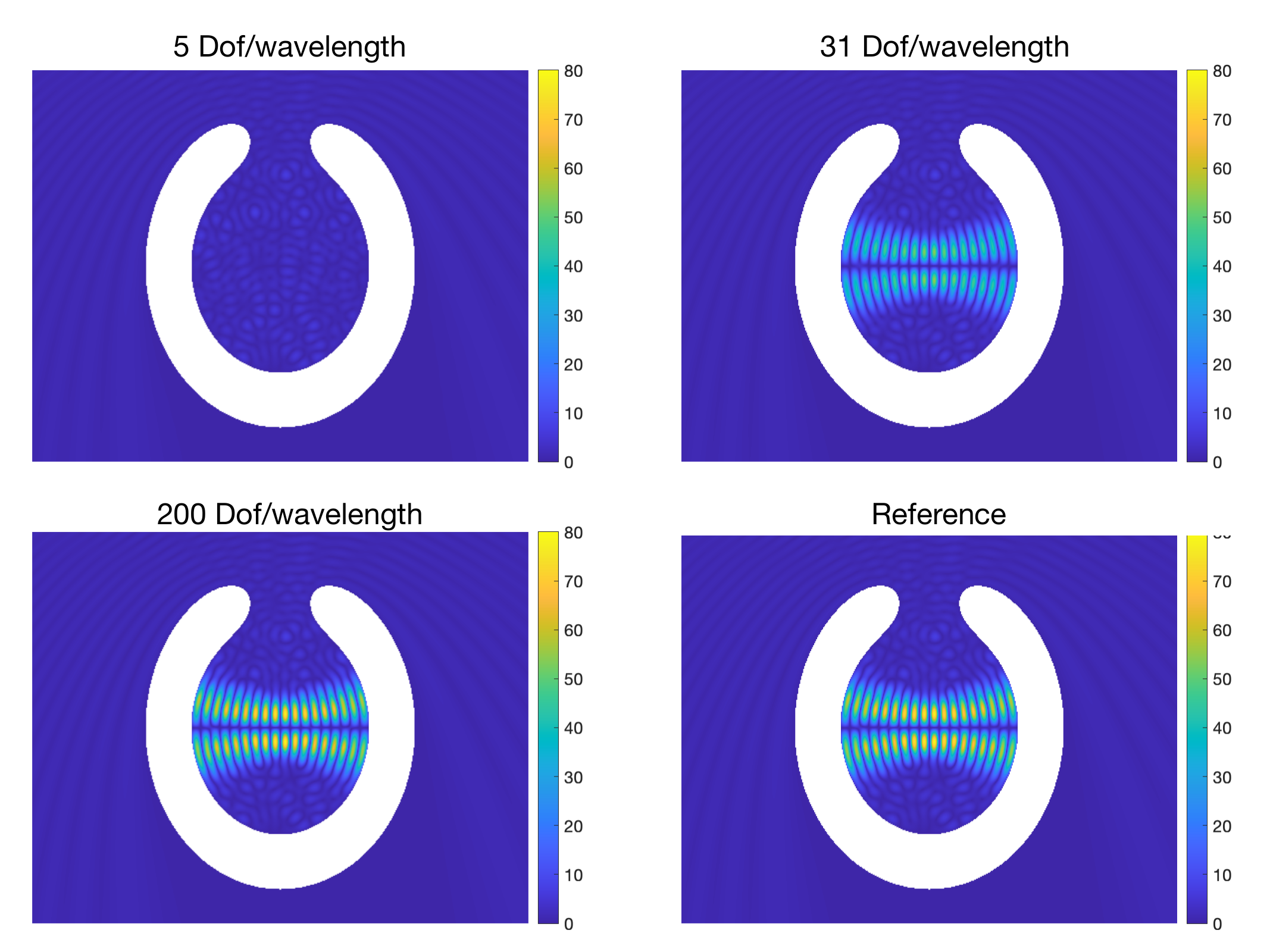}};
\end{tikzpicture}
\bigskip

Piecewise linear elements

\begin{tikzpicture}
\node at(0,0){\includegraphics[width=.8\textwidth]{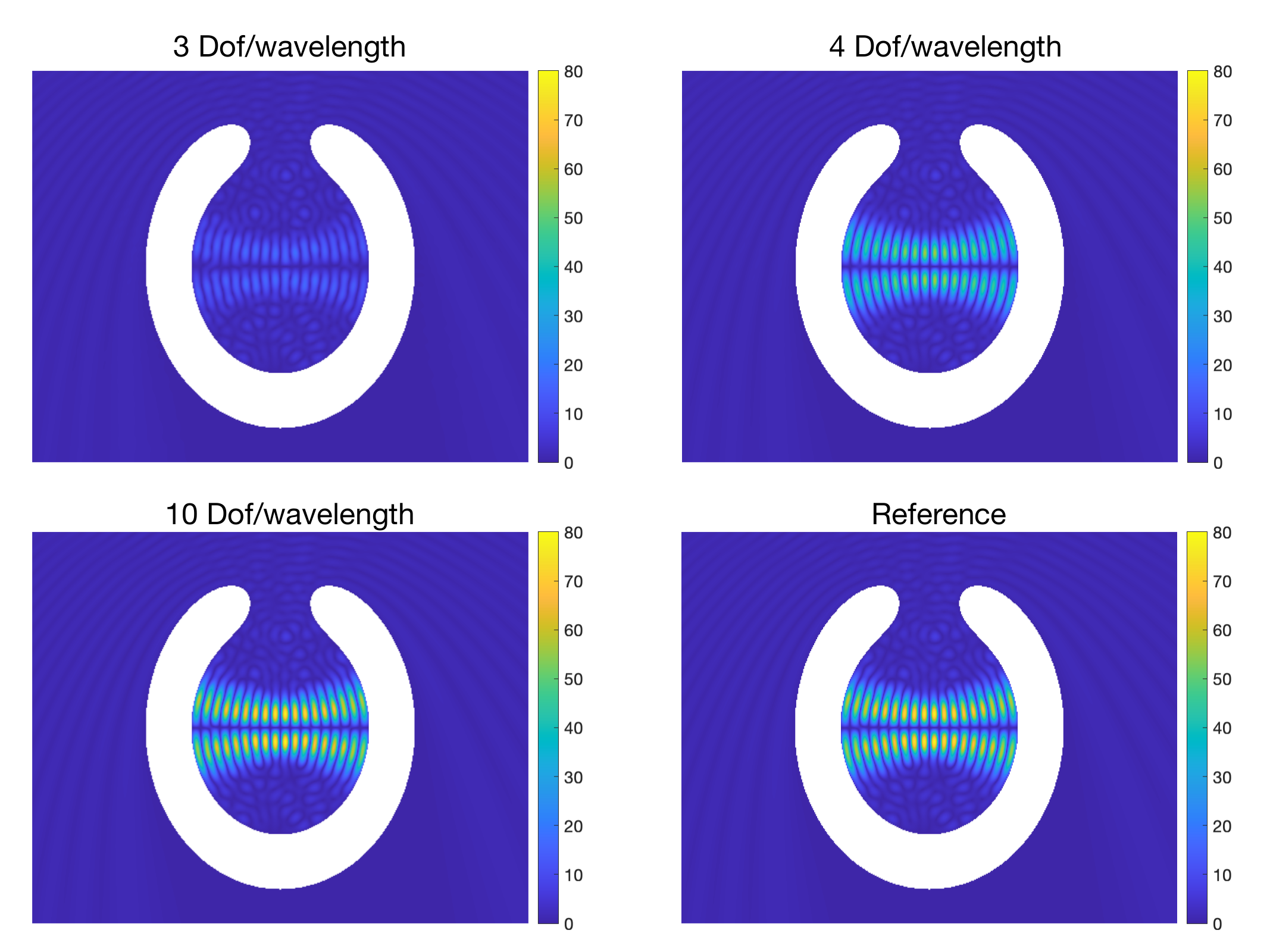}};
\end{tikzpicture}
\end{center}
\caption{\label{f:cavityReallyCoolPicture} The plots show the absolute value of the total field for the reference solution (bottom right) at $k=37.213$ when the plane wave $e^{ik\langle \omega,x\rangle}$ with $\omega=(\cos (-\pi/2+0.2), \sin(-\pi/2+0.2))$ is incident on a sound-soft cavity as well as the corresponding Galerkin solutions (for the indirect BIE) at various numbers of points per wavelength with both piecewise constant (top) and piecewise linear (bottom) elements. These numbers of points per wavelength correspond to points before, on, and after the peak of the quasioptimality constant in Figure~\ref{f:cavityPollution}. At low numbers of points per wavelength the Galerkin solutions fail almost entirely to resolve the behaviour of the total field.}
\end{figure}

\begin{figure}[htbp]
\begin{center}
\begin{tikzpicture}
\node at(0,0){\includegraphics[width=1\textwidth]{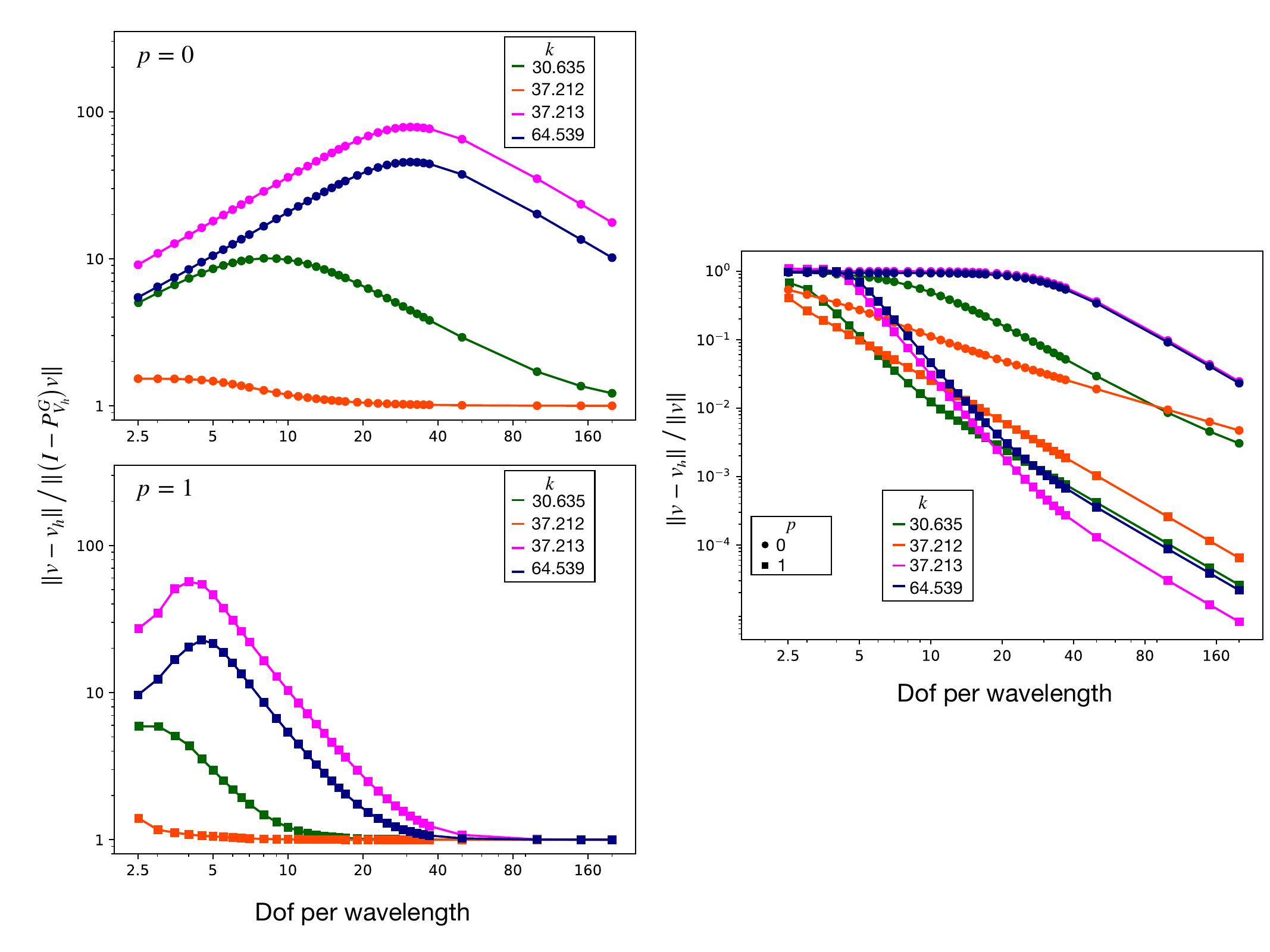}};
\end{tikzpicture}
\end{center}
\caption{\label{f:cavityPollution}The left plots shows the implied quasioptimality constant for piecewise constant (top) and piecewise linear (bottom) elements when the plane wave $e^{ik\langle \omega,x\rangle}$ with $\omega=(\cos (-\pi/2+0.2), \sin(-\pi/2+0.2))$  is incident on a sound-soft cavity domain (see Figure~\ref{f:cavityReallyCoolPicture}),  with this problem solved using the indirect BIE. The right plot shows the corresponding $L^2$ relative error in the density. 
As in Figure \ref{f:aReallyCoolPicture}, the pollution effect
can be seen in both plots.
While the magnitude of the effect can depend very strongly on the frequency (compare the curves at $k=37.212$ and $k=37.213$), the qualitative behavior is much less sensitive to the frequency chosen.}
\end{figure}

\subsection{Sufficient conditions for quasi-optimality with trigonometric-polynomial approximation spaces}

Let $d=2$ and assume that $\Oi$ is connected (the results can be extended to multiply-connected domains in a straightforward way).
Without loss of generality, assume that $\Gamma$ is given by $x=\gamma(t)$ for $\gamma: \mathbb{R}/2\pi \mathbb{Z} \to \mathbb{R}^2$ a smooth curve with $0<|\dot\gamma(t)| \leq c_{\max}$.
The subspace $\trig_N$ is defined in the $t$-variable by
\beq\label{e:V_N_trig}
\trig_N := {\rm span} \big\{ \exp(\ri mt) : |m|\leq N\big\};
\eeq
note that the dimension of $\trig_N$ is then $2N+1$.
We consider the collocation projection $P_{\trig_N}^C$ with evenly-spaced points in the $t$ variable; i.e.,
\beq\label{e:trig_points}
t_j = 2 \pi j/(2N+1), \quad j=0,1, \ldots, 2N.
\eeq
We assume that there are an odd number of points since the explicit expression for $P_{\trig_N}^Cv$ is simpler in this case (see \eqref{e:PNC}), but the results below also hold when there are an even number of points. In both cases the points are unisolvent by, e.g., \cite[\S11.3]{Kr:14}, \cite[\S3.2.2]{At:97}.

In the case of trigonometric polynomials, we change the $L^2$ norm used to define $P_{\trig_N}^G$. In particular, rather than the measure induced from $\mathbb{R}^2$, we use in~\eqref{e:Galerkin_def} the norm
\beq\label{e:newNormGamma}
\|u\|_{L^2(\Gamma)}^2:=\int_0^{2\pi} |(u\circ\gamma)(t)|^2dt.
\eeq

\begin{theorem}[Galerkin method with trigonometric polynomials]\label{thm:FG}
Suppose $d=2$, $\operator$ is given by one of~\eqref{e:DBIEs} or~\eqref{e:NBIEs},  and Assumptions~\ref{ass:parameters} and~\ref{ass:polyboundintro} hold.  Let $V_N$ be given by \eqref{e:V_N_trig}, $k_0>0$, and $\e>0$. Then there exist $C_1,C_2>0$, such that if $\Xi\geq 1$,
\beq\label{e:FGthres}
N\geq (\Xi c_{\max}+\e) k
\eeq
and $k\geq k_0$, $k\notin \cJ$, then the Galerkin solution, $v_N$ to~\eqref{e:basicForm} exists, is unique, and satisfies the quasi-optimal error bound
\beq\label{e:MR1}
\N{v-v_N}_\LtG\leq \big( 1+C_1 \Xi^{-1} +C_2k^{-1}\big) \N{(I-P_{\trig_N}^G)v}_{\LtG}.
\eeq
Moreover, if the right-hand side $f$ corresponds to plane-wave scattering (i.e., is given as in Theorem \ref{thm:BIEs}), then
\beq\label{e:MR2}
\frac{\N{v-v_N}_\LtG}{\N{v}_{\LtG}} = O(k^{-\infty}).
\eeq
\end{theorem}

\bre
We use the notation that $a= O(k^{-\infty})$ if $a$ decays faster than any algebraic power of $k$; i.e.,
given $k_0, N>0$, there exists $C(N,k_0)$ such that
\beqs
|a| \leq C(N,k_0) \,k^{-N} \quad \tfa k\geq k_0.
\eeqs
\ere

As in the case of piecewise polynomials, our estimates are weaker for collocation than Galerkin.
\begin{theorem}[Collocation method with trigonometric polynomials]\label{thm:FC}
Suppose $d=2$, $\operator$ is given by one of~\eqref{e:DBIEs} or~\eqref{e:NBIEs},  and Assumptions~\ref{ass:parameters} and~\ref{ass:polyboundintro} hold.  Let $V_N$ be given by \eqref{e:V_N_trig}, $P_{\trig_N}^C$ be the collocation projection (with $2N+1$ evenly-spaced points, as described above), $s>1/2$, and $k_0>0$. Then there exists $C>0$ such that if
\beq\label{e:FCthres}
N\geq (c_{\max} + \epsilon)k
\eeq
and $k\geq k_0$, $k\notin \cJ$, then the collocation solution, $v_N$, of~\eqref{e:basicForm} exists, is unique, and satisfies the quasi-optimal error bound
\beq\label{e:MR4}
\N{v-v_N}_{\Hsk}\leq \Big(\|I-P^C_{\trig_N}\|_{\Hshts}+C(k/N)+Ck^{-1}+C\rho\Big(\frac{k}{N}\Big)^{s}\Big) \N{(I-P_{\trig_N}^C)v}_{\Hsk}.
\eeq
Moreover, if the right-hand side $f$ corresponds to plane-wave scattering (i.e., is given as in Theorem \ref{thm:BIEs}), then
\beq\label{e:MR5}
\frac{\N{v-v_N}_\Hsk}{\N{v}_{\Hsk}} = O(k^{-\infty}).
\eeq

\end{theorem}


\bre[The Nyquist sampling rate]
When $\Gamma$ has length $2\pi$, the number of degrees of freedom per wavelength for $\trig_N$ \eqref{e:V_N_trig} equals $(2N+1)/k$. 
The Nyquist--Shannon--Whittaker sampling theorem then indicates that one requires $N\geq k-1/2$ to recover a function of frequency $\leq k$ using this many degrees of freedom. When $c_{\max}=1$, \eqref{e:FGthres} and \eqref{e:FCthres} become $N\geq (1+\e)k$; i.e., the Nyquist sampling rate is asymptotically sufficient to obtain existence of the Galerkin and collocation solutions.
\ere

%

\subsection{Results for the Nystr\"om method with Kress quadrature}

The quadrature we use for the Nystr\"om method is based on trigonometric polynomials in $2d$ and is described in detail in \S~\ref{sec:Nystrom}. This quadrature falls under the class of quadrature methods
introduced by \cite{Ma:63, Ku:69} and often referred to as ``Kress quadrature'' due to its use by Kress in \cite{Kr:91, Kr:95}.

We first consider the Dirichlet problem with the standard splitting of $kS_k$, $\DL_k$, $\DL_k'$ described in Lemma~\ref{l:standardSplitA} below. In particular, let 
$\operator=\frac{1}{2}(I+\pert)= A_{k}$ or $A_k'$ and for any $N\geq 0$, define $\pert_N$ according to these splittings and Definition~\ref{def:quad}. The discrete approximation of $\pert$ is then given by \eqref{e:nystromForm} with 
\begin{equation}
\label{e:dirichletNystrom}
\pert_V= \pert_{\trig_N}=\pert_N
\end{equation}
where $\trig_N$ is the space of trigonometric polynomials (see~\eqref{e:trig_points}).

\begin{theorem}[Dirichlet Nystr\"om \blue{in 2-d with trigonometric polynomials}]
\label{t:DNystrom}
Suppose that $d=2$ and Assumption \ref{ass:parameters} holds.  
Consider the Nystr\"om method with Kress quadrature (defined in \S\ref{sec:Nystrom}) for the Dirichlet BIE.

\noindent (i) Suppose $\Gamma$ is convex with non-vanishing curvature.  Then  there is $C_1>0$ such that for all $s>\frac{1}{2}$, there are $C>0$, $k_0>0$ such that if $k\geq k_0$, 
$$
N\geq C_1 k 
$$
the solution, $v_N$, to~\eqref{e:dirichletNystrom} exists, is unique, and satisfies the quasi-optimal error bound
$$
\|v-v_N\|_{\Hsh}\leq C\Big(\big\|(I-P_{\trig_N}^C)v\big\|_{\Hsk}+ \big\|P_{\trig_N}^C(\pert_N-\pert)P_{\trig_N}^Cv\big\|_{\Hsk}\Big).
$$
(ii) Suppose Assumption~\ref{ass:polyboundintro} holds with $P_{\rm inv}=0$. Then for all $\e>0$, there is $s_0>0$ and $C_1>0$ such that for all $s>s_0$, there are $C_2>0$, $k_0>0$ such that if $k\geq k_0$, $k\notin \cJ$,
$$
N\geq C_1 k (\log k)^\e 
$$
the solution, $v_N$, to~\eqref{e:dirichletNystrom}  exists, is unique, and satisfies the quasi-optimal error bound
$$
\|v-v_N\|_{\Hsh}\leq C_2\Big(\big\|(I-P_{\trig_N}^C)v\big\|_{\Hsk}+ \big\|P_{\trig_N}^C(\pert_N-\pert)P_{\trig_N}^Cv\big\|_{\Hsk}\Big).
$$
(iii) Suppose Assumption~\ref{ass:polyboundintro} holds. Then for all $\e>0$, there are $s_0>0$, $C_1>0$ such that for all  $M>0$, $s>s_0$,  there are $C_2>0$, $k_0>0$ such that if $k\geq k_0$, $k\notin \cJ$, 
$$
N\geq C_1 k^{1+\e} 
$$
the solution, $v_N$, to~\eqref{e:dirichletNystrom} exists, is unique, and satisfies the quasi-optimal error bound
$$
\|v-v_N\|_{\Hsh}\leq C_2\Big(\|(I-P_{\trig_N}^C)v\|_{\Hsk}+ \|P_{\trig_N}^C(\pert_N-\pert)P_{\trig_N}^Cv\|_{\Hsk} + k^{-M}\|v\|_{\Hsh}\Big) .
$$
Moreover, in all three cases, if the right-hand side $f$ corresponds to plane-wave scattering (i.e., is given as in Theorem \ref{thm:BIEs}), then
\beq\label{e:MR7}
\frac{\N{v-v_N}_\Hsk}{\N{v}_{\Hsk}} = O(k^{-\infty}).
\eeq

\end{theorem}

We next consider the Neumann problem with  Kress quadrature of $kS_k$, $\DL_k$, $\DL_k'$ described in Lemma~\ref{l:standardSplitA}, that for $kS_{ik}$ $\DL_{ik}$, $\DL_{ik}'$ described in Lemma~\ref{l:standardSplitB}, and that for $k^{-1}(H_k-H_{ik})$ described in Lemma~\ref{l:standardSplitC}. In particular, let 
$\operator=(i\frac{\eta_N}{2}-\frac{1}{4})(I+\pert)=B_{k,\rm reg} $ or $B'_{k,\rm reg}$ and for any $N\geq 0$, define $\pert_N$ according to the splittings above and Definition~\ref{def:quad} with the composition $L_aL_b$ approximated by $\LN_a^NP^C_{\trig_N}\LN_b^N$, where $P_{\trig_N}^C$ denotes the collocation projection with trigonometric polynomials.
The discrete approximation is then given by 
\begin{equation}
\label{e:NeumannNystrom}
\pert_{\trig_N}=\pert_N.
\end{equation}

\begin{theorem}[Neumann Nystr\"om \blue{in 2-d with trigonometric polynomials}]
\label{t:NeumannNystrom}
Suppose that $d=2$ and Assumptions \ref{ass:parameters} and \ref{ass:polyboundintro} hold.
Consider the Nystr\"om method 
with Kress quadrature (defined in \S\ref{sec:Nystrom}). For all $\e>0$, there are $s_0>0$, $C_1>0$ such that for all  $M>0$, $s>s_0$,  there are $C_2>0$, $k_0>0$ such that if $k\geq k_0$, $k\notin \cJ$,
$$
N\geq C_1 k^{1+\e} 
$$
the solution, $v_N$, to~\eqref{e:NeumannNystrom} exists, is unique, and satisfies the quasi-optimal error bound
\begin{equation}
\label{e:MR8}
\|v-v_N\|_{\Hsh}\leq C_2\Big(\big\|(I-P^C_{\trig_N})v\big\|_{\Hsh}+ \big\|P_{\trig_N}(\pert_N-\pert)P_{\trig_N}v\big\|_{\Hsk} + k^{-M}\|v\|_{\Hsk}\Big) .
\end{equation}
Moreover, if the right-hand side $f$ corresponds to plane-wave scattering (i.e., is given as in Theorem \ref{thm:BIEs}), then \eqref{e:MR7} holds.
\end{theorem}

\begin{figure}[htbp]
\begin{tikzpicture}
\node at(-1,0){\includegraphics[width=.5\textwidth]{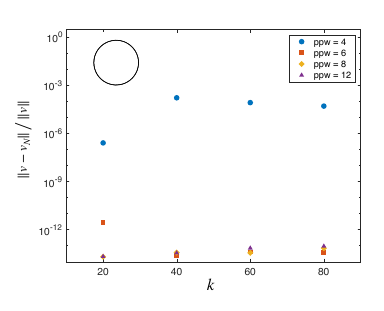}};
\node at(7,0){\includegraphics[width=.5\textwidth]{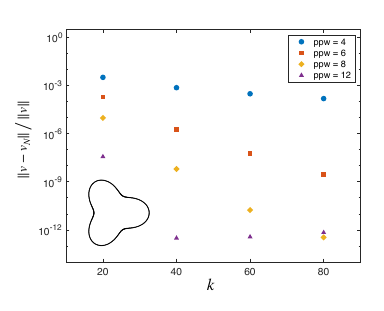}};
\node at(3,-6){\includegraphics[width=.5\textwidth]{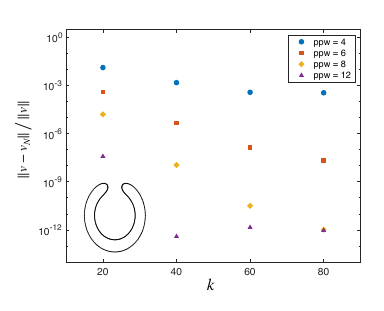}};
\end{tikzpicture}
\caption{The plots show the relative $L^2$ error when the Nystr\"om method is applied to the plane-wave sound-soft scattering problem in the exterior of a circle (top left), a star-shaped domain (top right), and cavity (bottom) at various frequencies and numbers of points per wavelength. As the number of points per wavelength increases or the frequency increases, 
the relative error decreases superalgebraically, 
which is consistent with Theorem~\ref{t:DNystrom}.}
\end{figure}
\paragraph{Comparison with previous results involving trigonometric polynomials.}
The analysis of Galerkin, 
collocation, and Nystrom methods with trigonometric polynomials applied to Helmholtz BIEs on smooth boundaries 
in the limit $N\to \infty$ with $k$ fixed has attracted a lot of attention in the literature; 
see, e.g., 
\cite{KrSl:93}
\cite{McPrWe:93}
\cite{SaVa:96}
\cite[Chapter 7]{At:97}
\cite[Chapters 9 and 10]{SaVa:02}
 and the references therein.
These results prove superalgebraic convergence for smooth data, and Theorems \ref{thm:FG}, \ref{thm:FC}, \ref{t:DNystrom}, and \ref{t:NeumannNystrom} are their $k$-explicit analogues.
Furthermore, e.g., \cite{SaVa:96} shows that the convergence is exponential if both the data and the boundary are analytic. 
For real analytic $\Gamma$, it should be possible to use methods from analytic microlocal analysis to improve 
 Theorems \ref{thm:FG}, \ref{thm:FC}, \ref{t:DNystrom}, and \ref{t:NeumannNystrom} by replacing the superalgebraic decay by exponential decay, but we do not consider this here.

\subsection{Discussion of the proof techniques}\label{s:discussion}

\paragraph{The $k$-explicit upper bounds on the projection-method errors.}

The current state-of-the-art analyses of $k$-explicit convergence of numerical methods for the Helmholtz rely on frequency splitting; specifically one splits into 
high ($\gg k$) and low ($\lesssim k$) frequencies and uses fact that Helmholtz solution operator is ``well behaved" on high frequencies.

This idea was first used in the analysis of the $hp$-FEM for the constant-coefficient Helmholtz equation in \cite{MeSa:10, MeSa:11}, with the frequency splitting carried out via the Fourier transform in $\Rea^d$. This philosophy was then applied to analyse the $hp$-BEM in  \cite{LoMe:11, Me:12}, with frequency splitting in $\Rea^d$ combined with trace operators to produce appropriate frequency splittings on $\Gamma$. 

Semiclassical pseudodifferential operators give a natural and intrinsic way to achieve frequency splittings on general domains. Such operators were used to analyse the $hp$-FEM for the variable-coefficient Helmholtz equation in \cite{LSW3, LSW4, GLSW1}.

The first author's thesis \cite{Ga:19} used the fact that the Helmholtz solution operator is ``well behaved" (more precisely, semiclassically elliptic) on high frequencies to show that 
 the high-frequency components of $S_k, K_k, K_k',$ and $H_k$ are semiclassical pseudodifferential operators. This 
 fact was then used in \cite{GaSp:22} to show that the Galerkin method with piecewise polynomials applied to the Dirichlet BIEs does not suffer from the pollution effect when $\Omega^-$ is nontrapping. Conceptually, these results are linked to earlier convergence analyses of  BIEs that used the integral operators' structure as homogeneous pseudodifferential operators to obtain results for fixed $k$ as $N\to \infty$ (see, e.g., the books \cite{SaVa:02, HsWe:08, GwSt:18} and the references therein).

In the present paper, we use the results of \cite{Ga:19} again 
to give sufficient conditions 
for the Galerkin solution to exist with an optimal bound on 
the quasioptimality constant -- see Theorem \ref{thm:PG}. 
As discussed in Remark \ref{rem:preasymptotic}, Theorem \ref{thm:PG} is analogous to the preasymptotic estimates for the $h$-FEM proved in \cite{GS3}, with these $h$-FEM results using frequency splittings defined by functional calculus and thus intrinsic to the operators considered, together with an elliptic-projection type argument. 
The results in the present paper 
are obtained by carefully analyzing the data to approximate solution maps from high to high, high to low, low to high, and low to low frequencies and using a norm on $L^2$ that is tailored to the meshwidth and the frequency splitting -- see \eqref{e:newNorm} and \eqref{e:blockDecomp} below. Specifically, we weight the high frequency component of the function by $(hk)^{-t}\rho^{-1/2}$ for some $t\geq 0$ tailored to the particular projection method. This weighting allows us to take full advantage of the better behaviour of the BIEs at high frequency.

Our analyses of methods based on trigonometric polynomials use analogous frequency splittings together with the fact that trigonometric polynomials naturally respect frequency decomposition (see Lemma \ref{l:proj2Pseudo} below).

\paragraph{The proofs of pollution for the Galerkin method with piecewise polynomials.} \

We first discuss the proof of Theorem~\ref{t:pollutionIntro} (i.e., the result that quasimodes imply pollution) and then discuss the construction of good quasimodes.

\noindent{\emph{The proof of Theorem~\ref{t:pollutionIntro}.}}

For the projection method 
\beqs
v-v_N= (I+P^G_{V_h} \pert)^{-1}(I-P^G_{V_h}) (I-P^G_{V_h})v
\eeqs
(see \eqref{e:QOabs} below). Therefore, to prove pollution, we need to show that $(I+P^G_{V_h} \pert)^{-1}(I-P^G_{V_h})$ is large. A simple calculation shows that if 
\beq\label{e:introv}
v := (I + \pert)^{-1} (I-P^G_{V_h}) \widetilde{f}
\eeq
(i.e., the data is orthogonal to the approximation space) then $v_N=0$ and 
\beqs
(I+P^G_{V_h} \pert) v= (I-P^G_{V_h}) v =(I-P^G_{V_h})^2 v.
\eeqs
\blue{This observation shows that to prove that $(I+P^G_{V_h} \pert)^{-1}(I-P^G_{V_h})$ is large }
it is sufficient to find $v$ with $\|v\|_{L^2}=1$ such that \eqref{e:introv} holds and $(I-P^G_{V_h})v$ is ``small". 

Under the assumption that there is a quasimode for the BIE, there is a ``small" $f$ and a $u$ with $\|u \|_{\LtG}=1$ such that $u= (I+\pert)^{-1} f$. Moreover, both $f$ and $u$ are $k$-oscillating and thus 
for the Galerkin projection with a piecewise polynomial space satisfying Assumption~\ref{ass:ppG},
$\| (I-P^G_{V_h})u\|_{\LtG} \leq C (hk)^{p+1} \| u\|_{\LtG}$. If there is $\tilde{f}$ such that $f = (I-P^G_{V_h})\widetilde{f}$, then
\beqs
\big\|(I+ P^G_{V_h} \pert)^{-1} (I-P^G_{V_h})\big\|_{\LtGt}\geq c (hk)^{-p-1}.
\eeqs
However, this cannot be true for all $h$ since it contradicts the upper bound on the quasioptimality constant in Theorem \ref{thm:PG}; i.e., we cannot solve $f = (I-P^G_{V_h})\widetilde{f}$.

Motivated by this discussion, our goal is to find $\widetilde{f}$ that is as close as possible to satisfying $(I-P^G_{V_h})\widetilde{f}=f$. Notice that, since $v=(I+\pert )^{-1}(I-P^G_{V_h})\tilde{f}$ and $(I+ \pert)^{-1}$ is $O(1)$ on high frequencies (see Lemma \ref{lem:HFinverse}), low-frequency errors cost more than high-frequency errors. Therefore, we aim to solve $(I-P^G_{V_h})\tilde{f}=f+e$, where $e$ has as few low-frequencies as possible.

The lower bounds on Galerkin piecewise-polynomial approximations of $k$-oscillating functions from \cite{Ga:25} imply that 
if $f$ has frequencies between $\e_0 k$ and $\Xi_0 k$ then there exists $\widetilde{f}$ such that 
\beq\label{e:blueCar1}
\big\| (I-P^G_{V_h})\widetilde{f} \big\|_{\LtG} \leq C(hk)^{-p-1}\big\| f\big\|_{\LtG}
\eeq
and
\begin{gather*}
(I-P^G_{V_h})\widetilde{f}=f +\chi_{\mathscr{L}}(I-P^G_{V_h})\widetilde{f}+\chi_{\mathscr{H}}(I-P^G_{V_h})\widetilde{f}
\end{gather*}
(see~\eqref{e:weSolvedIt} and Corollary~\ref{c:polysAreGreat}).
Here $\chi_{\mathscr{L}}$ is a finite-rank operator (which can be taken to be zero when $p=0$) and 
\beqs
\big\| \chi_{\mathscr{L}}\big\|_{H^{-M}(\Gamma)\to H^M(\Gamma)} \leq C_M 
\quad\tand\quad 
\big\| (I+ \pert)^{-1} \chi_{\mathscr{H}}\big\|_{\LtGt} \leq C 
\eeqs
(i.e., $\chi_{\mathscr{L}}$ and $\chi_{\mathscr{H}}$ are low- and high-frequency cutoffs, respectively).
Thus 
\begin{align}\label{e:blueCar2}
v &= (I+\pert)^{-1} f + (I+\pert)^{-1} \chi_{\mathscr{L}}(I-P^G_{V_h})\widetilde{f}+(I+\pert)^{-1}\chi_{\mathscr{H}}(I-P^G_{V_h})\widetilde{f}.
\end{align}
In the rest of this sketch, we ignore the contribution from $\chi_{\mathscr{L}}$; including this contribution leads to the finite-dimensional condition \eqref{e:normalOk1}.
Using \eqref{e:blueCar2}, \eqref{e:blueCar1}, and the fact that $u= (I+L)^{-1}f$ (with $\| u\|_{L^2(\Gamma)}= 1$), we see that $\|v\|_{\LtG}\geq 1/2$ if $(hk)^{-p-1} \|f \|_{\LtG} \ll 1$.
Finally, by \eqref{e:introv},
\begin{align*}
(I- P^G_{V_h}) v = (I- P^G_{V_h}) \chi_{\rm H}(I+\pert)^{-1} (I- P^G_{V_h}) \widetilde{f} + (I-P^G_{V_h}) \chi_{\rm L} v
\end{align*}
where $\chi_{\rm H}$ and $\chi_{\rm L}$ are high- and low-frequency cutoffs, so that
\begin{align*}
\big\|(I- P^G_{V_h}) v\big\|_{\LtG} \leq C (hk)^{-p-1} \| f\|_{\LtG} + (hk)^{p+1} \|v\|_{\LtG}.
\end{align*}
In summary, 
\begin{align*}
\big\|(I+ P^G_{V_h} \pert)^{-1} (I-P^G_{V_h})\big\|_{\LtGt}
&\geq
\frac{\|v\|_{\LtG}}{\| (I-P^G_{V_h})v\|_{\LtG}}\\
&\geq 
 c
\begin{cases}
(hk)^{-p-1}, & \|f\|_{\LtG} \leq (hk)^{2(p+1)} ,\\
(hk)^{p+1} \| f\|^{-1}_{\LtG}, & (hk)^{2(p+1)}\leq \|f\|_{\LtG} \ll (hk)^{p+1} .
\end{cases} 
\end{align*}
Since $\|v\|_{\LtG}/\| (I-P^G_{V_h})v\|_{\LtG}\geq 1$, by reducing $c$ if necessary, 
\beqs
\big\|(I+ P^G_{V_h} \pert)^{-1} (I-P^G_{V_h})\big\|_{\LtGt}\geq 
\frac 12 + 
c
\begin{cases}
(hk)^{-p-1}, & 1\leq (hk)^{2(p+1)} \|f\|_{\LtG}^{-1},\\
(hk)^{p+1} \| f\|^{-1}_{\LtG}, & \|f\|_{\LtG}^{-1}(hk)^{2(p+1)}\leq 1  ,
\end{cases} 
\eeqs
which is \eqref{e:thePollution}.

\noindent{\emph{Construction of quasimodes used to prove pollution in concrete settings.} }

With Theorem~\ref{t:pollutionIntro} in hand, it now remains to construct $f$ oscillating between $\e_0 k $ and $\Xi_0k$ such that $\|(I+\pert)^{-1}f\|_{\LtG}\gg \|f\|_{\LtG}$. In the case of the unit disk, this can be done using the explicit representation of the Helmholtz BIEs in the basis of trigonometric polynomials (see \S~\ref{s:disk}). When the geometry is more complicated, we consider only the Dirichlet problem and construct solutions of 
\begin{equation}
\label{e:quasi12345}
(-\Delta -k^2)u=0\text{ in }\Omega^+,\qquad u|_{\Gamma}=g,
\end{equation}
where, for $\chi\in C_c^\infty(\overline{\Omega^+})$, $\|g\|_{\LtG}\ll \|\chi u\|_{L^2(\Omega^+)}$ and $g$ has certain microlocal properties used to guarantee that $f$ is oscillating between $\e_0 k$ and $\Xi_0 k$. These properties are somewhat different for the direct and indirect formulation of the Dirichlet problem because the solution of the indirect equation has more immediate physical meaning (it is the Neumann trace of the scattering solution). To construct $f$ itself from $u$ we then use the description of $(I+\pert)^{-1}$ in terms of the outgoing Dirichlet-to-Neumann map on $\Omega^+$ and the impedance-to-Dirichlet map on $\Omega^-$ (see~\eqref{e:BIEinverse}).

In turn, $u$ in~\eqref{e:quasi12345} is constructed by finding $\tilde{u}$ satisfying
$$
(-\Delta -k^2)\tilde{u}=\tilde{g}\text{ in }\Omega^+,\qquad \tilde{u}|_{\Gamma}=0,
$$
where, for  $\chi\in C_c^\infty(\overline{\Omega^+})$, $\|\tilde{g}\|_{L^2(\Omega^+)}\ll \|\chi \tilde{u}\|_{L^2(\Omega^+)}$, and $\tilde{g}$ has microlocal properties that guarantee that appropriate billiard trajectories avoid the directions normal to $\Gamma$ and tangent to $\Gamma$. The function $u$ is then constructed by applying propagation singularities--type arguments to $\tilde{u}$. 
In fact, the necessary propagation of singularities estimates required to construct $u$ 
are not available in the literature. The technical reasons for this are discussed in Remark~\ref{r:propagation} and the propagation results are then proved in Appendix~\ref{a:propagate}.

\subsection{Outline of the rest of the paper}

\setlength{\leftmargini}{.5em}
\begin{itemize}
\item[]
\S\ref{sec:SCA} recaps standard results about semiclassical pseudodifferential operators.
\item[]
\S\ref{sec:abstract} gives sufficient conditions for quasioptimality of projection method under fairly-general assumptions about $\operator$.
\item[]
\S\ref{sec:PPproofs} proves the Galerkin and collocation results for piecewise polynomials (Theorems \ref{thm:PG} and \ref{thm:PC}).
\item[]
\S\ref{sec:trig} defines the projections for trigonometric polynomials in 2-d.
\item[]
\S\ref{sec:approx} bounds $\|(I-P_{\trig_N}^G)v\|_{H^s_k(\Gamma)}$ when $v$ is the BIE solution corresponding to the plane-wave scattering problem.
\item[]
\S\ref{sec:proofs} proves the Galerkin and collocation results for trigonometric polynomials (Theorems \ref{thm:FG} and \ref{thm:FC}).
\item[]
\S\ref{sec:Nystrom} defines the Nystr\"om method analysed in Theorems \ref{t:DNystrom} and \ref{t:NeumannNystrom}.
\item[]
\S\ref{sec:NystromProof} gives an abstract result about the convergence of the Nystr\"om method and \S\ref{sec:NystromApplied} uses this to prove Theorem \ref{t:DNystrom} and \ref{t:NeumannNystrom}.
\item[]
\S\ref{s:pollution} proves Theorem \ref{t:pollutionIntro}
and 
\S\ref{s:quasimodes} proves Theorems~\ref{t:diamond} and 
\ref{t:qualitative1} and 
\S\ref{s:disk} proves Theorems \ref{t:neumannDisk} and~\ref{t:dirichletDisk}.

\item[]
\S\ref{sec:numerical} describes the set up in the numerical experiments in \S\ref{sec:main}.
\item[]
Appendix \ref{app:A} defines the scattering problems and BIEs in \eqref{e:DBIEs} and \eqref{e:NBIEs}.
\item[]
Appendix \ref{a:propagate} proves the propagation-of-singularities result of Theorem \ref{t:basicPropagate}.
\end{itemize}
\setlength{\leftmargini}{2.5em}

\section{Review of semiclassical pseudodifferential operators}
\label{sec:SCA}

In this section, we review standard results about semiclassical pseudodifferential operators, with our default references being \cite{Zw:12} and \cite[Appendix E]{DyZw:19}. 
Recall that semiclassical pseudodifferential operators are pseudodifferential operators with a large/small parameter, where behaviour with respect to this parameter is explicitly tracked in the associated calculus.
In our case, the small parameter is $k^{-1}$. and we 
Normally the small parameter is denoted by $h$, but we use $k^{-1}$ to avoid a notational clash with the meshwidth of the $h$-BEM. 

The counterpart of ``semiclassical'' involving differential/pseudodifferential operators without a small parameter is
  usually called ``homogeneous'', and the homogeneous analogues of these results can be found in, e.g.,
\cite[Chapter 7]{Ta:96}, \cite[Chapter 7]{SaVa:02}, \cite[Chapters
6]{HsWe:08}.


\subsection{Weighted Sobolev spaces}\label{sec:SC1}

We first define weighted Sobolev spaces on $\Rea^d$, and then use these to define analogous weighted Sobolev spaces on $\Gamma$.
The \emph{semiclassical Fourier transform} is defined by
$$
(\mathcal F_{k}u)(\xi) := \int_{\mathbb R^d} \exp\big( -\ri k x \cdot \xi\big)
u(x) \, \rd x,
$$
with inverse
\beq\label{e:SCFTinverse}
(\mathcal F^{-1}_{k}u)(x) := \Big(\frac{k}{2\pi}\Big)^d \int_{\mathbb R^d} \exp\big( \ri k x \cdot \xi\big)
u(\xi)\, \rd \xi;
\eeq
see \cite[\S3.3]{Zw:12}; i.e., the semiclassical Fourier transform is just the usual Fourier transform with the transform variable scaled by $\hsc$. These definitions imply that, with $D:= -\ri \partial$,
\beqs
\cF_k \big( (k^{-1} D)^\alpha) u\big) = \xi^\alpha \cF_k u \quad \tand\quad \N{u}_{L^2(\Rea^d)} = 
\Big(\frac{k}{2\pi}\Big)^{d/2}
\N{\cF_k u}_{L^2(\Rea^d)};
\eeqs
see, e.g., \cite[Theorem 3.8]{Zw:12}.
Let
\beq\label{e:Hsk}
H_k^s(\Rea^d):= \Big\{ u\in \mathcal{S}'(\Rea^d) \,\tst\, \langle \xi\rangle^s (\cF_k u) \in L^2(\Rea^d) \Big\},
\eeq
where $\langle \xi \rangle := (1+|\xi|^2)^{1/2}$, $\mathcal{S}(\Rea^d)$ is the Schwartz space (see, e.g., \cite[Page 72]{Mc:00}), and $\mathcal{S}'(\Rea^d)$ its dual.
Define the norm
\beq\label{e:Hhnorm}
\Vert u \Vert_{H_k^m(\Rea^d)} ^2 = 
\Big(\frac{k}{2\pi}\Big)^d
\int_{\Rea^d} \langle \xi \rangle^{2m}
|\mathcal F_k u(\xi)|^2 \, \rd \xi;
\eeq
for example, with $m=1$,
\beq\label{e:H1norm}
\Vert u \Vert_{H_k^1(\Rea^d)} ^2 
=  k^{-2} \N{\nabla u}^2_{L^2(\Rea^d)}
+\N{u}^2_{L^2(\Rea^d)}.
\eeq
Working in a weighted $H^1$ norm with the derivative weighted by $k^{-1}$ is ubiquitous in the literature on the numerical analysis of the Helmholtz equation, except that usually one works with the weighted $H^1$ norm squared being $\|\nabla u\|^2_{L^2}+k^2 \|u\|_{L^2}$.
Here we work with \eqref{e:H1norm}/\eqref{e:Hhnorm} since weighting the $j$th derivative by $k^{-j}$ is easier to keep track of than weighting it by $k^{-j+1}$, especially when working with higher-order derivatives.


We define the norm for weighted Sobolev spaces on $\Gamma$ as follows
\beq\label{e:weightedNormGamma}
\|u\|_{H_{k}^s(\Gamma)}:= \|(1-k^{-2}\Delta_{\Gamma})^{s/2}u\|_{\LtG},
\eeq
where $\Delta_{\Gamma}$ denotes the Laplacian on $\Gamma$ and we use the spectral theorem to define powers of $(1-\hsc^2\Delta_{\Gamma})$. For example,
$$
\|u\|_{H_{k}^1(\Gamma)}^2=k^{-2}\N{\nabla_{\Gamma}u}_{\LtG}^2+\N{u}_{\LtG}^2,
$$
where $\nabla_\Gamma$ is the surface gradient operator, defined in terms of a parametrisation of the boundary by, e.g., \cite[Equation A.14]{ChGrLaSp:12}.
Note that these same norms can be defined via interpolation from the integer powers of $(1-k^{-2}\Delta_\Gamma)$.
The weighted spaces $H^s_k(\Gamma)$ 
can also be defined by charts; see, e.g., \cite[Pages 98 and 99]{Mc:00} for the unweighted case and \cite[\S5.6.4]{Ne:01} or \cite[Definition E.20]{DyZw:19} for the weighted case (but note that \cite[\S5.6.4]{Ne:01} uses a different weighting with $k$ to us); these other definitions all lead to equivalent norms.

The unweighted $H^s(\Gamma)$ norms are defined as above with $k=1$.
We use below that, given $k_0>0$ and $s>0$, there exists $C>0$ such that, for all $k\geq k_0$,
\beq\label{e:weightedineq}
k^{-s}\N{w}_{H^s(\Gamma)}\leq C \N{w}_{H^s_k(\Gamma)}.
\eeq


\subsection{Phase space, symbols, quantisation, and semiclassical pseudodifferential operators}\label{sec:331}

For simplicity of exposition, we begin by discussing semiclassical pseudodifferential operators on $\Rea^d$, and then
outline in \S\ref{sec:ReatoGamma} below how to extend the results from $\Rea^d$ to $\Gamma$.

The set of all possible positions $x$ and momenta (i.e.~Fourier variables) $\xi$ is denoted by $T^*\Rea^d$; this is known informally as ``phase space''. Strictly, $T^*\Rea^d :=\Rea^d \times (\Rea^d)^*$, i.e. the cotangent bundle to $\mathbb{R}^d$, but
for our purposes, we can consider $T^*\Rea^d$ as $\{(x,\xi) : \bx\in \Rea^d, \xi\in\Rea^d\}$.

A symbol is a function on $T^*\Rea^d$ that is also allowed to depend on $k$, and can thus be considered as an $k$-dependent family of functions.
Such a family $a=(a_k)_{k\geq k_0}$, with $a_k \in C^\infty({T^*\mathbb R^d})$,
is a \emph{symbol
of order $m$}, written as $a\in S^m(T^*\Rea^d)$,
if for any multiindices $\alpha, \beta$
\beq\label{e:Sm}
| \partial_x^\alpha \partial^\beta_\xi a_k(x,\xi) | \leq C_{\alpha, \beta}
\langle \xi\rangle^{m -|\beta|}
\quad\tfa (x,\xi) \in T^* \Rea^d \text{ and for all } k\geq k_0,
\eeq
(where recall that $\langle\xi\rangle:= (1+ |\xi|^2)^{1/2}$) and
$C_{\alpha, \beta}$ does not depend on $k$; see \cite[p.~207]{Zw:12}, \cite[\S E.1.2]{DyZw:19}.

For $a \in S^m$, we define the \emph{semiclassical quantisation} of $a$, denoted by $a(x,k^{-1} D):\mathcal{S}(\mathbb{R}^d)\to \mathcal{S}(\mathbb{R}^d)$, by 
\beq \label{e:quant}
a(x,k^{-1} D) v(x) := \Big(\frac{k}{2\pi}\Big)^d \int_{\Rea^d} \int_{\Rea^d} 
\exp\big(\ri k (x-y)\cdot\xi\big)\,
a(x,\xi) v(y) \,\rd y  \rd \xi
\eeq
where $D:= -\ri \partial$; see, e.g., \cite[\S4.1]{Zw:12} \cite[Page 543]{DyZw:19}.
We also write $a(x,k^{-1} D)= \Op(a)$. The integral in \eqref{e:quant} need not converge, and can be understood \emph{either} as an oscillatory integral in the sense of \cite[\S3.6]{Zw:12}, \cite[\S7.8]{Ho:83}, \emph{or} as an iterated integral, with the $y$ integration performed first; see \cite[Page 543]{DyZw:19}.

Conversely, if $A$ can be written in the form above, i.\,e.\ $A =a(x,k^{-1} D)$ with $a\in S^m(T^*\Rea^d)$, we say that $A$ is a \emph{semiclassical pseudo-differential operator of order $m$} and
we write $A \in \Psi^m(\Rea^d)$. We use the notation $a \in k^{-\ell} S^m$  if $k^\ell a \in S^m$; similarly
$A \in k^{-\ell} \Psi^m$ if
$k^{\ell} A \in \Psi^m$. We define $\Psi^{-\infty} := \cap_m \Psi^{-m}$.

\begin{theorem}\mythmname{Composition and mapping properties of
semiclassical pseudo-differential operators \cite[Theorem 8.10]{Zw:12}, \cite[Propositions E.17, E.19, and E.24]{DyZw:19}}\label{thm:basicP} If $A\in \Psi^{m_1}$ and $B  \in \Psi^{m_2}$, then
\begin{itemize}
\item[(i)]  $AB \in \Psi^{m_1+m_2}$.
\item[(ii)]  For any $s \in \mathbb R$, $A$ is bounded uniformly in $k$ as an operator from $H_k^s$ to $H_k^{s-m_1}$.
\end {itemize}
\end{theorem}

A key fact we use below is that if $\psi\in C_{c}^\infty(\mathbb{R})$ then, given $t\in\mathbb{R}$, $N>0$ and $k_0>0$ there exists $C>0$ such that for all $k\geq k_0$
\beq\label{e:frequencycutoff}
\N{\psi(|k^{-1} D|^2)}_{H_k^{t}(\Rea^d)\to H_k^{t+N}(\Rea^d)} \leq C;
\eeq
this can easily be proved using the semiclassical Fourier transform, since $\psi(|k^{-1} D|^2)$ is a Fourier multiplier (i.e., $\psi(|k^{-1} D|^2)$ is defined by \eqref{e:quant} with $a(x,\xi)= \psi(|\xi|^2)$, which is independent of $x$).

\subsection{The principal symbol map $\sigma$}
Let the quotient space $ S^m/k^{-1} S^{m-1}$ be defined by identifying elements
of  $S^m$ that differ only by an element of $k^{-1} S^{m-1}$.
For any $m$, there is a linear, surjective map
$$
\sigma^m:\Psi^m \to S^m/(k^{-1} S^{m-1}),
$$
called the \emph{principal symbol map},
such that, for $a\in S^m$,
\beq\label{e:symbolone}
\sigma^m\big(\Op(a)\big) = a \quad\text{ mod } k^{-1} S^{m-1};
\eeq
see \cite[Page 213]{Zw:12}, \cite[Proposition E.14]{DyZw:19} (observe that \eqref{e:symbolone} implies that
$\operatorname{ker}(\sigma^m_{k}) = k^{-1}\Psi ^{m-1}$).
When applying the map $\sigma^m_{k}$ to
elements of $\Psi^m_k$, we denote it by $\sigma$ (i.e.~we omit the $m$ dependence) and we use $\sigma(A)$ to denote one of the representatives
in $S^m$ (with the results we use then independent of the choice of representative).

Key properties of the principal symbol that we use below are that
\beq \label{e:multsymb}
\sigma(AB)=\sigma(A)\sigma(B)
\quad\tand \quad
\sigma(A^*) =\overline{\sigma(A)},
\eeq
see \cite[Proposition E.17]{DyZw:19}, and for $A\in \Psi^0$,
\begin{equation}
\label{e:pseudoL2}
\N{A}_{L^2\to L^2}\leq \sup|\sigma(A)|+Ck^{-1},
\end{equation}
see~\cite[Theorem 13.13]{Zw:12}.
              
\subsection{Extension of the above results from $\Rea^d$ to $\Gamma$}\label{sec:ReatoGamma}

While the definitions above are written for operators on $\Rea^d$, semiclassical pseudodifferential operators and all of their properties above have analogues on compact manifolds  (see e.g.~\cite[\S14.2]{Zw:12},~\cite[\S E.1.7]{DyZw:19}). Roughly speaking, the class of semiclassical pseudodifferential operators of order $m$ on a compact manifold $\Gamma$, $\Psi^m_k(\Gamma)$, are operators that, in any local coordinate chart, have kernels of the form~\eqref{e:quant} where the function $a\in S^m$ modulo a remainder operator $R$ that has the property that
\begin{equation}
\label{e:remainder}
\|R\|_{H_k^{-N}(\Gamma)\to H_{k}^N(\Gamma)}\leq C_{N} k^{-N};
\end{equation}
we say that an operator $R$ satisfying~\eqref{e:remainder} is $O(k^{-\infty})_{\Psi^{-\infty}(\Gamma)}$.

Semiclassical pseudodifferential operators on manifolds continue to have a natural principal symbol map
\beq\label{e:sigmaGamma}
\sigma:\Psi^m\to S^m(T^*\Gamma)/(k^{-1} S^{m-1}(T^*\Gamma)),
\eeq
where now $S^m(T^*\Gamma)$ is the class of functions on $T^*\Gamma$, the cotangent bundle of $\Gamma$, that satisfy the estimate~\eqref{e:Sm}
with $x$ replaced by a local coordinate variable $x'$ and $\xi$ replaced by $\xi'$ -- the dual variables to $x'$.
The property \eqref{e:multsymb} holds as before.
Furthermore, there is a noncanonical quantisation map $\Op:S^m(T^*\Gamma)\to \Psi^m(\Gamma)$ (involving choices of cut-off functions and coordinate charts) that satisfies
$$
\sigma (\Op(a))=a.
$$
and for all $A\in \Psi^m(\Gamma)$, there is $a\in S^m(T^*\Gamma)$ such that
$$
A=\Op(a)+O(k^{-\infty})_{\Psi^{-\infty}}.
$$

Let $g$ be the metric induced on $T^*\Gamma$ from the standard metric on $\mathbb{R}^d$. Then, in exact analogy with \eqref{e:frequencycutoff}, if $\psi\in C_{c}^\infty(\mathbb{R})$ then, given $t\in\mathbb{R}$, $N>0$ and $k_0>0$ there exists $C>0$ such that for all $k\geq k_0$,
\beq\label{e:frequencycutoff2}
\N{\psi\hDarg}_{H_k^{t}(\Gamma)\to H_k^{t+N}(\Gamma)} \leq C,
\eeq
where $D':= -\ri \partial_{x'}$.

Finally, we record the following consequence of \eqref{e:pseudoL2} for bounds on $H_k^s(\Gamma)$ norms.
\begin{lemma}
\label{l:HsEstimates}
Suppose that $A\in \Psi^m(\Gamma)$. Then, for all $s\in \mathbb{R}$ there is $C_s>0$ such that
$$
\|A\|_{H_{k}^{s}(\Gamma)\to H_{k}^{s-m}(\Gamma)}\leq \sup_{(x',\xi')\in T^*\Gamma}|\sigma(A)(x',\xi')(1+|\xi'|^2_g)^{-m/2}|+C_s k^{-1}.
$$
\end{lemma}
\begin{proof}
By \eqref{e:weightedNormGamma},
\begin{align*}
\|Au\|_{H_{k}^{s-m}(\Gamma)}&=\|(1-k^{-2}\Delta_\Gamma)^{(s-m)/2}Au\|_{\LtG}\\
&=\|(1-k^{-2}\Delta_\Gamma)^{(s-m)/2}A(1-k^{-2}\Delta_\Gamma)^{-s/2}(1-k^{-2}\Delta_\Gamma)^{s/2}u\|_{\LtG}\\
&\leq \|(1-k^{-2}\Delta_\Gamma)^{(s-m)/2}A(1-k^{-2}\Delta_\Gamma)^{-s/2}\|_{\LtG\to \LtG}\|u\|_{H_{k}^{s}(\Gamma)}.
\end{align*}
Now, $(1-k^{-2}\Delta_\Gamma)^{s/2}\in \Psi^{s}(\Gamma)$ with symbol $(1+|\xi|_g^2)^{s/2}$
and hence, by~\eqref{e:pseudoL2},
$$
\|(1-k^{-2}\Delta_\Gamma)^{(s-m)/2}A(1-k^{-2}\Delta_\Gamma)^{-s/2}\|_{\LtG\to \LtG}\leq \sup_{(x,\xi)\in T^*\Gamma}(1+|\xi'|_g^2)^{-m/2}|\sigma(A)|(x',\xi')+C_s k^{-1},
$$
as claimed.
\end{proof}

\subsection{Wavefront set and ellipticity}\label{sec:335}

\begin{definition}[Ellipticity]\label{def:elliptic}
$B\in \Psi^{\ell}(\Gamma)$ is \emph{elliptic} on a set $U\subset T^*\Gamma$ if
\beq\label{e:elliptic}
\liminf_{k\to \infty}\inf_{(x',\xi')\in U}\big|\sigma(B)(x',\xi')\langle \xi'\rangle^{-\ell}\big|>0,
\eeq
where $\langle \xi'\rangle:=(1+|\xi'|_g^2)^{\frac{1}{2}}$.
\end{definition}

\begin{definition}[Wavefront set of a pseudodifferential operator]\label{def:WF}
The wavefront set $\WF(A)$ of $A\in \Psi^m(\Gamma)$ is defined as follows:~$(x_0,\xi_0)\in (\WF(A))^c$ if there exists $B\in \Psi^{-m}(\Gamma)$, elliptic in a neighbourhood of $(x_0,\xi_0)$ such that
\beqs
BA = O(k^{-\infty})_{\Psi^{-\infty}(\Gamma)}.
\eeqs
\end{definition}

We make three remarks.

(i) Definition \ref{def:WF} implies that
\beq\label{e:WF_residual}
\WF(A) = \emptyset \quad\iff \quad A=  O(k^{-\infty})_{\Psi^{-\infty}(\Gamma)}.
\eeq

(ii) Strictly speaking, $\WF(A)$ is a subset of the fiber-radially compactified cotangent bundle $\overline{T}^*\Gamma$ (see \cite[\S E.1.3]{DyZw:19}), but this notion is not important in what follows.

(iii) One can show that Definition \ref{def:WF} is equivalent to \cite[Definition E.27]{DyZw:19} by using the composition formula \cite[E.1.21]{DyZw:19}. This composition formula also implies that
\beq\label{e:WF_prod}
\WF(AB) \subset \WF (A) \cap \WF(B)
\eeq
and if $a$ is independent of $\hsc$ then
\beq\label{e:WFquant}
\WF\big( \Op(a)\big) \subset \supp \,a.
\eeq
The related adjoint formula \cite[E.1.22]{DyZw:19} implies that
\beq\label{e:WFadjoint}
\WF(A^*) =\WF (A)
\eeq
(compare \eqref{e:WF_prod} and \eqref{e:WFadjoint} to \eqref{e:multsymb}).

We use repeatedly below the corollary of \eqref{e:WF_residual} and \eqref{e:WF_prod} that
\beq\label{e:WFdisjoint}
\text{ if } \quad\WF(A)\cap \WF(B) = \emptyset\quad\text{ then } \quad AB = O(k^{-\infty})_{\Psi^{-\infty}(\Gamma)};
\eeq
i.e., pseudodifferential operators act microlocally (i.e., pseudo locally in phase space).

\begin{theorem}[Elliptic parametrix]\label{thm:elliptic_para}
Suppose that $A\in \Psi^m(\Gamma)$ 
and that $B\in \Psi^\ell(\Gamma)$ is elliptic on $\WF (A)$. Then there exist $Q, Q' \in \Psi^{m-\ell}(\Gamma)$ such that
\beqs
A= BQ + O(k^{-\infty})_{\Psi^{-\infty}(\Gamma)} = Q'B+O(k^{-\infty})_{\Psi^{-\infty}(\Gamma)}.
\eeqs
\end{theorem}

 \bpf[Reference for the proof]
This is proved in \cite[Proposition E.32]{DyZw:19}.
\epf

\begin{corollary}[Elliptic estimates]\label{cor:elliptic}

(i) If $B\in \Psi^\ell(\Gamma)$ is \emph{elliptic} on $T^*\Gamma$
then there exists $k_0>0$ such that, for all $k\geq k_0$,  $B^{-1} : H_k^{s-\ell}(\Gamma) \rightarrow H_k^{s}(\Gamma)$ exists and is bounded (with norm independent of $\hsc$) for all $s$. Furthermore, if $B\in \Psi^0$, then $B^{-1}\in \Psi^0$ and
$$
\sigma^0(B^{-1})=(\sigma^0(B))^{-1}.
$$

(ii) Suppose that $A\in \Psi^m(\Gamma)$ 
and that $B\in \Psi^\ell(\Gamma)$ is elliptic on $\WF (A)$. Then, given $s,M,N \in \Rea$, there exists $C, k_0>0$ such that, for all $k\geq k_0$,
\beq\label{e:ellipticestimate}
\N{Au}_{H^{s-m}_k(\Gamma)} \leq C\Big(
\N{Bu}_{H^{s-m-\ell}_k(\Gamma)} + k^{-M} \N{u}_{H^{s-m-N}_k(\Gamma)}.
\Big)
\eeq
\end{corollary}

\bpf
The first statement in Part (i) follows from Part (ii) with $A=I$, $B=B$, and $A=I$, $B=B^*$. In both cases, the $O(k^{-M})$ error term on the right-hand side is absorbed on the left-hand side, making $k_0$ bigger if necessary, so that the bound \eqref{e:ellipticestimate} is just the statement of injectivity of $B$ and its analogue with $B=B^*$ the statement of surjectivity of $B$. For the statement about $B\in \Psi^0$, observe that by Theorem~\ref{thm:elliptic_para}, there is $Q\in \Psi^0(\Gamma)$ such that
$$
I=QB+R,\qquad R=O(k^{-\infty})_{\Psi^{-\infty}(\Gamma)}.
$$
In particular, since $(I-R)^{-1}=I+R(I-R)^{-1}$,
 $B^{-1}=Q+R(I-R)^{-1}Q=Q+O(k^{-\infty})_{\Psi^{-\infty}(\Gamma)}\in \Psi^0(\Gamma)$. 
Furthermore, the symbol formula~\eqref{e:multsymb} then implies that $\sigma^{0}(B^{-1})=1/\sigma^0(B)$ as claimed.

Part (ii) follows from Theorem \ref{thm:elliptic_para} and the definition \eqref{e:remainder} of $O(k^{-\infty})_{\Psi^{-\infty}(\Gamma)}$.
\epf

\subsection{Recap of the results of \cite{Ga:19, GaMaSp:21N}}\label{sec:SCAnew}

The key ingredient for the proof of the main results is the following result from \cite[Theorem 4.3]{GaMaSp:21N}, adapted from the results in \cite[Chapter 4]{Ga:19}.

\begin{theorem}\mythmname{The high-frequency components of $S_k, \DL_k, \DL'_k$, and $H_k$ are semiclassical pseudodifferential operators}\label{thm:HFSD}
Let $\chi \in C^\infty_{c}(\Rea)$ with
$\supp(1-\chi) \cap [-1,1]=\emptyset$. 
Then
\beqs
\begin{gathered}
\big(I-\chi\big(|k^{-1} D'|_g^2\big)\big) S_k ,\,\, \,S_k\big(I-\chi\big(|k^{-1} D'|_g^2\big)\big)\in k^{-1} \Psi^{-1}(\Gamma),\\
\big(I-\chi\big(|k^{-1} D'|_g^2\big)\big) \DL_k' ,\,\, \, \DL_k'\big(I-\chi\big(|k^{-1} D'|_g^2\big)\big)\in k^{-1} \Psi^{-1}(\Gamma),\\
\big(I-\chi\big(|k^{-1} D'|_g^2\big)\big) \DL_k ,\,\, \, \DL_k\big(I-\chi\big(|k^{-1} D'|_g^2\big)\big)\in k^{-1} \Psi^{-1}(\Gamma),\\
\big(I-\chi\big(|k^{-1} D'|_g^2\big)\big) H_k ,\,\, \, H_k\big(I-\chi\big(|k^{-1} D'|_g^2\big)\big)\in k \Psi^{1}(\Gamma).
\end{gathered}
\eeqs
Moreover,
\beq\label{e:prin_symb_S}
\sigma\big(\big(I-\chi\big(|k^{-1} D'|_g^2\big)\big) S_k\big)=\sigma\big(S_k\big(1-\chi\big(|k^{-1} D'|_g^2\big)\big)= \dfrac{\big(1-\chi(|\xi'|_g^2)\big)}{2k\sqrt{|\xi'|_g^2-1}}
\eeq
and
\beq\label{e:H_prin_symb}
\sigma\big(\big(I-\chi\big(|k^{-1} D'|_g^2\big)\big) H_k\big)=\sigma\big(H_k\big(1-\chi\big(|k^{-1} D'|_g^2\big)\big)=-\big(1-\chi(|\xi'|_g^2)\big) \dfrac{k\sqrt{|\xi'|_g^2-1}}{2}.
\eeq
\end{theorem}

\bre
The statement of Theorem \ref{thm:HFSD} in \cite[Theorem 4.4]{GaMaSp:21N} differs from Theorem \ref{thm:HFSD} in the following two ways. First, \cite[Theorem 4.4]{GaMaSp:21N} has $\chi \in C^\infty_{c}(\Rea)$ satisfying $\supp (1-\chi) \cap [-2,2]=\emptyset$; however, the result holds under the condition that $\supp (1-\chi) \cap [-1,1]=\emptyset$ (with this condition appearing in the key ingredient for the proof of \cite[Theorem 4.4]{GaMaSp:21N}, namely \cite[Lemma 4.1]{GaMaSp:21N}). Second, \cite[Theorem 4.4]{GaMaSp:21N} is stated with high-frequency cut-offs of the form $1-\blue{\chi(|k^{-1} D'|_g)}$, whereas above we have $1-\chi(|k^{-1} D'|_g^2)$ -- either choice is possible, since $|\xi'|_g =1$ iff $|\xi'|_g^2=1$.
\ere

\subsection{Pseudodifferential Properties of functions of $-k^{-2}\Delta_\Gamma$}

Theorem \ref{thm:HFSD} is stated with the high-frequency cutoffs given by $( I- \chi\hDarg)$
where $\chi \in C^\infty_{c}(\Rea)$ with
$\supp(1-\chi) \cap [-1,1]=\emptyset$. 
The arguments in the rest of the paper also use high-frequency cutoffs defined in terms of functions of $-k^{-2}\Delta_\Gamma$; i.e., $( I- \chi(-k^{-2}\Delta_\Gamma))$ for $\chi$ as above. 
The following results show how one can replace instances of $( I- \chi(-k^{-2}\Delta_\Gamma))$ by $( I- \widetilde\chi\hDarg)$, where $\widetilde{\chi}$ is smaller than $\chi$.
Note that this replacement was  used in \cite{GaSp:22} to present the results of Theorem \ref{thm:HFSD} without explicitly using pseudodifferential operators; see \cite[Remark 4.6]{GaMaSp:21N}.

\ble[Pseudodifferential properties of $( I- \chi(-k^{-2}\Delta_\Gamma))$]\label{lem:WFcutoff}
If $\chi \in C^\infty_{c}(\Rea)$
 with $\supp(1-\chi)\cap[-\Xi,\Xi]=\emptyset$, then $( I- \chi(-k^{-2}\Delta_\Gamma))\in \Psi^0(\Gamma)$ with
\beqs
\WF\big( I- \chi(-k^{-2}\Delta_\Gamma)\big) \subset 
\overline{\big\{ (x',\xi') : 1-\chi(|\xi'|_g^2)>0\big\}}
\subset \big\{(x',\xi') : |\xi'|_g^2 >\Xi \big\}.
\eeqs
\ele

\begin{proof}
First recall from the Helffer-Sj\"ostrand functional calculus (see
e.g.~\cite[Theorem 14.9]{Zw:12}) that $\chi(-k^{-2}\Delta_\Gamma)\in \Psi^{-\infty}(\Gamma)$ which implies $(I-\chi(-k^{-2}\Delta_\Gamma))\in \Psi^0(\Gamma)$.
By the definition of the wavefront set (Definition \ref{def:WF}), to prove the lemma, we need to show that for any $(x'_0,\xi'_0)\in T^*\Gamma$ such that $\chi(|\xi_0'|_{g(x_0')}^2)=1$, there is $A\in \Psi^0(T^*\Gamma)$ elliptic at $(x_0,\xi_0)$ such that
\begin{equation}
\label{e:waveFrontNeed}
\|A(I-\chi(-k^{-2}\Delta_\Gamma))v\|_{H_k^N}\leq C_Nk^{-N}\|v\|_{H_k^{-N}}.
\end{equation}
To do this, fix such an $(x_0',\xi_0')$ and let $\widetilde{\chi}\in C_{c}^\infty(\mathbb{R})$ with $\supp \widetilde{\chi}\cap \supp(1-\chi)=\emptyset$ and 
$\widetilde{\chi}(|\xi_0'|_{g(x_0')}^2)=1$. 
Then, by \cite[Theorem 14.9]{Zw:12} again, $\widetilde{\chi}(-k^{-2}\Delta_\Gamma)\in \Psi^{-\infty}(T^*\Gamma)$ with principal symbol
$$
\sigma\big(\widetilde{\chi}(-k^{-2}\Delta_\Gamma)\big)=\widetilde{\chi}(|\xi'|_g^2).
$$
We now show that \eqref{e:waveFrontNeed} is satisfied with $A= \widetilde{\chi}(-k^{-2}\Delta_\Gamma)$. Indeed, 
By the definition of $\widetilde\chi$, 
$$
\widetilde{\chi}(-k^{-2}\Delta_\Gamma)\big(I-\chi(-k^{-2}\Delta_\Gamma)\big)=0,
$$
so~\eqref{e:waveFrontNeed} certainly holds. 
Furthermore, since $\widetilde{\chi}(|\xi_0'|_{g(x_0')})=1$, $A=\widetilde{\chi}(-k^{-2}\Delta_\Gamma)$ is elliptic at $(x_0',\xi_0')$. 
Since $(x_0',\xi_0')$ with $\chi(|\xi_0'|_{g(x_0')})=1$ was arbitrary, the proof is complete.
\end{proof}

\begin{corollary}[Replacing $I- \chi(-k^{-2}\Delta_\Gamma)$ by $( I- \widetilde{\chi}\hDarg\big)$]\label{cor:WFcutoff}
Suppose $\chi, \widetilde{\chi} \in C^\infty_{c}(\Rea)$ with both $\supp(1-\chi)\cap[-\Xi,\Xi]=\emptyset$ and $\supp(1-\widetilde\chi)\cap[-\Xi,\Xi]=\emptyset$ and
$\supp\widetilde{\chi}\cap \supp(1-\chi)=\emptyset$.
Then
\begin{equation}
\begin{gathered}
I- \chi(-k^{-2}\Delta_\Gamma)
=\big( I- \widetilde{\chi}\hDarg\big)\big(I- \chi(-k^{-2}\Delta_\Gamma)\big)+ O(k^{-\infty})_{\Psi^{-\infty}(\Gamma)},\\
I- \chi(-k^{-2}\Delta_\Gamma)= \big(I- \chi(-k^{-2}\Delta_\Gamma)\big)\big( I- \widetilde{\chi}\hDarg\big)+ O(k^{-\infty})_{\Psi^{-\infty}(\Gamma)}.
\end{gathered}
\label{e:HSparametrix}
\end{equation}
\end{corollary}
\bpf[Proof of Corollary \ref{cor:WFcutoff}]By Lemma \ref{lem:WFcutoff} and the assumption that $\supp\widetilde{\chi}\cap \supp(1-\chi)=\emptyset$,
\beqs
\WF\big( I- \chi(-k^{-2}\Delta_\Gamma)\big)\cap \WF(\widetilde{\chi}\hDarg)=\emptyset
\eeqs
Therefore,
\begin{align*}
\big(I-\widetilde{\chi}\hDarg\big)\big(I-\chi(-k^{-2}\Delta_\Gamma)\big)&=\big(I-\chi(-k^{-2}\Delta_\Gamma)\big) +\widetilde{\chi}\hDarg\big(I-\chi(-k^{-2}\Delta_\Gamma)\big)\\
&=\big(I-\chi(-k^{-2}\Delta_\Gamma)\big) +O(k^{-\infty})_{\Psi^{-\infty}(\Gamma)},
\end{align*}
and
\begin{align*}
\big(I-\chi(-k^{-2}\Delta_\Gamma)\big)\big(I-\widetilde{\chi}\hDarg\big)&=\big(I-\chi(-k^{-2}\Delta_\Gamma)\big) +\big(I-\chi(-k^{-2}\Delta_\Gamma)\big)\widetilde{\chi}\hDarg\\
&=\big(I-\chi(-k^{-2}\Delta_\Gamma)\big) +O(k^{-\infty})_{\Psi^{-\infty}(\Gamma)}.
\end{align*}
\epf

\section{
Convergence of projection methods in an abstract framework
}\label{sec:abstract}

\subsection{New abstract result on convergence of the projection method}

\begin{assumption}[Assumptions on $\operator$ and its high-frequency components]
\label{ass:abstract1}

\

(i) $\operator:= I+L$ is bounded and invertible on $H^s(\Gamma)$ for all $s\in \Rea$, with $\pert :H^s(\Gamma)\to H^{s+1}(\Gamma)$
and $k\mapsto \pert$ continuous.

(ii) For any $\chi \in C^\infty_{c}(\Rea;[0,1])$ with $\supp(1-\chi)\cap [-1,1]=\emptyset$, 
$(I-\chi\hDarg)\pert$ and $\pert(I-\chi\hDarg)$ are both in $\Psi^{-1}(\Gamma)$.

(iii) There is $L_{\max}>0$ such that for any $\chi \in C^\infty_{c}(\Rea;[0,1])$ with $\supp(1-\chi)\cap [-1,1]=\emptyset$,
\beqs
\sup_{T^*\Gamma} \big|1+\sigma^0 \big( \big(I- \chi \hDarg) \pert \big)(x',\xi')\big|^{-1}\leq L_{\max} 
\eeqs
\end{assumption}

\begin{theorem}[New abstract result on convergence of the projection method]\label{thm:MAT}
Suppose that $\operator$ satisfies Assumptions \ref{ass:abstract1} and \ref{ass:polyboundintro} \blue{and $\rho(k)$ is as in the latter assumption}.
Let $k_0>0$, $C_1>0$, let $V_N$ be a finite-dimensional subspace of $H^s_k(\Gamma)$ and let $P_N: H^s_k(\Gamma) \to V_N$ be a projection satisfying
\beq\label{e:MATproj}
\N{(I-P_N)}_{H^t_k\to H^q_k} \leq C_1 \left(\frac{k}{N}\right)^{t-q}
\eeq
for all $N\geq k$, 
\beq\label{e:approx_pairs}
\blue{(t,q)\in \{(t_{\max},s), (s+1,s),  (s,q_{\min}), (s+1,q_{\min}), (s,s)\},}
\eeq
with $q_{\min}\leq s\leq t_{\max}-1$.

Given $\e>0$, there are $c,C>0$ such that the following holds. For all $k\geq k_0,\,k\notin\cJ$, $f\in H^s_{k}(\Gamma)$,  and $N$ satisfying 
\beq\label{e:MATthres}
\left(\frac{k}{N}\right)^{t_{\max} -q_{\min}} \rho\blue{(k)} \leq c
\eeq
the solution $v_N \in V_N$ to
\beq\label{e:MATprojection}
(I+P_N \pert) v_N = P_N f
\eeq
exists, is unique, and satisfies the quasi-optimal error estimate
\begin{multline}
\label{e:MATqo}
\N{v-v_N}_{H^s_k(\Gamma)}
\\\leq \Bigg(  L_{\max} \|I-P_N\|^2_{\Hshts} +C\bigg(\left(\frac{k}{N}\right)^{s-q_{\min}}\rho\blue{(k)}+k^{-1}+k/N\bigg) \Bigg)
\N{(I-P_N)v}_{H^s_k(\Gamma)},
\end{multline}
where $v\in H^s_k(\Gamma)$ is the solution of
\beq\label{e:MATequation}
\operator v:=(I+L)v =f.
\eeq

Moreover, 
\begin{multline}
\label{e:MATqoHF}
\N{
(I-1_{[-1-\e,1+\e]}(-k^{-2}\Delta_\Gamma))
(v-v_N)}_{H^s_k(\Gamma)}
\\\leq \Bigg(  L_{\max} \|I-P_N\|^2_{\Hshts} +C\bigg(\left(\frac{k}{N}\right)^{t_{\max}-q_{\min}}\rho\blue{(k)}+k^{-1}+k/N\bigg) \Bigg)
\N{(I-P_N)v}_{H^s_k(\Gamma)}.
\end{multline}
\end{theorem}

\subsection{Results about microlocal properties of $\operator$ under Assumption \ref{ass:abstract1} needed for the proof of Theorem \ref{thm:MAT}}\label{sec:microlocalA}

\ble[Wavefront sets of $(I- \chi \hDarg) \pert$ and $\pert(I- \chi \hDarg)$]\label{lem:WF}
If $\operator$ satisfies Assumption \ref{ass:abstract1}, then for any $\chi \in C^\infty_{c}(\Rea;[0,1])$ with $\supp(1-\chi)\cap [-1,1]=\emptyset$ 
\beq\label{e:WFlemma}
\WF \Big( \big(I- \chi \hDarg\big) \pert \Big)\cup
\WF \Big(\pert \big(I- \chi \hDarg\big)\Big)
\subset
\overline{\big\{ (x',\xi') : 1-\chi(|\xi'|_g^2)>0\big\}}.
\eeq
\ele

\bpf
Let
\beqs
(x_0,\xi_0) \in
\big(\overline{\big\{ (x',\xi') : 1-\chi(|\xi'|_g^2)>0\big\}}\big)^c
\eeqs
Let $B$ be elliptic in a neighbourhood of $(x_0,\xi_0)$ and such that
\beq\label{e:WFB}
\WF(B) \Subset 
\big(\overline{\big\{ (x',\xi') : 1-\chi(|\xi'|_g^2)>0\big\}}\big)^c
\eeq
(such a $B$ exists since the set on the right is open). 
By the definition of $\WF$ (Definition \ref{def:WF}) it is sufficient to prove that
\beq\label{e:STPincl1}
B(I- \chi \hDarg) \pert=O(k^{-\infty})_{\Psi^{-\infty}(\Gamma)}\quad\tand\quad
B\pert (I- \chi \hDarg)=O(k^{-\infty})_{\Psi^{-\infty}(\Gamma)}.
\eeq
By \eqref{e:WFquant},
\beq\label{e:WFchi}
\WF  \big(I- \chi \hDarg\big) \subset
\overline{\big\{ (x',\xi') : 1-\chi(|\xi'|_g^2)>0\big\}}.
\eeq
Therefore, by \eqref{e:WFB}, $\WF (B) \cap \WF (I- \chi \hDarg) = \emptyset$, and thus
the first equation in \eqref{e:STPincl1} holds by \eqref{e:WFdisjoint}.


To prove the second equation in \eqref{e:STPincl1},
choose $\widetilde{\chi}\in C^\infty_{c}(\Rea^d;[0,1])$ with $\widetilde{\chi} \equiv 1$ on $[-1,1]$, 
$\supp \widetilde \chi \cap \supp (1-\chi)=\emptyset$, and 
\beqs
\WF(B) \subset
\big(\overline{\big\{ (x',\xi') : 1-\widetilde\chi(|\xi'|_g^2)>0\big\}}\big)^c;
\eeqs
such a choice is possible because of \eqref{e:WFB} (i.e., because there is space between $\WF(B)$ and where $\chi \not\equiv 1$).
Observe that this choice of $\widetilde{\chi}$ and the property \eqref{e:WFchi} (with $\chi$ replaced by $\widetilde{\chi}$) imply that
\beq\label{e:2crows1}
\WF(B) \cap \WF \big(I- \widetilde{\chi} \hDarg\big) = \emptyset.
\eeq
Then, using that $\supp \widetilde \chi \cap \supp (1-\chi)=\emptyset$, \eqref{e:WFquant}, 
\eqref{e:WFdisjoint}, and Part (ii) of Theorem \ref{thm:basicP},
\beq\label{e:2crows2}
B\pert \big(I- \chi \hDarg\big)= B\pert \big(I- \chi \hDarg\big)(I- \widetilde{\chi} \hDarg\big) + O(k^{-\infty})_{\Psi^{-\infty}(\Gamma)}.
\eeq
By \eqref{e:WF_prod} and the assumption that $\pert \big(I- \chi \hDarg\big)\in\Psi^{-1}(\Gamma)$ (in Part (ii) of Assumption \ref{ass:abstract1}),
\beqs
\WF\big(B\pert \big(I- \chi \hDarg\big)\big) \subset \WF(B).
\eeqs
This inclusion combined with \eqref{e:2crows1}, \eqref{e:WFdisjoint},
and \eqref{e:2crows2}, imply that $B\pert \big(I- \chi \hDarg\big)=O(k^{-\infty})_{\Psi^{-\infty}(\Gamma)}$, which proves the second equation in \eqref{e:STPincl1} and thus completes the proof.
\epf

\ble[Ellipticity of $ (I-\chi\hDarg)\operator$ on high frequencies]\label{lem:elliptic_HF}
If $\operator$ satisfies Assumption \ref{ass:abstract1}, then for any $\chi \in C^\infty_{c}(\Rea;[0,1])$ with $\supp(1-\chi)\cap [-1,1]=\emptyset$

(a) $(I-\chi\hDarg)\operator$ and $\operator(I-\chi\hDarg) \in \Psi^0(\Gamma)$ are elliptic on $\{ (x',\xi') : 1-\chi(|\xi'|_g^2)>0\}$, and

(b) $I + (I-\chi\hDarg)\pert$ and $I + \pert(I-\chi\hDarg) \in \Psi^0(\Gamma)$ are elliptic on $T^*\Gamma$.
\ele

\bpf
Part (b) follows immediately from Part (iii) of Assumption \ref{ass:abstract1} and the definition of ellipticity (Definition \ref{def:elliptic}).
For Part (a), given $\chi$ as in the statement, let $\widetilde{\chi} \in C^\infty_{c}(\Rea;[0,1])$ with $\supp(1-\widetilde\chi)\cap [-1,1]=\emptyset$
and $\supp(1-\chi)\cap \supp\widetilde\chi=\emptyset$. Then, by Part (a), $I + (1-\widetilde{\chi}\hDarg) \pert$ is elliptic on $T^*\Gamma$, and thus 
$(1-\chi\hDarg)(I + (1-\widetilde{\chi}\hDarg) \pert)$ is elliptic on $\{ (x',\xi') : 1-\chi(|\xi'|_g^2)>0\}$.
Since $(1-\chi)(1-\widetilde{\chi}) = (1-\chi)$, Part (a) follows.
\epf

\begin{corollary}[Ellipticity of $ (I-\chi\hDarg)\operator^*$ on high frequencies]
\label{cor:adjoint_elliptic}
If $\operator$ satisfies Assumption \ref{ass:abstract1}, then $(I-\chi\hDarg)\operator^*$ and $\operator^*(I-\chi\hDarg) \in \Psi^0(\Gamma)$ are elliptic on $\{ (x,\xi') : 1-\chi(|\xi'|_g^2)>0\}$.
\end{corollary}

\bpf
We prove the result for $(I-\chi\hDarg)\operator^*$, the proof for $\operator^*(I-\chi\hDarg)$ is analogous.
Since $(I-\chi\hDarg)\operator^* = (\operator(I-\chi\hDarg)^*)^*$, by the second equation in \eqref{e:multsymb} and the
definition of ellipticity (Definition \ref{def:elliptic}), it is sufficient to prove that
\beq\label{e:chest2}
\operator\big(I-\chi\hDarg\big)^*\quad\text{ is elliptic on } \quad\big\{ (x',\xi') : 1-\chi(|\xi'|_g^2)>0\big\}.
\eeq
Choose $\widetilde{\chi}\in C^\infty_{c}(\Rea;[0,1])$ with $\widetilde{\chi}\equiv 1$ on $[-1,1]$ and $\supp \widetilde{\chi} \cap \supp(1-\chi)=\emptyset$. Then, by \eqref{e:WFadjoint} and \eqref{e:WFdisjoint},
\beq\label{e:chest1}
\operator\big(I-\chi\hDarg\big)^* = \operator \big(I-\widetilde{\chi}\hDarg\big)\big(I-\chi\hDarg\big)^*.
\eeq
By Lemma \ref{lem:elliptic_HF}, $\operator  (I-\widetilde{\chi}\hDarg)$ is elliptic on $\{ (x,\xi') : 1-\widetilde{\chi}(|\xi'|_g^2)>0\}$. Since $\{ 1- \chi >0\} \subset \{ 1- \widetilde{\chi}>0\}$, $\operator  (I-\widetilde{\chi}\hDarg)$ is elliptic on $\{ (x,\xi') : 1-\chi(|\xi'|_g^2)>0\}$. Then \eqref{e:chest2} follows by \eqref{e:chest1}, both equations in \eqref{e:multsymb}, and the definition of ellipticity (Definition \ref{def:elliptic}).
\epf

\

Recall that, since the operator $\operator$ is a compact perturbation of the identity, $\rho$ is bounded below by a $k$-independent constant (see 
Lemma \ref{lem:inversebound} below).

\ble[Norm of $\operator^{-1}$ on $\Hsh$ bounded by norm on $\LtG$
]\label{lem:2}
Suppose that $\operator$ satisfies Assumption \ref{ass:abstract1} \blue{and let $\rho(k):=\|\operator^{-1}\|_{\LtGt}$}.
Then, given $s\in \mathbb{R}$ and $k_0>0$, there exists $C>0$ such that, for all $k\geq k_0$,
$\operator$ is invertible on $H^s_k(\Gamma)$ with
\beq\label{e:norminv2}
\N{\operator^{-1}}_{H^{s}_k(\Gamma)\rightarrow H^{s}_k(\Gamma)} \leq C\rho(k).
\eeq
\ele

\bpf
Because $\operator$ is invertible on $H^s(\Gamma)$ for all $k>0$ by Part (i) of Assumption \ref{ass:abstract1},
it is sufficient to prove the result for $k\geq k_0$ with $k_0$ sufficiently large.
We first prove this result for $s>0$.
Proving the bound \eqref{e:norminv2} is equivalent to proving
that if $\operator\phi = g$ with $\phi, g\in H^s_k(\Gamma)$, then
\beq\label{e:norminv3}
\N{\phi}_{H^s_k(\Gamma)}\leq C \rho(k) \N{g}_{H^s_k(\Gamma)}.
\eeq
Let $\chi,\widetilde{\chi}\in C^\infty_{c}(\Rea;[0,1])$ with both $\supp(1-\chi)\cap[-1,1]=\emptyset$ and $\supp(1-\widetilde\chi)\cap[-1,1]=\emptyset$, and $\supp \chi \cap \supp(1-\widetilde{\chi})=\emptyset$. 
Thus, by \eqref{e:WFchi} and Lemma \ref{lem:elliptic_HF},
\beq\label{e:2crows3}
(I-\chi\hDarg)\operator\quad \text{ is elliptic on } \quad\WF\big( I- \widetilde{\chi}\hDarg\big).
\eeq

The idea of the proof is to write
\beqs
\phi = \big(1-\widetilde{\chi}\hDarg\big)\phi + \widetilde{\chi}\hDarg\phi,
\eeqs
estimate the high-frequency components of $\phi$ (i.e., $(1-\widetilde{\chi}\hDarg)\phi$) using ellipticity of $(I-\chi\hDarg)\operator$ on high frequencies 
and estimate the low-frequency components
(i.e., $\widetilde{\chi}\hDarg\phi$) by the property \eqref{e:frequencycutoff2} of frequency cut-offs.

Indeed, by \eqref{e:2crows3} and \eqref{e:ellipticestimate}, given $s\in \Rea$ and $M>0$, there exists $C,C',k_0>0$ such that, for all $k\geq k_0$,
\begin{align}\nonumber
\N{(1-\widetilde{\chi}\hDarg)\phi}_{H^s_k(\Gamma)} &\leq C\Big( \N{(I-\chi\hDarg)g}_{H^s_k(\Gamma)} + k^{-M} \N{\phi}_{H^{s}_k(\Gamma)}\Big) \\
&\leq C'\Big( \N{g}_{H^s_k(\Gamma)} + k^{-M} \N{\phi}_{H^{s}_k(\Gamma)}\Big).\label{e:lem22}
\end{align}
For the bound on the low-frequencies, by \eqref{e:frequencycutoff2}, given $s\geq 0$ and $k_0>0$, there exists $C''>0$ such that, for all $k\geq k_0$,
\begin{align}\nonumber
\N{\widetilde{\chi}\hDarg \phi }_{H^s_k(\Gamma)} \leq C'' \N{\phi}_{\LtG} &\leq C''\rho(k) \N{g}_{\LtG}\\
& \leq C''
\rho(k) \N{g}_{H^s_k(\Gamma)}\label{e:lem23};
\end{align}
the result \eqref{e:norminv3} then follows by combining \eqref{e:lem22} and \eqref{e:lem23} and reducing $\hsc_0$ if necessary to absorb the
$k^{-M} \N{\phi}_{L^2(\Gamma)}$ term into the left-hand side.

Since $\|\operator^{-1}\|_{H^{s}_k(\Gamma) \to H^{s}_k(\Gamma)} = \|(\operator^*)^{-1} \|_{H^{-s}_k(\Gamma) \to H^{-s}_k(\Gamma)}$, the result for $s<0$ follows by applying the above argument to $\operator^*$ using Corollary \ref{cor:adjoint_elliptic}.
\epf

\

We record the following simple corollary of Lemma \ref{lem:2}.
\begin{corollary}\label{cor:invertMe}
Suppose that $\operator$ satisfies Assumption \ref{ass:abstract1}.
Then, given $s\in\Rea$ and $k_0>0$, there exists $C>0$ such that, for all $k\geq k_0$,
\begin{equation}
\|L(I+L)^{-1}\|_{\Hshts}\leq C\rho(k),
\end{equation}
\end{corollary}
\bpf
Since 
$$
\pert (I+\pert)^{-1}= I-(I+L)^{-1}
$$
the result follows from~\eqref{e:norminv2}.
\epf

\

A result analogous to Lemma \ref{lem:2} about $\operator$ (as opposed to $\operator^{-1}$) also holds. Although we only use it in \S\ref{sec:approx} below, we state and prove it here because of its similarity with Lemma \ref{lem:2}.

\ble[Norm of $\operator$ on $\Hsh$ bounded by norm on $\LtG$]\label{lem:normA}
Suppose that $\operator$ satisfies Assumption \ref{ass:abstract1}. Then, given $s\in \Rea, k_0>0$, there exists $C>0$ such that
\beqs
\N{\operator}_{\Hsht} \leq C \Big( 1+ \N{\operator}_{\LtGt}\Big) \quad\tfa k\geq k_0.
\eeqs
\ele

\bpf
Let  $\chi \in C^\infty_{c}(\Rea;[0,1])$ with $\chi\equiv 1$ on a neighbourhood of $[-1,1]$.
If $v\in H^s(\Gamma)$ with $s>0$, then by \eqref{e:frequencycutoff2}, the fact that $\big(I-\chi\hDarg\big)\operator \in \Psi_h^0(\Gamma)$ (by, e.g., Part (i) of Lemma \ref{lem:elliptic_HF}), and Part (ii) of Theorem \ref{thm:basicP}, there exist $C_1, C_2>0$ such that, for all $k\geq k_0$,
\begin{align*}
\N{\operator v}_{\Hsh} &\leq \N{\chi\hDarg \operator v}_{\Hsh} + \N{\big(I-\chi\hDarg\big)\operator v}_{\Hsh},\\
&\leq C_1 \N{\operator}_{\LtGt} \N{v}_{\LtG} + C_2\N{v}_{\Hsh};
\end{align*}
the result (with $s>0$) then follows since $\|v\|_{\LtG}\leq \|v\|_{\Hsh}$.

Similarly, if $v\in H^s(\Gamma)$ with $s<0$, then 
there exist $C_1, C_2>0$ such that, for all $k\geq k_0$,
\begin{align*}
\N{\operator v}_{\Hsh} &\leq \N{\operator \chi\hDarg v}_{\Hsh} + \N{\operator \big(I-\chi\hDarg\big)v}_{\Hsh},\\
&\leq \N{\operator \chi\hDarg v}_{\LtG} + C_2\N{v}_{\Hsh}\leq C_1 \N{\operator}_{\LtGt} \N{v}_{\Hsh} + C_2\N{v}_{\Hsh},
\end{align*}
which is the result with $s<0$.
\epf

\

Finally, we use the following description of $\operator^{-1}$ acting on high frequencies.
\ble[The high-frequency components of $\operator^{-1}$]\label{lem:HFinverse}
Suppose that $\operator$ satisfies Assumptions \ref{ass:abstract1} and \ref{ass:polyboundintro}. Let $k_0>0$ and
$\chi_1,\chi_2 \in C^\infty_{c}$ with $\supp(1-\chi_j) \cap [-1,1]=\emptyset$ 
and $\supp \chi_1 \cap \supp(1-\chi_2)=\emptyset$.

Then for all $k>k_0$, $k\notin \cJ$,  $I+ \pert(I-\chi_1\hDarg)$ and $I+ (I-\chi_1\hDarg)\pert$ are both in $\Psi^0(\Gamma)$, are elliptic on $T^*\Gamma$,
\beq\label{e:inverse_cutoff_right}
(I+\pert)^{-1}\big(I-\chi_2\hDarg\big) = \big( I+ \pert\big(I-\chi_1\hDarg\big)\big)^{-1}\big(I-\chi_2\hDarg\big) + O(k^{-\infty})_{
\Psi^{-\infty}}
\eeq
and
\beq\label{e:inverse_cutoff_left}
\big(I-\chi_2\hDarg\big) (I+\pert)^{-1}= \big(I-\chi_2\hDarg\big)\big( I+ \big(I-\chi_1\hDarg\big)\pert\big)^{-1} + O(k^{-\infty})_{
\Psi^{-\infty}}.
\eeq
Furthermore, given $s \in \Rea$ and $k_0>0$, there is $C>0$ such that for all $k\geq k_0$
\begin{equation}
\label{e:invEst}
\begin{gathered}
\big\|\big(I+L\big(I-\chi_1\hDarg\big)\big)^{-1}\big\|_{H_{k}^{s}(\Gamma)\to H_{k}^{s}(\Gamma)}\leq L_{\max}+Ck^{-1},\\\
\big\|\big(I+\big(I-\chi_2\hDarg\big)L\big)^{-1}\big\|_{H_{k}^{s}(\Gamma)\to H_{k}^{s}(\Gamma)}\leq L_{\max}+Ck^{-1}.
\end{gathered}
\end{equation}
\ele

\bpf
We prove the result \eqref{e:inverse_cutoff_right} with the cutoffs on the right; the proof of \eqref{e:inverse_cutoff_left} with the cutoffs on the left is analogous.
By Lemma \ref{lem:elliptic_HF}, $I+ L(I-\chi_1\hDarg)\in \Psi^0(\Gamma)$ is elliptic on $T^* \Gamma$ and hence, by Part (i) of Corollary \ref{cor:elliptic}, is invertible with inverse satisfying
$$
|\sigma^0\big((I+ L(I-\chi_1\hDarg))^{-1}\big)|=1/|\sigma^0\big(I+ L(I-\chi_1\hDarg)\big)|.
$$
In particular, by Part (iii) of Assumption~\ref{ass:abstract1}, and Lemma~\ref{l:HsEstimates}, the first estimate in~\eqref{e:invEst} holds.
Then
\begin{align*}
(I+L)^{-1} - \big( I+ L \big(I-\chi_1\hDarg\big)\big)^{-1} = -(I+L)^{-1} L \chi_1\hDarg \big( I+ L \big(I-\chi_1\hDarg\big)\big)^{-1}.
\end{align*}
Multiplying by $(1- \chi_2\hDarg)$, we have that
\begin{align*}
&(I+L)^{-1}\big(1- \chi_2\hDarg\big)  - \big( I+ L \big(I-\chi_1\hDarg\big)\big)^{-1}\big(1- \chi_2\hDarg\big)  \\
&\hspace{2cm}= -(I+L)^{-1} L \chi_1\hDarg \big( I+ L \big(I-\chi_1\hDarg\big)\big)^{-1}\big(1- \chi_2\hDarg\big).
\end{align*}
By Lemma \ref{lem:2} and Assumption \ref{ass:polyboundintro}, $(I+L)^{-1}: H^s_k(\Gamma)\to H^s_k(\Gamma)$ is polynomially bounded for all $s\in \Rea$. Then, the fact that $\supp \chi_1 \cap \supp(1-\chi_2)=\emptyset$ and the properties  \eqref{e:WF_prod}, \eqref{e:WFdisjoint}, and \eqref{e:WFchi} imply that the right-hand side of the last displayed equality is
$ O(k^{-\infty})_{\Psi^{-\infty}}$, and the result follows.
\epf

Combining Lemma \ref{lem:HFinverse} with Corollary \ref{cor:WFcutoff}, we obtain the following result.
\begin{corollary}\label{cor:nicebound1}
Suppose that $\operator$ satisfies Assumptions \ref{ass:abstract1} and \ref{ass:polyboundintro} and 
$\chi \in C^\infty_{c}(\Rea;[0,1])$ with $\supp(1-\chi) \cap [-1,1]=\emptyset$. Given $s\in \Rea$ and $k_0>0$ there exists $C>0$ such that for all $k\geq k_0$
\beqs
\big\| L (I+L)^{-1} \big( I - \chi(\hsc^2 \Delta_\Gamma)\big) \big\|_{H^{s}_k \to H^{s+1}_k} \leq C.
\eeqs
\end{corollary}
\bpf
Let $\widetilde\chi \in C^\infty_{c}(\Rea;[0,1])$ be such that $\supp(1-\widetilde\chi) \cap [-1,1]=\emptyset$ and 
$\supp(1-\chi)\cap \supp \widetilde\chi=\emptyset$. 
Then, by Corollary \ref{cor:WFcutoff} and Assumption \ref{ass:polyboundintro}
\begin{align*}
L (I+L)^{-1} \big( I - \chi(\hsc^2 \Delta_\Gamma)\big) = L (I+L)^{-1} \big( I - \widetilde\chi\hDarg\big)\big( I - \chi(\hsc^2 \Delta_\Gamma)\big) + O(k^{-\infty})_{
\Psi^{-\infty}}.
\end{align*}
On the one hand,
\begin{align*}
&\widetilde\chi\hDarg L (I+L)^{-1} \big( I - \widetilde\chi\hDarg\big)\big( 1 - \chi(\hsc^2 \Delta_\Gamma)\big)\\
&\qquad\qquad=\widetilde\chi\hDarg\big( I - (I+L)^{-1}\big)\big( 1 - \widetilde\chi\hDarg\big)\big( 1 - \chi(\hsc^2 \Delta_\Gamma)\big),
\end{align*}
and the bound 
\beq\label{e:STY1}
\big\|\widetilde\chi\hDarg L (I+L)^{-1} \big( 1 - \chi(\hsc^2 \Delta_\Gamma)\big)\big\|_{H^{s}_k\to H^{s+1}_k}\leq C
\eeq
then follows from \eqref{e:inverse_cutoff_right}, \eqref{e:invEst}, and \eqref{e:frequencycutoff2}.

On the other hand, 
\beq\label{e:STY2}
\big\|\big(I-\widetilde\chi\hDarg \big)L (I+L)^{-1} \big( I - \widetilde\chi\hDarg\big)\big( I - \chi(\hsc^2 \Delta_\Gamma)\big)\big\|_{H^{s}_k\to H^{s+1}_k}\leq C
\eeq
by Part (ii) of Assumption \ref{ass:abstract1}, \eqref{e:inverse_cutoff_right}, and \eqref{e:invEst}. The result then follows by combining \eqref{e:STY1} and \eqref{e:STY2}.
\epf

\subsection{The idea of the proof of Theorem \ref{thm:MAT}}

The proof of Theorem \ref{thm:MAT} starts by using the following lemma.

\ble[Quasi-optimality in terms of the norm of the discrete inverse]\label{lem:QOabs}
If $P_N: \Hsh \to V_N$ is a projection and $I+ P_N\pert : H^s_k(\Gamma)\to H^s_k(\Gamma)$ is invertible, then
the solution $v_N\in V_N$ to \eqref{e:MATprojection} exists, is unique, and satisfies
\beq\label{e:QOabs}
v-v_N= (I+P_N \pert)^{-1}(I-P_N) (I-P_N)v.
\eeq
\ele

\bpf
We first consider the equation \eqref{e:MATprojection} as an equation in $H^s_k(\Gamma)$.
Since $I+P_N \pert:H^s_k(\Gamma)\to H^s_k(\Gamma)$ is invertible, the solution $v_N$ to \eqref{e:MATprojection} in $H^s_k(\Gamma)$ exists and is unique as an element of $H^s_k(\Gamma)$. Applying $(I-P_N)$ to \eqref{e:MATprojection}, we see that $(I-P_N)v_N=0$ and thus the solution $v_N\in V_N$; i.e., the equation \eqref{e:MATprojection} has a unique solution in $V_N$.
Then, by \eqref{e:MATequation} and \eqref{e:MATprojection},
\begin{align*}
(I+P_N \pert) (v-v_N) = (I+P_N \pert) v - P_N f & = v + P_N \pert v - P_N\big( (I+\pert)v\big)  = (I-P_N)v.
\end{align*}
Therefore, since $ (I-P_N)=  (I-P_N)^2$,
\begin{align*}
(v-v_N) = (I+P_N \pert)^{-1} (I-P_N)v=(I+P_N \pert)^{-1} (I-P_N)(I-P_N)v.
\end{align*}
\epf

To obtain conditions under which $I+ P_N\pert : H^s_k(\Gamma)\to H^s_k(\Gamma)$ is invertible (and thus under which we can use Lemma \ref{lem:QOabs}),
we write
\beq\label{e:JeffFav1}
(I+P_N\pert) = I +\pert + (P_N -I)\pert = \Big( I+ (P_N-I)\pert (I+\pert)^{-1}\Big) (I+\pert),
\eeq
and study when $ I+ (P_N-I)\pert (I+\pert)^{-1}$ is invertible. 
We then use that 
\beqs
\Big( I+ (P_N-I) \pert(I+\pert)^{-1}\Big)^{-1}(I-P_N) = (I-P_N)\Big( I+ (P_N-I) \pert(I+\pert)^{-1}\Big)^{-1}(I-P_N)
\eeqs
i.e., we can add a $(I-P_N)$ on the left.
Indeed, if $( I+ (P_N-I) \pert(I+\pert)^{-1})v= (I-P_N)f$, then $P_N v =0$ so that $v=(I-P_N)v$.
Therefore the operator on the right-hand side of \eqref{e:QOabs} can be written as 
\beq\label{e:thething}
(I+P_N \pert )^{-1}(I-P_N) = (I+\pert)^{-1} (I-P_N)\Big( I+ (P_N-I) \pert(I+\pert)^{-1}\Big)^{-1}(I-P_N).
\eeq

The first natural attempt is to show that $ I+ (P_N-I)\pert (I+\pert)^{-1}$ is invertible is to show that
\begin{equation}
\label{e:Neumann0}
\|(P_N-I)\pert (I+\pert)^{-1}\|_{\Hshts}<\frac{1}{2}, 
\end{equation}
in which case $\|(I+(P_N-I)\pert (I+\pert)^{-1})^{-1}\|_{\Hshts}\leq 2$. 

To understand when the condition \eqref{e:Neumann0} holds, given $\e>0$, we define low- and high-frequency projectors, 
\beq\label{e:highandlow}
\Pi_{\mathscr{L}}:= 1_{[-1-\e,1+\e]}(-k^{-2}\Delta_\Gamma)\quad\tand\quad
\Pi_{\mathscr{H}}:=(I-\Pi_{\mathscr{L}}),
\eeq
 and decompose $H_{k}^{s}$ into an orthogonal sum $\Pi_{\mathscr{L}}H_{k}^{s}\oplus \Pi_{\mathscr{H}}H_{k}^{s}$. 
That is, we write
\beq\label{e:split1}
\begin{aligned}
&(P_N-I)\pert (I+\pert)^{-1}= \begin{pmatrix}\Pi_{\mathscr{L}}(P_N-I)\pert (I+\pert)^{-1}\Pi_{\mathscr{L}}&\Pi_{\mathscr{L}}(P_N-I)\pert (I+\pert)^{-1}\Pi_{\mathscr{H}}\\

\Pi_{\mathscr{H}}(P_N-I)\pert (I+\pert)^{-1}\Pi_{\mathscr{L}}&\Pi_{\mathscr{H}}(P_N-I)\pert (I+\pert)^{-1}\Pi_{\mathscr{H}}\end{pmatrix}.
\end{aligned}
\eeq

Using the results in \S\ref{sec:microlocalA}, we show below (see Lemma \ref{lem:split2}) that 
\begin{equation}
\label{e:blockDecomp1}
\begin{aligned}
&\big\|(P_N-I)\pert (I+\pert)^{-1}\big\|_{\Hshts} \leq 
C\begin{pmatrix}(k/N)^{t_{\max}-q_{\min}}\rho &(k/N)^{s+1-q_{\min}}\\
(k/N)^{t_{\max}-s}\rho&(k/N)
\end{pmatrix},
\end{aligned}
\end{equation}
where here and below we use the notation that
$$
\|A\|_{\Hshts}\leq C_1\begin{pmatrix} a_{\mathscr{LL}}&a_{\mathscr{LH}}\\a_{\mathscr{HL}}&a_{\mathscr{HH}}\end{pmatrix}, \quad\text{if}\quad
A=
\begin{pmatrix}
 \Pi_{\mathscr{L}} A \Pi_{\mathscr{L}}
 &\Pi_{\mathscr{L}} A \Pi_{\mathscr{H}}
 \\\Pi_{\mathscr{H}} A \Pi_{\mathscr{L}}&\Pi_{\mathscr{H}} A \Pi_{\mathscr{H}}\end{pmatrix}, 
$$
with
$$
\begin{aligned}
\|\Pi_{\mathscr{L}} A \Pi_{\mathscr{L}}\|_{\Hshts}&\leq C_1a_{\mathscr{LL}}&\|\Pi_{\mathscr{H}} A \Pi_{\mathscr{L}}\|_{\Hshts}&\leq C_1a_{\mathscr{HL}}\\
\|\Pi_{\mathscr{L}} A \Pi_{\mathscr{H}}\|_{\Hshts}&\leq C_1a_{\mathscr{LH}}&\|\Pi_{\mathscr{H}} A \Pi_{\mathscr{H}}\|_{\Hshts}&\leq C_1a_{\mathscr{HH}}.
\end{aligned}
$$
(Once can then easily check that if $\|A\|\leq C_1 \mathscr{A}$ and $\|B\|\leq C_2 \mathscr{B}$, then $\|AB\|\leq C_1 C_2 \mathscr{A}\mathscr{B}$.)

The largest matrix entry on the right-hand side of \eqref{e:blockDecomp1} is $(k/N)^{t_{\max}-s} \rho$, and demanding that this be sufficiently small is a more restrictive condition than \eqref{e:MATthres}.

We instead change the norm on $H_{k}^{s}$ using an invertible operator $\mathcal{C}:\Hshts$. More precisely, we observe that 
$$
 I+ (P_N-I)\pert (I+\pert)^{-1}= \mathcal{C}^{-1} \Big(I +\mathcal{C}(P_N-I)\pert (I+\pert)^{-1}\mathcal{C}^{-1}\Big)\mathcal{C}.
$$
We now choose $\mathcal{C}$ to be a diagonal matrix with entries chosen so that the bottom-left and top-right entries of 
$\mathcal{C}(P_N-I)\pert (I+\pert)^{-1}\mathcal{C}^{-1}$ are equal; i.e., compared to $(P_N-I)\pert (I+\pert)^{-1}$ we make the bottom left entry smaller, at the cost of making the top right entry bigger.
The choice of $\mathcal{C}$ that achieves this is 
\begin{equation}
\mathcal{C}:=\begin{pmatrix} I&0\\0&(k/N)^{s-t_{\max}/2+1/2-q_{\min}/2}\rho ^{-1/2}I\end{pmatrix},
\label{e:newNorm}
\end{equation}
with then
\begin{equation}
\label{e:blockDecomp}
\begin{aligned}
&\|\mathcal{C}(P_N-I)\pert (I+\pert)^{-1}\mathcal{C}^{-1}\|_{\Hshts}\\
&\hspace{2cm} \leq 
C\begin{pmatrix}(k/N)^{t_{\max}-q_{\min}}\rho &
\Big( (k/N)^{t_{\max}-q_{\min}}\rho \Big)^{1/2} (k/N)^{1/2}
\\
\Big( (k/N)^{t_{\max}-q_{\min}}\rho \Big)^{1/2} (k/N)^{1/2}
&(k/N)
\end{pmatrix};
\end{aligned}
\end{equation}
observe that the largest entry in the matrix on the right-hand side is the top left entry, which is $(k/N)^{t_{\max}-q_{\min}}\rho$ as desired. 
The condition \eqref{e:MATthres} then ensures that 
\begin{equation}
\label{e:Neumann1}
\big\|\mathcal{C}(P_N-I)\pert (I+\pert)^{-1}\mathcal{C}^{-1}\big\|_{\Hshts}<\frac{1}{2},
\end{equation}
and so $I+(P_N-I)\pert (I+\pert)^{-1}$ is invertible and 
\beq\label{e:magicC1}
\Big(I+(P_N-I)\pert (I+\pert)^{-1}\Big)^{-1}= \mathcal{C}^{-1}\Big(I+\mathcal{C}(P_N-I)\pert (I+\pert)^{-1}\mathcal{C}^{-1}\Big)^{-1}\mathcal{C};
\eeq
The result then follows by decomposing $(I+L)^{-1}(I-P_N)$ into high- and low-frequency components similar to in \eqref{e:split1} (see Lemma \ref{lem:split1} below).

\subsection{Proof of Theorem \ref{thm:MAT}}\label{sec:proofMAT}

First observe that with $\Pi_{\mathscr{L}}$ and $\Pi_{\mathscr{H}}$ defined by \eqref{e:highandlow}, 
by the definition of $\|\cdot\|_{H^{s}_k}$ \eqref{e:weightedNormGamma},
\beq\label{e:projectionbounds}
\|\Pi_{\mathscr{L}}\|_{H_k^{-N}\to H_k^{N}}\leq C,\quad\text{ so that }\quad \|\Pi_{\mathscr{H}}\|_{H_k^t\to H_k^t}\leq C.
\eeq

\ble[Bounds on the decomposition of $(P_N-I)\pert (I+\pert)^{-1}$]\label{lem:split2}
Let $k_0>0$, $q_{\min}\leq s\leq t_{\max}-1$, $\operator$ satisfy Assumptions \ref{ass:abstract1} and \ref{ass:polyboundintro}, $P_N: H^{s}_k(\Gamma) \to V_N$ be a projection satisfying \eqref{e:MATproj}
with $(t,q)$ the pairs
\beqs
(t_{\max},s),\quad (t_{\max},q_{\min}),\quad (s+1,q_{\min}) \quad\tand\quad \quad (s+1,s), 
\eeqs
Then there is $C>0$ such that for all $k>k_0$, $k\notin\cJ$, and $N\geq k$.
\beq\label{e:yoga3}
\begin{aligned}
\|\Pi_{\mathscr{L}} (I-P_N)L(I+L)^{-1}\Pi_{\mathscr{L}}\|_{\Hshts}&\leq C(k/N)^{t_{\max}-q_{\min}}\rho ,\\
\|\Pi_{\mathscr{H}} (I-P_N)L(I+L)^{-1}\Pi_{\mathscr{L}}\|_{\Hshts}&\leq C(k/N)^{t_{\max}-s}\rho ,\\
\|\Pi_{\mathscr{L}} (I-P_N)L(I+L)^{-1}\Pi_{\mathscr{H}}\|_{\Hshts}&\leq C(k/N)^{s+1-q_{\min}},\\
\|\Pi_{\mathscr{H}} (I-P_N)L(I+L)^{-1}\Pi_{\mathscr{H}}\|_{\Hshts}&\leq C(k/N).
\end{aligned}
\eeq
\ele

\bpf
For the term involving two low frequency projections, 
by \eqref{e:projectionbounds}, Corollary \ref{cor:invertMe}  and~\eqref{e:MATproj} with $(t,q)=(t_{\max},q_{\min})$,
\begin{align*}
&\|\Pi_{\mathscr{L}} (I-P_N)L(I+L)^{-1}\Pi_{\mathscr{L}}\|_{\Hshts}\\
&\leq \|\Pi_{\mathscr{L}}\|_{H_k^{q_{\min}}\to H_{k}^{s}}\|(I-P_N)\|_{H_k^{t_{\max}}\to H_k^{q_{\min}}}\|L(I+L)^{-1}\|_{H_k^{t_{\max}}\to H_k^{t_{\max}}}\|\Pi_{\mathscr{L}}\|_{H_{k}^{s}\to H_k^{t_{\max}}}\\
&\leq C(k/N)^{t_{\max}-q_{\min}}\rho.
\end{align*}
For the low-to-high frequency term, by \eqref{e:projectionbounds}, Corollary \ref{cor:invertMe}  and~\eqref{e:MATproj} with $(t,q)=(t_{\max},s)$,
\begin{align*}
&\|\Pi_{\mathscr{H}} (I-P_N)L(I+L)^{-1}\Pi_{\mathscr{L}}\|_{\Hshts}\\
&\leq \|\Pi_{\mathscr{H}}\|_{H_{k}^{s}\to H_{k}^{s}}\|(I-P_N)\|_{H_k^{t_{\max}}\to H_{k}^{s}}\|L(I+L)^{-1}\|_{H_k^{t_{\max}}\to H_k^{t_{\max}}}\|\Pi_{\mathscr{L}}\|_{H_{k}^{s}\to H_k^{t_{\max}}}\\
&\leq C(k/N)^{t_{\max}-s}\rho.
\end{align*}
For the high-to-low frequency term, let $\chi \in C_{c}^\infty((-1-\epsilon,1+\epsilon))$ 
with $\supp (1-\chi)\cap [-1,1]=\emptyset$ and observe that, by the definition of $\Pi_{\mathscr{H}}$ \eqref{e:highandlow}, $(1-\chi(-k^{-2}\Delta_\Gamma))\Pi_{\mathscr{H}}=\Pi_{\mathscr{H}}$. Using this, along with \eqref{e:projectionbounds}, Corollary \ref{cor:nicebound1},  and~\eqref{e:MATproj} with $(t,q)=(s+1,q_{\min})$, we obtain that
\begin{align*}
&\|\Pi_{\mathscr{L}} (I-P_N)L(I+L)^{-1}\Pi_{\mathscr{H}}\|_{\Hshts}\\
&\leq \|\Pi_{\mathscr{L}}\|_{H_k^{q_{\min}}\to H_{k}^{s}}\|(I-P_N)\|_{H_k^{s+1}\to H_k^{q_{\min}}}\|L(I+L)^{-1}(1-\chi(-k^{-2}\Delta_\Gamma))\|_{H_{k}^{s}\to H_k^{s+1}}\|\Pi_{\mathscr{H}}\|_{H_{k}^{s}\to H_{k}^{s}}\\
&\leq C(k/N)^{s+1-q_{\min}}\|L(I+L)^{-1}(1-\chi(-k^{-2}\Delta_\Gamma))\|_{H_{k}^{s}\to H_k^{s+1}}\\
&\leq C(k/N)^{s+1-q_{\min}}.
\end{align*}
Finally, for the high-to-high frequency term, by \eqref{e:projectionbounds}, Corollary \ref{cor:nicebound1},  and~\eqref{e:MATproj} with $(t,q)=(s+1,s)$,
\begin{align*}
&\|\Pi_{\mathscr{H}} (I-P_N)L(I+L)^{-1}\Pi_{\mathscr{H}}\|_{\Hshts}\\
&\leq \|\Pi_{\mathscr{H}}\|_{H_{k}^{s}\to H_{k}^{s}}\|(I-P_N)\|_{H_k^{s+1}\to H_{k}^{s}}\|L(I+L)^{-1}(1-\chi(-k^{-2}\Delta_\Gamma))\|_{H_{k}^{s}\to H_k^{s+1}}\|\Pi_{\mathscr{H}}\|_{H_{k}^{s}\to H_{k}^{s}}\\
&\leq C(k/N).
\end{align*}
\epf

\ble[Bounds on the decomposition of $(I+\pert)^{-1}(I-P_N)$]\label{lem:split1}
Let $k_0>0$, $q_{\min}\leq s\leq t_{\max}$, $\operator$ satisfy Assumptions \ref{ass:abstract1} and \ref{ass:polyboundintro}, $P_N: H^{s}_k(\Gamma) \to V_N$ be a projection satisfying \eqref{e:MATproj}
with $(t,q)$ the pairs
\beqs
(t_{\max},s),\quad (t_{\max},q_{\min}),\quad (s,q_{\min}) \quad\tand\quad (s,s), 
\eeqs
Then there is $C>0$ such that for all $k>k_0$, $k\notin\cJ$, and $N\geq k$
\beq\label{e:yoga4}
\begin{aligned}
\|\Pi_{\mathscr{L}} (I+L)^{-1}(I-P_N)\Pi_{\mathscr{L}}\|_{\Hshts}&\leq C(k/N)^{t_{\max}-q_{\min}}\rho ,\\
\|\Pi_{\mathscr{H}} (I+L)^{-1}(I-P_N)\Pi_{\mathscr{L}}\|_{\Hshts}&\leq C(k/N)^{t_{\max}-s},\\
\|\Pi_{\mathscr{L}} (I+L)^{-1}(I-P_N)\Pi_{\mathscr{H}}\|_{\Hshts}&\leq C(k/N)^{s-q_{\min}}\rho ,\\
\|\Pi_{\mathscr{H}} (I+L)^{-1}(I-P_N)\Pi_{\mathscr{H}}\|_{\Hshts}&\leq \big(L_{\max}+Ck^{-1}\big)\|I-P_N\|_{\Hshts}.
\end{aligned}
\eeq
\ele

\bpf
For the low-to-low frequency term, by \eqref{e:projectionbounds}, Lemma \ref{lem:2},  and~\eqref{e:MATproj} with $(t,q)=(t_{\max},q_{\min})$,
\begin{align*}
&\|\Pi_{\mathscr{L}} (I+L)^{-1}(I-P_N)\Pi_{\mathscr{L}}\|_{\Hshts}\\
&\leq \|\Pi_{\mathscr{L}}\|_{H_k^{q_{\min}}\to H_{k}^{s}}\|(I+L)^{-1}\|_{H_k^{q_{\min}}\to H_k^{q_{\min}}}\|(I-P_N)\|_{H_k^{t_{\max}}\to H_k^{q_{\min}}}\|\Pi_{\mathscr{L}}\|_{H_{k}^{s}\to H_k^{t_{\max}}}\\
&\leq C(k/N)^{t_{\max}-q_{\min}}\rho.
\end{align*}
For the low-to-high frequency term, let $\chi \in C_{c}^\infty((-2,2);[0,1])$ with $\supp (1-\chi)\cap [-1,1]=\emptyset$ and observe that 
\beq\label{e:Dean2}
\Pi_{\mathscr{H}}=\Pi_{\mathscr{H}}(1-\chi(-k^{-2}\Delta_\Gamma)).
\eeq
By Corollary \ref{cor:WFcutoff}, 
\eqref{e:inverse_cutoff_left},~\eqref{e:invEst},
\beq\label{e:Dean1}
\|(1-\chi(-k^{-2}\Delta_\Gamma))(I+L)^{-1}\|_{\Hshts}\leq L_{\max} + Ck^{-1};
\eeq
we record for later that analogous arguments also show that 
\beq\label{e:Dean1a}
\|(I+L)^{-1}(1-\chi(-k^{-2}\Delta_\Gamma))\|_{\Hshts}\leq L_{\max} + Ck^{-1}.
\eeq
Therefore, by \eqref{e:Dean2}, \eqref{e:Dean1}, \eqref{e:projectionbounds},  and~\eqref{e:MATproj} with $(t,q)=(t_{\max},s)$, we obtain that
\begin{align*}
&\|\Pi_{\mathscr{H}} (I+L)^{-1}(I-P_N)\Pi_{\mathscr{L}}\|_{\Hshts}\\
&\leq \|\Pi_{\mathscr{H}}\|_{H_{k}^{s}\to H_{k}^{s}}\|(1-\chi(-k^{-2}\Delta_\Gamma))(I+L)^{-1}\|_{\Hshts}\|(I-P_N)\|_{H_k^{t_{\max}}\to H_{k}^{s}}\|\Pi_{\mathscr{L}}\|_{H_{k}^{s}\to H_k^{t_{\max}}}\\
&\leq C(k/N)^{t_{\max}-s}.
\end{align*}
For the high-to-low frequency term, by \eqref{e:projectionbounds}, Lemma \ref{lem:2},  and~\eqref{e:MATproj} with $(t,q)=(s,q_{\min})$, 
\begin{align*}
&\|\Pi_{\mathscr{L}} (I+L)^{-1}(I-P_N)\Pi_{\mathscr{H}}\|_{\Hshts}\\
&\leq \|\Pi_{\mathscr{L}}\|_{H_k^{q_{\min}}\to H_{k}^{s}}\|(I-P_N)\|_{H_{k}^{s}\to H_k^{q_{\min}}}\|(I+L)^{-1}\|_{H_{k}^{s}\to H_{k}^{s}}\|\Pi_{\mathscr{H}}\|_{H_{k}^{s}\to H_{k}^{s}}\\
&\leq C(k/N)^{s-q_{\min}}\|(I+L)^{-1}\|_{H_{k}^{s}\to H_{k}^{s}}\\
&\leq C(k/N)^{s-q_{\min}}\rho.
\end{align*}
Finally, for the high-to-high frequency term, by \eqref{e:Dean1}, and the fact that $\|\Pi_{\mathscr{H}}\|_{\Hshts}=1$,
\begin{align*}
&\|\Pi_{\mathscr{H}} (I+L)^{-1}(I-P_N)\Pi_{\mathscr{H}}\|_{\Hshts}\\
&\leq \|\Pi_{\mathscr{H}}\|_{H_{k}^{s}\to H_{k}^{s}}\|(1-\chi(-k^{-2}\Delta_\Gamma))(I+L)^{-1}\|_{H_{k}^{s}\to H_{k}^{s}}\|(I-P_N)\|_{H_{k}^{s}\to H_{k}^{s}}\|\Pi_{\mathscr{H}}\|_{H_{k}^{s}\to H_{k}^{s}}\\
&\leq \big(L_{\max}+C\hsc\big)\|(I-P_N)\|_{\Hshts}.
\end{align*}
\epf

\


\bpf[Proof of Theorem \ref{thm:MAT}]
By Lemma \ref{lem:QOabs}, it is enough to study $(I+P_N \pert)^{-1}(I-P_N)$ and hence we use~\eqref{e:thething}. Decomposing $H_{k}^{s}=\Pi_{\mathscr{L}}H_{k}^{s}\oplus \Pi_{\mathscr{H}}H_{k}^{s}$, letting $\mathscr{C}$ be as in~\eqref{e:newNorm}, and using Lemma~\ref{lem:split1}, we obtain~\eqref{e:blockDecomp}. 
Therefore, under the condition~\eqref{e:MATthres}, \eqref{e:Neumann1} holds, and then $I+(P_N-I)\pert (I+\pert)^{-1}$ is invertible with \eqref{e:magicC1}.

We now claim that 
\begin{align}\nonumber
&\big\| \big(I+\mathcal{C}(P_N-I)\pert (I+\pert)^{-1}\mathcal{C}^{-1}\big)^{-1}\big\|_{\Hshts}\\
&\hspace{2cm}\leq \begin{pmatrix}1  +C(k/N)^{t_{\max}-q_{\min}}\rho&
C\Big( (k/N)^{t_{\max}-q_{\min}}\rho \Big)^{1/2} (k/N)^{1/2}
\\
C\Big( (k/N)^{t_{\max}-q_{\min}}\rho \Big)^{1/2} (k/N)^{1/2}
&1  + Ck/N
\end{pmatrix}.\label{e:sumSeries}
\end{align}
To see this, let 
\beqs
M:=
\begin{pmatrix}
\delta & \delta^{1/2} (k/N)^{1/2}\\
\delta^{1/2}(k/N)^{1/2} & (k/N)
\end{pmatrix}
\eeqs
and observe that $M^2= (\delta + k/N) M$. Thus 
\beqs
\sum_{n=0}^\infty M^n = I + \bigg( \sum_{n=0}^\infty (\delta + k/N)^n \bigg) M \leq I + C M
\eeqs 
if both $\delta$ and $k/N$ are sufficiently small. Applying this last inequality with $\delta = (k/N)^{t_{\max}-q_{\min}}\rho$, we obtain \eqref{e:sumSeries}. 

Then, by \eqref{e:magicC1} and the definition of $\mathcal{C}$ \eqref{e:newNorm}, 
\begin{equation}
\label{e:neumannOut}
\begin{aligned}
&\big\|\big(I+ (P_N-I)\pert (I+\pert)^{-1}\big)^{-1}\big\|_{\Hshts}\leq 
\begin{pmatrix}
1  +C(k/N)^{t_{\max}-q_{\min}}\rho&C(k/N)^{s+1-q_{\min}}\\
C(k/N)^{t_{\max}-s}\rho &1+Ck/N
\end{pmatrix}.
\end{aligned}
\end{equation}
By \eqref{e:MATproj},
\begin{gather*}
\|(I-P_N)\|_{ H_k^{t_{\max}}\to H_k^{q_{\min}}}\leq (k/N)^{t_{\max}-q_{\min}},\qquad
 \|(I-P_N)\|_{ H_{k}^{s}\to H_k^{q_{\min}}}\leq (k/N)^{s-q_{\min}},\\
 \|(I-P_N)\|_{H_k^{t_{\max}}\to H_{k}^{s}}\leq (k/N)^{t_{\max}-s},
\end{gather*}
so that, by \eqref{e:projectionbounds},
\beq\label{e:Dean3}
\|(I-P_N)\|_{\Hshts}\leq \begin{pmatrix} C(k/N)^{t_{\max}-q_{\min}}&C(k/N)^{s-q_{\min}}\\ C(k/N)^{t_{\max}-s}&\|I-P_N\|_{\Hshts}\end{pmatrix}.
\eeq
Thus, by \eqref{e:neumannOut}, \eqref{e:Dean3}, and the fact that $(k/N)^{t_{\max}-q_{\min}}\rho \leq c$ by \eqref{e:MATthres}, 
\begin{equation}
\label{e:neumannOut2}
\begin{aligned}
&\big\|\big(I+ (P_N-I)\pert (I+\pert)^{-1}\big)^{-1}(I-P_N)\big\|_{\Hshts}\\
&\qquad\leq \begin{pmatrix}C(k/N)^{t_{\max}-q_{\min}}&C(k/N)^{s-q_{\min}}\\C(k/N)^{t_{\max}-s}&\|I-P_N\|_{\Hshts}+ C(k/N)+C(k/N)^{t_{\max}-q_{\min}}\rho
\end{pmatrix},
\end{aligned}
\end{equation}
where we have bounded $\|I-P_N\|_{\Hshts}$ by a constant (via \eqref{e:MATproj}) when it is multiplied by a term that is small.
By Lemma \ref{lem:split1},
\begin{equation}
\label{e:oneMoreTime}
\begin{aligned}
&\|(I+L)^{-1}(I-P_N)\|_{\Hshts}\leq \begin{pmatrix} C(k/N)^{t_{\max}-q_{\min}}\rho & C(k/N)^{s-q_{\min}}\rho \\
C(k/N)^{t_{\max}-s}&(L_{\max}+C\hsc)\|I-P_N\|_{H_{k}^{s}\to H_{k}^{s}}.\end{pmatrix}.
\end{aligned}
\end{equation}
%
Combining \eqref{e:thething}, \eqref{e:neumannOut2}, and \eqref{e:oneMoreTime}, we obtain that
\begin{equation*}
\begin{aligned}&\|(I+P_N \pert)^{-1}(I-P_N)\|_{\Hshts}\\
&\leq \|(I+\pert)^{-1}(I-P_N)\|_{\Hshts}\|I+(P_N-I)\pert(I+\pert)^{-1}(I-P_N)\|_{\Hshts}\\
&\leq C\begin{pmatrix}C(k/N)^{t_{\max}-q_{\min}}\rho &C(k/N)^{s-q_{\min}}\rho \\
C(k/N)^{t_{\max}-s}&(L_{\max}+C\hsc)\|I-P_N\|^2_{\Hshts}+C\big((k/N)+(k/N)^{t_{\max}-q_{\min}}\rho \big).
\end{pmatrix},
\end{aligned}
\end{equation*}
where we have bounded $(L_{\max}+C\hsc)\|I-P_N\|_{\Hshts}$ by a constant when it is multiplied by a term that is small.
The bounds \eqref{e:MATqo} and \eqref{e:MATqoHF} now follow by combining the last displayed bound with Lemma~\ref{lem:QOabs}.

\section{Proofs of the Galerkin and collocation results for piecewise polynomials}\label{sec:PPproofs}

\subsection{Checking that the Dirichlet and Neumann BIEs in \S\ref{sec:BIEs} satisfy Assumption \ref{ass:abstract1}}\label{sec:check}

\subsubsection{The Dirichlet BIEs.}

We let
\beqs
\pert: = 2 \big( \DL_k - \ri \eta_D S_k\big) \quad\tor\quad 2 \big( \DL_k' - \ri \eta_D S_k\big)
\eeqs
so that $\operator = 2 A_k$ or $2A_k'$, respectively.
The properties of $S_k, \DL_k,$ and $\DL_k$ as standard (i.e., not semiclassical) pseudodifferential operators imply that $\pert:H^s(\Gamma)\to H^{s+1}(\Gamma)$ for all $s\in \Rea$ (see, e.g., \cite[Theorem 4.4.1]{Ne:01}); thus $\operator$ is bounded on $H^s(\Gamma)$.
Under Assumption \ref{ass:parameters}, \blue{the parameter} $\eta_D \in \Rea$ with $|\eta|$ proportional to $k$, and then $\operator$ is injective on $H^s(\Gamma)$; see, e.g., \cite[Proof of Theorem 2.27]{ChGrLaSp:12}. Fredholm theory then implies that $\operator$ is invertible on $H^s(\Gamma)$, and Part (i) of Assumption \ref{ass:abstract1} holds.

%
%
Part (ii) of Assumption \ref{ass:abstract1} then follows from Theorem \ref{thm:HFSD}.

By Theorem \ref{thm:HFSD} and the results about the principal symbol in \S\ref{sec:ReatoGamma},
\beqs
\sigma^0 \Big((1-\chi\hDarg)\DL_k\Big) =\sigma^0 \Big((1-\chi\hDarg)\DL_k'\Big) =0,
\eeqs
so that
\beqs
\sigma^0 \Big( \big(1-\chi\hDarg\big) \pert\Big) =
-2\ri \eta\,\sigma^0 \Big( \big(1-\chi\hDarg\big)  S_k\big)\Big)
= -\ri \left(\frac{\eta_D}{k}\right)\dfrac{\big(1-\chi(|\xi'|_g^2)\big)}{\sqrt{|\xi'|_g^2-1}}.
\eeqs
Since $\eta_D$ is real with modulus proportional to $k$ (by Assumption \ref{ass:parameters}), Part (iii) of Assumption \ref{ass:abstract1} follows with $L_{\max}=1$.

\subsubsection{The Neumann BIEs.}

%

We give the proof for $\Breg$; the proof for $\Breg'$ is very similar.
Let
\beqs
\widetilde{L}:= -\ri \eta_N \DL_k +S_{\ri k}H_k + \frac{1}{4}I
\quad\text{ so that }\quad
\Breg= \left(\frac{\ri \eta_N}{2} - \frac{1}{4}\right)I + \widetilde{L}.
\eeqs
We then let
\beqs
\pert:= \left(\frac{\ri \eta_N}{2} - \frac{1}{4}\right)^{-1} \widetilde{L} 
\quad\text{ so that } \quad
\left(\frac{\ri \eta_N}{2} - \frac{1}{4}\right)^{-1}\Breg = I + \pert;
\eeqs
i.e., $\operator = (\ri \eta_N/2- 1/4)^{-1}\Breg$.
The first of the Calder\'on relations
\beq\label{e:Calderon}
S_k H_k = -\frac{1}{4}I + (\DL_k)^2
\quad\tand\quad
H_k S_k= -\frac{1}{4}I + (\DL_k')^2
\eeq
(see, e.g., \cite[Equation 2.56]{ChGrLaSp:12})
implies that
\beq\label{e:widetildeL}
\widetilde{L}= -\ri \eta_N \DL_k + \big(S_{\ri k}-S_k\big)H_k + (\DL_k)^2,
\eeq
and then $\widetilde{L}: H^s(\Gamma)\to H^{s+1}(\Gamma)$ for all $s\in \Rea$ by the properties of $\DL_k$, $H_k$, and $S_{\ri k }-S_k$ as standard pseudodifferential operators (see, e.g., \cite[Theorem 4.4.1]{Ne:01} and \cite[Theorems 2.1 and 2.2]{BoDoLeTu:15}); thus $\Breg$ is bounded on $H^s(\Gamma)$.
Under Assumption \ref{ass:parameters}, $\eta_N\in \mathbb{R}\setminus\{0\}$ is independent of $k$, and then $\Breg$ is injective on $H^s(\Gamma)$; see \cite[Theorem 2.2]{GaMaSp:21N}. Invertibility of $\Breg$ on $H^s(\Gamma)$ then follows from Fredholm theory, and Part (i) of Assumption \ref{ass:abstract1} holds.

We now show that Parts (ii) and (iii) of Assumption \ref{ass:abstract1} are satisfied. By Theorem \ref{thm:HFSD},
\beqs
\sigma\big(\big(1-\chi\hDarg\big) \widetilde{L}\big) = \sigma \big(\big(1-\chi\hDarg\big)S_{\ri k }H_k\big) + \frac{\big(1-\chi(|\xi'|_g^2)\big)}{4}.
\eeqs
Now, by, e.g., \cite[Lemma 4.8]{GaMaSp:21N}, $S_{\ri k} = \hsc \widetilde{S}$ where $\widetilde{S} \in\Psi^{-1}(\Gamma)$ with
\beq\label{e:Sik}
\sigma(\widetilde{S}) = \frac{1}{2 \sqrt{|\xi'|_g^2+1}};
\eeq
observe that $S_{\ri k}$ is then elliptic on $T^*\Gamma$, and hence invertible by Part (i) of Corollary \ref{cor:elliptic}.
Let $\widetilde{\chi}\in C^{\infty}_{c}(\Rea^d;[0,1])$ be such that $\widetilde{\chi}\equiv 1$ on $[-1,1]$ and $\supp \widetilde{\chi}\cap \supp(1-\chi)=\emptyset$. Then, by \eqref{e:Sik}, \eqref{e:H_prin_symb}, \eqref{e:multsymb}, and \eqref{e:WFdisjoint}
\beqs
\sigma \big(\big(1-\chi\hDarg\big)S_{\ri k }H_k\big)
=\sigma \big(\big(1-\chi\hDarg\big)S_{\ri k }\big(1-\widetilde{\chi}\hDarg\big)H_k\big)
=-\frac{1-\chi(|\xi'|_g^2)}{4} \sqrt{\frac{|\xi'|_g^2-1}{|\xi'|_g^2+1}},
\eeqs
so that
\begin{align}
\sigma \big( (1- \chi\hDarg)\widetilde{L}\big) &=  \frac{1-\chi(|\xi'|_g^2)}{4}\left(1- \sqrt{\frac{|\xi'|_g^2-1}{|\xi'|_g^2+1}}\right)=\frac{1-\chi(|\xi'|_g^2)}{2(|\xi'|_g^2+1)}\left(1+ \sqrt{1-\frac{2}{|\xi'|_g^2+1}}\right)^{-1}.\label{e:Macron}
\end{align}
Since \eqref{e:Macron} is real on $\{(x',\xi') : 1-\chi(|\xi'|_g^2)>0\}$ and $\eta_N\neq 0$, Part (iii) of Assumption \ref{ass:abstract1} holds with $L_{\max}=1$.
The expression \eqref{e:Macron} implies that $(1- \chi\hDarg)\widetilde{L}$ is the sum of an operator in $\Psi^{-2}(\Gamma)$ and an operator in $k^{-1} \Psi^{-1}(\Gamma)$, and thus Part (ii) of Assumption \ref{ass:abstract1} holds.

\bre[More general regularising operators than $S_{\ri k}$]
The properties of $S_{\ri k}$ that are used in the above arguments to show that $\Breg$ and $\Bregp$ satisfy Assumption \ref{ass:abstract1} are that $S_{\ri k} \in k^{-1}\Psi^{-1}(\Gamma)$, is elliptic, and its semiclassical principal symbol is real (with these assumptions equal to \cite[Assumption 1.1]{GaMaSp:21N}). The results in this paper therefore hold with $S_{\ri k}$ replaced by any other operator satisfying these assumptions (such as the quantisation of the principal symbol of $S_{\ri k}$, which is considered in \cite{BoTu:13}).
\ere

\bre[Regularisation understood via the Calder\'on relations]
The smoothing property in Part (ii) of Assumption \ref{ass:abstract1} can also be checked by using
the expression \eqref{e:widetildeL} and the fact that 
\beqs
\big(1-\chi\hDarg\big) (S_k-S_{\ri k}) \quad \tand\quad (S_k-S_{\ri k})\big(1-\chi\hDarg\big)\in k^{-1} \Psi^{-3}(\Gamma).
\eeqs
with real principal symbols, which follows from the results in \cite[\S3,4]{GaMaSp:21N}.
%
\ere

\subsection{Proof of Theorem \ref{thm:PG}  (Galerkin method with piecewise polynomials)}\label{sec:Galerkin_poly_proof}



\ble\label{lem:GalerkinPP}
Let $p\geq 0$ and $V_h$ be a space of piecewise polynomials on a mesh of width $h$ satisfying  Assumption \ref{ass:ppG}. Then 
the approximation property \eqref{e:MATproj} holds with $N=h$ and $P_{V_h}^G$ defined by \eqref{e:Galerkin_def},
for all the pairs $(t,q)$ in \eqref{e:approx_pairs} with $s=0$, $t_{\max}=p+1$, and $q_{\min}=-p-1$.
\ele

\bpf
Multiplying \eqref{e:assppG} by $k^{-q}$,
then using that $k^{-t} \N{v}_{H^t(\Gamma)}\leq \N{v}_{H^t_k(\Gamma)}$ by \eqref{e:weightedineq}, we find that
\beqs
\N{v- \cI_h v}_{L^2(\Gamma)} \leq C(hk)^{t}\N{v}_{H^t_k(\Gamma)} \quad\tfor 0\leq t\leq p+1.
\eeqs
The minimisation property \eqref{e:Galerkin_def} (with $s=0$) then implies that 
\beq\label{e:MATproj2}
\N{I-P_{V_h}^G}_{H_{k}^{t}(\Gamma)\to L^2(\Gamma)} \leq C(hk)^{t} \quad\tfor 0\leq t\leq p+1.
\eeq
i.e., \eqref{e:MATproj} holds for $(t,q)=(p+1,0)$, $(t,q)=(1,0)$, and $(t,q)=(0,0)$.
Since $I-P_{V_h}^G$ is the $L^2$-orthogonal projection, $(I-P_N^G)^*= I-P_N^G$, 
\eqref{e:MATproj2}
implies that 
\beq\label{e:MATproj2new}
\N{I-P_{V_h}^G}_{L^2(\Gamma)\to H^{-t}_k(\Gamma) } \leq C(hk)^{t} \quad\tfor 0\leq t\leq p+1,
\eeq
so that \eqref{e:MATproj} holds for $(t,q)=(0,-p-1)$. It only remains to show that  \eqref{e:MATproj} holds for $(t,q)=(1,-p-1)$, but this follows from the fact that
\begin{align*}
\N{I-P_{V_h}^G}_{H_{k}^{1}(\Gamma)\to H_{k}^{-p-1}(\Gamma) } 
&=\N{(I-P_{V_h}^G)^2}_{H_{k}^{1}(\Gamma)\to H_{k}^{-p-1}(\Gamma) } \\
&\leq \N{I-P_{V_h}^G}_{H_{k}^{1}(\Gamma)\to L^2(\Gamma) } \N{I-P_{V_h}^G}_{L^2(\Gamma)\to H_{k}^{-p-1}(\Gamma) }\leq C(hk)^{p+2}
\end{align*}
by \eqref{e:MATproj2} and \eqref{e:MATproj2new}.
\epf

\

With Lemma \ref{lem:GalerkinPP} in hand, we
apply Theorem \ref{thm:MAT} with $s=0$, $t_{\max}=p+1$, and $q_{\min}=-p-1$. We see that if $c$ is sufficiently small, then the condition \eqref{e:Pthres} ensures that \eqref{e:MATthres} holds, and then the quasi-optimal error estimate \eqref{e:PPMR1} follows from \eqref{e:MATqo}, together with the facts that $\|I-P_{V_h}^G\|_{\LtGt}=1$, $L_{\max}=1$.

To prove the bound \eqref{e:PPMR2} on the relative error, observe that the combination of \eqref{e:MATproj} with $(t,q)=(p+1,0)$ and \eqref{e:oscil} imply that
\beq\label{e:delivery1}
\N{(I-P_{V_h}^G)v}_{\LtG} \leq C(hk)^{p+1} \N{v}_{H^{p+1}_k(\Gamma)}\leq C\Creg (hk)^{p+1} \N{v}_{\LtG}.
\eeq
Inputting \eqref{e:delivery1} into the quasi-optimal error estimate \eqref{e:PPMR1}, we find \eqref{e:PPMR2}.
%
%

\subsection{Proof of Theorem \ref{thm:PC} (collocation method with piecewise polynomials)}

\ble\label{lem:collocationPP}
Let $d=2$, $p\geq 1$ and $V_h$ be a space of piecewise polynomials on a mesh of width $h$ satisfying  Assumption \ref{ass:ppC}. Then 
the approximation property \eqref{e:MATproj} holds with $P_{V_h}^C$
given by the interpolation operator,
for the pairs $(t,q)$ equal to \eqref{e:approx_pairs} with
$\frac{1}{2}<s\leq 1$,  $t_{\max}=p+1$, $q_{\min}=0$.
\ele

\bpf
Applying \eqref{e:assppC} with first $q=0$ and then $q=1$, and using that $k^{-t} |v|_{H^t(\Gamma)}\leq \N{v}_{H^t_k(\Gamma)}$
by \eqref{e:weightedineq}, we obtain that
\beq\label{e:HoD1}
\N{v-\cI_h v}_{\LtG} \leq C (hk)^t \N{v}_{H_{k}^{t}(\Gamma)}
\eeq
and
\beq\label{e:HoD2}
 k^{-1}|v-\cI_h v|_{\HoG} \leq C (hk)^{t-1} \N{v}_{H_{k}^{t}(\Gamma)}
\eeq
for $\frac{1}{2}< t\leq p+1$.
Combining \eqref{e:HoD2} and \eqref{e:HoD1} to obtain a bound on $\|v-\cI_h v\|_{H^1_k(\Gamma)}$, and then interpolating the result with \eqref{e:HoD1}, we obtain that 
\beq\label{e:interpolate1}
\N{v-\cI_h v}_{H^q_k(\Gamma)} \leq C (hk)^{t-q} \N{v}_{H_{k}^{t}(\Gamma)}
\eeq
for $0\leq q\leq1$ and $\frac{1}{2}< t\leq p+1$. Setting $t_{\max}=p+1$ and $q_{\min}=0$, we see that
the approximation property \eqref{e:MATproj} holds for the pairs $(t,q)$ equal to \eqref{e:approx_pairs} only when $\frac{1}{2}<s\leq 1$ (for the first pair in \eqref{e:approx_pairs} we need $s\leq 1$, for the second pair we need $0\leq s\leq1$, and for the third pair we need $s> \frac{1}{2}$).
\epf

\

With Lemma \ref{lem:collocationPP} in hand, the rest of the proof is the essentially the same as the proof of Part (i) of Theorem \ref{thm:PG} (i.e., using Theorem \ref{thm:MAT}), except that now $s=1$ instead of $s=0$.
\epf

\section{Definitions of the projections for trigonometric polynomials in 2-d}\label{sec:trig}

\subsection{Recap of Fourier series results and the definition of $H^{s}_k(0,2\pi)$ in terms of Fourier series.}

Given $f\in L^1(0,2\pi)$, let
\beqs
\widehat{f}(n):= \frac{1}{\sqrt{2\pi}}\int_0^{2\pi} \exp(-\ri nt) f(t)\, \rd t.
\eeqs
Recall that if $f\in L^2(0,2\pi)$ then
\beq\label{e:Fourier}
f(t) =  \frac{1}{\sqrt{2\pi}}\sum_{n=-\infty}^\infty \exp(\ri nt ) \,\widehat{f}(n) \quad\tand\quad
\N{f}^2_{L^2(0,2\pi)} = \sum_{n=-\infty}^\infty |\widehat{f}(n)|^2.
\eeq
Let
\beqs
\langle f, g\rangle_{H^{s}_k(0,2\pi)} := \sum_{n=-\infty}^\infty \widehat{f}(n) \, \overline{\widehat{g}(n)}\langle n k^{-1}\rangle^{2s}
\eeqs
(where recall that $\langle \xi\rangle:= (1+|\xi|^2)^{1/2}$)
and
\beq\label{e:norm_circle}
\vertiii{f}^2_{H^{s}_k(0,2\pi)}:= \langle f, f\rangle_{H^{s}_k} := \sum_{n=-\infty}^\infty \big|\widehat{f}(n)\big|^2 \langle n k^{-1}\rangle^{2s}.
\eeq
Given $s\in \Rea$ there exist $C_1, C_2>0$ such that, for all $v\in \Hsh$,
\beq\label{e:normequiv}
C_1 \|v\|_{\Hsh} \leq \vertiii{v\circ \gamma}_{H^{s}_k(0,2\pi)} \leq C_2 \|v\|_{\Hsh},
\eeq
where $\|\cdot\|_{H^s_k(\Gamma)}$ is defined by \eqref{e:weightedNormGamma}.
Let
\beq\label{e:nottrig1}
\phi_{m,\hsc}(t) := \frac{\exp(\ri m t)}{\sqrt{2\pi}\langle m k^{-1}\rangle^s}.
\eeq
Then
$\widehat{\phi}_m(n)= \delta_{mn}\langle m \hsc\rangle^{-s}$, $\vertiii{\phi_m}^2_{H^{s}_k(0,2\pi)}=1$, and the Fourier expansion in \eqref{e:Fourier} implies that
\beqs
f(t) = \sum_{n=-\infty}^\infty \langle f ,\phi_{m,\hsc}\rangle_{H^{s}_k(0,2\pi)}\phi_{m,\hsc}(t).
\eeqs

\subsection{The Galerkin projection.}
For $v\in L^2(0,2\pi)$, let
\beq\label{e:PNG}
P_{\trig_N}^G v(t):= \sum_{m=-N}^N \langle v ,\phi_{m,\hsc}\rangle_{L^2(0,2\pi)} \phi_{m,\hsc}(t)
=\sum_{m=-N}^N \widehat{v}(m) \frac{\exp (\ri m t)}{\sqrt{2\pi}}.
\eeq
Then $P_{V_N}^G: L^2(0,2\pi) \to V_N$ is a projection, where $V_N$ is defined by \eqref{e:V_N_trig}
and \eqref{e:PNG} is equivalent to \eqref{e:Galerkin_def} with the the norm \eqref{e:newNormGamma}.


\ble[Approximation property of Galerkin projection]\label{lem:approx_FG}
For all $t\geq q$ there is $C>0$ such that 
\beq\label{e:approxFG}
\N{ I-P_{\trig_N}^G}_{H_{k}^{t}(\Gamma) \to H_{k}^{q}(\Gamma)} \leq \frac{C}{ \langle (N+1)k^{-1}\rangle^{t-q}}.
\eeq
\ele
\bpf
By \eqref{e:PNG} and \eqref{e:norm_circle}, 
for $v\in H_{k}^{t}(0,2\pi)$,
\beqs
\vertiii{( I-P_{\trig_N}^G)v}_{H_{k}^{q}(0,2\pi)}^2 = \frac{1}{2\pi} \sum_{|m|\geq N+1} |\widehat{v}(n)|^2 \frac{ \langle nk^{-1}\rangle^{2t}}{ \langle n k^{-1} \rangle^{2(t-q)}}
\leq \frac{1}{ \langle (N+1)k^{-1}\rangle^{2(t-q)}}\vertiii{v}_{H_{k}^{t}(0,2\pi)}^2;
\eeqs
the result then follows from the norm equivalence \eqref{e:normequiv}. 
\epf

%
%
%

\subsection{The collocation projection.}
With $V_N$ defined by \eqref{e:V_N_trig} and $t_j$ defined by \eqref{e:trig_points}, let
\beq\label{e:collocation_def_recap}
(P_{\trig_N}^C v)(t_j) = v(t_j), \quad j=0, 1, \ldots, 2N, \quad\tand\quad P_{\trig_N}^C v \in V_N
\eeq
(i.e., \eqref{d:collocation} with $x_j$ replaced by $t_j$).
The points $t_j$ are unisolvent for $V_N$ -- and hence $P_{\trig_N}^C$ is a well-defined projection -- by, e.g., \cite[\S11.3]{Kr:14}, \cite[\S3.2.2]{At:97}.
We now recall the explicit expression for $P_{\trig_N}^C$ in terms of the Lagrange basis functions (see, e.g., \cite[\S3.2.2, Page 62]{At:97}).
For $j=0, 1, \dots, 2N$, let
\beqs
\ell_j(t) := \frac{2}{2N+1} D_N(t-t_j), 
\quad\text{ where }\quad
D_N(t):= \frac{1}{2} + \sum_{j=1}^N \cos(jt) = \frac{\sin (N+1/2)t}{2 \sin (t/2)};
\eeqs
the definition of $D_N(t)$ implies that $\ell_j(t_m) = \delta_{jm}$.
Then
\beq\label{e:PNC}
P_{\trig_N}^C v(t) = \sum_{j=0}^{2N} v(t_j) \ell_j(t).
\eeq
The definition \eqref{e:collocation_def_recap} implies that
\beq\label{e:alias0}
\tif\, m = \ell (2N+1) + \mu\, \text{ with }\, |\mu|\leq N\, \text{ then } \,
\big(P_{\trig_N}^C (\exp(\ri m \cdot)\big)(t) = \exp(\ri \mu t).
\eeq
Observe that the definition of $P_{V_N}^C$ \eqref{e:collocation_def_recap} implies that for $v\in\trig_N$,
$P_{\trig_N}^Cv= v$, so $(I-P_{\trig_N}^C)P_{\trig_N}^G=0$,  and thus
\beq\label{e:SVclevertrick}
I-P_{\trig_N}^C=(I-P_{\trig_N}^C)(I-P_{\trig_N}^G).
\eeq

\ble[Approximation property of collocation projection]
\label{lem:approx_FC}
Given $t\geq q\geq 0$ with $t>1/2$, then there exists $C>0$ such that if $(N+1) \geq k$ then
\beq\label{e:approx_FC}
\N{ I-P_{\trig_N}^C}_{H_{k}^{t}(\Gamma) \to H^q_k(\Gamma)} \leq C \left( \frac{ k}{N} \right)^{t-q}.
\eeq
\ele

\bpf
The analogous bound when $\hsc=1$ (i.e., the unweighted case) is proved in, e.g., \cite[Theorem 2.1]{KrSl:93}, \cite[Lemma 4.1]{SaVa:98}, \cite[Theorem 11.8]{Kr:14}, \cite[Theorem 8.3.1]{SaVa:02}. The bound \eqref{e:approx_FC} follows from repeating these proofs, now with $k\neq 1$; the requirement that $(N+1)k^{-1} \geq 1$ is used in the inequality $\langle a x\rangle \leq |a|\langle x\rangle$ when $|a|\geq 1$.
\epf

\section{Best approximation by trigonometric polynomials in plane-wave scattering}
\label{sec:approx}


\subsection{Statement of the two main results in this section}

The abstract results in \S\ref{sec:abstract} give conditions under which the projection-method error, $v-v_N$, is bounded in terms of the projection error $\|(I-P_{\trig_N}^G)v\|_{H^{s}_k}$.
In this section we bound the projection error $\|(I-P_{\trig_N}^G)v\|_{H^{s}_k}$ in terms of $\|v\|_{H^{s}_k}$ when $v$ is the BIE  solution corresponding to the plane-wave scattering problem (i.e., when the right-hand side $f$ of $\operator v=f$ is given by Theorem \ref{thm:BIEs}).
These bounds then lead to the bounds on the relative error for plane wave data \eqref{e:PPMR2}, \eqref{e:PPMR5}, \eqref{e:MR2}, and~\eqref{e:MR5}).

The two main results of this section are the following two theorems.

\begin{theorem}[Galerkin and collocation projection error for trigonometric polynomials]\label{thm:superalgebraicFG}
Suppose $\trig_N$ is given by \eqref{e:V_N_trig} and that $\operator$ is one of $A_k, A_k'$, $\Breg$, or $\Bregp$ and the right-hand side $f$ is as described in Theorem \ref{thm:BIEs}.

%
%

Then $v\in H^s(\Gamma)$ for all $s\geq 0$ and, given $\eps>0$ and $k_0>0$, if $N\geq (c_{\max}+\eps)k$ and $k\geq k_0$, then
\beq\label{e:superalgebraicFG}
\frac{
\N{(I-P_{\trig_N}^G)v}_{H^s_k(\Gamma)}}
{\N{v}_{H^s_k(\Gamma)}}
=O(k^{-\infty}),
\eeq
where $P_{\trig_N}^G$ is defined by \eqref{e:Galerkin_def}/\eqref{e:PNG}.

Furthermore, if $s>1/2$, then \eqref{e:superalgebraicFG} holds with $P_{\trig_N}^G$ replaced by $P_{\trig_N}^C$ defined by
Definition~\ref{d:collocation}/Equation \eqref{e:collocation_def_recap}.
\end{theorem}

The result in Theorem \ref{thm:superalgebraicFG} for $I-P_N^C$ follows immediately from the result for $I-P_N^G$  by using
the identity \eqref{e:SVclevertrick} and the fact that $I-P_N^C$ is bounded on $H^s_k$ (by Lemma \ref{lem:approx_FC}).

\begin{theorem}[Oscillatory behaviour of BIE solution under plane-wave scattering]\label{thm:oscil}
Suppose that $\operator$ is one of $A_k, A_k'$, $\Breg$, or $\Bregp$ and the right-hand side $f$ is as described in Theorem \ref{thm:BIEs}.
Then $v\in H^t(\Gamma)$ for all $t\geq 0$ and
given $t\geq s\geq 0$ and $k_0>0$ there exists $\Creg(t,s,k_0)>0$ such that
\beq\label{e:oscil}
\N{v}_{H^t_k(\Gamma)} \leq \Creg \N{v}_{H^s_k(\Gamma)} \quad\tfa k\geq k_0.
\eeq
\end{theorem}

In \S\ref{sec:Galerkin_poly_proof}, the bound in Theorem \ref{thm:oscil} is combined with the polynomial-approximation assumptions Assumptions \ref{ass:ppG} and \ref{ass:ppC} to prove bounds on $\|(I-P_N)v\|_{H^s_k(\Gamma)}$ (where $P_N$ is either the Galerkin or collocation projection).


\subsection{Bounds on the high-frequency components of $f$ and $v$}

In the next lemma we use the notation that $\gamma^+$ is the Dirichlet trace on $\Gamma$ from $\Omega^+$ (as used in \S\ref{app:A}).

\begin{lemma}[The high-frequency components of $f$ are superalgebraically small]\label{lem:HFf}
Suppose that $B\in \Psi^m(\Gamma)$ and $\chi\in C_{c}^\infty(\Rea;[0,1])$ with
$\supp(1-\chi)\cap[-R,R]=\emptyset$ for some $R>0$. Then for all $N,k_0>0$ there exists $C_N>0$ such that for all $a\in \mathbb{R}^d$ with $|a|\leq R$ and for all $0<\hsc\leq\hsc_0$,
$$
\|(I-\chi\hDarg)B\gamma^+\re^{\ri kx\cdot a}\|_{H_k^{N}(\Gamma)}\leq C_N k^{-N}.
$$
\end{lemma}

The idea behind the proof is that if $|a|\leq R$ then $B\gamma^+ \re^{\ri k x\cdot a}$ contains frequencies $\leq kR$, and $(I-\chi\hDarg$ is a frequency cut-off to frequencies $>kR$; we now proof this rigorously using the quantisation definition \eqref{e:quant} and integration by parts.

\

\begin{proof}
By a partition of unity, we can work in local coordinates where
$$
\Gamma=\big\{ (x',F(x'))\,:\, x'\in U\subset \mathbb{R}^{d-1}\big\}.
$$
By the composition formula for symbols \cite[Theorem 4.14]{Zw:12}, \cite[Proposition E.8]{DyZw:19}, $(I-\chi\hDarg)B=\Op(b) +O(k^{-\infty})_{\Psi^{-\infty}(\Gamma)}$ for some $b\in S^m$ satisfying $\supp b\cap \{|\xi'|_g\leq R\}=\emptyset$.
For such a $b$ 
and $\psi\in C_{c}^\infty(\mathbb{R}^{d-1})$ supported in $U$, by \eqref{e:quant},
\beq\label{e:trustingJeff}
\Op(b)\psi(x')\gamma^+\re^{\ri kx\cdot a}= (2\pi \hsc)^{-d}\int_{\Rea^d}\int_{\Rea^d} \re^{\frac{\ri}{\hsc}(\langle x'-y',\xi'\rangle+\langle a',y'\rangle+F(y')a_1)}b(x',\xi')\psi(y')\rd y'\rd\xi'.
\eeq
In these coordinates, the metric $g$ on $\Gamma$ is given by
$$
g(\partial_{x^i},\partial_{x^j})= \delta_{ij} + \partial_{x^i}F\,\partial_{x^j}F
$$
and hence, in these coordinates,
$$
|\xi'|_g^2(x')= \big\langle G(x')\xi',\xi'\big\rangle,\qquad G(x')= I-\frac{1}{1+|\partial_{x'}F(x')|^2}(\partial_{x'}F(x'))^t(\partial_{x'}F(x')).
$$
The phase in the integrand in \eqref{e:trustingJeff} is stationary in $y'$ when
$
\xi'_*=a'+a_1\partial_{y'}F(y'),
$
and 
$$
|\xi'_*|^2_g= \Big\langle G(x')\big(a'+a_1\partial_{y'}F(y')\big),\big(a'+a_1\partial_{y'}F(y')\big)\Big\rangle.
$$
By direct computation, we find that $|\xi'_*|^2_g\leq |a|^2\leq R^2.$
Therefore, since $|\xi'|_g>R$ on the support of the integrand in \eqref{e:trustingJeff}, the phase is non-stationary in $y'$ and repeated integration by parts in $y'$
(see, e.g., \cite[Lemmas 3.10 and 3.14]{Zw:12})
shows that
$$
\N{\Op(b)\psi(x')\gamma^+\re^{\ri kx\cdot a}}_{H_k^N(\Gamma)}\leq C_Nk^{-N}.
$$
\end{proof}

\ble[The high-frequency components of $v$ are superalgebraically small]\label{lem:HFv}
Suppose that $\operator$ is one of $A_k, A_k'$, $\Breg$, or $\Bregp$ and the right-hand side $f$ is as described in Theorem \ref{thm:BIEs}.
For any $\widetilde{\chi}\in C^\infty_{c}(\Rea^d;[0,1])$ with $\supp(1-\widetilde{\chi})\cap[-1,1]=\emptyset$ and any $s>0$,
\beqs
\frac{\N{(I- \widetilde{\chi}\hDarg) v}_{\Hsh}
}{
\N{v}_{\Hsh}
} = O(k^{-\infty}) .
\eeqs
\ele

To prove Lemma \ref{lem:HFv} we need the following result.

\ble[Lower bound on $\|v\|_{\Hsh}$]\label{lem:vlowerbound}
Suppose that $\operator$ is one of $A_k, A_k'$, $\Breg$, or $\Bregp$ and the right-hand side $f$ is as described in Theorem \ref{thm:BIEs}.
Given $s>0$ there exists $M\in \Rea$ such that, given $k_0>0$, there exists $C>0$ such that
\beqs
\|v\|_{\Hsh}\geq C \hsc^{M}\quad\tfa 0<\hsc\leq \hsc_0.
\eeqs
\ele

\bpf
Since
\beqs
\|v\|_{\Hsh} \geq \|f\|_{\Hsh} \big( \| \operator \|_{\Hsht}\big)^{-1},
\eeqs
we see that the required bound from below on $\|v\|_{\LtG}$ follows if we can show that (i) $\|f\|_{\Hsh}$ is bounded below algebraically in $\hsc$, and (ii) $\| \operator \|_{\Hsht}$ is bounded above algebraically in $\hsc^{-1}$.

Condition (ii) follows from Lemma \ref{lem:normA}, the bounds on $\|A_k'\|_\LtGt = \|A_k\|_\LtG$ in \cite{ChGrLaLi:09}, \cite[Theorem 2]{GaSm:15}, \cite[Theorem A.1]{HaTa:15}, \cite[Theorem 4.5]{Ga:19},  and the bounds on $\|\Breg\|_{\LtGt}=\|\Breg'\|_{\LtGt}$ in \cite[Theorems 2.1]{GaMaSp:21N}.

Condition (i) follows from explicitly calculating the $\hsc$ dependence of $\|f\|_{\Hsh}$ for the four different $f$s given in Theorem \ref{thm:BIEs}; for the $f$ for $\Breg$, we additionally have to use the bound $\|(S_{\ri k})^{-1}\|_{H_{k}^{s-1}(\Gamma)\to H^s_k(\Gamma)} \lesssim \hsc^{-1}$ from \cite[Corollary 3.7]{GaMaSp:21N}.
\epf

\

\bpf[Proof of Lemma \ref{lem:HFv}]
By \S\ref{sec:check}, $\operator$ satisfies Assumption \ref{ass:abstract1} and thus the conclusions of Lemma \ref{lem:elliptic_HF} hold.
Given $\widetilde{\chi}$, let $\chi\in C^\infty_{c}(\Rea;[0,1])$ be such that $\supp(1-\chi)\cap[-1,1]=\emptyset$, and $\supp \chi \cap \supp(1-\widetilde{\chi})=\emptyset$; this implies that  $\{\widetilde{\chi}\not\equiv 1\} \Subset \{ \chi \not\equiv 1\}$, and thus, by \eqref{e:WFchi} and Lemma \ref{lem:elliptic_HF}, \eqref{e:2crows3} holds.
The elliptic estimate \eqref{e:ellipticestimate} then implies that, given $s,M, N \in \Rea$, there exists $C, k_0>0$ such that, for all $k\geq k_0$,
\begin{align}\nonumber
\N{\big( 1- \widetilde{\chi}\hDarg\big) v }_{H^s_k(\Gamma)}
&\leq C \Big(
\N{
\big( 1- \chi\hDarg\big) \operator v}_{H^{s}_k(\Gamma)} + k^{-M} \N{v}_{H^{s-N}_k(\Gamma)}
\Big)\\
&= C \Big(
\N{
\big( I-\chi\hDarg\big)f}_{H^{s}_k(\Gamma)} + k^{-M} \N{v}_{H^{s-N}_k(\Gamma)}
\Big).\label{e:Munich1}
\end{align}
The result then follows by choosing $N=0$ and using Lemmas \ref{lem:HFf} and \ref{lem:vlowerbound}.
\epf

\subsection{Proof of Theorem \ref{thm:oscil}}

First observe that it is sufficient to prove the bound \eqref{e:oscil} for $k_0$ sufficiently large.
Let $\widetilde{\chi}$ be as in Lemma \ref{lem:HFv}, and write
\beqs
v = \big( 1- \widetilde{\chi}\hDarg\big) v +\widetilde{\chi}\hDarg v.
\eeqs
By \eqref{e:frequencycutoff2}, given $t\geq s$ and $k_0>0$, there exists $C>0$ such that, for all $k\geq k_0$,
\beqs
\N{\widetilde{\chi}\hDarg v}_{H_{k}^{t}(\Gamma)} \leq C\N{v}_{H^s_k(\Gamma)}.
\eeqs
Combining this with the result of Lemma \ref{lem:HFv}, we have
\beqs
\N{v}_{H_{k}^{t}(\Gamma)} \leq O(k^{-\infty}) \N{v}_{H_{k}^{t}(\Gamma)} + C \N{v}_{H^s_k(\Gamma)},
\eeqs
and the result follows.

\subsection{Proof of Theorem \ref{thm:superalgebraicFG}}

As noted after the statement of the theorem, the result in Theorem \ref{thm:superalgebraicFG} for $I-P_{\trig_N}^C$ follows immediately from the result for $I-P_{\trig_N}^G$  by using
the identity \eqref{e:SVclevertrick} and the fact that $I-P_{\trig_N}^C$ is bounded on $H^s_k$ (by Lemma \ref{lem:approx_FC}).

To prove the result for $I-P_{\trig_N}^G$,
we first show that the operator $I-P_{\trig_N}^G$ can be naturally expressed in terms of functions of $\partial_t^2$.
Indeed, we claim that
\beq\label{e:PNGSL}
P_{\trig_N}^G = \ind_{(-\infty,1]} \big( - N^{-2} \partial_t^2\big),
\eeq
where $\partial_t$ is understood as an operator on functions on $\Gamma$ via $\partial_t u = (\partial_t (u\circ \gamma^{-1}))\circ \gamma$.

To see this, recall that for $f\in L^\infty(\Rea)$ and $v\in \LtG$,
\beq\label{e:func_calc}
f(-\partial_t^2)v:=\sum_{m=1}^\infty f(\lambda_m)( v,u_{\lambda_m})_{L^2(0,2\pi)}u_{\lambda_m},
\eeq
where 
%
%
$\{u_{\lambda_m}\}_{m=1}^\infty$ is an orthonormal basis for $L^2([0,2\pi])$ of eigenfunctions of $-\partial_t^2$; i.e.,
$$
(-\partial_t^2-\lambda_m)u_{\lambda_m}=0\quad\tand\quad \N{u_{\lambda_m}}_{L^2(\Gamma)}=1.
$$
Thus
\beqs
u_{\lambda_m}\big(\gamma(t)\big) = \frac{ \exp(\ri m t)}{\sqrt{2\pi}}, \quad \text{i.e.,} \quad u_{\lambda_m}(x) =  \frac{ \exp(\ri m \gamma^{-1}(x))}{\sqrt{2\pi}},
\eeqs
where $\lambda_m = m^2$.
Therefore, \eqref{e:func_calc} implies that
\begin{align*}
\ind_{(-\infty,1]} \big( - N^{-2} \partial_t^2\big) &= \sum_{m=1}^\infty 1_{(-\infty,1]} \big( N^{-2} \lambda_m\big) \big( v, u_{\lambda_m}\big)_{L^2(0,2\pi)} u_{\lambda_m}\\
&= \sum_{\lambda_m\leq N^2} \widehat{v}(n) \frac{\exp(\ri n t)}{\sqrt{2\pi}} =
\sum_{m=-N}^N \widehat{v}(n) \frac{\exp(\ri n t)}{\sqrt{2\pi}},
\end{align*}
which is $P_{\trig_N}^G$ given by \eqref{e:PNG}, so that \eqref{e:PNGSL} holds, and also
\beq\label{e:rain1}
(I-P_{\trig_N}^G) =\ind_{(1,\infty)}\big( -N^{-2} \partial_t^2\big)= \ind_{\big(( Nk^{-1})^2,\infty\big)} \big(-k^{-2} \partial_t^2\big).
\eeq

Theorem \ref{thm:superalgebraicFG} is a consequence of the following result. 

\begin{lemma}
\label{l:proj2Pseudo}
Given $\Theta,\epsilon>0$, 
let $\chi \in C_{c}^\infty(\mathbb{R};[0,1])$ be such that 
$\supp(1-\chi)\cap[-\Theta^2,\Theta^2]=\emptyset$ and $\supp \chi \subset [-(\Theta + \epsilon/(2c_{\max}))^2, (\Theta + \epsilon/(2c_{\max})^2)]$.
If $N\geq (\Theta c_{\max} +\epsilon)k$, then
$$
(I-P_{\trig_N}^G)=(I-P_{\trig_N}^G)\big(I - \chi(|k^{-1} D'|_{g}^2)\big)  +O(k^{-\infty})_{\Psi^{-\infty}(\Gamma)}=\big(I - \chi(|k^{-1} D'|_{g}^2)\big)(I-P_{\trig_N}^G) + O(k^{-\infty})_{\Psi^{-\infty}(\Gamma)}.
$$
\end{lemma}

The result \eqref{e:superalgebraicFG} follows by applying Lemma~\ref{l:proj2Pseudo} to $v$ with $\Theta=1$ and using Lemma \ref{lem:HFv} (noting that $\widetilde\chi$ in Lemma \ref{lem:HFv} can be taken to be $\chi$ via the choice $\Theta=1$). It therefore remains to prove Lemma \ref{l:proj2Pseudo}.

\

\begin{proof}[Proof of Lemma \ref{l:proj2Pseudo}]
Given $\eps>0$, let $\chi_0 \in C^{\infty}_{c}(\Rea;[0,1])$ be such that $\supp \chi_0 \subset (-(\Theta c_{\max}+\eps)^2, (\Theta c_{\max}+\eps)^2)$ and $\supp(1-\chi_0)\cap [-(\Theta c_{\max}+\eps/2)^2,(\Theta c_{\max}+\eps/2)^2]=\emptyset$.

If $N\geq (\Theta c_{\max}+\eps)k$ then
\beqs
(I-P^G_{\trig_N})\chi_0 (-k^{-2}\partial_t^2)=\chi_0 (-k{-2}\partial_t^2)(I-P^G_{\trig_N})=0,
\eeqs
since the image of $\chi_0 (-\hsc^{2}\partial_t^2)$ contains frequencies $< (\Theta c_{\max}+\eps)k$, and $I-P_{\trig_N}^G$ restricts to frequencies $\geq (\Theta c_{\max}+\eps)k$ by \eqref{e:rain1}.
In particular,
\beqs
(I-P_{\trig_N}^G)=(I-P_{\trig_N}^G)\big(I - \chi_0(-k^{-2} \partial_t^2)\big)=\big(I - \chi_0(-k^{-2} \partial_t^2)\big)(I-P_{\trig_N}^G).
\eeqs
Now, $-\partial_t^2=-\Delta_{\Gamma,g_{\rm flat}}$ where $g_{\rm flat}$ is the metric on $\Gamma$ given by $|\dot{\gamma}(t)|_{g_{\rm flat}(\gamma(t))}=1$. Therefore, by Lemma \ref{lem:WFcutoff}, 
\beqs
\WF(I-P_{\trig_N}^G) \subset \Big\{ |\xi'|_{g_{\rm flat}}^2 \geq (\Theta c_{\max} + \epsilon)^2 \Big\}.
\eeqs
Since the arc length metric equals $|\dot{\gamma}(t)|^2 dt^2$,
$$
\frac{|\xi'|_{g(\gamma(t))}}{|\xi'|_{g_{\rm flat}(\gamma(t))}}= |\dot{\gamma}(t)|^{-1}\geq \inf_{t\in[0,2\pi]}|\dot\gamma(t)|^{-1}= c_{\max}^{-1}.
$$
By the support properties of $\chi$ and \eqref{e:WFquant},
\begin{align*}
& \WF\, \chi \hDarg\subset \Big\{ |\xi'|_g^2 \leq (\Theta + \epsilon/(2c_{\max}))^2\Big\} \subset \Big\{ |\xi'|_{g_{\rm flat}}^2 \leq (\Theta c_{\max} + \epsilon/2)^2\Big\}
\end{align*}
so that 
\beqs
\WF\, \chi \hDarg \cap \WF(I-P_{\trig_N}^G) = \emptyset;
\eeqs
the result then follows from \eqref{e:WF_prod}.
%
\end{proof}

\section{Proofs of the Galerkin and collocation results for trigonometric polynomials}
\label{sec:proofs}

Theorems \ref{thm:FG} and \ref{thm:FC} (on, respectively, the Galerkin and collocation method with trigonometric polynomials) are proved using the arguments in the proof of Theorem \ref{thm:MAT} with the additional structure of the Galerkin projection (in particular, \eqref{e:rain1}), and we prove these last.

\subsection{Proof of Theorem \ref{thm:FG} (Galerkin method with trigonometric polynomials)}

First observe that the bound on the relative error \eqref{e:MR2} follows from the combination of \eqref{e:MR1} and \eqref{e:superalgebraicFG}. 
Therefore, we only need to prove the quasi-optimality bound \eqref{e:MR1} 

Theorem \ref{thm:MAT} was based on the fact that the identity \eqref{e:thething} holds under the condition \eqref{e:Neumann1}. In the case of trigonometric polynomials, we instead use the simpler setup that \eqref{e:thething} holds under the condition \eqref{e:Neumann0}.

Let $\Xi\geq 1$, $\epsilon>0$, and let $\chi$ be as in Lemma \ref{l:proj2Pseudo} with $\Theta:= \Xi$; i.e. $\supp(1-\chi)\cap[-\Xi^2,\Xi^2]=\emptyset$ and $\supp \chi \subset [-(\Xi + \epsilon/(2c_{\max}))^2, (\Xi+ \epsilon/(2c_{\max})^2)]$.

Then, if $N\geq (\Xi c_{\max} +\epsilon)k$, the combination of \eqref{e:thething}, Lemma \ref{l:proj2Pseudo}, and Assumption \ref{ass:polyboundintro} implies that 
\begin{align}\nonumber
&(I+P_{\trig_N}^G \pert )^{-1}(I-P_{\trig_N}^G) \\ \nonumber
&= (I+\pert)^{-1}(I-\chi \hDarg)(I-P_{\trig_N}^G)\Big( I+ (P_{\trig_N}^G-I)(I-\chi \hDarg)) \pert(I+\pert)^{-1}\Big)^{-1}(I-P_{\trig_N}^G) \\
&\qquad\qquad+O(k^{-\infty})_{\Psi^{-\infty}(\Gamma)}.\label{e:theotherthing}
\end{align}
Now, by  \eqref{e:inverse_cutoff_right} and~\eqref{e:invEst} from Lemma \ref{lem:HFinverse} and the fact that $L_{\max}=1$ (by \S\ref{sec:check}),
\beq\label{e:JeffGenius2}
\N{(I+\pert)^{-1}(I-\chi \hDarg)}_{\LtGt}\leq (1+C\hsc )
\eeq
(compare to \eqref{e:yoga4}). 

By the combination of Lemma \ref{lem:QOabs}, \eqref{e:theotherthing} and the bounds \eqref{e:JeffGenius2} and $\|(I-P_{\trig_N}^G)\|_{L^2\to L^2}\leq 1$, 
to prove \eqref{e:MR1}, we only need to show that 
if 
$N\geq (\Xi R_{\min} c_{\max} +\epsilon)k$ then 
\beqs
\N{(P_{\trig_N}^G-I)(I-\chi \hDarg) \pert(I+\pert)^{-1}}_{\LtGt}\leq 
C \big( \Xi^{-1} +\hsc\big)
\eeqs
(compare to \eqref{e:yoga3}).
Similarly, since $\|(I-P_{\trig_N}^G)\|_{L^2\to L^2}\leq 1$, it is sufficient to show that
\beq\label{e:JeffGenius1}
\N{(I-\chi \hDarg) \pert(I+\pert)^{-1}}_{\LtG\to \LtG}\leq C \big( \Xi^{-1} +\hsc\big).
\eeq

To prove \eqref{e:JeffGenius1}, we choose $\psi_i\in C^{\infty}_{c}(\Rea;[0,1])$, $i=0,1$ with $\supp(1-\psi_i)\cap [-1,1]=\emptyset$,  $\supp \psi_i \cap \supp(1-\chi)=\emptyset$, $\supp \psi_0\cap \supp(1-\psi_1)=\emptyset$ (i.e., $\psi_0$ is ``smaller than" $\psi_1$ which is ``smaller than" $\chi$) -- such a choice is possible since $\Xi \geq 1$, and thus $[-\Xi^2,\Xi^2]\supset [-1,1]$. 
Then, by Lemma \ref{lem:WF}, \eqref{e:inverse_cutoff_left}, \eqref{e:WFdisjoint}, and Assumption \ref{ass:polyboundintro},
\begin{align}\nonumber
&\big(I-\chi \hDarg\big) \pert (I+L)^{-1}  \\ \nonumber
&\quad=\big(I-\chi \hDarg \big)\pert \big(I-\psi_1\hDarg\big)(I+L)^{-1}   + O(k^{-\infty})_{\Psi_h^{-\infty}(\Gamma)}\\ \nonumber
&\quad=\big(I-\chi \hDarg \big)\pert \big(I-\psi_1\hDarg\big)\big(I+(1-\psi_0\hDarg)L\big)^{-1}   + O(k^{-\infty})_{\Psi_h^{-\infty}(\Gamma)}\\ \nonumber
&\quad=\big(I-\chi \hDarg \big)\pert \big(I+(1-\psi_0\hDarg)L\big)^{-1}   + O(k^{-\infty})_{\Psi_h^{-\infty}(\Gamma)}\\
&\quad=\big(I-\chi \hDarg \big)\big(1-\psi_0\hDarg\big)\pert \big(I+\big(1-\psi_0\hDarg\big)L\big)^{-1}   + O(k^{-\infty})_{\Psi_h^{-\infty}(\Gamma)}.
\label{e:WW1}
\end{align}
By \eqref{e:invEst} from Lemma \ref{lem:HFinverse} and the fact that $L_{\max}=1$ (by \S\ref{sec:check}), 
\beqs
\big\|\big(I+\big(1-\psi_0\hDarg\big)L\big)^{-1}\big\|_{\LtGt}\leq 1 + C\hsc,
\eeqs
and thus to prove \eqref{e:JeffGenius1} it is sufficient to prove that
\beq\label{e:cleaners1}
\big\|\big(I-\chi \hDarg \big)\big(1-\psi_0\hDarg\big)\pert \big\|_{\LtGt} \leq C \big( \Xi^{-1} +\hsc\big).
\eeq

Now, by Part (ii) of Assumption \ref{ass:abstract1}, $\big(1-\psi_0\hDarg\big)\pert \in \Psi^{-1}(\Gamma)$ so that, by the definition of $S^m$ \eqref{e:Sm}
and \eqref{e:sigmaGamma},
\beqs
\big| \sigma\big(\big(1-\psi_0\hDarg\big)\pert \big) \big| \leq C \langle \xi'\rangle^{-1} \quad\tfa (x',\xi')\in T^*\Gamma.
\eeqs
Then, since $1-\chi(|\xi'|_g^2)=0$ for $|\xi'|_g\leq \Xi$,
\beqs
\big| \sigma\big(\big(I-\chi \hDarg \big)\big(1-\psi_0\hDarg\big)\pert \big) \big| \leq C \langle \xi'\rangle^{-1}\leq C' \Xi^{-1} \quad\tfa (x',\xi')\in T^*\Gamma.
\eeqs
The bound \eqref{e:cleaners1} then holds by Lemma~\ref{l:HsEstimates}, and the the proof is complete.

\subsection{Proof of Theorem \ref{thm:FC}  (collocation method with trigonometric polynomials)}

First observe that the bound \eqref{e:MR5} on the relative error then follows
from combining \eqref{e:MR4}
with Theorem \ref{thm:superalgebraicFG}, \eqref{e:SVclevertrick}, and Assumption \ref{ass:polyboundintro}, since the superalgebraic decay in \eqref{e:superalgebraicFG} (with $P_{\trig_N}^G$ replaced by $P_{\trig_N}^C$) absorbs the $\rho$ on the right-hand side of \eqref{e:MR4}.
Therefore, we only need to prove the quasi-optimality bound \eqref{e:MR4}.

As in the proof of Theorem \ref{thm:FG}, we use the identity \eqref{e:thething} under the condition \eqref{e:Neumann0}.
Let $\chi\in C_{c}^\infty(\Rea;[0,1])$ be such that Lemma \ref{l:proj2Pseudo} holds with $\Theta=1$. Then, by \eqref{e:SVclevertrick} and Lemma \ref{l:proj2Pseudo}, if $N\geq (c_{\max} + \epsilon)k$, 
\begin{align}\nonumber
(P_{\trig_N}^C-I)\pert (I+\pert)^{-1} 
&= (P_{\trig_N}^C-I)(I-P_{\trig_N}^G)\pert(I+\pert)^{-1}\\ \nonumber
&= (P_{\trig_N}^C-I)(I-P_{\trig_N}^G)(1-\chi\hDarg)\pert(I+\pert)^{-1}+O(k^{-\infty})_{\Psi^{-\infty}(\Gamma)}\\
&= (P_{\trig_N}^C-I)(1-\chi\hDarg)\pert(I+\pert)^{-1}+O(k^{-\infty})_{\Psi^{-\infty}(\Gamma)}.\label{e:theThingItIs}
\end{align}
We now claim that if $N\geq (c_{\max} + \epsilon)k$ then 
\beq\label{e:collocateOK}
\N{(P^C_N-I)(I-\chi \hDarg) \pert(I+\pert)^{-1}}_{\Hsht}\leq C\frac{k}{N}.
\eeq
Indeed, since 
 $$
 \|I-P_{\trig_N}^C\|_{H_k^{s+1}(\Gamma)\to \Hsh}\leq C\frac{k}{N}, 
 $$
for $s>1/2$ (by \eqref{e:approx_FC}),
to prove \eqref{e:collocateOK} it is sufficient to prove that 
\beq\label{e:collocateOK2}
\big\|(I-\chi \hDarg) \pert(I+\pert)^{-1}\big\|_{H^s_k(\Gamma)\to H^{s+1}_k(\Gamma)}\leq C.
\eeq
Let $\widetilde{\chi} \in C_{c}^\infty(\mathbb{R};[0,1])$ be such that $\supp(1-\widetilde{\chi})=\emptyset$ and $\supp(1-\chi)\cap \supp\widetilde \chi = \emptyset$. Then, by Lemma \ref{lem:WF}, \eqref{e:WF_prod}, \eqref{e:WF_residual}, and Assumption \ref{ass:polyboundintro}, 
\begin{align*}
(I-\chi \hDarg) \pert(I+\pert)^{-1}= (I-\chi \hDarg) \pert(I-\widetilde\chi \hDarg)(I+\pert)^{-1} + O(k^{-\infty})_{\Psi^{-\infty}}.
\end{align*}
Part (ii) of Assumption \ref{ass:abstract1} and  Part (ii) of Theorem \ref{thm:basicP} imply that $\|(1-\chi\hDarg)\pert\|_{\Hsh\to H_k^{s+1}(\Gamma)}\leq C$, and then the bound \eqref{e:collocateOK2} follows by combining this with \eqref{e:inverse_cutoff_left}, and \eqref{e:invEst}.

The combination of~\eqref{e:collocateOK} and~\eqref{e:theThingItIs} implies that if $N\geq (c_{\max} + \epsilon)k$ then 
 \begin{equation}
 \label{e:collocateOK2a}
 \|(I+(P_{\trig_N}^C-I)\pert(I+\pert)^{-1})^{-1}\|_{\Hsht}\leq 1+C\frac{k}{N}.
 \end{equation}
Therefore,  by~\eqref{e:thething}, to prove Theorem~\ref{thm:FC} we only need to bound $(I+L)^{-1}(I-P_{\trig_N}^C)$. 
Now, by \eqref{e:Dean1a} and the fact that $L_{\max}=1$,
\begin{align*}
&\N{(I+\pert)^{-1}(I-P_{\trig_N}^C) }_{\Hsht} \\
&\leq\N{(I+\pert)^{-1}(1-\chi(-\hsc^2\Delta_g))(I-P_{\trig_N}^C) }_{\Hsht} +\N{(I+\pert)^{-1}\chi(-\hsc^2\Delta_g)(I-P_{\trig_N}^C) }_{\Hsht}\\
&\leq  \|I-P_{\trig_N}^C\|_{\Hsht}+C\hsc +\N{(I+\pert)^{-1}\chi(-\hsc^2\Delta_g)(I-P_{\trig_N}^C) }_{\Hsht}.
\end{align*}
Finally, by the approximation property in Lemma \ref{lem:approx_FC} with $t=s$ and $q=0$, the smoothing property of $\chi(-\hsc^2\Delta_g)$, and Lemma \ref{lem:2},
\beq\label{e:laundry1}
\N{(I+\pert)^{-1}(I-P_{\trig_N}^C) }_{\Hsht} \leq \|I-P_{\trig_N}^C\|_{\Hsht}+C\hsc +C\left(\frac{k}{N}\right)^{s} \rho.
\eeq
The combination of \eqref{e:thething}, \eqref{e:collocateOK2a}, and  \eqref{e:laundry1} completes the proof of \eqref{e:MR4}, and the proof of Theorem \ref{thm:FC} is complete.

\

Before ending this section, we record the following corollary of our proof that we use in our study of the Nystr\"om method.
\begin{lemma}
\label{lem:inverseForm}
Suppose that $\operator$ satisfies Assumptions~\ref{ass:abstract1} and~\ref{ass:polyboundintro}. Then for all $s\geq 0$, $k_0>0$ there there is $C>0$ such that for all $k\geq k_0$, $k\notin \mathcal{J}$, and $N\geq Ck$, 
\begin{equation}
\label{e:invFormula}
(I+P_{\trig_N}^C\pert)^{-1}= (I+\pert)^{-1}\Big(I+(P_{\trig_N}^C-I)\pert(I+\pert)^{-1}\Big)^{-1}.
\end{equation}
\end{lemma} 
\bpf
The combination of \eqref{e:theThingItIs} and \eqref{e:collocateOK}
imply that
\begin{equation*}
\|(P_{\trig_N}^C-I)\pert(I+\pert)^{-1}\|_{\Hshts}\leq C'k/N.
\end{equation*}
The result  \eqref{e:invFormula} then follows from \eqref{e:JeffFav1}.
\epf

\section{{Description of the Nystr\"om method}}
\label{sec:Nystrom}

\begin{assumption}[Class of weakly-singular integrals]\label{ass:Nystrom}
There exist $\blue{L_j\in C^\infty( \mathbb{R}^2/2\pi \mathbb{Z}^2)}$, $j=1,2$, such that
\beq\label{e:NystromL}
(Lv)(t):=\int^{2\pi}_0 \log\left( 4\sin^2 \left(\frac{t-\tau}{2}\right)\right) L_1(t,\tau) \, v(\tau) \, \rd \tau + \int_0^{2\pi} L_2(t,\tau)\, v(\tau)\, \rd \tau.
\eeq
\end{assumption}

\begin{definition}[Approximation by quadrature]\label{def:quad}
With $L$ defined by \eqref{e:NystromL} and the collocation projection $P_{\trig_N}^C$ defined 
by \eqref{e:collocation_def_recap}
with an odd number of evenly-spaced points \eqref{e:trig_points},
\beq\label{e:LN}
(\LN^N v)(t):=\int^{2\pi}_0 \log\left( 4\sin^2 \left(\frac{t-\tau}{2}\right)\right) P_{\trig_N,\tau}^C \big(L_1(t,\tau) \, v(\tau)\big) \, \rd \tau + \int_0^{2\pi}
P_{\trig_N,\tau}^C \big(L_2(t,\tau)\, v(\tau)\big)\, \rd \tau.
\eeq
\end{definition}

Using the expression \eqref{e:PNC} for $P_{\trig_N,\tau}^C$ and the explicit expressions for integrals of trigonometric polynomials against the $\log$ factor (see, e.g., \cite[Lemma 8.23]{Kr:14}),
one can write $(\LN^N v)(t)$ defined by \eqref{e:LN} in terms of $L_j(t,t_j)v(t_j), j=0,\ldots,2N$, multiplying trigonometric polynomials in $t$; see, e.g., \cite[Equations 12.18-12.20]{Kr:14}.
The Nystr\"om method we consider for computing approximations to $(I+L)v = c_0^{-1}f$ is then \eqref{e:nystromForm} with $L_V= \LN^N$ and $P_V=  P_{N}^C$.

Our estimates for the Nystr\"om method are given in terms the following quantities (see Theorem~\ref{thm:MAT2} below): for $j=1,2$, let
\beq\label{e:widehatL}
\widehat{L}_{j,m}(t):=\frac{1}{\sqrt{2\pi}}\int_0^{2\pi} e^{-i\tau m}L_j(t,\tau)d\tau,
\eeq
and for $N\geq 0$, $0\leq \oldK\leq N$, $s\in \mathbb{R}$, $0<\e<1$, let
\begin{equation}
\label{e:defFs}
\begin{gathered}
\FL^{s,\e}(N,\oldK ,L):=k^{-1}\sum_{N-\oldK <|m|\leq (2-\e)N-\oldK }\|\widehat{L}_{1,m}\|_{\Hsk}\langle m/k\rangle^s,\\
\FH^{s,\e}(N,\oldK ,L):=\sum_{|m|> (2-\e)N-\oldK }\Big(\|\widehat{L}_{1,m}\|_{\Hsk}+\|\widehat{L}_{2,m}\|_{\Hsk}\Big)\langle m/k\rangle^s.
\end{gathered}
\end{equation}
We also write $\FL^{s,\e}(N,L):=\FL^{s,\e}(N,N,L)$ and $\FH^{s,\e}(N,L):=\FH^{s,\e}(N,N,L)$.

We now recall the standard method (a.k.a.~``Kress quadrature") for writing the Dirichlet \eqref{e:DBIEs} and Neumann  \eqref{e:NBIEs} BIEs with the perturbation $\pert$ satisfying Assumption~\ref{ass:Nystrom}; this method is based on the splitting
\beq\label{e:Kress1}
\frac{\ri }{4} H_0^{(1)}(\mu) = -\frac{1}{4\pi} J_0\big(\mu\big) \log \mu^2 +T\big(\mu,k\big),
\eeq 
where both $J_0(\mu)$ and $T(\mu,k)$ are analytic in $\mu$.

\begin{lemma}
\label{l:standardSplitA}
 The operators $\eta_D S_k$, $\DL_k$, and $\DL_k'$ satisfy  Assumption \ref{ass:Nystrom}.
In all three cases, there are functions $\widetilde{L}_j:\mathbb{R}^2\times \mathbb{T}_t\times \mathbb{T}_\tau$ and $N_0\in \mathbb{R}$ such that 
\begin{equation}
\begin{gathered}
\label{e:fourierSplit}
L_j(t,\tau)=\widetilde{L}_j\big(k(\gamma(t)-\gamma(\tau)),t,\tau\big),\qquad j=1,2,\\
\widetilde{L}_1(\mu,t,\tau)=k\int_{\mathbb{S}^1} e^{i\langle \mu,\omega\rangle}f(\omega,t,\tau)dS(\omega),\qquad |\partial^\alpha f|\leq C_\alpha,\\
\supp \mathcal{F}_{\mu\to \xi}\big(\partial_t^\alpha\partial_\tau^\beta\widetilde{L}_2(\cdot,t,\tau)\big)(\xi)\subset \big\{|\xi|\leq 1\big\}, \qquad |\partial^\alpha \widetilde{L}_2(\mu ,t,\tau)|\leq C_\alpha k^{N_0}.
\end{gathered}
\end{equation}
\blue{Here, we write $dS(\omega)$ for the arc length measure on 
$$
\mathbb{S}^1:=\{ x\in \mathbb{R}^2\,:\, |x|=1\}.
$$}
\end{lemma}
\begin{proof}
We first write 
\begin{align*}
&\frac{\ri }{4} H_0^{(1)}(k|x-y|)\\
&=-\frac{1}{2\pi} J_0\big(k|x-y|\big) \log (k|x-y|)+ \frac{\ri }{4} H_0^{(1)}(k|x-y|)+\frac{1}{2\pi} J_0\big(k|x-y|\big) \log (k|x-y|) \\
&=-\frac{1}{4\pi} J_0\big(k|x-y|\big) \log |x-y|^2+ \frac{\ri }{4} H_0^{(1)}(k|x-y|)+\frac{1}{2\pi} J_0\big(k|x-y|\big) \log (k|x-y|) \\
&\hspace{1.5cm}-\frac{\log k}{2\pi} J_0\big(k|x-y|\big) \\
&=:-\frac{1}{4\pi} J_0\big(k|x-y|\big) \log |x-y|^2 + S_1\big(k|x-y|,k\big).
\end{align*}
By the asymptotics of $H_0^{(1)}(z)$ as $z\to 0$ (see, e.g., \cite[\S10.8]{Di:25}), 
$\frac{\ri }{4} H_0^{(1)}(z)+\frac{1}{2\pi} J_0(z) \log (z)$ is analytic at $z=0$, and thus $S_1(\mu,k)$ is analytic in $\mu$.

Next, observe that, since $\gamma:\mathbb{R}/2\pi \mathbb{Z}\to \Gamma$ is smooth, $|\gamma'|\geq c>0$, and $\gamma$ is bijectitve,
\begin{equation*}
e(t,\tau):=\frac{|\gamma(t)-\gamma(\tau)|^2}{4\sin^2\big(\frac{t-\tau}{2}\big)}\in C^\infty\big((\mathbb{R}/2\pi \mathbb{Z})^2\big),\qquad e(t,\tau)>c>0.
\end{equation*}
In particular, with $f(t,\tau):=\log e(t,\tau)\in C^\infty$, 
\begin{equation}
\label{e:distance}
\log\big(|\gamma(t)-\gamma(\tau)|^2\big)=:\log\Big(4\sin^2\big(\tfrac{t-\tau}{2}\big)\Big)+f(t,\tau),
\end{equation}
so that 
\begin{align*}
&\frac{\ri }{4} H_0^{(1)}(k|\gamma(t)-\gamma(\tau)|)\\
&=-\frac{1}{4\pi} J_0\big(k|\gamma(t)-\gamma(\tau)|\big) \log \Big(4\sin^2\big(\tfrac{t-\tau}{2}\big)\Big) -\frac{1}{4\pi} J_0\big(k|\gamma(t)-\gamma(\tau)|\big)f(t,\tau)+ S_1\big(k|\gamma(t)-\gamma(\tau)|,k\big).
\end{align*}
Thus, for $\eta_D S_k$ we set
\begin{gather*}
L_1(t,\tau):=-\frac{\eta_D}{4\pi} J_0\big(k|\gamma(t)-\gamma(\tau)|\big) ,\\ L_{2}(t,\tau):=-\frac{\eta_D}{4\pi} J_0\big(k|\gamma(t)-\gamma(\tau)|\big)f(t,\tau)+ \eta_D S_1\big(k|\gamma(t)-\gamma(\tau)|,k\big).
\end{gather*}
The required properties for $L_1(t,\tau)$ now follow from a standard integral representation of $J_0(|x-y|)$ as the Fourier transform of the surface measure on $\mathbb{S}^1$ (see, e.g., \cite[Page 154]{StWe:71}). The required properties for $L_2(t,\tau)$ follow from the Payley--Wiener theorem (see, e.g., \cite[Theorem 4.1]{StWe:71}), the fact that $J_0(\mu)$ and $S_1(\mu,k)$ are analytic and the asymptotic growth 
$$
|J_0(z)|+|H_0^{(1)}(z)|\leq C\log|z| e^{|\Im z|} \quad\tfa z \in \mathbb{C}\setminus (-\infty,0]
$$
see, e.g., \cite[\S10.8 and \S10.17]{Di:25}. 
The proofs for $\DL_k$ and $\DL_k'$ are nearly identical, using instead the  same analyticity and growth properties for $\partial_\mu S_1(\mu,k)$. 
\end{proof}

\begin{lemma}
\label{l:standardSplitB}
The operators $k\Reg$, $\DL_{\ri k}$, $\DL_{\ri k}'$ satisfy Assumption~\ref{ass:Nystrom} and in all three cases, there are functions $\widetilde{L}_j:\mathbb{R}^2\times \mathbb{T}_t\times \mathbb{T}_\tau$ such that 
\beq\label{e:cutSplit0}
L_j(t,\tau)=\widetilde{L}_j\big(k(\gamma(t)-\gamma(\tau)),t,\tau,k\big),\qquad j=1,2,
\eeq
and for all $j=1,2$, $\alpha\in \mathbb{N}$, $N>0$, there is $C_{\alpha N}>0$ such that
\begin{equation}
\label{e:cutSplit}
 |\partial^\alpha \widetilde{L}_1(\mu,t,\tau)|\leq C_{\alpha N}k\log k \langle \mu\rangle^{-N},
 \qquad
|\partial^\alpha \widetilde{L}_2(\mu,t,\tau)|\leq C_{\alpha N}k \langle \mu\rangle^{-N}.
\end{equation}
\end{lemma}
\begin{proof}
The main difference here compared to Lemma \ref{l:standardSplitA} is that $J_0(ix)$ grows exponentially and $H_0^{(1)}(ix)$ decays exponentially as $x\to \infty$ (see, e.g., \cite[\S10.17]{Di:25}). We therefore use the upper bounds that for any $\alpha \in \mathbb{N}$ and  $\e>0$ there is $C_{\alpha \e}>0$ such that 
\beq\label{e:Bessel}
\begin{aligned}
|\partial^\alpha J_0(\ri x)|&\leq C_{\alpha \e} e^{(1+\e)|x|},&x\in \mathbb{R},\\ 
|\partial^\alpha H_0^{(1)}(\ri x)|&\leq C_{\alpha \e} e^{-(1-\e)|x|}\log |x|,& x>0.
\end{aligned}
\eeq
As in Lemma~\ref{l:standardSplitA}, with $f$ as in~\eqref{e:distance},
\begin{align*}
\frac{\ri }{4} H_0^{(1)}(\ri k|\gamma(t)-\gamma(\tau)|)&=-\frac{1}{4\pi} J_0\big(\ri k|\gamma(t)-\gamma(\tau)|\big)  \log \Big(4\sin^2\big(\tfrac{t-\tau}{2}\big)\Big) \\
&\hspace{1cm}-\frac{1}{4\pi} J_0\big(\ri k|\gamma(t)-\gamma(\tau)|\big)f(t,\tau)+ S_1(\ri k|\gamma(t)-\gamma(\tau)|,\ri k),
\end{align*}
with 
$$
|\partial^\alpha S_1(\mu,\ri k)|\leq C_{\alpha,\e} e^{(1+\e)|\mu|} \log k
$$
by \eqref{e:Bessel}. 
We let $\chi \in \mathcal{S}(\mathbb{R})$ with $\chi \equiv 1$ near 0, and 
\beq\label{e:chiBound}
|\chi(x)|\leq Ce^{-r |x|}\quad\text{ for some }\quad r>1+\e,
\eeq
and write
\begin{align*}
&\frac{\ri }{4} H_0^{(1)}(\ri k|\gamma(t)-\gamma(\tau)|)=-\frac{1}{4\pi} J_0\big(\ri k|\gamma(t)-\gamma(\tau)|\big)\chi\big(k|\gamma(t)-\gamma(\tau)|\big)  \log \Big(4\sin^2\big(\tfrac{t-\tau}{2}\big)\Big) \\
&\hspace{3.5cm}+\chi(k|\gamma(t)-\gamma(\tau)|)\Big(S_1(\ri k|\gamma(t)-\gamma(\tau)|,\ri k)-\frac{1}{4\pi} J_0\big(\ri k|\gamma(t)-\gamma(\tau)|\big)f(t,\tau)\Big)\\
&\hspace{3.5cm}+\Big(1-\chi\big(k|\gamma(t)-\gamma(\tau)|\big)\Big)\frac{\ri }{4} H_0^{(1)}\big(\ri k|\gamma(t)-\gamma(\tau)|\big).
\end{align*}
Therefore, if
\begin{align*}
L_1(t,\tau)&:=-\frac{k}{4\pi} J_0\big(\ri k|\gamma(t)-\gamma(\tau)|\big)\chi\big(k|\gamma(t)-\gamma(\tau)|\big),\\
L_2(t,\tau)&:=k\chi(k|\gamma(t)-\gamma(\tau)|)\Big(S_1(\ri k|\gamma(t)-\gamma(\tau)|,\ri k)-\frac{1}{4\pi} J_0\big(\ri k|\gamma(t)-\gamma(\tau)|\big)f(t,\tau)\Big)\\
&\qquad+k\Big(1-\chi\big(k|\gamma(t)-\gamma(\tau)|\big)\Big)\frac{\ri }{4} H_0^{(1)}\big(\ri k|\gamma(t)-\gamma(\tau)|\big),
\end{align*}
then the result for $kS_{\rm i k}$ follows from the bounds \eqref{e:Bessel} and \eqref{e:chiBound}.
For $\DL_{\rm ik}$ and $\DL_{\rm ik}'$, the estimates follow similarly after taking appropriate derivatives of $H_0^{(1)}$, $J_0$, and $S_1$.
\end{proof}

\begin{lemma}
\label{l:standardSplitC}
The operator $k^{-1}(H_k-H_{\ri k})$ satisfies Assumption~\ref{ass:Nystrom} and there are functions $\widetilde{L}_{j,i}:\mathbb{R}^2\times \mathbb{T}_t\times \mathbb{T}_\tau$, $j=1,2$, $i=1,2$, and $N_0\in \Rea$ such that 
$$
\begin{gathered}
L_j(t,\tau)=\widetilde{L}_{j,1}\big(k(\gamma(t)-\gamma(\tau)\big),t,\tau,k)+\widetilde{L}_{j,2}\big(k(\gamma(t)-\gamma(\tau)),t,\tau,k\big),\qquad j=1,2,
\end{gathered}
$$
and for $j=1,2$,  all $\alpha\in \mathbb{N}$, $N>0$, there are $C_\alpha$ and $C_{\alpha N}>0$ such that
$$
\begin{gathered}
\widetilde{L}_{1,1}(\mu,t,\tau)=k\int_{\mathbb{S}^1} e^{i\langle \mu,\omega\rangle}f(\omega,t,\tau)dS(\omega),\qquad |\partial^\alpha f|\leq C_\alpha,\\
\supp \mathcal{F}_{\mu\to \xi}\big(\partial_t^\alpha\partial_\tau^\beta\widetilde{L}_{2,1}(\cdot,t,\tau)\big)(\xi)\subset \{|\xi|\leq 1\}, \qquad |\partial^\alpha \widetilde{L}_{2,1}(\mu ,t,\tau)|\leq C_\alpha k^{N_0}
\end{gathered}
$$
(compare to \eqref{e:fourierSplit})
and
$$
|\partial^\alpha \widetilde{L}_{1,2}(\mu,t,\tau)|\leq C_{\alpha N}k \log k\langle \mu\rangle^{-N},\qquad |\partial^\alpha \widetilde{L}_{2,2}(\mu,t,\tau)|\leq C_{\alpha N}k \langle \mu\rangle^{-N}
$$
(compare to \eqref{e:cutSplit}).
\end{lemma}
\begin{proof}
Recall (from, e.g, \cite[\S10.8]{Di:25}) that 
$$
H^{(1)}_1(z)+i\frac{2}{\pi z} -i\frac{2}{\pi}J_1(z)\log \frac{z}{2}=:S_2(z)
$$
is analytic. 
Bessel's equation and the fact that $H^{(1)'}_0(\mu)= - H^{(1)}_1(\mu)$ imply that 
$$
H^{(1)''}_0(\mu)=-\mu^{-1} H^{(1)'}_0(\mu)-H^{(1)}_0(\mu)=\mu^{-1} H^{(1)}_1(\mu)-H^{(1)}_0(\mu).
$$
Therefore, 
\begin{align*}
&H^{(1)''}_0(\mu)+H^{(1)''}_0(i\mu)\\
&=\mu^{-1}\Big(-iH_1^{(1)}(i\mu)+H_1^{(1)}(\mu)\Big)-H_0^{(1)}(\mu)-H_0^{(1)}(i\mu)\\
&=\mu^{-1}\bigg(-iS_2(i\mu)+S_2(\mu)+\frac{2}{\pi}J_1(i\mu)\log \frac{i\mu}{2}+i\frac{2}{\pi}J_1(\mu)\log\frac{\mu}{2}\bigg)-H^{(1)}_0(\mu)-H^{(1)}_0(i\mu)\\
&=\mu^{-1}\bigg(iJ_1(i\mu)-iS_2(i\mu)+S_2(\mu)+\frac{2}{\pi}\big(J_1(i\mu)+iJ_1(\mu)\big)\log \frac{\mu}{2}\bigg)-H^{(1)}_0(\mu)-H^{(1)}_0(i\mu).
\end{align*}
Recall that $J_1$ is analytic, vanishes at $0$, and for all $\alpha\in \mathbb{N}$, there are $C>0$ and $N>0$ such that 
$$
|J_1(z)|\leq C \langle z\rangle^Ne^{|\Im z|}\quad\tand\quad
|S_2(z)|\leq C \langle z\rangle^Ne^{|\Im z|}\quad \tfa z\in \Com.
$$
The result then follows by introducing cutoffs on all terms where Bessel functions, $J_0, J_1$ are evaluated at $ik|\gamma(t)-\gamma(\tau)|$ as in Lemma~\ref{l:standardSplitB} to obtain $\widetilde{L}_{j,2}$, and analyzing the remaining terms as in Lemma~\ref{l:standardSplitA} to obtain $\widetilde{L}_{j,1}$.
\end{proof}

\section{Abstract result about the convergence of the Nystr\"om method}
\label{sec:NystromProof}

For our abstract theorem on the Nystr\"om method, we assume that  $\operator=I+\pert$ satisfies Assumption~\ref{ass:abstract1} with
\beq\label{e:nystromSetup}
\pert= \sum_{i=1}^I \widetilde{\pert}_i+\sum_{j=1}^J  \widetilde{\pert}_{j,a}\widetilde{\pert}_{j,b}, \qquad \pert_N:=\sum_{i=1}^I\widetilde{\LN}_{i}^N+\sum_{j=1}^J  \widetilde{\LN}_{j,a}^NP_{\trig_N}^C\widetilde{\LN}_{j,b}^N,
\eeq
and $\widetilde{\pert}_i$, $\widetilde{\pert}_{j,\cdot}$ satisfying Assumption~\ref{ass:Nystrom} and $\widetilde{\LN}_{i}^N$ and $\widetilde{\LN}_{j,\cdot}^N$ defined as in Definition~\ref{def:quad}.
We further assume that for any $\chi \in C_{c}^\infty$ with $\supp(1-\chi)\cap[-1,1]=\emptyset$, $(1-\chi\hDarg)\widetilde{\pert}_{j,b},\,\widetilde{\pert}_{j,b}(1-\chi\hDarg)\in \Psi^{-1}(\Gamma)$. 

We define
\begin{gather*}
\FLm^{s,\e}(N,\oldK ):=\max\Big(\sup_{i}\FL^{s,\e}(N,\oldK ,\widetilde{\pert}_{i}),\sup_{j}\FL^{s,\e}(N,\oldK ,\widetilde{\pert}_{j,a}),\sup_j \FL^{s,\e}(N,\oldK ,\widetilde{\pert}_{j,b})\Big),\\ 
\FHm^{s,\e}(N,\oldK ):=\max\Big(\sup_{i}\FH^{s,\e}(N,\oldK ,,\widetilde{\pert}_{i}),\sup_{j}\FH^{s,\e}(N,\oldK ,\widetilde{\pert}_{j,a}),\sup_j \FH^{s,\e}(N,\oldK ,\widetilde{\pert}_{j,b})\Big),\\
\|\widetilde{\pert}\|_{s,r}:= \sum_{j}\|\widetilde{L}_{j,a}\|_{\Hsh\to \Hsh}+\|\widetilde{L}_{j,b}\|_{H_{k}^{r}(\Gamma)\to H_{k}^{r}(\Gamma)},\\ \mathcal{B}^{r,s}(\widetilde{L}):=\sum_{j} \|\widetilde{L}_{j,a}\|_{\Hsh\to \Hsh}\|\widetilde{L}_{j,b}\|_{H_{k}^{r}(\Gamma)\to H_{k}^{r}(\Gamma)}.
\end{gather*}
We also write $\FLm^{s,\e}(N):=\FLm^{s,\e}(N,N)$ and $\FHm^{s,\e}(N):=\FHm^{s,\e}(N,N)$. 

\begin{theorem}[New abstract result about the convergence of the Nystr\"om method]\label{thm:MAT2}
Suppose that $\operator$ satisfies Assumptions \ref{ass:abstract1} and \ref{ass:polyboundintro}, and $L$ and $L_N$ are as in \eqref{e:nystromSetup}.
Given $k_0>0$, $M>0$, $t^*\geq s\geq t>1/2$, there exist $c,C>0$ such that the following holds.

Given $f\in H^s_k(\Gamma)$, let $v\in H^s_k(\Gamma)$ be the solution of
\beq\label{e:MATequation2}
\operator v:=c_0(I+L)v =f.
\eeq
If $k\geq k_0$,
\beq\label{e:MATthresa0}
\left( \frac{k}{N}\right)^{t^*-s} \rho \leq c,
\eeq
\begin{align}\label{e:MATthresb0a}
\bigg[1+\rho\Big(\frac{k}{N}\Big)^{s-t}\bigg]\Big(\frac{k}{N}\Big)\FLm^{s,\e}(N)\Big(1+\|\widetilde{L}\|_{s,s}+\|\widetilde{L}\|_{t,t}\Big)&\leq c\\
\label{e:MATthresb0b}
\bigg[1+\rho\Big(1+\Big(\frac{k}{N}\Big)^{s+1}\Big)\bigg]\FHm^{s,\e}(N)\Big(1+\|\widetilde{L}\|_{s,s}+\|\widetilde{L}\|_{t,t}\Big)&\leq c, \quad\tand\\
\label{e:MATthresb0c}
\bigg[1+\rho\Big(\frac{k}{N}\Big)^{s+1}\bigg]\Big(\frac{k}{N}\Big)\mathcal{B}^{s,s}(\widetilde{L})+
\rho \Big(\frac{k}{N}\Big)^{s-t+1}
\mathcal{B}^{t,s}(\widetilde{L})&\leq c,
\end{align}
then the solution $v_N \in V_N$ to
\beq\label{e:MATprojection2}
(I+P_{\trig_N}^C \pert_N) v_N = P_{\trig_N}^C f
\eeq
exists, is unique, and satisfies the error estimate
\begin{align}\nonumber
&\N{v-v_N}_{H^s_k(\Gamma)}\\
\nonumber
&\hspace{1cm}\leq C\bigg[\Big(1+\rho \Big(\frac{k}{N}\Big)^{s}\Big)\|(I-P_{\trig_N}^C)v\big\|_{\Hsh}+ \Big(1+\rho \Big(\frac{k}{N}\Big)^{s+1}\Big)\big\|P_{\trig_N}^C(\pert_N-\pert)P_{\trig_N}^Cv\big\|_{\Hsh}\\
&\hspace{3cm} +\rho\big\|P_{\trig_N}^C(\pert_N-\pert)P_{\trig_N}^Cv\big\|_{L^2(\Gamma)}\bigg],
\label{e:MATqo2}
\\ \nonumber
&\leq C\bigg[\Big(1+\rho \Big(\frac{k}{N}\Big)^{s}\Big)\Big(\big\|(I-P_{\trig_N}^C)v\big\|_{\Hsh}+ \big\|P_{\trig_N}^C(\pert_N-\pert)P_{\trig_N}^Cv\big\|_{\Hsh}\Big)\\
&\hspace{1cm} +\rho\bigg( \Big[\Big(\frac{k}{N}\Big)^{s+1} \FLm^{s,\e}(N)+\FHm^{0,\e}(N)\Big]
\big(1+\|\widetilde{L}\|_{0,0}+ \|\widetilde{L}\|_{s,s}\big)
+ \Big(\frac{k}{N}\Big)^{s+1}\mathcal{B}^{0,s}(\widetilde{L})\bigg)\|v\|_{\Hsh}\bigg].
\label{e:MATqo2new}
\end{align}
\end{theorem}


Before proving Theorem~\ref{thm:MAT2}, we need a few technical lemmas; the first is an analogue of Lemma~\ref{lem:HFinverse} for $(I+P_{\trig_N}^C\pert)^{-1}$.
\begin{lemma}
\label{l:splitCollocate}
Suppose that $\operator$ satisfies Assumptions~\ref{ass:abstract1} and~\ref{ass:polyboundintro}. 
Given $t^*> 1/2$ 
and
$\chi \in C_{c}^\infty$ with $\supp(1-\chi)\cap[-1,1]=\emptyset$, there exist $C,c>0$ such that if $(k/N)^{t^*-s} \rho \leq c$ then 
\begin{equation}
\label{e:recordHighLow1}
\big\|(I+P_{\trig_N}^C\pert)^{-1}\big(I-\chi(-k^{-2}\Delta_\Gamma)\big)\big\|_{\Hsht}\leq C\Big(1+\rho \Big(\frac{k}{N}\Big)^{s+1}\Big)
\eeq
and
\beq
\label{e:recordHighLow2}
\big\|(I+P_{\trig_N}^C\pert)^{-1}\chi(-k^{-2}\Delta_\Gamma)\big\|_{L^2(\Gamma)\to \Hsh}\leq C\rho .
\end{equation}
\end{lemma}
\bpf
By Lemma~\ref{lem:inverseForm}, if $N\geq C k$ then
\beq\label{e:Friday1}
(I+P_{\trig_N}^CL)^{-1}=(I+L)^{-1}\Big(I+(P_{\trig_N}^C-I)\pert(I+\pert)^{-1}\Big)^{-1}.
\eeq
We prove the bounds \eqref{e:recordHighLow1} and \eqref{e:recordHighLow2} using \eqref{e:neumannOut}.
Observe that $P_{\trig_N}^C$ satisfies~\eqref{e:MATproj} for $0\leq q\leq t\leq \infty$ and $t>1/2$ by \eqref{e:approx_FC}; i.e., $t_{\max}$ in the arguments in \S\ref{sec:proofMAT} can be taken to be arbitrarily large. We therefore apply \eqref{e:neumannOut} 
with $q_{\min}=0$ and $t_{\max}=t^*$, so that the condition \eqref{e:MATthres} under which \eqref{e:neumannOut} holds becomes $(k/N)^{t^*} \rho \leq c$.

Given $\chi$ as in the statement of the lemma, choose $\e>0$ in the definitions of $\Pi_{\mathscr{L}}$ and $\Pi_{\mathscr{H}}$ \eqref{e:highandlow} so that 
\beqs
1-\chi(-k^{-2}\Delta_\Gamma)=\Pi_{\mathscr{H}}(1-\chi(-k^{-2}\Delta_\Gamma)).
\eeqs
Then choose $\chi^< \in C_{c}^\infty $ with $\supp (1-\chi^<)\cap [-1,1]=\emptyset$  such that 
\beqs
\chi^<(-k^{-2}\Delta_\Gamma)=\chi^<(-k^{-2}\Delta_\Gamma)\Pi_{\mathscr{L}}
\eeqs
(this requires, in particular, that $\supp \chi^{<} \cap\supp(1-\chi)=\emptyset$).
Finally, choose $\chi^{\ll} \in C_{c}^\infty $ with $\supp (1-\chi^<)\cap [-1,1]=\emptyset$ such that $\supp \chi^{\ll}\cap \supp(1-\chi^<)=\emptyset$.
Now, by \eqref{e:Friday1}, Lemma \ref{lem:2}, Lemma \ref{lem:HFinverse}, and Corollary \ref{cor:WFcutoff},
\begin{align*}
&\big\|(I+P_{\trig_N}^C\pert)^{-1}\big(I-\chi(-k^{-2}\Delta_\Gamma)\big)\big\|_{\Hshts}\\
&=
\big\|(I+L)^{-1}\big(I+(P_{\trig_N}^C-I)\pert(I+\pert)^{-1}\big)^{-1}\big(I-\chi(-k^{-2}\Delta_\Gamma)\big)\big\|_{\Hsht}
\\
&\leq 
\big\|(I+L)^{-1}\chi^<(-k^{-2}\Delta_\Gamma)\big(I+(P_{\trig_N}^C-I)\pert(I+\pert)^{-1}\big)^{-1}\big(I-\chi(-k^{-2}\Delta_\Gamma)\big)\big\|_{\Hshts}
\\
&\quad
+
\big\|(I+L)^{-1}(I-\chi^{\ll}(-k^{-2}\Delta_\Gamma))\big(I-\chi^<(-k^{-2}\Delta_\Gamma)\big)\big(I+(P_{\trig_N}^C-I)\pert(I+\pert)^{-1}\big)^{-1}\big(I-\chi(-k^{-2}\Delta_\Gamma)\big)\big\|_{\Hshts}\\
&\leq 
C\rho \big\|\chi^<(-k^{-2}\Delta_\Gamma)\big(I+(P_{\trig_N}^C-I)\pert(I+\pert)^{-1}\big)^{-1}\big(I-\chi(-k^{-2}\Delta_\Gamma)\big)\big\|_{\Hshts}
\\
&\qquad\qquad
+
C\big\|\big(I-\chi^<(-k^{-2}\Delta_\Gamma)\big)\big(I+(P_{\trig_N}^C-I)\pert(I+\pert)^{-1}\big)^{-1}\big(I-\chi(-k^{-2}\Delta_\Gamma)\big)\big\|_{\Hshts}
\\
&\leq 
C\rho \big\|\Pi_{\mathscr{L}}\big(I+(P_{\trig_N}^C-I)\pert(I+\pert)^{-1}\big)^{-1}\Pi_{\mathscr{H}}\big\|_{\Hshts}
\\
&\qquad\qquad
+
C\big\|\big(I+(P_{\trig_N}^C-I)\pert(I+\pert)^{-1}\big)^{-1}\Pi_{\mathscr{H}}\big\|_{\Hshts}.
\end{align*}
Therefore, by \eqref{e:neumannOut}, if $(k/N)^{t^*} \rho \leq c$ then
\begin{align*}
\big\|(I+P_{\trig_N}^C\pert)^{-1}\big(I-\chi(-k^{-2}\Delta_\Gamma)\big)\big\|_{\Hshts}\leq C\rho (k/N)^{s+1} + C (k/N)^{s+1} + 1 + C k/N,
\end{align*}
which implies~\eqref{e:recordHighLow1}.

For \eqref{e:recordHighLow2}, 
we apply again
\eqref{e:neumannOut} (with $q_{\min}=0$ and $t_{\max}=t^*$), except this time, given $\chi$ as in the statement of the lemma, we choose $\e>0$  
 in the definitions of $\Pi_{\mathscr{L}}$ and $\Pi_{\mathscr{H}}$ \eqref{e:highandlow} so that 
\beqs
\chi(-k^{-2}\Delta_\Gamma)=\Pi_{\mathscr{L}}\chi(-k^{-2}\Delta_\Gamma).
\eeqs
Then, by \eqref{e:Friday1}, \eqref{e:neumannOut}, and Lemma \ref{lem:2}, 
 if  $(k/N)^{t^*-s}\rho \leq c$ then 
\begin{align}\nonumber
\big\|(I+P_{\trig_N}^C\pert)^{-1}\chi(-k^{-2}\Delta_\Gamma)\big\|_{\Hshts}
&\leq C \rho \big\|\big(I+(P_{\trig_N}^C-I)\pert(I+\pert)^{-1}\big)^{-1}\chi(-k^{-2}\Delta_\Gamma)\big\|_{\Hshts}\\
&\leq C \rho \big\|\big(I+(P_{\trig_N}^C-I)\pert(I+\pert)^{-1}\big)^{-1}\Pi_{\mathscr{L}}\big\|_{\Hshts} \nonumber
\\
& \leq C \rho
\label{e:tempBound1}
\end{align}
Finally, 
let $\chi^>\in C_{c}^\infty$ with $\supp(1-\chi^>)\cap[-1,1]=\emptyset$, and $\supp \chi \cap \supp(1-\chi^>)=\emptyset$. By the definition of $\|\cdot\|_{\Hsh}$ \eqref{e:weightedNormGamma} and the fact that $\widetilde\chi$ has compact support, $\chi^>(-k^{-2}\Delta_\Gamma): L^2(\Gamma)\to \Hsh$ with bounded norm. 
The result \eqref{e:recordHighLow2} then follows 
from the bound \eqref{e:tempBound1}, the fact that $\chi = \chi \chi^>$, and the smoothing property of $\chi^>(-k^{-2}\Delta_\Gamma)$.
\epf

\ble[Error estimate in terms of the discrete inverse]\label{lem:QOabs2}
If $P_N: \Hsh \to V_N$ is a projection,
$\pert_N : \Hsh\to \Hsh$, and $I+ P_N\pert_NP_N : H^s_k(\Gamma)\to H^s_k(\Gamma)$ is invertible, then
the solution $v_N\in V_N$ to \eqref{e:MATprojection2}
exists, is unique, and satisfies
\begin{equation}\label{e:QOabs2}
v-v_N = \big(I+P_NL_NP_N\big)^{-1}\Big((I-P_NL)(I-P_N)v +P_N(L_N-L)P_Nv\Big).
\end{equation}
\ele

\bpf
First observe that if $v_N\in H_{k}^{s}(\Gamma)$ and $(I+P_NL_N)v_N=P_Nf$, then $v_N\in V_N$, and hence $(I+P_NL_NP_N)v_N=P_Nf$. Furthermore, if $(I+P_NL_NP_N):H^s_k(\Gamma)\to H^s_k(\Gamma)$ is invertible, then there is $v_N\in H_{k}^{s}(\Gamma)$ such that $(I+P_NL_NP_N)v_N=P_Nf$. This $v_N$ satisfies $(I-P_N)v_N=0$ so that $v_N\in V_N$ and thus $(I+P_NL_N)v_N=P_Nf$.
Therefore, if $(I+P_NL_NP_N):H_{k}^{s}(\Gamma)\to H_{k}^{s}(\Gamma)$ is invertible, then
the equation \eqref{e:MATprojection2} has a unique solution in $V_N$ and it, in addition, satisfies $(I+P_NL_NP_N)v_N=P_Nf$.
By this and the fact that $(I+\pert)v=f$, 
\begin{align*}
(I+P_N \pert_NP_N) (v-v_N) = (I+P_N \pert_NP_N) v - P_N f & 
=(I-P_N)v + P_N L_N P_N v - P_N Lv
\\
&= (I-P_NL)(I-P_N)v +P_N(L_N-L)P_Nv
\end{align*}
and the result follows.
\epf

\ble\label{lem:mainevent}
Suppose that $\pert$ satisfies Assumption \ref{ass:Nystrom}, $L_N$ is defined by Definition \ref{def:quad}. 
Then, given $r\geq s\geq 0$, $0<\e<1$, there exist $C>0$ such that for all $0\leq \oldK \leq \oldK '\leq N$,
\begin{equation}\label{e:mainevent2}
\begin{aligned}
&\N{(\pert - \pert_N)P_{\trig_\oldK}^C }_{H_{k}^{r}(\Gamma)\to H^s_k(\Gamma)} \\
&\hspace{.2cm}\leq C\Bigg(\Big( \Big(\frac{k}{N}\Big)^{r-s+1} \FLm^{r,\e}(N,\oldK ')+\FHm^{s,\e}(N,\oldK ')\Big)(1+\|\widetilde{L}\|_{s,s}+ \|\widetilde{L}\|_{r,r})\\
&\hspace{.4cm}+\Big(\Big(\frac{k}{N}\Big) \FLm^{s,\e}(N)+\FHm^{s,\e}(N)\Big)\Big( \Big(\frac{k}{N}\Big)^{r-s+1} \FLm^{r,\e}(N,\oldK ')+\FHm^{s,\e}(N,\oldK ')\Big)+ \Big(\frac{k}{N}\Big)^{r-s+1}\mathcal{B}^{s,r}(\widetilde{L})\\
&\hspace{.4cm}+\max_j\Big( \Big(\frac{k}{N}\Big)^{r-s+1} \FLm^{r,\e}(N)+\FHm^{s,\e}(N)\Big)\cdot\|P_{\trig_{N}}^C(I-P_{\trig_{\oldK'}}^C)\|_{H_{k}^{r}\to H_{k}^{r}}\|(I-P_{\trig_{\oldK '}}^G)\widetilde{L}_{j,b}P_{\trig_{\oldK} }^G\|_{H_{k}^{r}\to H_{k}^{r}}\Bigg).
\end{aligned}
\end{equation}
In particular, if $\oldK =\oldK '=N$,
\begin{align}\nonumber
&\N{(\pert - \pert_N)P_{\trig_{N}}^C}_{H_{k}^{r}(\Gamma)\to H^s_k(\Gamma)} \\ \nonumber
&\hspace{.2cm}\leq C\Bigg[\Big( \Big(\frac{k}{N}\Big)^{r-s+1} \FLm^{r,\e}(N)+\FHm^{s,\e}(N)\Big)(1+\|\widetilde{L}\|_{s,s}+ \|\widetilde{L}\|_{r,r})\\
&\hspace{.4cm}+\Big(\Big(\frac{k}{N}\Big) \FLm^{s,\e}(N)+\FHm^{s,\e}(N)\Big)\Big( \Big(\frac{k}{N}\Big)^{r-s+1} \FLm^{r,\e}(N)+\FHm^{s,\e}(N)\Big)+ \Big(\frac{k}{N}\Big)^{r-s+1}\mathcal{B}^{s,r}(\widetilde{L})\Bigg].\label{e:mainevent}
\end{align}
Furthermore, if $\oldK '\geq \oldK +\e k$,  and $\oldK =\Xi k$ with $\Xi\geq c_{\max}+\e$, then for all $M>0$ there is $C_M>0$ such that 
\begin{align}\nonumber
&\N{(\pert - \pert_N)P_{\trig_\oldK}^C }_{H_{k}^{r}(\Gamma)\to H^s_k(\Gamma)} \\ \nonumber
&\hspace{0.5cm}\leq C\Bigg[\Big( \Big(\frac{k}{N}\Big)^{r-s+1} \FLm^{r,\e}(N,\oldK ')+\FHm^{s,\e}(N,\oldK ')\Big)(1+\|\widetilde{L}\|_{s,s}+ \|\widetilde{L}\|_{r,r})\\ \nonumber
&\hspace{1cm}+\Big(\Big(\frac{k}{N}\Big) \FLm^{s,\e}(N)+\FHm^{s,\e}(N)\Big)\Big( \Big(\frac{k}{N}\Big)^{r-s+1} \FLm^{r,\e}(N,\oldK ')+\FHm^{s,\e}(N,\oldK ')\Big)+ \Big(\frac{k}{N}\Big)^{r-s+1}\mathcal{B}^{s,r}(\widetilde{L})\Bigg]\\
&\hspace{1cm}+C_M\Big( \Big(\frac{k}{N}\Big)^{r-s+1} \FLm^{r,\e}(N)+\FHm^{s,\e}(N)\Big)\|\widetilde{L}\|_{0,0}k^{-M}.
\label{e:mainevent3}
\end{align}
\ele

We postpone the proof of Lemma \ref{lem:mainevent} and first prove Theorem \ref{thm:MAT2}.

\

\bpf[Proof of Theorem \ref{thm:MAT2}]
In a similar way to \eqref{e:JeffFav1}, we write
\begin{align}\nonumber
&I+P^C_{\trig_N} \pert_N P_{\trig_N}^C = I+ P_{\trig_N}^C\pert + P_{\trig_N}^C(\pert_N-\pert)P_{\trig_N}^C+P_{\trig_N}^C\pert(P_{\trig_N}^C-I) \\
&\qquad= (I+P_{\trig_N}^C\pert)\Big( I + (I+P_{\trig_N}^C\pert)^{-1} P_{\trig_N}^C \pert(P_{\trig_N}^C-I) + (I+P_{\trig_N}^C\pert )^{-1}P_{\trig_N}^C (\pert_N-\pert)P_{\trig_N}^C\Big) .\label{e:JeffFav2}
\end{align}
By decreasing $c$ if necessary, we see that the condition \eqref{e:MATthresa0} implies that $N\geq Ck$ for $C$ as in Lemma \ref{lem:inverseForm}; thus $I+P_{\trig_N}^C\pert$ is invertible.
Then,  
by \eqref{e:JeffFav2} and Neumann series, 
if
\beq\label{e:STPN1}
\N{(I+P_{\trig_N}^C\pert)^{-1} P_{\trig_N}^C \pert(P_{\trig_N}^C-I)}_{\Hsht}\leq 1/4 
\eeq
and
\beq\label{e:STPN2}
\N{(I+P_{\trig_N}^C\pert )^{-1}P_{\trig_N}^C (\pert_N-\pert)P_{\trig_N}^C}_{\Hsht} \leq 1/4,
\eeq
then $I+P_{\trig_N}^CL_NP_{\trig_N}^C$ is invertible with 
\beq\label{e:goodOldNeumann}
\big\|(I+P_{\trig_N}^C \pert_N P_{\trig_N}^C)^{-1}v\big\|_{\Hsh}\leq 2\big\|\big(I+P_{\trig_N}^CL\big)^{-1}v\big\|_{\Hsh}\quad\tfa v\in H^s_k(\Gamma).
\eeq
We first prove that \eqref{e:STPN1} holds, by establishing that
\begin{align}
\big\|(I+P_{\trig_N}^C\pert)^{-1}P_{\trig_N}^{C}\pert(P_{\trig_N}^C-I)\big\|_{\Hshts} \leq
C \bigg( 1 + \rho \Big( \frac k N\Big)^{s+1}\bigg)\frac k N + C\rho  \Big( \frac k N\Big)^{s},
\label{e:moreCareful}
\end{align}
We then prove the result \eqref{e:MATqo2} under the assumption that \eqref{e:STPN2} holds, and come back and establish \eqref{e:STPN2} at the end; we proceed in  this order because we use \eqref{e:moreCareful} in the proof of \eqref{e:MATqo2}.

Let $\chi\in C_{c}^\infty(-1-\e,1+\e)$ with $\supp(1-\chi)\cap [-1,1]=\emptyset$. 
Then,
\begin{align*}
&(I+P_{\trig_N}^C\pert)^{-1} P_{\trig_N}^C \pert\big(1-\chi(-k^{-2}\Delta_\Gamma)\big)(P_{\trig_N}^C-I)\\
&\hspace{1cm}=(I+P_{\trig_N}^C\pert)^{-1} \Big(\chi(-k^{-2}\Delta_\Gamma)+\big(1-\chi(-k^{-2}\Delta_\Gamma)\big)\Big)P_{\trig_N}^C \pert\big(1-\chi(-k^{-2}\Delta_\Gamma)\big)(P_{\trig_N}^C-I).
\end{align*}
Now, by \eqref{e:recordHighLow1}, \eqref{e:approx_FC} with $t=q=s>1/2$, Part (ii) of Assumption \ref{ass:abstract1}, Corollary \ref{cor:WFcutoff}, and \eqref{e:approx_FC} with $t=s$ and $q=s-1$,
\begin{align}\nonumber
&\big\|(I+P_{\trig_N}^C\pert)^{-1}\big(1-\chi(-k^{-2}\Delta_\Gamma)\big)P_{\trig_N}^C \pert\big(1-\chi(-k^{-2}\Delta_\Gamma)\big)(P_{\trig_N}^C-I)\big\|_{\Hshts}\\ 
&\hspace{0.5cm}\leq C\Big(1+\rho \Big(\frac{k}{N}\Big)^{s+1}\Big)\big\| \pert\big(1-\chi(-k^{-2}\Delta_\Gamma)\big)(P_{\trig_N}^C-I)\big\|_{\Hshts}
 \label{e:splitMe1}
\leq  C\Big(1+\rho  \Big(\frac{k}{N}\Big)^{s+1}\Big)\frac{k}{N}.
\end{align}
Next, by \eqref{e:recordHighLow2}, \eqref{e:approx_FC} with $t=q=1$, Part (ii) of Assumption \ref{ass:abstract1} and Corollary \ref{cor:WFcutoff}, and \eqref{e:approx_FC} with $t=s, q=0$,
\begin{align}\nonumber
&\big\|(I+P_{\trig_N}^C\pert)^{-1}\chi(-k^{-2}\Delta_\Gamma)P_{\trig_N}^C \pert\big(1-\chi(-k^{-2}\Delta_\Gamma)\big)(P_{\trig_N}^C-I)\big\|_{\Hshts}\\ \nonumber
&\leq \big\|(I+P_{\trig_N}^C\pert)^{-1}\chi(-k^{-2}\Delta_\Gamma)\|_{H_{k}^{1} \to \Hshs} \big\|P_{\trig_N}^C\big\|_{H_{k}^{1}\to H_{k}^{1}} \big\| \pert\big(1-\chi(-k^{-2}\Delta_\Gamma)\big)\big\|_{L^2\to H_{k}^{1}} \big\| (P_{\trig_N}^C-I)\big\|_{\Hshs \to L^2}\\ 
&\hspace{0cm}\leq C\rho \Big(\frac{k}{N}\Big)^{s}.\label{e:splitMe2}
\end{align}
Finally, by \eqref{e:recordHighLow2} again, Lemma \ref{lem:2}, the fact that $\chi(-k^{-2}\Delta_\Gamma): L^2 \to H^s_k$ with bounded norm, and \eqref{e:approx_FC} with $t=s$ and $q=0$,
\begin{align}\nonumber
&\big\|(I+P_{\trig_N}^C\pert)^{-1}P_{\trig_N}^C\pert  \chi(-k^{-2}\Delta_\Gamma)(P_{\trig_N}^C-I)\big\|_{\Hshts}\\ 
&\hspace{0cm}=\big\|(I-(I+P_{\trig_N}^C\pert)^{-1})\chi(-k^{-2}\Delta_\Gamma)(P_{\trig_N}^C-I)\big\|_{\Hshts}\leq C\rho \big\|(P_{\trig_N}^C-I)\big\|_{H_{k}^{s}\to L^2}
\leq C\rho\Big(\frac{k}{N}\Big)^{s}.\label{e:splitMe3}
\end{align}
Combining~\eqref{e:splitMe1},~\eqref{e:splitMe2}, and~\eqref{e:splitMe3}, we obtain \eqref{e:moreCareful},
and thus \eqref{e:STPN1} is ensured if \eqref{e:MATthresa0} holds.

We now prove \eqref{e:MATqo2} under the assumption that \eqref{e:STPN2} holds. 
By \eqref{e:goodOldNeumann},
the right-hand side of \eqref{e:QOabs2} can be bounded as follows:
\begin{align}\nonumber
&\Big\|\big(I+P_{\trig_N}^C\pert_NP_{\trig_N}^C\big)^{-1}\Big[(I-P_{\trig_N}^C\pert)(I-P_{\trig_N}^C)v +P_{\trig_N}^C(\pert_N-\pert)P_{\trig_N}^Cv\Big]\Big\|_{\Hsh}\\ \nonumber
&\leq2 \Big\|\big(I+P_{\trig_N}^C\pert\big)^{-1}\Big[(I-P_{\trig_N}^C\pert)(I-P_{\trig_N}^C)v +P_{\trig_N}^C(\pert_N-\pert)P_{\trig_N}^Cv\Big]\Big\|_{\Hsh}\\ \nonumber
&=2 \Big\| \Big[I- 2(I+P_{\trig_N}^C\pert)^{-1}P_{\trig_N}^C\pert\Big](I-P_{\trig_N}^C)v +(I+P_{\trig_N}^C\pert)^{-1}P_{\trig_N}^C(\pert_N-\pert)P_{\trig_N}^Cv\Big\|_{\Hsh}\\ \nonumber
&\leq 2 \big\| (I-P_{\trig_N}^C)v\big\|_{\Hsh}+4\big\|(I+P_{\trig_N}^C\pert)^{-1}P_{\trig_N}^{C}\pert(I-P_{\trig_N}^C)\big\|_{\Hsh\to \Hsh}\big\|(I-P_{\trig_N}^C)v\big\|_{\Hsh}\\
&\qquad +2 \big\|(I+P_{\trig_N}^C\pert)^{-1}P_{\trig_N}^C(\pert_N-\pert)P_{\trig_N}^Cv\big\|_{\Hsh}.
\label{e:Friday2a}
\end{align}
The second term on the right-hand side of \eqref{e:Friday2a} can be estimated via \eqref{e:moreCareful}. 
To bound the third term on the right-hand side of \eqref{e:Friday2a}, we observe that, by Lemma \ref{l:splitCollocate}, 
\begin{align}\nonumber
&\big\|(I+P_{\trig_N}^C\pert)^{-1}P_{\trig_N}^C(\pert_N-\pert)P_{\trig_N}^Cv\big\|_{\Hsh}\\ \nonumber
&\hspace{1cm}\leq
\big\|(I+P_{\trig_N}^C\pert)^{-1}\big(I-\chi(-\hsc^2\Delta_g)\big)\big\|_{\Hsh\to \Hsh} \big\|P_{\trig_N}^C(\pert_N-\pert)P_{\trig_N}^Cv\big\|_{\Hsh}\\ \nonumber
&\hspace{2cm}+\|(I+P_{\trig_N}^C\pert)^{-1}\chi(-\hsc^2\Delta_g)\|_{L^2(\Gamma)\to \Hsh}\big\|P_{\trig_N}^C(\pert_N-\pert)P_{\trig_N}^Cv\big\|_{L^2(\Gamma)}\\ 
&\hspace{1cm}
\leq
C\Big(1+\rho \Big(\frac{k}{N}\Big)^{s+1}\Big)\big\|P_{\trig_N}^C(\pert_N-\pert)P_{\trig_N}^Cv\big\|_{\Hsh}
+ C \rho \big\|P_{\trig_N}^C(\pert_N-\pert)P_{\trig_N}^Cv\big\|_{L^2(\Gamma)}.
\label{e:Friday3}
\end{align}
Therefore, inserting the bounds \eqref{e:moreCareful} and \eqref{e:Friday3} into \eqref{e:Friday2a}, we obtain that
\begin{align*}
&\Big\|\big(I+P_{\trig_N}^C\pert_NP_{\trig_N}^C\big)^{-1}\Big[(I-P_{\trig_N}^C\pert)(I-P_{\trig_N}^C)v +P_{\trig_N}^C(\pert_N-\pert)P_{\trig_N}^Cv\Big]\Big\|_{\Hsh}\\
& \leq C \Big( 1 + \rho \Big( \frac k N \Big)^s \Big) \big\| (I- P_{\trig_N}^C)v\big\|_{\Hsh} + 
C\Big(1+\rho \Big(\frac{k}{N}\Big)^{s+1}\Big)\big\|P_{\trig_N}^C(\pert_N-\pert)P_{\trig_N}^Cv\big\|_{\Hsh}\\
&\quad+ C \rho \big\|P_{\trig_N}^C(\pert_N-\pert)P_{\trig_N}^Cv\big\|_{L^2(\Gamma)},
\end{align*}
which combined with \eqref{e:QOabs2} implies the bound \eqref{e:MATqo2}.

To complete the proof of \eqref{e:MATqo2}, we now need to check that~\eqref{e:STPN2} holds.
By \eqref{e:recordHighLow2}, \eqref{e:recordHighLow1},  
and the fact that $P_{\trig_N}^C:\Hsh\to \Hsh$ is uniformly bounded (for $s>1/2$) by \eqref{e:approx_FC}, 
for $s\geq t>1/2$, 
\begin{align*}
&\N{(I+P_{\trig_N}^C\pert)^{-1}P_{\trig_N}^C(\pert_N-\pert)P_{\trig_N}^C}_{\Hshts} \\
&\hspace{.25cm}\leq \N{(I+P_{\trig_N}^C\pert)^{-1}\chi(-k^{-2}\Delta_\Gamma)P_{\trig_N}^C(\pert_N-\pert)P_{\trig_N}^C}_{\Hshts}\\
&\hspace{5cm} +\N{(I+P_{\trig_N}^C\pert)^{-1}(1-\chi(-k^{-2}\Delta_\Gamma))P_{\trig_N}^C(\pert_N-\pert)P_{\trig_N}^C}_{\Hshts} \\ 
&\hspace{.25cm}\leq C\rho \N{(\pert_N-\pert)P_{\trig_N}^C}_{H_k^{s} \to H_k^{t}} +C\Big(1+\rho \Big(\frac{k}{N}\Big)^{s+1}\Big)\N{(\pert_N-\pert)P_{\trig_N}^C}_{\Hshts}.
\end{align*}
Now, by the bound \eqref{e:mainevent2} from Lemma \ref{lem:mainevent}, first applied with $r=t$ and then with $r=s$, 
\begin{align*}\nonumber
&\N{(I+P_{\trig_N}^C\pert)^{-1}P_{\trig_N}^C(\pert_N-\pert)P_{\trig_N}^C}_{\Hshts} \\ \nonumber
&\hspace{.25cm}\leq
C\rho  \bigg(\Big(\frac{k}{N}\Big)^{s-t+1}\mathcal{B}^{t,s}(\widetilde{L})+\Big[\Big(\frac{k}{N}\Big)^{s-t+1} \FLm^{s,\e}(N)+ \FHm^{t,\e}(N)\Big]
\\
&\hspace{5cm}\times
\Big[1+\|\widetilde{L}\|_{s,s}+\|\widetilde{L}\|_{t,t}+
\left(\frac{k}{N}\right)\FLm^{s,\e}(N) + \FHm^{s,\e}(N)\Big]\bigg) \\
&\hspace{.5cm}+C\Big(1+\rho \Big(\frac{k}{N}\Big)^{s+1}\Big)\bigg(\Big(\frac{k}{N}\Big)\mathcal{B}^{s,s}(\widetilde{L})
+ \Big[\Big(\frac{k}{N}\Big) \FLm^{s,\e}(N)+\FHm^{s,\e}(N)\Big]\\
&\hspace{5cm}\times\Big[1+\|\widetilde{L}\|_{s,s}+
\left(\frac{k}{N}\right)\FLm^{s,\e}(N) + \FHm^{s,\e}(N)\Big]\bigg)\\
&\hspace{.25cm}\leq C\Big( \frac k N + \rho\Big(\frac{k}{N}\Big)^{s-t+1}\Big)\FLm^{s,\e}(N)
\Big[1+\|\widetilde{L}\|_{s,s}+\|\widetilde{L}\|_{t,t}\Big]
\\
&\hspace{.5cm}+C\bigg(1+\rho\Big(1+\Big(\frac{k}{N}\Big)^{s+1}\Big)\bigg)\FHm^{s,\e}(N)\Big[1+\|\widetilde{L}\|_{s,s}+\|\widetilde{L}\|_{t,t}\Big]\\
&\hspace{.5cm}
+C\bigg[1+\rho\Big(\frac{k}{N}\Big)^{s+1}\bigg]\Big(\frac{k}{N}\Big)\mathcal{B}^{s,s}(\widetilde{L})+
\rho \Big(\frac{k}{N}\Big)^{s-t+1}
\mathcal{B}^{t,s}(\widetilde{L})
\end{align*}
where in the last line we have used~\eqref{e:MATthresb0b} and \eqref{e:MATthresb0c} to see that 
$$
\Big|\left(\frac{k}{N}\right)\FLm^{s,\e}(N) +\FHm^{s,\e}(N)\Big|\leq C.
$$
Therefore, the conditions \eqref{e:MATthresb0a}, \eqref{e:MATthresb0b}, and \eqref{e:MATthresb0c} ensure that \eqref{e:STPN2} holds, and the proof of \eqref{e:MATqo2} is complete.

Finally, to obtain \eqref{e:MATqo2new}, we bound
$\big\|P_{\trig_N}^C(\pert_N-\pert)P_{\trig_N}^Cv\big\|_{L^2(\Gamma)}$ in \eqref{e:MATqo2} by \eqref{e:mainevent}.
Indeed, by \eqref{e:mainevent},
\begin{align*}
&\big\|P_{\trig_N}^C(\pert_N-\pert)P_{\trig_N}^Cv\big\|_{L^2(\Gamma)}\\
&\leq C\Bigg[\Big( \Big(\frac{k}{N}\Big)^{s+1} \FLm^{s,\e}(N)+\FHm^{0,\e}(N)\Big)(1+\|\widetilde{L}\|_{0,0}+ \|\widetilde{L}\|_{s,s})\\
&\hspace{1cm}+\Big(\Big(\frac{k}{N}\Big) \FLm^{0,\e}(N)+\FHm^{0,\e}(N)\Big)\Big( \Big(\frac{k}{N}\Big)^{s+1} \FLm^{s,\e}(N)+\FHm^{0,\e}(N)\Big)+ \Big(\frac{k}{N}\Big)^{s+1}\mathcal{B}^{0,s}(\widetilde{L})\Bigg]\|v\|_{\Hsh},
\\
&\leq C\bigg[ \Big(\Big(\frac{k}{N}\Big)^{s+1} \FLm^{s,\e}(N)+\FHm^{0,\e}(N)\Big]
\big(1+\|\widetilde{L}\|_{0,0}+ \|\widetilde{L}\|_{s,s}\big)
+ \Big(\frac{k}{N}\Big)^{s+1}\mathcal{B}^{0,s}(\widetilde{L})\bigg]\|v\|_{\Hsh}.
\end{align*}
\epf

Before proving Lemma \ref{lem:mainevent}, we record the following two lemmas. In these lemmas and the proof of  Lemma \ref{lem:mainevent}, we use the notation that
\beq\label{e:nottrig2}
(\varphi_{m}\circ \gamma)(t) := \exp(\ri m t)
\eeq
(compare to \eqref{e:nottrig1}).

\ble[Multiplication by trigonometric polynomials] Given $s\in \Rea$ there exists $C>0$ such that, for all $v \in \Hsh$ and all $m\in \Rea$,
\beq\label{e:SVineq}
\N{\varphi_m v}_{\Hsh} \leq C \langle m \hsc\rangle^s \N{v}_{\Hsh}.
\eeq
\ele

\bpf
This result is well-known (see, e.g., \cite[Prop.~3]{SaVa:96}), but since the proof is very short, we give it here. By \eqref{e:norm_circle} and \eqref{e:normequiv},
\begin{align*}
\N{\varphi_m v}_{\Hsh}& \leq C_2 \sum_{n=-\infty}^\infty \big| \widehat{(v\circ \gamma)}(n-m) \big|^2 \langle n\hsc\rangle^{2s} \\
&=C_2 \sum_{n=-\infty}^\infty \big| \widehat{(v\circ \gamma)}(n) \big|^2 \langle (n+m)\hsc\rangle^{2s}
\leq C' \sum_{n=-\infty}^\infty \big| \widehat{(v\circ \gamma)}(n) \big|^2 \langle n\hsc\rangle^{2s} \langle m\hsc\rangle^{2s},
\end{align*}
by Peetre's inequality. The result \eqref{e:SVineq} then follows from using \eqref{e:norm_circle} and \eqref{e:normequiv} again.
\epf

\begin{lemma}
\label{l:alias}
Let $\oldK \leq N$. 
If $\oldK+|m|\leq N $ then
\begin{equation}
\label{e:alias2}
(I-P_{\trig_N}^C)\varphi_m P_{\trig_{\oldK}}^C =0.
\end{equation}
Furthermore, if $\oldK+|m|\leq 2N$ then
\begin{equation}
\label{e:alias}
P^G_{\trig_{2N-\oldK-|m| }}(I-P_{\trig_N}^C)\varphi_m P_{\trig_{\oldK}}^C =0.
\end{equation}
\end{lemma}
\begin{proof}
For $(I-P_{\trig_N}^C)\varphi_m P_{\trig_{\oldK}}^C$ not to be zero, there must exist $\xi\in \mathbb{Z}$ such that
\beq\label{e:lastDay1}
N+1\leq |\xi| \leq n+|m|,
\eeq
so that if $n+|m|\leq N$ then \eqref{e:alias2} holds. 
Then for 
$P^G_{\trig_{2N-\oldK-|m| }}(I-P_{\trig_N}^C)\varphi_m P_{\trig_{\oldK}}^C$ not to be zero, we need, in addition to \eqref{e:lastDay1},
that, by \eqref{e:alias0}, there exists $\ell\in0,1,2,\ldots$ such that 
\beq\label{e:lastDay2}
\big| |\xi| - \ell (2N+1)\big| \leq 2N-\oldK -|m|.
\eeq
If $n+|m|\leq N$, then \eqref{e:lastDay1} 
cannot hold, so we can restrict attention to the case  $n+|m|>N$.
In this case, \eqref{e:lastDay2} 
implies that $| |\xi| - \ell (2N+1)| \leq N$, which
cannot hold with $\ell=0$ by the lower bound in \eqref{e:lastDay1}.
Finally, \eqref{e:lastDay2} implies that
\beqs
|\xi| -\ell(2N+1) \geq |m| + n -2N \quad\text{ so that }\quad |\xi| \geq |m| + n +\ell + 2N(\ell-1),
\eeqs
which contradicts \eqref{e:lastDay1} if $\ell\geq 1$; i.e., we have proved \eqref{e:alias}.
%
%
\end{proof}

\

We adopt the notation of \cite[\S12.4]{Kr:14}, \cite{KrSl:93} and define
\beq\label{e:A0}
(A_0 w)(t):= \int_0^{2\pi} \log \left(4\sin^2\left(\frac{t-\tau}{2}\right)\right) w(\tau) \rd \tau.
\eeq
Up to a factor of $\pi^{-1}$, $A_0$ is the Laplace single-layer operator on the unit circle with arc-length parametrisation
(since there $|x(t)-x(\tau)|^2 = 4 \sin^2 ((t-\tau)/2)$). Therefore, $A_0: H^s(0,2\pi)\to H^{s+1}(0,2\pi)$ for all $s\in \Rea$ (by, e.g., \cite[Theorem 4.4.1]{Ne:01}).

\begin{lemma}
If $L$ satisfies Assumption \ref{ass:Nystrom} 
then for all $c,\e>0$ there exists $C>0$ such that for $N\geq ck$,  $0\leq \oldK \leq N$,  $r\geq s>1/2$, and $\LN^N$ as in~\eqref{e:LN},
\begin{equation}
\label{e:singleOp}
\|(L-\LN^N)P_{\trig_{\oldK}}^C\|_{H_{k}^{r}(\Gamma)\to \Hsh}\leq  C \bigg[\Big(\frac{k}{N}\Big)^{r-s+1} \FL^{r,\e}(N,\oldK ,L)+\Big(\frac{k}{N}\Big)^{\min\{s,1\}}\FH^{s,\e}(N,\oldK ,L)\bigg].
\end{equation}
\end{lemma}

\bpf
By the definitions of $L$  in~\eqref{e:NystromL} and $\LN^N$ in \eqref{e:LN}, 
\beqs
(L-\LN^N)P_{\trig_{\oldK}}^C  = I_1 + I_2,
\eeqs
where
\beqs
(I_1v)(t) := \int^{2\pi}_0 \log\left( 4\sin^2 \left(\frac{t-\tau}{2}\right)\right) \big(I- P_{\trig_N,\tau}^C\big) \big(L_1(t,\tau) \, P_{\trig_{\oldK}}^C v(\tau)\big) \, \rd \tau
\eeqs
and
\beqs
(I_2 v)(t) := \int^{2\pi}_0 \big(I- P_{\trig_N,\tau}^C\big) \big(L_2(t,\tau) \, P_{\trig_{\oldK}}^Cv(\tau)\big) \, \rd \tau.
\eeqs
We prove \eqref{e:singleOp} by proving that 
\begin{align}\nonumber
&\N{I_1}_{H_{k}^{r}(\Gamma)\to \Hsh}\leq C \Big(\frac{k}{N}\Big)^{r-s+1} k^{-1}\sum_{N-\oldK \leq |m|\leq (2-\e)N-\oldK } \big\| \widehat{L}_{1,m}(t)\big\|_{\Hsh}\langle m/k\rangle^{r} \\
&\hspace{3cm}+C\Big( \frac k N \Big)^{\min\{s,1\}}\sum_{|m| >(2-\e)N-\oldK } \big\| \widehat{L}_{1,m}(t)\big\|_{\Hsh}\langle m/k\rangle^{s}
\label{e:VV1}
\end{align}
and
\beq\label{e:VV2}
\N{I_2}_{H_{k}^{r}(\Gamma)\to \Hsh}\leq C\Big(\frac k N\Big)^s\sum_{|m|>2N -\oldK }\big\|\widehat{L}_{2,m}(t)\|_{\Hsh}\langle m/k\rangle^{s}.
\eeq
By the definition of 
$\widehat{L}_{j,m}(t) $ from \eqref{e:widehatL}
and  \eqref{e:Fourier},
\beqs
L_j(t,\tau) = \frac{1}{\sqrt{2\pi}} \sum_{m=-\infty}^\infty \widehat{L}_{j,m}(t)\re^{\ri m \tau}.
\eeqs
Then
\beq\label{e:I1}
(I_1v)(t)= \frac{1}{\sqrt{2\pi}}\sum_{m=-\infty}^\infty \widehat{L}_{1,m}(t) \, A_0\Big( (I- P_{\trig_{N}}^C)(\varphi_m P_{\trig_{\oldK}}^C v)\Big)(t),
\eeq
where $A_0$ is defined by \eqref{e:A0} and
\beq\label{e:I2}
\begin{aligned}
(I_2v)(t)&= \frac{1}{\sqrt{2\pi}}\sum_{m=-\infty}^\infty \widehat{L}_{2,m}(t) \int_0^{2\pi} (I- P_{\trig_{N},\tau}^C)\big(\varphi_m(\tau) P_{\trig_{\oldK}}^Cv(\tau)\big)\rd \tau\\
&=\sum_{m=-\infty}^\infty \widehat{L}_{2,m}(t) \ P_{\trig_0}^G(I- P_{\trig_{N}}^C)\big(\varphi_m P_{\trig_{\oldK}}^Cv\big).
\end{aligned}
\eeq
We first prove the bound on $I_2$ in \eqref{e:VV1} (since this is slightly easier than proving the bound on $I_1$).
By, e.g., \cite[Prop.~2]{SaVa:96}, \cite[\S5.13]{SaVa:02},  for all $v\in \Hsh$ and $s>1/2$,
\beq\label{e:SVmult}
\big\|\widehat{L}_{2,m}v\big\|_{\Hsht} \leq C \big\|\widehat{L}_{2,m}\big\|_{\Hsh} \N{v}_{\Hsh}.
\eeq
By \eqref{e:SVmult}, \eqref{e:alias}, \eqref{e:SVineq}, \eqref{e:PNG}, 
\eqref{e:approx_FC}, 
\begin{align*}
\N{I_2 v}_{\Hsh} &\leq C \sum_{|m|\geq 2N-\oldK }\big\|\widehat{L}_{2,m}\big\|_{\Hsh} \N{ P_{\trig_0}^G(I- P_{\trig_{N}}^C)(\varphi_m P_{\trig_{\oldK}}^Cv)}_{H_{k}^{s}(\Gamma)}\\
&\leq C\|v\|_{\Hsh} \sum_{|m|\geq 2N-\oldK }\big\|\widehat{L}_{2,m}\|_{\Hsh}\langle mk^{-1}\rangle^{s} (k/N)^{s},\\
&\leq C\|v\|_{\Hrh} \sum_{|m|\geq 2N-\oldK }\big\|\widehat{L}_{2,m}\|_{\Hsh}\langle mk^{-1}\rangle^{s} (k/N)^{s},
\end{align*}
which is \eqref{e:VV2}.
We now bound $I_1$. Using in the expression \eqref{e:I1} the properties \eqref{e:SVmult}, \eqref{e:alias2}, \eqref{e:alias}, and \eqref{e:MATproj2}, we find that
\begin{align*}
\N{I_1 v}_{\Hsh} &\leq  C \sum_{m=-\infty}^\infty \big\| \widehat{L}_{1,m}(t)\big\|_{\Hsh} \|A_0 (I-P_{\trig_N}^C)\varphi_mP_{\trig_{\oldK}}^Cv\|_{\Hsh} \\
&=  C \sum_{|m|> N-n} \big\| \widehat{L}_{1,m}(t)\big\|_{\Hsh} \|A_0 (I-P_{\trig_N}^C)\varphi_mP_{\trig_{\oldK}}^Cv\|_{\Hsh} \\
&\leq C \sum_{N-\oldK <|m|\leq (2-\e)N-\oldK } \big\| \widehat{L}_{1,m}(t)\big\|_{\Hsh} \|A_0 (I-P_{\trig_{2N-|m|-\oldK} }^{G})(I-P_{\trig_N}^C)\varphi_mP_{\trig_{\oldK}}^Cv\|_{\Hsh}\\
&\qquad+C\sum_{|m| >(2-\e)N-\oldK } \big\| \widehat{L}_{1,m}(t)\big\|_{\Hsh} \|A_0 (I-P_{\trig_N}^C)\varphi_mP_{\trig_{\oldK}}^Cv\|_{\Hsh}.
\end{align*}
We now claim that $\|A_0\|_{H_{k}^{s-1}(\Gamma) \to \Hsh}\leq C$. Indeed, for $s\geq 1$, since $A_0: H^{s-1}(\Gamma) \to H^s(\Gamma)$,
\beqs
\big\| A_0 u \big\|_{L^2(\Gamma)} \leq C \| u\|_{L^2(\Gamma)} \quad\tand\quad \big\| A_0 u\big\|_{H^s(\Gamma)} \leq C \| u \|_{H^{s-1}(\Gamma)}
\eeqs
so that 
\beqs
\big\| A_0u\big\|_{\Hsh}\leq C \Big( \big\| A_0 u \big\|_{\LtG} + \hsc^s \big\| A_0 u\big\|_{\Hsh}\Big) \leq C \big( \| u \|_{\LtG} + \hsc^s \| u \|_{H^{s-1}(\Gamma)} \big) \leq C \| u\|_{H_{k}^{s-1}(\Gamma)}.
\eeqs
The result for $s<1$ follows by duality and interpolation. Therefore,
\begin{align}\nonumber
\N{I_1 v}_{\Hsh}
&\leq C \sum_{N-\oldK <|m|\leq (2-\e)N-\oldK } \big\| \widehat{L}_{1,m}(t)\big\|_{\Hsh}\|A_0 (I-P_{\trig_{2N-|m|-\oldK} }^{G})(I-P_{\trig_N}^C)\varphi_mP_{\trig_{\oldK}}^Cv\|_{\Hsh}\\\nonumber
&\qquad+\sum_{|m| >(2-\e)N-\oldK } \big\| \widehat{L}_{1,m}(t)\big\|_{\Hsh} \| (I-P_{\trig_N}^C)\varphi_mP_{\trig_{\oldK}}^Cv\|_{H_k^{s-1}(\Gamma)}\\
\nonumber
&\leq C \sum_{N-\oldK <|m|\leq (2-\e)N-\oldK } \big\| \widehat{L}_{1,m}(t)\big\|_{\Hsh}\|A_0 (I-P_{\trig_{2N-|m|-\oldK} }^{G})(I-P_{\trig_N}^C)\varphi_mP_{\trig_{\oldK}}^Cv\|_{\Hsh}\\
&\qquad+\sum_{|m| >(2-\e)N-\oldK } \big\| \widehat{L}_{1,m}(t)\big\|_{\Hsh} \Big( \frac k N \Big)^{s-\max\{s-1,0\}}\langle m/k\rangle^s\| v\|_{H_{k}^{s}(\Gamma)} .
\label{e:quinoa1}
\end{align}
where we have used \eqref{e:approx_FC} and \eqref{e:SVineq} in the last step. 
Now, since $A_0$ is a (non-semiclassical) pseuoddifferential operator of order $-1$, for any $M$
$$
\|P_{\trig_{N-\frac{|m|-n}{2}}}^GA_0(I-P_{\trig_{2N-|m|-n}}^G)\|_{H^{-M}\to H^M}\leq C_M (2N-|m|-n)^{-M}.
$$
Next, observe that for any $s\in \mathbb{R}$, since $2N-|m|-n\geq \e  N\geq c\e k$, then
\begin{gather*}
\|(I-P_{\trig_{N-\frac{|m|-n}{2}}}^G)u\|_{H_{k}^{s}(\Gamma)}\leq Ck^{-s}\|(I-P_{\trig_{N-\frac{|m|-n}{2}}}^G)u\|_{H^s(\Gamma)},\\
 \|(I-P_{\trig_{2N-|m|-n}}^G)u\|_{H^s(\Gamma)}\leq Ck^{s}\|(I-P_{\trig_{2N-|m|-n}}^G)u\|_{H_{k}^{s}(\Gamma)},\\
 \|(I-P_{\trig_{2N-|m|-n}}^G)u\|_{H^{s-1}(\Gamma)}\leq C (2N-|m|-n+1)^{-1}\|(I-P_{\trig_{2N-|m|-n}}^G)u\|_{H^{s}(\Gamma)}.
\end{gather*}
By \eqref{e:approxFG} and the fact that $2N-|m|-n\geq \e  N\geq c\e k$,
\beqs
\| I -P_{\trig_{N-\frac{|m|-n}{2}}}^G\|_{\Hsh\to \Hsh}\leq C.
\eeqs
Combining these last five displayed bounds, we obtain that 
\begin{align*}
&\| A_0(I-P_{\trig_{2N-|m|-n}}^G)w \|_{\Hsh}\\
 &\quad\leq \big\|  \big( I -P_{\trig_{N-\frac{|m|-n}{2}}}^G\big) A_0(I-P_{\trig_{2N-|m|-n}}^G)w \big\|_{\Hsh} + C_M (2N-|m|-n)^{-M}\N{w}_{\Hsh},\\
&\quad\leq  C \hsc^s \big\|  \big( I -P_{\trig_{N-\frac{|m|-n}{2}}}^G\big) A_0(I-P_{\trig_{2N-|m|-n}}^G)w \big\|_{H^s(\Gamma)}  + C_M (2N-|m|-n)^{-M}\N{w}_{\Hsh},\\
&\quad\leq  C \hsc^s \big\|(I-P_{\trig_{2N-|m|-n}}^G)w \big\|_{H^{s-1}(\Gamma)}  + C_M (2N-|m|-n)^{-M}\N{w}_{\Hsh},\\
&\quad\leq  C (2N-|m|-n+1)^{-1} \hsc^s \big\|(I-P_{\trig_{2N-|m|-n}}^G)w \big\|_{H^{s}(\Gamma)}  + C_M (2N-|m|-n)^{-M}\N{w}_{\Hsh},\\
&\quad\leq  C (2N-|m|-n+1)^{-1}  \big\|(I-P_{\trig_{2N-|m|-n}}^G)w \big\|_{H^{s}_k(\Gamma)}  + C_M (2N-|m|-n)^{-M}\N{w}_{\Hsh}.
\end{align*}
Using this along with \eqref{e:approx_FC} and \eqref{e:SVineq} in the estimate \eqref{e:quinoa1}, we obtain 
\begin{align*}
\N{I_1 v}_{\Hsh} &\leq C \sum_{N-\oldK <|m|\leq (2-\e)N-\oldK } \big\| \widehat{L}_{1,m}(t)\big\|_{\Hsh} (2N-\oldK +1-|m|)^{-1}\langle m/k\rangle^{r}\Big(\frac{k}{N}\Big)^{r-s}\|v\|_{H_{k}^{r}(\Gamma)}\\
&\qquad +\sum_{|m| >(2-\e)N-\oldK } \big\| \widehat{L}_{1,m}(t)\big\|_{\Hsh}\Big( \frac k N \Big)^{s-\max\{s-1,0\}}\langle m/k\rangle^{s} \| v\|_{H_{k}^{s}(\Gamma)} \\
&\leq C\|v\|_{H_{k}^{r}(\Gamma)}\Big(\frac{k}{N}\Big)^{r-s+1} k^{-1}\sum_{N-\oldK <|m|\leq (2-\e)N-\oldK } \big\| \widehat{L}_{1,m}(t)\big\|_{\Hsh}\langle m/k\rangle^{r} \\
&\qquad+C\| v\|_{H_{k}^{r}(\Gamma)}\Big( \frac k N \Big)^{\min\{s,1\}}\sum_{|m| >(2-\e)N-\oldK } \big\| \widehat{L}_{1,m}(t)\big\|_{\Hsh}\langle m/k\rangle^{s},
\end{align*}
where we have used \eqref{e:approx_FC} and the fact that $N\geq ck $ in the last step.
\epf

The only thing the remains to do in the proof of Lemma~\ref{lem:mainevent} is to estimate the difference
$$
(\widetilde{L}_{j,a}\widetilde{L}_{j,b}-\widetilde{\LN}_{j,a}^NP_{\trig_N}^C\widetilde{\LN}_{j,b}^N)P_{\trig_N}^C.
$$

\begin{lemma}
Let $T_a$ and $T_b$ Assumptions~\eqref{e:NystromL}. Then, for all $r\geq s\geq 0$ and $0<\e<1$, there is $C>0$ such that for $0\leq \oldK \leq \oldK '\leq N$, and $\TN_a^N,\,\TN_b^N$ as in~\eqref{e:LN},
\begin{equation*}
\begin{aligned}
&\|(T_aT_b-\TN_a^NP_{\trig_N}^C\TN_b^N)P_{\trig_{\oldK}}^C\|_{H_{k}^{r}(\Gamma)\to \Hsh}\\
&\hspace{.5cm}\leq C\Bigg(\Big( \Big(\frac{k}{N}\Big)^{r-s+1} \FL^{r,\e}(N,\oldK ,T_b)+\FH^{s,\e}(N,\oldK ,T_b)\Big)\|T_a\|_{H_{k}^{s}\to H_{k}^{s}}\\\
&\hspace{0.75cm}+ \Big( \Big(\frac{k}{N}\Big)^{r-s+1} \FL^{r,\e}(N,\oldK ',T_a)+\FH^{s,\e}(N,\oldK ',T_a)\Big)\|T_b\|_{H_{k}^{r}\to H_{k}^{r}}\\
&\hspace{0.75cm}+\Big(\Big(\frac{k}{N}\Big) \FL^{s,\e}(N,T_a)+\FH^{s,\e}(N,T_a)\Big)\Big( \Big(\frac{k}{N}\Big)^{r-s+1} \FL^{r,\e}(N,\oldK ,T_b)+\FH^{s,\e}(N,\oldK ,T_b)\Big)
&\\
&\hspace{0.75cm}+ \Big(\frac{k}{N}\Big)^{r-s+1}\|T_a\|_{H_{k}^{s}\to H_{k}^{s}}\|T_b\|_{H_{k}^{r}\to H_{k}^{r}}\\
&\hspace{0.75cm}+\Big( \Big(\frac{k}{N}\Big)^{r-s+1} \FL^{r,\e}(N,T_a)+\FH^{s,\e}(N,T_a)\Big)\cdot\|P_{\trig_N}^C(I-P_{\trig_{\oldK '}}^C)\|_{H_{k}^{r}\to H_{k}^{r}}\|(I-P_{\trig_{\oldK '}}^G)T_bP_{\trig_{\oldK}}^G\|_{H_{k}^{r}\to H_{k}^{r}}\Bigg).
\end{aligned}
\end{equation*}
In particular, if $\oldK =\oldK '=N$,
\begin{equation*}
\begin{aligned}
&\|(T_aT_b-\TN_a^NP_{\trig_N}^C\TN_b^N)P_{\trig_N}^C\|_{H_{k}^{r}(\Gamma)\to \Hsh}\\
&\hspace{.5cm}\leq C\Bigg(\Big( \Big(\frac{k}{N}\Big)^{r-s+1} \FL^{r,\e}(N,T_b)+\FH^{s,\e}(N,T_b)\Big)\|T_a\|_{H_{k}^{s}\to H_{k}^{s}}\\\
&\hspace{.75cm}+ \Big( \Big(\frac{k}{N}\Big)^{r-s+1} \FL^{r,\e}(N,T_a)+\FH^{s,\e}(N,T_a)\Big)\|T_b\|_{H_{k}^{r}\to H_{k}^{r}}\\
&\hspace{.75cm}+\Big(\Big(\frac{k}{N}\Big) \FL^{s,\e}(N,T_a)+\FH^{s,\e}(N,T_a)\Big)\Big( \Big(\frac{k}{N}\Big)^{r-s+1} \FL^{r,\e}(N,T_b)+\FH^{s,\e}(N,T_b)\Big)
&\\
&\hspace{.75cm}+ \Big(\frac{k}{N}\Big)^{r-s+1}\|T_a\|_{H_{k}^{s}\to H_{k}^{s}}\|T_b\|_{H_{k}^{r}\to H_{k}^{r}}\Bigg).
\end{aligned}
\end{equation*}
Suppose, in addition, that for all $\chi \in C_{c}^\infty(\mathbb{R})$ with $\chi \equiv 1$  near $[-1,1]$, $(1-\chi\hDarg)T_b\in \Psi^{-1}(\Gamma)$. Then, for $\oldK '\geq \oldK +\e k$ and $\oldK =\Xi k$ with $\Xi\geq c_{\max}+\e$, then for all $M>0$ there is $C_M>0$ such that
\begin{equation*}
\begin{aligned}
&\|(T_aT_b-\TN_a^NP_{\trig_N}^C\TN_b^N)P_{\trig_{\oldK}}^C\|_{H_{k}^{r}(\Gamma)\to \Hsh}\\
&\hspace{.2cm}\leq C\Bigg(\Big( \Big(\frac{k}{N}\Big)^{r-s+1} \FL^{r,\e}(N,\oldK ,T_b)+\FH^{s,\e}(N,\oldK ,T_b)\Big)\|T_a\|_{H_{k}^{s}\to H_{k}^{s}}\\\
&\hspace{1cm}+ \Big( \Big(\frac{k}{N}\Big)^{r-s+2} \FL^{r,\e}(N,\oldK ',T_a)+\FH^{s,\e}(N,\oldK ',T_a)\Big)\|T_b\|_{H_{k}^{r}\to H_k^{r+1}}\\
&\hspace{1cm}+\Big(\Big(\frac{k}{N}\Big) \FL^{s,\e}(N,T_a)+\FH^{s,\e}(N,T_a)\Big)\Big( \Big(\frac{k}{N}\Big)^{r-s+1} \FL^{r,\e}(N,\oldK ,T_b)+\FH^{s,\e}(N,\oldK ,T_b)\Big)
&\\
&\hspace{1cm}+ \Big(\frac{k}{N}\Big)^{r-s+1}\|T_a\|_{H_{k}^{s}\to H_{k}^{s}}\|T_b\|_{H_{k}^{r}\to H_{k}^{r}}\\
&\hspace{1cm}+C_M\Big( \Big(\frac{k}{N}\Big)^{r-s+1} \FL^{r,\e}(N,T_a)+\FH^{s,\e}(N,T_a)\Big)\|T_b\|_{\LtGt}k^{-M}\Bigg).
\end{aligned}
\end{equation*}
\end{lemma}
\bpf
We write
\begin{align*}
(T_aT_b-\TN_a^NP_{\trig_N}^C\TN_b^N)P_{\trig_{\oldK}}^C&=T_aP_{\trig_N}^CT_bP_{\trig_{\oldK}}^C-\TN_a^NP_{\trig_N}^C\TN_b^NP_{\trig_{\oldK}}^C+T_a(I-P_{\trig_N}^C)T_bP_{\trig_{\oldK}}^C\\
&=T_aP_{\trig_N}^C(T_b-\TN_b^N)P_{\trig_{\oldK}}^C +(T_a-\TN_a^N)P_{\trig_N}^C\TN_b^NP_{\trig_{\oldK}}^C+T_a(I-P_{\trig_N}^C)T_bP_{\trig_{\oldK}}^C\\
&=T_a(T_b-\TN_b^N)P_{\trig_{\oldK}}^C +(T_a-\TN_a^N)P_{\trig_{\oldK '}}^CT_bP_{\trig_{\oldK}}^C+(T_a-\TN_a^N)P_{\trig_N}^C(\TN_b^N-T_b)P_{\trig_{\oldK}}^C\\
&\qquad +T_a(I-P_{\trig_N}^C)T_bP_{\trig_{\oldK}}^C +(T_a-\TN_a^N)P_{\trig_N}^C(I-P_{\trig_{\oldK '}}^C)T_bP_{\trig_{\oldK}}^C
\end{align*}
So, 
\begin{align*}
&\|(T_aT_b-\TN_a^NP_{\trig_N}^C\TN_b^N)P_{\trig_{\oldK}}^C\|_{H_{k}^{r}\to H_{k}^{s}}\\
&\hspace{.2cm}\leq C\Big(\|T_a\|_{H_{k}^{s}\to \Hsh}\|(T_b-\TN_b^N)P_{\trig_{\oldK}}^C\|_{H_{k}^{r}\to H_{k}^{s}}+\|(T_a-\TN_a^N)P_{\trig_{\oldK '}}^C\|_{H_{k}^{r}(\Gamma)\to H_{k}^{s}}\|T_b\|_{H_{k}^{r}(\Gamma)\to H_{k}^{r}} \\
&\hspace{.4cm}+\|(T_a-\TN_a^N)P_{\trig_N}^C\|_{H_{k}^{s}\to H_{k}^{s}}\|(\TN_b^N-T_b)P_{\trig_{\oldK}}^C\|_{H_{k}^{r}\to H_{k}^{s}}\\
&\hspace{.4cm}+\|T_a\|_{H_{k}^{s}\to H_{k}^{s}}\|I-P_{\trig_N}^C\|_{H_{k}^{r}\to H_{k}^{s}}\|T_b\|_{H_{k}^{r}\to H_{k}^{r}}\\
&\hspace{.4cm}+\|(T_a-\TN_a^N)P_{\trig_N}^C\|_{H_{k}^{s}\to H_{k}^{s}}\|P_{\trig_N}^C(I-P_{\trig_{\oldK '}}^C)\|_{H_{k}^{s}\to H_{k}^{s}}\|(I-P_{\trig_{\oldK '}}^G)T_bP_{\trig_{\oldK}}^G\|_{H_{k}^{r}(\Gamma)\to H_{k}^{s}}\Big)\\
&\hspace{.2cm}\leq C\Bigg(\Big( \Big(\frac{k}{N}\Big)^{r-s+1} \FL^{r,\e}(N,\oldK ,T_b)+\FH^{s,\e}(N,\oldK ,T_b)\Big)\|T_a\|_{H_{k}^{s}\to H_{k}^{s}}\\\
&\hspace{.4cm}+ \Big( \Big(\frac{k}{N}\Big)^{r-s+1} \FL^{r,\e}(N,\oldK ',T_a)+\FH^{s,\e}(N,\oldK ',T_a)\Big)\|T_b\|_{H_{k}^{r}\to H_{k}^{r}}\\
&\hspace{.4cm}+\Big(\Big(\frac{k}{N}\Big) \FL^{s,\e}(N,T_a)+\FH^{s,\e}(N,T_a)\Big)\Big( \Big(\frac{k}{N}\Big)^{r-s+1} \FL^{r,\e}(N,\oldK ,T_b)+\FH^{s,\e}(N,\oldK ,T_b)\Big)
&\\
&\hspace{.4cm}+ \Big(\frac{k}{N}\Big)^{r-s+1}\|T_a\|_{H_{k}^{s}\to H_{k}^{s}}\|T_b\|_{H_{k}^{r}\to H_{k}^{r}}\\
&\hspace{.4cm}+\Big( \Big(\frac{k}{N}\Big)^{r-s+2} \FL^{r,\e}(N,T_a)+\FH^{s,\e}(N,T_a)\Big)\cdot\|P_{\trig_N}^C(I-P_{\trig_{\oldK '}}^C)\|_{H_{k}^{r}\to H_{k}^{r}}\|(I-P_{\trig_{\oldK '}}^G)T_bP_{\trig_{\oldK}}^G\|_{H_{k}^{r}\to H_{k}^{r}}\Bigg).
\end{align*}
The implication when $\oldK =\oldK '=N$ follows immediately since $P_{\trig_N}^C(I-P_{\trig_N}^C)=0$. 

Now, suppose that $(1-\chi\hDarg)T_b\in \Psi^{-1}(\Gamma)$ for all $\chi \in C_{c}^\infty$ with $\chi \equiv 1$ near $[-1,1]$, $\oldK '\geq \oldK +\e k$, and $\oldK =\Xi k$ with $\Xi\geq (1+\e)c_{\max}$. Then let $\chi_0,\chi_1\in C_{c}^\infty(\mathbb{R})$ with $\chi_0\equiv 1$ near $[-\Xi^2,\Xi^2]$, $\supp \chi_0\subset \{\chi_1\equiv 1\}$, and $\supp \chi_1\subset  (-(\Xi+\e)^2,(\Xi+\e)^2)$ we observe 
$$
P_{\trig_{\oldK}}^G=\chi_0(k^{-2}\partial_t^2)P_{\trig_{\oldK}}^G,\qquad (I-P_{\trig_{\oldK '}}^G)=(I-P_{\trig_{\oldK '}}^G)(1-\chi_1(k^{-2}\partial_t^2))
$$
By Lemma~\ref{l:proj2Pseudo} (with $\Xi=1$) together with the fact that $\oldK '\geq (c_{\max}+\e)k$, there is $\chi \in C_{c}^\infty$ with $\chi \equiv 1$ near $[-1,1]$ such that
$$
(I-P_{\trig_{\oldK '}}^G)=(I-P_{\trig_{\oldK '}}^G)(1-\chi\hDarg)+O(k^{-\infty})_{\Psi^{-\infty}(\Gamma)}.
$$
Therefore, 
\begin{align*}
&(I-P_{\trig_{\oldK '}}^G)T_bP_{\trig_{\oldK}}^G\\
&=(I-P_{\trig_{\oldK '}}^G)(1-\chi\hDarg)T_bP_{\trig_{\oldK}}^G+O(k^{-\infty}\|T_b\|_{\LtGt})_{\Psi^{-\infty}(\Gamma)}\\
&=(I-P_{\trig_{\oldK '}}^G)(1-\chi_1(|k^{-2}\partial_t|^2)(1-\chi\hDarg)T_b\chi_0(k^{-2}\partial_t^2))P_{\trig_{\oldK}}^G+O(k^{-\infty}\|T_b\|_{\LtGt})_{\Psi^{-\infty}(\Gamma)}\\
&=O(k^{-\infty}\|T_b\|_{\LtGt})_{\Psi^{-\infty}(\Gamma)},
\end{align*}
where the last line follows from the fact that  $(1-\chi\hDarg)T_b\in \Psi^{-1}$,  that both $(1-\chi_1(|k^{-1}\partial_t|^2))\in \Psi^0$ and $\chi_0(|k^{-1}\partial_t|^2))\in \Psi^0(\Gamma)$, 
and 
$$
\WF(\chi_0(k^{-2}\partial_t^2))\cap \WF(1-\chi_1(k^{-2}\partial_t^2))=\emptyset.
$$ 
\epf

The next lemma is used to prove the results for plane-wave data.
\begin{lemma}
\label{l:NystromCompactMicrolocalizedD}
Suppose that there is $\Xi>0$ such that for all $\psi \in C_{c}^\infty(\mathbb{R})$ with $\supp(1-\psi)\cap [-\Xi,\Xi]=\emptyset$  and all $M>0$, there is $C_M>0$ such that 
\begin{equation}
\label{e:compactMicro}
\|(1-\psi\hDarg)v\|_{H_{k}^M(\Gamma)}\leq C_Mk^{-M},
\end{equation}
and for all $\chi\in C_{c}^\infty(\mathbb{R})$ with $\chi \equiv 1$ near $[-\Xi,\Xi]$, $(1-\chi\hDarg)\widetilde{\pert}_b\in \Psi^{-1}(\Gamma)$. Then for all $r\geq s\geq 0$, $M>0$, and $s\in \mathbb{R}$, there are $C_1,C_2>0$ such that for all $N\geq C_1k$ 
$$
\begin{aligned}
&\|P_{\trig_N}^C(\pert_N-\pert)P^C_{\trig_N}v\|_{\Hsh} \\
&\hspace{.2cm}\leq C\Bigg(\Big( \Big(\frac{k}{N}\Big)^{r-s+1} \FLm^{r,\e}(N,C_2k)+\FHm^{s,\e}(N,C_2k)\Big)\big(1+\|\widetilde{L}\|_{s,s}+ \|\widetilde{L}\|_{r,r}\big)\\
&\hspace{.4cm}+\Big(\Big(\frac{k}{N}\Big) \FLm^{s,\e}(N)+\FHm^{s,\e}(N)\Big)\Big( \Big(\frac{k}{N}\Big)^{r-s+1} \FLm^{r,\e}(N,C_2k)+\FHm^{s,\e}(N,C_2k)\Big)+ \Big(\frac{k}{N}\Big)^{r-s+1}\mathcal{B}^{s,r}(\widetilde{L})\\
&\hspace{.4cm}+C_M\Big( \Big(\frac{k}{N}\Big)^{r-s+1} \FLm^{r,\e}(N)+\FHm^{s,\e}(N)\Big)\|\widetilde{L}\|_{0,0}k^{-M}\Bigg)\|v\|_{H_{k}^{r}(\Gamma)}.
\end{aligned}
$$
\end{lemma}
\bpf
First observe that, if $N\geq n$,  
\begin{align}\nonumber
(\pert-\pert_N)P_{\trig_N}^C v &= (\pert-\pert_N)P_{\trig_n}^C v - (\pert-\pert_N)(P_{\trig_n}^C - P_{\trig_N}^C) v \\
&= (\pert-\pert_N)P_{\trig_n}^C v - (\pert-\pert_N)P_{\trig_n}^C(I-P_{\trig_N}^C) v.\label{e:alarm1}
\end{align}
We now claim that there exists $C_3>0$ such that 
\begin{equation}
\label{e:claimJubilee}
\|(I-P_{\trig_{C_3c_{\max}k}}^C)v\|_{H_{k}^{r}(\Gamma)}=O(k^{-\infty}).
\end{equation}
Assuming this claim, the result follows from the combination of \eqref{e:alarm1}, \eqref{e:claimJubilee}, and \eqref{e:mainevent3}, by setting $n$ and $n'$ as in \eqref{e:claimJubilee} and then setting $C_1=\max\{C_3,1\} c_{\max}+ 2\e$.

To see~\eqref{e:claimJubilee}, recall that from \eqref{e:SVclevertrick} and Lemma~\ref{l:proj2Pseudo} that there is $\chi_1 \in C_{c}^\infty(\mathbb{R})$ with $\chi_1 \equiv 1 $ on $[-C_3,C_3]$ such that
\begin{align}\nonumber
(I-P_{\trig_{C_3c_{\max}k}}^C)&=(I-P_{\trig_{C_3c_{\max}k}}^C)(I-P_{C_3c_{\max}k}^G)\\
&=(I-P_{\trig_{C_3c_{\max}k}}^C)(I-P_{C_3c_{\max}k}^G)(1-\chi_1\hDarg)+O(k^{-\infty})_{\Psi^{-\infty}(\Gamma)}.\label{e:alarm2}
\end{align}
Let $C_3>\Xi$. Then, for $\psi\in C_{c}^\infty((-C_3,C_3))$ with $\psi\equiv 1$ near $[-\Xi,\Xi]$,
\beq\label{e:alarm3}
(1-\chi_1\hDarg) \psi\hDarg=O(k^{-\infty})_{\Psi^{-\infty}(\Gamma)}.
\eeq
Combining \eqref{e:alarm2} and \eqref{e:alarm3}, we see that
$$
I-P_{\trig_{C_3c_{\max}k}}^C= (I-P_{\trig_{C_3c_{\max}k}}^C)\big(1-\psi\hDarg\big)+O(k^{-\infty})_{\Psi^{-\infty}(\Gamma)},
$$
and then \eqref{e:claimJubilee} holds by~\eqref{e:compactMicro}.
\epf

\section{Application of the abstract Nystr\"om results to the second-kind Helmholtz BIEs}\label{sec:NystromApplied}
\subsection{Technical bounds on the Kress splittings}

The next few lemmas bound $\FL$ and $\FH$ (defined by \eqref{e:defFs}) for the splittings discussed in Lemmas~\ref{l:standardSplitA} to~\ref{l:standardSplitC}  for any curve $\Gamma$.

\begin{lemma}
\label{l:generalEstimates}
For $L_1$ and $L_2$ satisfying~\eqref{e:fourierSplit}, $0<\e<1$, $k_0>0$, $s\in \mathbb{R}$, $M>0$, there is $C>0$ such that for all $k>k_0$,  $N>(1+\e)(1-\e)^{-1}c_{\max}k$
\begin{equation}
\label{e:basicEstimates}
\FL^{s,\e}(N,L)\leq C\sqrt{\log k},\qquad \FH^{s,\e}(N,L)\leq Ck^{-M}.
\end{equation}
Furthermore, if $N-n\geq (1+\e) c_{\max} k$, then 
\beq\label{e:theEnd1}
{\max}\big\{ \FL^{s,\e}(N,n,L),\FH^{s,\e}(N,n,L) \big\} \leq  C k^{-M}.
\eeq
\end{lemma}
%
%
%

\begin{proof}
By \eqref{e:widehatL} and \eqref{e:fourierSplit},
$$
D_t^\ell \widehat{L}_{1,m}=\frac{k}{\sqrt{2\pi}}\int_0^{2\pi}\int_{\mathbb{S}^1} D_t^\ell\Big(e^{ik\big(\langle \gamma(t)-\gamma(\tau),\omega\rangle-m\tau/k\big)}f(\omega,t,\tau)\Big)dS(\omega)d\tau.
$$
When $m>(1+\e)c_{\max} k$, 
$$
\big|-\langle \gamma'(\tau),\omega\rangle-m/k\big|>c_\e |m/k|,
$$
and thus we can integrate by parts in $\tau$ to obtain
\begin{equation}
\label{e:highFreqL1}
|D_t^\ell\widehat{L}_{1,m}|\leq C_M k^{-M}\langle m/k\rangle^{-M},\qquad |m|>(1+\e)c_{\max}k. 
\end{equation}
Next, by the Cauchy--Schwarz inequality,
\begin{equation}
\label{e:bumping}
k^{-1}\sum_{|m|\leq (1+\e)c_{\max}k}\|\widehat{L}_{1,m}\|_{\Hsh}\langle m/k\rangle^s\leq C_sk^{-1/2}\sqrt{\sum_{|m|\leq (1+\e)c_{\max}k}\|\widehat{L}_{1,m}\|_{\Hsh}^2},
\end{equation}
where we have used that 
\beqs
\sqrt{\sum_{|m|\leq (1+\e)c_{\max}k}\langle m/k\rangle^{2s}}\leq C_s k^{1/2}.
\eeqs
Let $L_1$ be the operator with kernel $L_1(t,\tau)$. By \eqref{e:widehatL},
it is enough to estimate for any $\ell\geq 0$
\begin{align}\nonumber
\sum_{|m|\leq (1+\e)c_{\max}k}\big\|(k^{-1}D_t)^\ell\widehat{L}_{1,m}\big\|_{L^2(\Gamma)}^2
&=\sum_{|m|\leq (1+\e)c_{\max}k}\big\|(k^{-1}D_t)^\ell L_{1}(e^{-im\tau})\big\|_{L^2(\Gamma)}^2\\ \nonumber
&\leq C\sum_{|m|\leq (1+\e)c_{\max}k}\big\|(k^{-1}D_t)^\ell L_{1}(e^{-im\tau})\big\|_{L^2(0,2\pi)}^2\\
&\leq C\big\|(k^{-1}D_t)^\ell L_1\big\|^2_{\operatorname{HS}},\label{e:eating}
\end{align}
where $\|\cdot\|_{\operatorname{HS}}$ denotes the Hilbert--Schmidt norm (see, e.g., \cite[Page 32]{Kr:14}).

Now, the kernel of $(k^{-1}D_t)^\ell L_1$ is given by
$$
K(\tau,s):=k\int_{\mathbb{S}^1} e^{ik\big(\langle \gamma(t)-\gamma(\tau),\omega\rangle\big)}f_\ell(\omega,t,\tau)dS(\omega),
$$
for some smooth $f_\ell$ with all derivatives bounded uniformly in $k$. Therefore, since the Hilbert--Schmidt norm is equal to the $L^2$ norm of the kernel (see, e.g., \cite[Equation 1.2.33]{At:97}),
\begin{align*}
&\|(k^{-1}D_t)^\ell L_1\|^2_{\operatorname{HS}}\\
&=k^2\int_0^{2\pi}\int_0^{2\pi}\int_{\mathbb{S}^1}\int_{\mathbb{S}^1} e^{ik(\langle \gamma(t)-\gamma(\tau),\omega\rangle-\langle \gamma(t)-\gamma(\tau),\zeta\rangle)}\widetilde{f}_\ell(\omega,t,\tau,\zeta)dS(\omega)dS(\zeta)dtd\tau\\
&=k^2\int_0^{2\pi}\int_0^{2\pi}\int_{\mathbb{S}^1}\int_{\mathbb{S}^1} e^{ik|\gamma(t)-\gamma(\tau)|\big(\big\langle \frac{\gamma(t)-\gamma(\tau)}{|\gamma(t)-\gamma(\tau)|},\omega\big\rangle-\big\langle \frac{\gamma(t)-\gamma(\tau)}{|\gamma(t)-\gamma(\tau)|},\zeta\big\rangle\big)}\widetilde{f}_\ell(\omega,t,\tau,\zeta)dS(\omega)dS(\zeta)dtd\tau\\
&\leq Ck+ k^2\int_{0}^{2\pi}\int_{|t-\tau|\geq Ck^{-1}}\int_{\mathbb{S}^1}\int_{\mathbb{S}^1} e^{ik|\gamma(t)-\gamma(\tau)|\big(\big\langle \frac{\gamma(t)-\gamma(\tau)}{|\gamma(t)-\gamma(\tau)|},\omega\big\rangle-\big\langle \frac{\gamma(t)-\gamma(\tau)}{|\gamma(t)-\gamma(\tau)|},\zeta\big\rangle\big)}\widetilde{f}_\ell(\omega,t,\tau,\zeta)dS(\omega)dS(\zeta)dtd\tau.
\end{align*}
We now apply stationary phase in $\zeta$ and $\omega$. Let 
\beqs
\Phi:=
\langle \gamma(t)-\gamma(\tau),\omega\rangle-\langle \gamma(t)-\gamma(\tau),\zeta\rangle.
\eeqs
When $\omega \in \mathbb{S}^1$, $\omega'=\omega^\perp$ and $(\omega^\perp)' = -\omega$ so that
\beqs
\partial_\omega \Phi= \big\langle \gamma(t)-\gamma(\tau),\omega^\perp\big\rangle
\quad\tand\quad
\partial_\zeta \Phi = -\big\langle \gamma(t)-\gamma(\tau),\zeta^\perp\big\rangle
\eeqs
so that there are stationary points when both $\zeta$ and $\omega$ equal $\blue{\pm \frac{\gamma(t)-\gamma(\tau)}{|\gamma(t)-\gamma(\tau)|}}$. 
Since 
\beqs
\partial^2_{\tau\omega}\Phi=\begin{pmatrix}
 -\big\langle \gamma(t)-\gamma(\tau),\omega\big\rangle& 0\\
0& \big\langle \gamma(t)-\gamma(\tau),\zeta\big\rangle\\
\end{pmatrix}
\eeqs
and $|t-\tau|\geq k^{-1}$, these stationary phase are nondegenerate, and the principle of stationary phase (see, e.g., 
\cite[Theorem 3.16]{Zw:12}, \cite[Theorem 7.7.5]{Ho:83}) implies that 
$$
\begin{aligned}
\|(k^{-1}D_t)^\ell L_1\|_{\operatorname{HS}}&\leq Ck+Ck\int_{0}^{2\pi}\int_{|t-\tau|\geq Ck^{-1}} |\gamma(t)-\gamma(\tau)|^{-1}dtd\tau\leq Ck\log k.
\end{aligned}
$$
Combining this with~\eqref{e:bumping} and~\eqref{e:eating}, we obtain
$$
k^{-1}\sum_{|m|\leq (1+\e)c_{\max}k}\|\widehat{L}_{1,m}\|_{\Hsh}\langle m/k\rangle^s\leq C_s\sqrt{\log k}.
$$
The first estimate in~\eqref{e:basicEstimates} then follows from this last inequality combined 
the definition of $F_L^{s,\epsilon}(N,L)$ \eqref{e:defFs}, the choice $N>(1+\e)(1-\e)^{-1}c_{\max}k$, and~\eqref{e:highFreqL1}.

For the second estimate in~\eqref{e:basicEstimates}, by \eqref{e:widehatL} and \eqref{e:fourierSplit} (including, in particular, the support property of the Fourier transform of $\widetilde{L}$),
$$
D_t^\ell\widehat{L}_{2,m}=\frac{1}{\sqrt{2\pi}}\int_{0}^{2\pi}\int_{|\xi|\leq 1} e^{ik(\langle \gamma(t)-\gamma(\tau),\xi\rangle-m\tau/k)}\widehat{\widetilde{L}}_{2}(\xi,t,\tau)d\xi d\tau.
$$
As before, we integrate by parts in $\tau$ when $|m|>(1+\e)c_{\max}k$ to obtain that
\begin{equation}
\label{e:highFreqL2}
|D_t^\ell\widehat{L}_{2,m}|\leq C_M k^{-M}\langle m/k\rangle^{-M},\qquad |m|>(1+\e)c_{\max}k,
\end{equation}
which, together with~\eqref{e:highFreqL1} immediately implies the second estimate in~\eqref{e:basicEstimates}.

Finally, the result \eqref{e:theEnd1} follows from \eqref{e:highFreqL1} and  \eqref{e:highFreqL2}.
\end{proof}

\begin{lemma}
\label{l:cutSplitEst}
Let $L_1$ and $L_2$ satisfy~\eqref{e:cutSplit}. Then, for all $M>0$, $s\in \mathbb{R}$, and $0<\e<1$, there is $C>0$ such that 
\begin{equation}
\label{e:basicCutEstimates}
\FL^{s,\e}(N,L)\leq C,\qquad \FH^{s,\e}(N,L)\leq C\langle N/k\rangle^{-M}k\log k.
\end{equation}
\end{lemma}
\begin{proof}
By the definitions of $\FL$ and $\FH$ \eqref{e:defFs}, it is sufficient to prove that
$$
\|\widehat{L}_{1,m}\|_{\Hsk}\leq C_{M,s}\langle m/k\rangle^{-M},\qquad \|\widehat{L}_{2,m}\|_{\Hsk}\leq C_{M,s}\log k\langle m/k\rangle^{-M}.
$$
By \eqref{e:cutSplit0} and \eqref{e:cutSplit},
$$
D_t^\ell \widehat{L}_{1,m}(t)= \frac{1}{\sqrt{2\pi}}\int_0^{2\pi} e^{-i\tau m} D_t^\ell\big(\widetilde{L}_1\big(k(\gamma(t)-\gamma(\tau)), t,\tau\big)\big)d\tau.
$$
By integration by parts in $\tau$ and the property \eqref{e:cutSplit} of $\widetilde{L}_1$,
\begin{align*}
|D_t^\ell \widehat{L}_{1,m}(t)|&= \frac{1}{\sqrt{2\pi}} \int_0^{2\pi} e^{-i\tau m} D_t^\ell \frac{(1+ k^{-2}mD_\tau)^N}{\langle m/k\rangle^{2N}}\big(\widetilde{L}_1\big(k(\gamma(t)-\gamma(\tau)), t,\tau\big)\big)d\tau\\
&\leq C\int_0^{2\pi} k^{\ell+1}\langle m/k\rangle^{-N} \langle k(\gamma(t)-\gamma(\tau))\rangle^{-2}d\tau\\
&\leq k^{\ell}\langle m/k\rangle^{-N}.
\end{align*}
The proof of the estimate on $\widehat{L}_{2,m}$ is identical.
\end{proof}

\begin{lemma}
\label{l:genuineSplit}
Let $L_1$ and $L_2$ satisfy~\eqref{e:fourierSplit}. If $\Gamma$ is convex with non-vanishing curvature and unit parametrized, then for all $k_0>0$ and $s\in \mathbb{R}$ there is $C>0$ such that for all $k>k_0$, $0<\e<1$ and $N\in \mathbb{R}$, 
\begin{equation}
\label{e:advancedEstimates}
\FL^{s,\e}(N,L)\leq C.
\end{equation}
\end{lemma}
\begin{proof}
We decompose $L_1$ into two pieces. Let $\chi \in C_{c}^\infty(\mathbb{R})$ with $\supp(1-\chi)\cap [-1,1]=\emptyset$ and set $\chi_\e(x):=\chi(\e^{-1}x)$. Then
\beq\label{e:splitConvex}
(k^{-1}D_t)^\ell L_1(t,\tau)= (k^{-1}D_t)^\ell L_1(t,\tau)\chi_\e(t-\tau) +(k^{-1}D_t)^\ell L_1(t,\tau)\big(1-\chi_\e(t-\tau)\big)=:I_1(t,\tau)+I_2(t,\tau).
\eeq
Let $\widehat{I}_{j,m}(t)$ be the Fourier transform in $\tau$ of $I_j(t,\tau)$ (compare to \eqref{e:widehatL}),
so that, by the definition of $\FL$ \eqref{e:defFs}, it is sufficient to prove that
\beq\label{e:BH1}
k^{-1}\sum_{m}\|\widehat{I}_{1,m}\|_{\LtG}\langle m/k\rangle^s+ k^{-1}\sum_{m}\|\widehat{I}_{2,m}\|_{\LtG}\langle m/k\rangle^s\leq C.
\eeq
Now, by \eqref{e:fourierSplit},
$$
 \widehat{I}_{1,m}(t)=k\int\int_{\mathbb{S}^1} e^{ik\big(\langle \gamma(\tau)-\gamma(t),\omega\rangle-\tau m/k\big) }f_\ell(\omega,t,\tau)\chi_\e(t-\tau)dS(\omega)d\tau.
$$
 For general $m$, we may perform stationary phase in $\omega$ alone \blue{with parameter $k|\gamma(t)-\gamma(\tau)|$} to find
 \begin{align*}
  \widehat{I}_{1,m}(t)=k\int\sum_{\pm}e^{\pm ik|\gamma(t)-\gamma(\tau)|-i\tau m} \langle k|\gamma(t)-\gamma(\tau)|\rangle^{-1/2}f_{\pm,\ell}(t,\tau)\chi_\e(t-\tau)d\tau,
  \end{align*}
  where 
  $$
  | f_{\pm,\ell}(t,\tau)|\leq C.
  $$
  Integrating in $\tau$ and using that $|\dot\gamma |>c>0$ and $\gamma:\mathbb{R}/2\pi \mathbb{Z}\to \mathbb{R}^2$ is a bijection \blue{onto its image}, we then obtain for any $m$ that
\begin{equation}
\label{e:i0}
|
\widehat{I}_{1,m}(t)|\leq  C_{\ell} k^{1/2}.
\end{equation}
Recalling the goal \eqref{e:BH1}, we see
it is therefore enough to obtain estimates when $|m-k|\geq M' k^{1/2}$ and $|m+k|\geq M' k^{1/2}$, i.e., when
$Mk^{-1/2}<|1-(m/k)^2|$. We consider the three cases
\beq\label{e:threeCases}
(m/k)^2 -1 \leq -\delta, \quad -\delta < (m/k)^2- 1 < -Mk^{-1/2}, \quad \tand \quad (m/k)^2 -1 > Mk^{-1/2}
\eeq
separately. 
Let 
$$
\Phi:=\langle \gamma(\tau)-\gamma(t),\omega\rangle-\tau m/k
$$
be the phase function.  Then,
\beq\label{e:statPoint}
\partial_\tau \Phi=\langle \gamma'(\tau),\omega\rangle -m/k,\qquad \partial_\omega \Phi=\langle \gamma(\tau)-\gamma(t),\omega^\perp\rangle.
\eeq
Since $\omega$ is a variable on the unit circle and $|\gamma'(t)|=1$ for all $t$,
at a critical point, $(\tau_c,\omega_c)$, where $\partial_{\tau}\Phi=\partial_\omega\Phi=0$,
\beq\label{e:Sat1}
\gamma(t)-\gamma(\tau_c)=\pm |\gamma(t)-\gamma(\tau_c)|\omega_c,\qquad \gamma'(\tau_c)=(m/k)\omega_c\pm \sqrt{1-(m/k)^2}\,\omega_c^\perp,
\eeq
and 
\beq\label{e:Hessian}
\partial^2_{\tau\omega}\Phi=\begin{pmatrix}\langle \gamma''(\tau_c),\omega_c\rangle& \langle \gamma'(\tau_c),\omega_c^\perp\rangle\\
\langle \gamma'(\tau_c),\omega_c^\perp\rangle&-\langle \gamma(\tau_c)-\gamma(t),\omega_c\rangle \end{pmatrix}
=\begin{pmatrix}\langle \gamma''(\tau_c),\omega_c\rangle& \pm \sqrt{1-(m/k)^2}\\
\pm \sqrt{1-(m/k)^2}&-\langle \gamma(\tau_c)-\gamma(t),\omega_c\rangle \end{pmatrix}.
\eeq
\blue{Observe that at $\tau_c=t$, with $\omega_c$ satisfying~\eqref{e:Sat1}, $\partial^2_{\tau \omega}\Phi= (1-m^2/k^2)$. Hence, when $1-(m/k)^2\geq \delta$  there is $c>0$ depending on $\delta$ such that every critical point $(0,\omega_c)$ is isolated inside $B((0,\omega_c),c\delta)$. In particular, there is $c>0$ such that if $\e <c\delta$, then all critical points of the phase in the support integrand of $\widehat{I}_{1,m}$ occur with $\tau =t$ and 
}
$$
\partial^2_{\tau\omega}\Phi
=\begin{pmatrix}\langle \gamma''(\tau_c),\omega_c\rangle& \mp \sqrt{1-(m/k)^2}\\
\mp \sqrt{1-(m/k)^2}&0 \end{pmatrix}.
$$
Thus $|\det \partial^2_{\tau\omega}\Phi|>c>0$ and the principle of stationary phase \blue{in $(\tau,\omega)$} (see, e.g., \cite[Theorem 3.16]{Zw:12}) implies that 
\begin{equation}
\label{e:i1}
| \widehat{I}_{1,m}(t)|\leq  C_{\ell \delta}
\qquad\tfa 1-(m/k)^2\geq \delta.
\end{equation}

We next consider $1-(m/k)^2<-Mk^{-1/2}$. In this case, we integrate by parts with $\frac{k^{-1}\langle D\Phi, D_{\tau,\omega}\rangle}{|D\Phi|^2}$, using that 
$$
|\partial_\tau \Phi|+|\partial_\omega \Phi| \geq |m/k|-1 + 1-|\langle\gamma'(\tau),\omega\rangle|+c|\tau-t|^2,\qquad 
$$
to obtain 
 \begin{align}
|\widehat{I}_{1,m}(t)|&\leq Ck^{1-N'}\int \Big(|m/k|-1 + 1-|\langle\gamma'(\tau),\omega\rangle|+c|\tau-t|^2\Big)^{-2N'}dS(\omega)d\tau\\
&\leq Ck^{1-N'}\int_0^{2\pi}\int_{0}^{\pi} \Big(|m/k|-1 + 1-\sin\theta+c|t-\tau|^2\Big)^{-2N'}d\theta d\tau.
\end{align}
Now 
\begin{align*}
\int_{-\pi/2-\e}^{\pi/2+\e} \big( A + 1 -\sin \theta\big)^{-2N'} d \theta
\leq \int_{-\e}^\e \big( A + C \phi^2\big)^{-2N'} d \phi 
&= A^{1/2-2N'} \int_{-A^{1/2}\e}^{A^{1/2}\e} \big(1 + C s^2\big)^{-2N'} d \phi \\
&\leq C A^{1/2-2N'}
\end{align*}
and 
\beqs
\int_{|\theta-\pi/2|>\e} \big( A + 1 -\sin \theta\big)^{-2N'} d \theta \leq C A^{-2N'},
\eeqs
so that 
\begin{align}
|\widehat{I}_{1,m}(t)|&\leq Ck^{1-N'} \int_0^{2\pi} \Big(|m/k|-1+c|t-\tau|^2\Big)^{1/2-2N'}d\tau\leq Ck^{1+\ell-N'}  \big(|m/k|-1\big)^{1-2N'}.
\label{e:i2}
 \end{align}

Finally, we consider $M/k^{1/2}<1-(m/k)^2<\delta.$ For this, we again perform stationary phase in $(\omega,\tau)$. 
Recall that $\gamma''(\tau_c) = - n(\tau_c) \kappa(\tau_c)$, where $n$ is the outward-pointing unit normal vector to $\Omega^-$ and $\kappa$ is the (signed) curvature. Since $n$ is perpendicular to $\gamma'$, by both equations in \eqref{e:Sat1}, 
$$
-\langle \gamma''(\tau_c),\omega_c\rangle\langle \gamma(\tau_c)-\gamma(t),\omega_c\rangle = |\kappa(\tau_c)|\sqrt{1-(m/k)^2}|\gamma(t)-\gamma(\tau_c)|.
$$
Therefore, by \eqref{e:Hessian},
\beq\label{e:Hessian2}
\det \partial^2_{\tau\omega}\Phi= \sqrt{1-(m/k)^2}\Big(|\kappa(\tau_c)||\gamma(t)-\gamma(\tau_c)|-\sqrt{1-(m/k)^2}\Big).
\eeq
We now seek to show that $|\det \partial^2_{\tau\omega}\Phi|>0$ when $|\kappa(\tau)|\geq c$ for all $\tau$; i.e., when $\Gamma$ is convex with non-vanishing curvature.
By the first equation in \eqref{e:Sat1}, with coordinates chosen so that $\omega_c=(1,0)$, for $r=\pm 1$,
\begin{align*}
&\pm \big(|\gamma(t)-\gamma(\tau_c)|,0\big)\\
&=\gamma(t)-\gamma(\tau_c)\\
&=\gamma'(\tau_c)(t-\tau_c)+\frac{1}{2}\gamma''(\tau_c)(t-\tau_c^2)+O((t-\tau_c)^3)\\
&=\big(m/k, r\sqrt{1-(m/k)^2}\big) (t-\tau_c)\pm \frac{1}{2}\kappa(\tau_c)\big(-r\sqrt{1-(m/k)^2},m/k\big)(t-\tau_c)^2+O\big((t-\tau_c)^3\big).
\end{align*}
The two components of this last equation imply that
\begin{gather*}
(t-\tau_c)\Big(r\sqrt{1-(m/k)^2}\pm \frac{1}{2}\kappa(\tau_c)(m/k)(t-\tau_c)+O\big((t-\tau_c)^2\big)\Big)=0,
\quad\tand
\\ |\gamma(t)-\gamma(\tau_c)|=\Big|m/k \pm \frac{1}{2}\kappa(\tau_c)(-r\sqrt{1-(m/k)^2})(t-\tau_c)+O((t-\tau_c)^2\Big|\big|t-\tau_c\big|
\end{gather*}
Therefore, either $t=\tau_c$, in which case $|\det \partial^2_{\tau\omega}\Phi|=-\sqrt{1-(m/k)^2}$ by \eqref{e:Hessian}, or 
$$
\frac{1}{2}\kappa(\tau_c)(t-\tau_c)=\mp r(k/m) \sqrt{1-(m/k)^2} +O((t-\tau_c)^2).
$$
In this latter case, since $\kappa(\tau_c)>0$, 
$$
|t-\tau_c|=O(\sqrt{1-(m/k)^2})
$$
and
$$
\frac{|\gamma(t)-\gamma(\tau_c)|}{|t-\tau_c|}=|m/k| +O\big(1-(m/k)^2\big).
$$
Thus, 
\begin{align*}
|\kappa(\tau_c)||\gamma(t)-\gamma(\tau_c)|&=|\kappa(\tau_c)||t-\tau_c| (|m/k|+O\big(1-(m/k)^2))\big)\\
&=\big|( 2k/m \sqrt{1-(m/k)^2} +O\big((1-(m/k)^2)\big)\big|\big(|m/k|+O\big((1-(m/k)^2)\big)\big)\\
&=2\sqrt{1-(m/k)^2}+O\big( 1-(m/k)^2\big),
\end{align*}
Thus, by \eqref{e:Hessian2}, at either critical point
$$
|\det \partial^2_{\tau\omega}\Phi| \geq c\big(1-(m/k)^2\big).
$$

Performing stationary phase, (using~\cite[Theorem 7.7.5]{Ho:83} with $k=1$) we then obtain 
\begin{equation}
\label{e:i3}
| \widehat{I}_{1,m}(t)|\leq C_{\ell} 
(1-(m/k)^2)^{-1/2}\quad\tfor M/k^{1/2}<1-(m/k)^2<\delta.
\end{equation}
We now combine~\eqref{e:i0},~\eqref{e:i1},~\eqref{e:i2}, and~\eqref{e:i3}, to see that the bound on $\widehat{I}_{1,m}$ in \eqref{e:BH1} holds (i.e., the contribution to $\FL$ from $I_1$ is uniformly bounded). Indeed, splitting the sum into the four regions (three as in \eqref{e:threeCases} and the fourth equal to $Mk^{-1/2}<|1-(m/k)^2|$) and inputting 
\eqref{e:i0},~\eqref{e:i1},~\eqref{e:i2}, and~\eqref{e:i3}, we obtain that 
\begin{align*}
\sum_{m}\|\widehat{I}_{1,m}\|_{\LtG}\langle m/k\rangle^s
&\leq 
C\sum_{Mk^{-1/2}<|1-(m/k)^2|} \langle m/k\rangle^s k^{1/2}
+C\sum_{(m/k)^2-1 \leq -\delta}\langle m/k\rangle^s \\
&\qquad+C\sum_{(m/k)^2-1> Mk^{-1/2}} \langle m/k\rangle^s k^{1-N'} \big( |m/k|-1\big)^{1-2N'}\\
&\qquad +C \sum_{Mk^{-1/2} < 1-(m/k)^2 < \delta }\langle m/k\rangle^s \big( 1- (m/k)^2\big)^{-1/2}
\\
&\quad\leq 
Ck+  C \int_{Mk^{-1/2} < 1-(m/k)^2 < \delta }\big( 1- (m/k)^2\big)^{-1/2}dm\\
&\leq Ck + C k\int_{Mk^{-1/2} < 1-x^2 < \delta } \big(1- x^2\big)^{-1/2}dx,
\end{align*}
so that the bound on $\widehat{I}_{1,m}$ in \eqref{e:BH1} holds.

Next, we consider $I_2$. For this, we use a partition of unity, $\{\psi_j\}_{j=1}^J$, with $\sup_j\diam(\supp\psi_j)<\frac{\e}{2}$. To write
$$
I_2(t,\tau)=\sum_{j,n}\psi_j(t)I_2(t,\tau)\psi_n(\tau).
$$ 
Since $\chi_\e \equiv 1$ on $[-\e,\e]$, we can assume that $\supp\psi_j\cap \supp \psi_n=\emptyset$. 

Since $\Gamma$ is convex with non-vanishing curvature, 
\beq\label{e:diagramsareforsissies1}
\frac{\gamma(t)-\gamma(\tau)}{|\gamma(t)-\gamma(\tau)|}\neq \pm \gamma'(t) \quad\tfor t\in \supp \psi_j, \tau\in \supp \psi_n;
\eeq
i.e., the ray from $\gamma(\tau)$ to $\gamma(t)$ is not tangent to $\Gamma$ at $\gamma(t)$.  We now write
\begin{align*}
&\int \psi_j(t) I_2(t,\tau)\psi_n(\tau)\overline{\psi_j(t) I_2(t,s)\psi_n(s)}dt\\
&= k^{2+2\ell}\int\int_{\mathbb{S}^1}\int_{\mathbb{S}^1} e^{ik(\langle \gamma(\tau)-\gamma(t),\omega\rangle-\langle \gamma(s)-\gamma(t),\zeta\rangle) }\\
&\hspace{1cm}\psi_j^2(t)\psi_n(\tau)\psi_n(s)(1-\chi_\e(t-s))(1-\chi_\e(t-\tau))f_\ell(\omega,t,\tau) \overline{f_\ell(\zeta,t,s) }dS(\omega)dS(\zeta)dt.
\end{align*}
Denote the phase function by
$$
\Psi:=\big\langle \gamma(\tau)-\gamma(t),\omega\big\rangle-\big\langle \gamma(s)-\gamma(t),\zeta\big\rangle.
$$
First, by integrating by parts in $\zeta$, we may assume that 
$$
\big|\big \langle \gamma(t)-\gamma(s),\zeta^\perp\big\rangle\big|\ll1,
$$
at the cost of an $O(k^{-\infty})$ error term. 
Next, we perform stationary phase in $(t,\omega)$; 
$$
\partial_t \Psi= \langle \gamma'(t),\zeta-\omega\rangle,\qquad \partial_\omega \Psi = \langle \gamma(t)-\gamma(\tau),\omega^\perp\rangle,
$$
and 
$$
\partial^2_{t\omega}\Psi=\begin{pmatrix} \langle \gamma''(t),\zeta-\omega\rangle &\langle \gamma'(t),\omega^\perp\rangle\\\langle \gamma'(t),\omega^\perp\rangle& \langle \gamma(t)-\gamma(\tau),\omega\rangle\end{pmatrix}.
$$
By our assumption on $\supp \psi_j$ and $\supp \psi_k$, we have $\omega_c=\zeta$  and $\gamma(t_c)-\gamma(\tau)=\pm |\gamma(t_c)-\gamma(\tau)|\omega$.
Thus, by \eqref{e:diagramsareforsissies1}, $|\langle \gamma'(t_c),\zeta^\perp\rangle|>c>0$.  In particular, $|\det \partial^2\Psi|>c>0$ and we can perform stationary phase to obtain
\begin{align*}
\int \psi_j(t)I_2(t,\tau)\psi_n(\tau)\overline{\psi_j(t)I_2(t,s)\psi_n(s)}dt
&= k^{1+2\ell}\int_{\mathbb{S}^1} e^{ik(\langle \gamma(\tau)-\gamma(s),\zeta\rangle) }\widetilde{f}_\ell(\zeta,s,\tau) \widetilde{\chi}_\e(s-\tau)dS(\zeta),
\end{align*}
for some $\widetilde{\chi}\in C_{c}^\infty(\mathbb{R})$ with $\widetilde{\chi}\equiv 1$ in a small neighborhood of 0. In particular, this has exactly the same form as $I_1$ and hence
$$
\| \widehat{I}_{2,m}\|_{L^2}=\sqrt{\big|\sum_{j,n}\langle  \psi_j I_2\psi_n e^{-im\tau}, \psi_j I_2\psi_n e^{-im\tau}\rangle\big|}\leq \sum_{j,n}\sqrt{\|( \psi_jI_2\psi_n)^* \psi_jI_2\psi_ne^{-im\tau}\|_{L^2}} .
$$
In particular, by~\eqref{e:i0},~\eqref{e:i1},~\eqref{e:i2}, and~\eqref{e:i3},
$$
k^{-1}\sum_m\| \widehat{I}_{2,m}\|_{L^2}\leq Ck^\ell.
$$
\end{proof}


Although, at the moment, we are only able to prove Lemma~\ref{l:genuineSplit} in the case of $\Gamma$ convex with non-vanishing curvature, and in other cases lose a $\sqrt{\log k}$ as in Lemma~\ref{l:generalEstimates}; we conjecture that this loss is technical.
\begin{conjecture}
\label{conj:fl}
Let $\Gamma$ be smooth and parametrized by $\gamma$ where $0<|\gamma'|<c_{\max}$. Then for all $k_0>0$ and $s\in \mathbb{R}$ there is $C>0$ such that for all $N>0$, $k>k_0$, $0<\e <1$, 
$$
\FL^{s,\e}(N,L)\leq C.
$$
\end{conjecture}

If this \blue{conjecture} holds, then the conclusion of Point (ii) of Theorem \ref{t:DNystrom} holds with $N\geq C_1 k$; i.e., we need not consider the case of convex obstacles separately in Theorem \ref{t:DNystrom}.


\subsection{Application of Theorem~\ref{thm:MAT2} to Dirichlet and Neumann BIEs}

\paragraph{Proof of Theorem~\ref{t:DNystrom}}

Parts (i)-(iii) of Theorem~\ref{t:DNystrom} follow from Theorem~\ref{thm:MAT2} combined with 
\bit
\item the fact $\mathcal{B}^{r,s}(\widetilde{L})=0= \|\widetilde{L}\|_{s,r}$ for all $s,r$ (since the sums over $j$ in \eqref{e:nystromSetup} are empty),
\item for Part (i), Lemmas \ref{l:standardSplitA} and \ref{l:genuineSplit} (applied to $\eta_D S_{k}$ $\DL_{k}$, and $\DL'_{k}$),
\item for Parts (ii) and (iii), Lemmas \ref{l:standardSplitA} and \ref{l:generalEstimates}.
\eit
For the result about plane wave data \eqref{e:MR7}, we use, in addition, Lemma~\ref{l:NystromCompactMicrolocalizedD} -- together with Lemmas~\ref{lem:HFv} and Lemma~\ref{lem:vlowerbound} -- and the bounds \eqref{e:theEnd1}.

\paragraph{Proof of Theorem~\ref{t:NeumannNystrom}}

Theorem~\ref{t:NeumannNystrom} follow from Theorem~\ref{thm:MAT2} after using Lemmas~\ref{l:standardSplitA} and \ref{l:generalEstimates} on $\DL_k$ and $\DL_k'$, Lemmas \ref{l:standardSplitB} and \ref{l:cutSplitEst} on $S_{ik}$, and the combination of Lemmas \ref{l:standardSplitC}, \ref{l:generalEstimates}, and \ref{l:cutSplitEst} on $k^{-1}(H_k-H_{ik})$.

As in the Dirichlet case, the result about plane wave data \eqref{e:MR7} uses, in addition, Lemma~\ref{l:NystromCompactMicrolocalizedD} together with Lemmas~\ref{lem:HFv} and Lemma~\ref{lem:vlowerbound}.


%

\section{From quasimodes to pollution for projection methods}
\label{s:pollution}

In this section, we first prove a version of Theorem~\ref{t:pollutionIntro} for an abstract projection method under an appropriate assumption on the approximation space (see Lemma~\ref{l:pollution}). We then specialize to the case of piecewise polynomial spaces, proving Theorem~\ref{t:pollutionIntro}.

\subsection{Abstract assumptions on the projection}
We start by stating some abstract assumptions on the approximation space and projection. Throughout, we let $\{V_h\}_{h>0}\subset L^2(\Gamma)$ be a family of finite dimensional approximation spaces with $P^G_{V_h}:L^2(\Gamma)\to V_h$ orthogonal projection. 
\begin{assumption}
\label{ass:Goodness}
 Let $t_{\max}\geq 0$. For any $k_0>0$, there is $C>0$ such that for all $0<h<1$, and $k_0<k$,
$$
\|I-P^G_{V_h}\|_{H^{t_{\max}}(\Gamma)\to \LtG}\leq Ch^{t_{\max}}. 
$$
\end{assumption}
It will also be convenient to assume that $V_h$ contains the constants.
\begin{assumption}
\label{ass:Goodness2}
The subspace $V_h$ contains the constants.
\end{assumption}

To state our next assumption, let $-\Delta_g$ denote the Laplace--Beltrami operator on $\Gamma$ and $\{\phi_{\lambda_j}\}_{j=1}^\infty$ be an orthonormal basis satisfying
$$
(-\Delta_g-\lambda_j^2)\phi_{\lambda_j}=0,
$$
and define the set of functions \emph{oscillating with frequency between $ka$ and $kb$} as
\beq\label{e:oscilSpace}
\mathcal{E}_k(a,b):=\operatorname{span} \big\{\phi_{\lambda_j}\,:\, ak<  \lambda_j\leq bk\big\}. 
\eeq
We say that a function $u=u(k)$ is \emph{$k$-oscillating} if there are $0<a<b$ and $k_0>0$ such that $u(k)$ is oscillating with frequency between $ka$ and $kb$ for all $k>k_0$. 

We denote by
$$
\Pi_j(a):\LtG\to \mathcal{E}_k(2^{-j-1}a,2^{-j}a),
$$
the orthogonal projection onto $\mathcal{E}_k(2^{-j-1}a,2^{-j}a)$.

We then make an assumption on $P^G_{V_h}$ that quantifies the angle between $(V_h)^\perp$ and oscillating functions.
\begin{assumption}
\label{ass:angleControl}
Let $\{V_h\}_{h>0}\subset L^2(\Gamma)$ and $P_{V_h}^G:L^2(\Gamma)\to V_h$ a corresponding family of projections.
Let $\Xi_0>1$, $k_0>0$ and $J=J(k):[k_0,\infty)\to (0,\infty)$.  Then, for $k\geq k_0$ there exists $\Ang(V_h,P^G_{V_h},k)<\infty$ such that
$$
\sum_{j=0}^{J-1}2^{jt_{\max}}\Pi_j(\Xi_0)(I-P^G_{V_h})\sum_{\ell=0}^{J-1}2^{\ell t_{\max}}\Pi_\ell(\Xi_0):H^{t_{\max}}_{k}(\Gamma)\to \mathcal{E}_k(2^{-J}\Xi_0,\Xi_0)
$$
is surjective and has a right inverse $\mc{R}(k,h):\mathcal{E}_k(2^{-J}\Xi_0,\Xi_0)\to \mathcal{E}(2^{-J}\Xi_0,\Xi_0)$ with 
$$
\|\mc{R}(k,h)\|_{\LtG\to H^{t_{\max}}_{k}(\Gamma)}\leq C\Ang(V_h,P^G_{V_h},k).
$$
\end{assumption}
\begin{remark}
The notation $\Ang$ comes from the fact that if $J=1$, then this number estimates the tangent of the angle between $(V_h)^{\perp}$ and the range of $1_{[2^{-2}\Xi_0^2,\Xi_0^2]}(-k^{-2}\Delta_g)$. 
\end{remark}

\subsection{Pollution under Assumptions~\ref{ass:Goodness} and~\ref{ass:angleControl}}
We now prove the analogue of Theorem~\ref{t:pollutionIntro} under Assumptions~\ref{ass:Goodness} and~\ref{ass:angleControl}. Before proceeding, we need a technical lemma.
For $a>0$, $J\geq 1$, define
$$
\mc{W}_J(a):=\sum_{j=0}^{J-1}2^{jt_{\max}}\Pi_j(a).
$$
\begin{lemma}
\label{l:sobolev}
Let $m\geq 0$, $a>0$. Then there is $C>0$ such that for all $k_0>0$, $k>k_0$, defining $J=J(k):=\log_2k -\log_2 k_0+1$  we have
$$
\|\mc{W}_Ju\|_{H^m}\leq C\max(k^{t_{\max}} k_0 ^{-t_{\max}}\langle k_0\rangle^{m},\langle k\rangle ^{m})\Big\|\sum_{j=0}^{J-1}\Pi_j(a) u\Big\|_{L^2}.
$$
\end{lemma}
\begin{proof}
Notice that
\begin{align*}
\|\mc{W}_J(a)u\|_{H^m}^2&\leq C\| (-\Delta_g+1)^{m/2}\mc{W}_Ju\|_{L^2}^2=C\sum_{j=0}^{J-1}\|(-\Delta_g+1)^{m/2}2^{jt_{\max}}\Pi_j(a)u\|_{L^2}^2\\
&\leq C\sum_{j=0}^{J-1}(k^22^{-2j-2}a^2+1)^{m}2^{2jt_{\max})}\|\Pi_j(a) u\|_{L^2}^2\\
&\leq C\sum_{j=0}^{J-1}2^{2jt_{\max}}\langle 2^{-j}k\rangle^{2m}\|\Pi_j(a) u\|_{L^2}^2\\
&\leq C\max(k^{2t_{\max}}k_0^{-2t_{\max}}\langle k_0\rangle^{2m},\langle k\rangle^{2m})\sum_{j=0}^{J-1}\|\Pi_j(a) u\|_{L^2}^2\\
&= C\max(k^{2t_{\max}}k_0^{-2t_{\max}}\langle k_0\rangle^{2m},\langle k\rangle^{2m})\Big\|\sum_{j=0}^{J-1}\Pi_j(a) u\Big\|_{L^2}^2.
\end{align*}
\end{proof}

\begin{lemma}
\label{l:pollution}
Let $C>0$, $t_{\max}\geq 0$ $\Xi_0>1$, $0<\e<\frac{1}{10}$ $k_0>0$, $0<\e_0<\Xi_0$, $\chi \in C_c^\infty((-(1+2\e)^2, (1+2\e)^2))$ with $\chi \equiv 1$ on $[-1-\e,1+\e]$, $\operator=I+L$, and $\cJ\subset (0,\infty)$ such that Assumptions~\ref{ass:polyboundintro} and~\ref{ass:abstract1} hold. Let $V_h,P_h$ satisfies Assumptions~\ref{ass:Goodness} and \ref{ass:angleControl} with 
$$
\Ang(\mc{V}_h,P_h,k)\leq C (hk)^{-\gamma},\qquad \gamma\geq t_{\max},\qquad 1\leq J(k)\leq \log_2 k+C.
$$
Then there is $c>0$ such that the following holds. For any $k_n\to \infty$, $k_0<k_n\notin \cJ$, $u_n,f_n\in L^2(\Gamma)$, $\beta_n,\alpha_n\in (0,1]$ satisfying
\beq\label{e:betaNorm}
\beta_n\|(I+\pert(k_n))^{-1}\|_{\LtGt}\leq c,
\eeq
and
\begin{gather}(I+\pert(k_n))u_n=f_n,\qquad \|f_n\|_{\LtG}\leq \alpha_n, \qquad \|u_n\|_{\LtG}=1,\label{e:BIEQuasi}\\
\| (1-\chi(\Xi_0^{-2}k_n^{-2}\Delta_g))f_n\|_{\LtG}\leq \beta_n\qquad \|\chi(\e^{-2}_0k_n^{-2}\Delta_g)f_n\|_{\LtG}\leq\beta_n,\label{e:kOscillate}\\
2^{-J(k_n)t_{\max}}(hk_n)^{2t_{\max}-\gamma}\alpha_n \|1_{(0,2^{-2J+4}\Xi_0^2]}(-k^{-2}\Delta_g)(I+L^*)^{-1}u_n\|_{L^2}\leq c, \label{e:normalOK}
 \end{gather}
then for any $0<h<1$ and $n$ such that $(I+P^G_{V_h}\pert(k_n))$ has an inverse,
$$
\big\|(I+P^G_{V_h}\pert(k_n))^{-1}(I-P^G_{V_h})\big\|_{\LtG\to \LtG}\geq c \begin{cases} (hk_n)^{-t_{\max}}, &  \alpha_n\leq (hk_n)^{\gamma}\,\\
(hk_n)^{\gamma-t_{\max}}\alpha_n^{-1},&(hk_n)^{\gamma}\leq \alpha_n\leq c (hk_n)^{\gamma-t_{\max}}.
\end{cases}
$$
\end{lemma}

\begin{proof}
To ease the notation, we omit the subscript $n$ in the proof.
By \eqref{e:kOscillate}, there is $j_0>0$ such that for all $k>k_0$,
\begin{equation}
\label{e:fHasOscillation}
f-\sum_{j=0}^{j_0}\Pi_{j}(2\Xi_0)f=O(\beta(\hsc))_{\LtG}.
\end{equation}
By Assumption~\ref{ass:angleControl} 
$$
g:= \mc{R}(k,h)\sum_{j=0}^{j_0}2^{jt_{\max}}\Pi_j(2\Xi_0)f
$$
satisfies
\beq
\label{e:preliminaryRightInverse}
\sum_{j=0}^{J-1}2^{jt_{\max}}\Pi_j(2\Xi_0)(I-P^G_{V_h}) \sum_{\ell=0}^{J-1}2^{\ell t_{\max}}\Pi_\ell(2\Xi_0)g = \sum_{j=0}^{j_0}2^{jt_{\max}}\Pi_j(2\Xi_0)f,
\eeq
and, using Lemma~\ref{l:sobolev},
\beq\label{e:tildefbound}
\|\tilde{f}\|_{H^{t_{\max}}(\Gamma)}\leq Ck^{t_{\max}}\Ang(\mc{V}_h,P_h,\hsc)\alpha,\quad\text{ where } \quad  \tilde{f}:=\mathcal{W}_J(2\Xi_0)g= \sum_{\ell=0}^{J-1}2^{\ell t_{\max}}\Pi_j(2\Xi_0)g.
\eeq
Applying $\Pi_\ell(2\Xi_0)$ to \eqref{e:preliminaryRightInverse}, we see that
\beqs
\Pi_\ell(2\Xi_0) (I- P^G_{V_h}))\tilde{f} = 
\begin{cases}
\Pi_\ell(2\Xi_0) f, & \ell=0,\ldots, j_0,\\
0, & \ell=j_0+1,\ldots,J-1,
\end{cases}
\eeqs
and combining this with 
\eqref{e:fHasOscillation} we find that 
\beq
\label{e:rightInverse}
\sum_{j=0}^{J-1}\Pi_j(2\Xi_0)(I-P^G_{V_h}) \tilde f = \sum_{j=0}^{j_0}\Pi_j(2\Xi_0)f=f+O(\beta(\hsc))_{\LtG}.
\eeq
In particular,  since $-\Delta_g\phi_0=0$ implies that $\phi_0$ is a constant, 
Assumption~\ref{ass:Goodness2} implies that $1_{\{0\}}(-k^{-2}\Delta_g) (I- P^G_{V_h})=0$ and then \eqref{e:rightInverse} implies that
\begin{equation}
\label{e:weSolvedIt}
\begin{gathered}
(I-P^G_{V_h})\tilde{f}=f +O(\beta )_{\LtG}+1_{(0,2^{-2J+4}\Xi_0^2]}(-k^{-2}\Delta_g)(I-P^G_{V_h})\tilde{f}+1_{(4\Xi_0^2,\infty)}(-k^{-2}\Delta_g)(I-P^G_{V_h})\tilde{f}.
\end{gathered}
\end{equation}
Define
\begin{equation}
\begin{aligned}
&v:=(I+\pert)^{-1}(I-P^G_{V_h})\tilde{f}\label{e:finalv}\\
&=u+O(\beta\|(I+\pert)^{-1}\|_{\LtGt})_{\LtG}\\
&\qquad+(I+\pert)^{-1}\Big(1_{(0,2^{-2J+4}\Xi_0^2]}(-k^{-2}\Delta_g)(I-P^G_{V_h})\tilde{f}+1_{(4\Xi_0^2,\infty)}(-k^{-2}\Delta_g)(I-P^G_{V_h})\tilde{f}\Big). 
\end{aligned}
\end{equation}
Then, $(I+\pert)v=(I-P^G_{V_h})\tilde{f}$, and hence
\begin{align*}
(I+P^G_{V_h}\pert)v&=(I+\pert)v-(I-P^G_{V_h})\pert v=(I-P^G_{V_h})\tilde{f}-(I-P^G_{V_h})\big[(I-P^G_{V_h})\tilde{f}-v\big]\\
&=(I-P^G_{V_h})v.
\end{align*}
Moreover, by Lemma \ref{lem:HFinverse}, the fact that $P^G_{V_h}$ is self-adjoint, Assumption \ref{ass:Goodness}, 
and \eqref{e:tildefbound},
\begin{align*}
\big\|v\big\|_{\LtG}^2&=1-O(\beta\|(I+L)^{-1}\|_{\LtGt})-O(\|(I-P^G_{V_h})\tilde f\|_{\LtG})\\
&\qquad -2\big\langle u, (I+L)^{-1}1_{(0,2^{-2J+4}\Xi_0^2]}(-k^{-2}\Delta_g)(I-P^G_{V_h})\tilde{f}\big\rangle \\
&\qquad +\|(I+L)^{-1}1_{(0,2^{-2J+4}\Xi_0^2]}(-k^{-2}\Delta)(I-P^G_{V_h})\tilde{f}\|_{\LtG}^2\\
&\geq 1-O(\beta\|(I+L)^{-1}\|_{\LtGt})-C(hk)^{t_{\max}-\gamma}\alpha\\
&\qquad -2\big\langle(I-P^G_{V_h})1_{(0,2^{-2J+4}\Xi_0^2]}(-k^{-2}\Delta_g) (I+L^*)^{-1}u, (I-P^G_{V_h})\tilde{f}\big\rangle \\
&\geq 1-O(\beta\|(I+L)^{-1}\|_{\LtGt})-C(hk)^{t_{\max}-\gamma}\alpha\\
&\qquad -\|(I-P^G_{V_h})1_{(0,2^{-2J+4}\Xi_0^2]}(-k^{-2}\Delta_g) (I+L^*)^{-1}u\|_{\LtG}\|(I-P^G_{V_h})\tilde{f}\|_{\LtG}\\
&\geq 1-O(\beta\|(I+L)^{-1}\|_{\LtGt})-C(hk)^{t_{\max}-\gamma}\alpha\\
&\qquad -C2^{-Jt_{\max}}(hk)^{2t_{\max}-\gamma} \alpha\|1_{(0,2^{-2J+4}\Xi_0^2]}(-k^{-2}\Delta_g) (I+L^*)^{-1}u\|_{\LtG}.
\end{align*}
Therefore, if $\alpha\ll (hk)^{\gamma-t_{\max}}$, then, by 
\eqref{e:betaNorm} and~\eqref{e:normalOK}, $\|v\|_{\LtG}\geq c>0$.
Now, by \eqref{e:finalv}, Lemma \ref{lem:HFinverse}, and \eqref{e:tildefbound},
\begin{align*}
\|(I-P^G_{V_h})v\|_{\LtG}&\leq \|(I-P^G_{V_h})(1-\chi(-\hsc^2\Delta_g))(I+\pert)^{-1}(I-P^G_{V_h})\tilde{f}\|_{\LtG}+ \|(I-P^G_{V_h})\chi(-\hsc^2\Delta_g)v\|_{\LtG}\\
&\leq C\|(I-P^G_{V_h})\tilde{f}\|_{\LtG}+C(hk)^{t_{\max}}\|v\|_{\LtG}\\
&\leq C(hk)^{t_{\max}} \Ang(V_h,P_h,k)\alpha+C(hk)^{t_{\max}}\|v\|_{\LtG}.
\end{align*}
Provided the inverse $(I+P^G_{V_h}\pert)^{-1}$ exists, we have $(I+P^G_{V_h}\pert)^{-1}(I-P^G_{V_h})v=v$. Thus, for $\alpha\ll (hk)^{\gamma-t_{\max}}$, we have $\alpha\geq \|(I+L)^{-1}\|^{-1}\geq \beta$, $(hk)^{t_{\max}} \Ang(V_h,P_h,k)\geq 1$, and 
\begin{align*}
\frac{\|(I+P^G_{V_h}\pert)^{-1}(I-P^G_{V_h})v\|_{\LtG}}{\|(I-P^G_{V_h})v\|_{\LtG}}
&\geq \frac{\|v\|_{\LtG}}{C (hk)^{t_{\max}} \Ang(V_h,P_h,k)\alpha +C(hk)^{t_{\max}}\|v\|_{\LtG}}\\
&\geq \frac{1}{C(hk)^{t_{\max}} \Ang(V_h,P_h,k)\alpha +C(hk)^{t_{\max}}}\\
&\geq \frac{1}{C(hk)^{t_{\max}-\gamma} \alpha +C(hk)^{t_{\max}}}\\
&\geq c \begin{cases} (hk)^{-t_{\max}}, & \alpha\leq (hk)^{\gamma} \,\\
(hk)^{\gamma-t_{\max}}\alpha^{-1}, &(h\hbar)^{\gamma}\leq \alpha\ll (hk)^{\gamma-t_{\max}}.
\end{cases}
\end{align*}
\end{proof}

\subsection{Galerkin projection onto piecewise polynomials satisfies Assumption~\ref{ass:angleControl}: proof of Theorem~\ref{t:pollutionIntro}}
First observe that Assumption~\ref{ass:ppG} implies Assumption~\ref{ass:Goodness}. Therefore, we only need to check Assumption~\ref{ass:angleControl}. 
\begin{lemma}
\label{l:lowerBound}
Let $p\geq 0$, $a>0$ $C>0$. Then there are $k_1\geq 0$, $c>0$ such that for any $k_0> k_1$, defining $J:=\log_2 k-\log_2k_0+1$, for all $hk<C$, 
$$
c(hk)^{p+1}\Big\|\sum_{j=0}^{J-1}\Pi_j(a) u\Big\|_{L^2}\leq \Big\|(I-P_{V_h}^G)\mc{W}_J(a)u\Big\|_{L^2}.
$$
Moreover, if $p=0$, then $k_1=0$. 
\end{lemma} 
\begin{proof}
First, observe that 
\begin{equation}
\label{e:low}
\begin{aligned}
\Big\langle \Delta_{g}^{p+1}\sum_{j=0}^{J-1}2^{j(p+1)}\Pi_j(a) u,\sum_{j=0}^{J-1}2^{j(p+1)}\Pi_j(a) u\Big\rangle &= \sum_{j=0}^{J-1}\Big\langle (2^{2j}\Delta_{g})^{p+1}\Pi_j(a) u,\Pi_j(a) u\Big\rangle \\
&\geq 2^{-2(p+1)}a^{2(p+1)}k^{2(p+1)}\sum_{j=0}^{J-1}\|\Pi_j(a) u\|_{L^2}^2\\
&\geq   2^{-2(p+1)}a^{2(p+1)}k^{2(p+1)}\Big\|\sum_{j=0}^{J-1}\Pi_j(a) u\Big\|_{L^2}^2.
\end{aligned}
\end{equation}
Set 
\beq\label{e:newv}
v:=\sum_{j=0}^{J-1}2^{j(p+1)}\Pi_j(a)u.
\eeq
Then, using Lemma~\ref{l:sobolev} in the arguments in~\cite[Lemma 3.2]{Ga:25}, 
we obtain (for $k_0>0$ large enough when $p\geq 1$  or $k_0>0$ when $p=0$)
\begin{equation}
\label{e:brokenEst}
\sum_{T\in \mc{T}_h}\sum_{|\alpha|=p+1}\|\partial_x^\alpha (v\circ \gamma_T)\|_{L^2(\Omega_h,dx)}^2\leq C\e k^{2(p+1)}\Big\|\sum_{j=0}^{J-1}\Pi_j(a) u\Big\|_{L^2}^2+\e^{-1}Ch^{-2(p+1)}\|(I-P_{V_h}^G)v\|_{L^2}^2.
\end{equation}
Arguing as in~\cite[Lemma 3.1]{Ga:25}, we also obtain that for $k>k_1$
\begin{equation}
\label{e:kToDer}
ck^{2(p+1)}\Big\|\sum_{j=0}^{J-1}\Pi_j(a) u\Big\|_{L^2}^2\leq \sum_{T\in \mc{T}_h}\sum_{|\alpha|=p+1}\|\partial_x^\alpha (v\circ \gamma_T)\|_{L^2(\Omega_h,dx)}^2
\end{equation}
and $k_1=0$ if $p=0$. (Note that the arguments of \cite[Lemmas 3.1 and 3.2]{Ga:25} use only estimates on the Sobolev norms of $v$, and the analogue of \eqref{e:low}, and hence can be easily adapted to the case when $v$ is given by \eqref{e:newv}.)

Combining~\eqref{e:brokenEst},~\eqref{e:kToDer}, and taking $\e>0$ small enough proves the lemma.
\end{proof}

\begin{corollary}
\label{c:polysAreGreat}
Let $p\geq 0$, $a>0$, $C>0$. Then there is $k_1\geq 0$ such that given $k_0>k_1$, there is $C_1>0$ such that for all $k>k_0$, $hk\leq C$, setting $J:=\log_2k-\log_2k_0+1$, the operator
$\mc{W}_J(a)(I-P_{V_h}^G)\mc{W}_J(a):\mc{E}_k(2^{-J}a,a)\to \mc{E}_k(2^{-J}a,a)$ is invertible and satisfies
$$
\Big\|\big[\mc{W}_J(a)(I-P_{V_h}^G)\mc{W}_J(a)\big]^{-1}\Big\|_{H_k^{-p-1}\to H_k^{p+1}}\leq C_1(hk)^{\blue{-}2(p+1)}.
$$
Moreover, if $p=0$, then $k_1=0$.
\end{corollary}
\begin{proof}
Observe that since $P_{V_h}^G$ and $\mathcal{W}_J$ are self-adjoint and $(I-P_{V_h}^G)^2=(I-P_{V_h}^G)$, by Lemma~\ref{l:lowerBound}
\begin{align*}
c(hk)^{2(p+1)}\Big\|\sum_{j=0}^{J-1}\Pi_j(a)u\Big\|_{L^2}^2&\leq\|(I-P_{V_h}^G)\mc{W}_Ju\|_{L^2}^2=\Big\langle \mc{W}_J(I-P_{V_h}^G)\mc{W}_Ju,\sum_{j=0}^{J-1}\Pi_j(a)u\Big\rangle.
\end{align*}
Hence, 
$$c(hk)^{2(p+1)}\Big\|\sum_{j=0}^{J-1}\Pi_j(a)u\Big\|_{L^2}\leq \| \mc{W}_J(I-P_{V_h}^G)\mc{W}_Ju\|_{L^2},$$
and the lemma follows since $\mc{E}_k(2^{-J}a,a)$ is finite dimensional and there is $C>0$ such that 
$$
C^{-1}\|u\|_{H_k^{p+1}}\leq \|u\|_{L^2}\leq C\|u\|_{H_k^{-p-1}},\qquad \text{for all }u\in \mc{E}_k(2^{-J}a,a).
$$
\end{proof}

We can now prove Theorem~\ref{t:pollutionIntro}.

\

\begin{proof}[Proof of Theorem~\ref{t:pollutionIntro}]
By Corollary~\ref{c:polysAreGreat}, there is $k_1\geq 0$ such that for all $k_0>k_1$, $(V_h,P_{V_h}^G)$ satisfies Assumptions~\ref{ass:angleControl} with $\Ang(V_h,P_{V_h}^G,k)\leq C(hk)^{-2(p+1)}$, and $J=\log_2k-\log_2 k_0+1$. Morevover, if $p=0$, then $k_1=0$. By Assumption~\ref{ass:ppG} part (iii), $(P_{V_h}^G,V_h)$ satisfies Assumption~\ref{ass:Goodness} with $t_{\max}=p+1$ and, since a constant is a piecewise polynomial, $V_h$ contains the constants. Therefore, the assumptions of Lemma~\ref{l:pollution} hold with $t_{\max}=p+1$ and $\gamma=2(p+1)$. 
Finally, we need to show that the condition \eqref{e:normalOK} (i.e., the second condition in \eqref{e:normalOk1}) is not required when $p=0$. To do this, observe that there is $k_*>0$ such that $1_{(0,k_*^2k^{-2})}(k^{-2}\Delta_g)=0$ and hence, choosing $k_0>0$ small enough when $p=0$ completes the proof.
\end{proof}

\section{Construction of quasimodes implying pollution}\label{s:quasimodes}

In this section, we present the construction of quasimodes that lead to Theorems~\ref{t:diamond} and \ref{t:qualitative1}.  In fact, we prove results that allow for the passage from quasimodes for the scattering problem i.e. outgoing $u$ such that $(-\hsc^2\Delta-1)u\in L^2_{\comp}(\Omega^+)$ and
$$
\|(-k^{-2}\Delta-1)u\|_{L^2(\Omega^+)}\ll k^{-1}\|\chi u\|_{L^2(\Omega^+)},\qquad u|_{\Gamma}=0,\qquad u\text{ outgoing},
$$
where $\chi\in C_c^\infty$ with $\chi \equiv 1$ in a neighborhood of the convex hull of $\Omega^-$, with a few additional properties to quasimodes usable in Theorem~\ref{t:pollutionIntro}. 

Before proceeding, we recall the following formulas for the inverses of $A_{k,\eta}$, $A_{k,\eta}'$, $B_{k,\rm{reg}}$, and $B_{k,\rm{reg}}'$~\cite[Theorem 2.33]{ChGrLaSp:12} \cite[Lemma 7.4]{GaMaSp:21N}
\begin{equation}
\label{e:BIEinverse}
\begin{aligned}
A_{k}^{-1}&=I-P_{\rm ItD}^{-,\eta_D}(P_{\rm DtN}^+-ik^{-1}\eta_D)\\
(A_{k}')^{-1}&=I-(P_{\rm DtN}^+-ik^{-1}\eta_D)P_{\rm ItD}^{-,\eta_D}\\
B_{k,\rm{reg}}^{-1}&=P_{\rm NtD}^+k^{-1}S_{ik}^{-1}-(I-i\eta_N P_{\rm NtD}^+k^{-1}S_{ik}^{-1})P_{\rm ItD}^{-,\eta_N,S_{ik}}\\
(B_{k,\rm{reg}}')^{-1}&=S_{ik}^{-1}k^{-1}P_{\rm NtD}^+-k^{-1}S_{ik}^{-1}P_{\rm ItD}^{-,\eta_N,S_{ik}}(kS_{ik}-i\eta_N P_{\rm NtD}^+)
\end{aligned}
\end{equation}
Here, we define $P_{\rm DtN}^{\pm}f:=k^{-1}\partial_\nu u_1$, where
$$
(-\Delta-k^2)u_1=0\text{ in }\Omega^+,\quad u_1|_{\Gamma}=f,\qquad u_1\text{ is outgoing ($+$)/ incoming ($-$)},
$$
 $P_{\rm NtD}^{\pm}f:= u_2|_{\Gamma}$, where
$$
(-\Delta-k^2)u_2=0\text{ in }\Omega^+,\quad k^{-1}\partial_\nu u_2|_{\Gamma}=f,\qquad u_2\text{ is outgoing($+$)/ incoming ($-$)},
$$
$P_{\rm ItD}^{\pm,\eta}f:=u_3|_{\Gamma}$, where
$$
(-\Delta-k^2)u_3=0\text{ in }\Omega^-,\quad (k^{-1}\partial_\nu u_3\pm ik^{-1}\eta  u_3)|_{\Gamma}=f,
$$
and $P_{\rm ItD}^{\pm,\eta,S_{ik}}f:=u_4|_{\Gamma}$, where
$$
(-\Delta-k^2)u_4=0\text{ in }\Omega^-,\quad (S_{ik}\partial_\nu u_4\pm i\eta u_4 )|_{\Gamma}=f.
$$

\blue{In this section, we use heavily machinery from semiclassical analysis. Recall from \S\ref{sec:SCA} that the small parameter is usually denoted by $h$ in the semiclassical literature. Here we use the notation $\hbar:=k^{-1}$, motivated by the fact that the semiclassical parameter is often related to Planck's constant, which is written as $2\pi\hsc$ see, e.g., \cite[S1.2]{Zw:12}, \cite[Page 82]{DyZw:19}.
We therefore use the notation $\hbar:=k^{-1}$, writing 
$$
\|\cdot\|_{\mac{s}}:=\|\cdot\|_{H_{\hsc^{-1}}^s}.
$$
}

\subsection{Dynamical Preliminaries}
\label{s:dynamics}

 Let $\Omega\subset \mathbb{R}^d$ be open with smooth, compact, boundary, $\Gamma$. In order to define the relevant dynamics, we recall the notion of the $b$-contangent bundle of $\Omega$. We have $T^*\Omega\rightharpoonup {}^bT^*\Omega$ via the canonical projection map $\pi_b:T^*\Omega\to {}^bT^*\Omega$ given, in local coordinates $(x_1,x')$ where $\partial\Omega=\{x_1=0\}$ and $\Omega=\{x_1>0\}$ by
$$
\pi_b(x_1,x',\xi_1,\xi')=(x_1,x',x_1\xi_1,\xi').
$$
Let $g$ be a metric on $\Omega$ and assume throughout this section that for every point $\rho\in \pi_b(S^*\blue{\overline{\Omega}})\blue{=:\mc{Z}}$, there is exactly one \blue{generalised broken bicharacteristic (GBB)} through $\rho$. This, for instance, is the case of $\partial\Omega$ is nowhere tangent to the geodesics in $\Omega$ to infinite order or $\partial\Omega$ has negative semidefinite second fundamental form with respect to the inward normal. 
We denote the generalized broken bicharacteristic flow on $\mc{Z}$, by
$$
\varphi_t^{\Omega}:\mc{Z}\to \mc{Z}.
$$
It will sometimes be convenient below to denote $x(\varphi_t^\Omega(\rho)):=\pi_{\mathbb{R}^d}(\varphi_t^\Omega(\rho))$, where $\pi_{\mathbb{R}^d}:\mc{Z}\to \mathbb{R}^d$ denote projection to the base.

Let $f:\mc{Z}\to \mathbb{R}$ be a boundary defining function for $\Omega$ and $g_{\Gamma}$ denote the metric induced on $T^*\Gamma$ by $g$. We write
\begin{gather*}
\mc{H}:=\{ (x',\xi')\in T^*\Gamma\,:\, |\xi'|_{g_\Gamma}<1\} , \quad  \mc{E}:=\{ (x',\xi')\in T^*\Gamma\,:\, |\xi'|_{g_\Gamma}>1\}\\
\mc{G}:=\{ (x',\xi')\in T^*\Gamma\,:\, |\xi'|_{g_\Gamma}=1\},\quad  \mc{G}_d:=\{ (x',\xi')\in \mc{G}\,:\, H_{|\xi|_g^2}^2f>0\},\qquad \mc{G}_g:=\mc{G}\setminus \mc{G}_d.
\end{gather*}

We then define the incoming ($\Gamma_-$) and outgoing ($\Gamma_+$) sets,
\begin{gather*}
\Gamma_\pm:=\{ \rho \in \mc{Z}\,:\, \limsup_{t\to \infty} |x(\varphi_{\blue{\mp} t}(\rho))|<\infty \},
\end{gather*}
and the trapped set
$$
K:=\Gamma_+\cap \Gamma_-.
$$
We say $(\Omega,g)$ is trapping if $K\neq \emptyset$. We denote by $\operatorname{ch}(U)$, the convex hull of a set $U$.

\paragraph{Basic properties of the trapped set}

\begin{lemma}\label{l:13point1}
Suppose that $\Omega^-\Subset \mathbb{R}^d$ has connected complement and $\Omega=\Omega^+:=\mathbb{R}^d\setminus \overline{\Omega}$ and $\supp (g-I)$ is compact. The trapped set is closed and satisfies
\begin{equation}
\label{e:trappedInclusion}
K\subset \operatorname{ch}\big(\overline{\Omega^-}\cup \supp (g-I)\big)\blue{\cap \overline{\Omega^+}}.
\end{equation}
\end{lemma}
\blue{
In both Lemma~\ref{l:13point1} and the rest of this section, we sometimes identify $U\subset \overline{\Omega^+}$ with subsets of $\cZ$ using the natural projection map. For instance, in~\eqref{e:trappedInclusion} one should, strictly speaking, have $\pi^{-1}(\operatorname{ch}(\overline{\Omega^-}\cup \supp (g-I))\blue{\cap \overline{\Omega^+}})$ instead of $\operatorname{ch}(\overline{\Omega^-}\cup \supp (g-I))\blue{\cap \overline{\Omega^+}}$, where $\pi:\cZ\to \overline{\Omega^+}$ denotes the natural projection. 
}

\begin{proof}[Proof of Lemma \ref{l:13point1}]
We start by proving the inclusion of $K$. Suppose that $\rho\in\mc{Z}$ with $x(\rho)\notin \operatorname{ch}(\overline{\Omega^-}\cup \supp (g-I))$. Then, for 
$$
|t|\leq d\big(x(\rho), \operatorname{ch}(\overline{\Omega^-}\cup \supp (g-I)\big)
$$
we have
$$
x(\varphi^{\Omega^+}_{t})= x(\rho)\pm t\xi(\rho).
$$
Now, by convexity, for one choice of $\pm$, 
$$
\{ x(\rho)\pm [0,\infty)\xi(\rho)\}\cap \operatorname{ch}(\overline{\Omega^-}\cup \supp (g-I))=\emptyset
$$
In particular,  $\rho$ not trapped either forward or backward in time.

Since $K$ is invariant under $\varphi_t^{\Omega^+}$ and  $K\subset \operatorname{ch}(\overline{\Omega^-}\cup \supp (g-I))$, there is $C>0$ such that for any $\{\rho_n\}_{n=1}^\infty\in K$ with $\rho_n\to \rho$, and any $t\in\mathbb{R}$,  $|x(\varphi_t^{\Omega^+}(\rho_n))|\leq C$.
In particular, since $\varphi_t^{\Omega^+}$ is continuous, $|x(\varphi_t^{\Omega^+}(\rho))|\leq C$. Since $t$ is arbitrary this implies $\rho\in K$ and hence $K$ is closed.
\end{proof}

\begin{lemma}
\label{l:incoming}
Suppose that $\Omega^-\Subset \mathbb{R}^d$ has connected complement and $\Omega=\Omega^+:=\mathbb{R}^d\setminus \overline{\Omega}$ and $\supp (g-I)$ is compact and $\Omega^+$ is trapping. Then $\Gamma_-\setminus K\neq \emptyset$.
\end{lemma}
\begin{proof}
Since $|\xi|_g^2$ is homogeneous degree $2$, $\Gamma_-$ is nonempty if and only if $\Gamma_+$ is nonempty. Therefore, it is enough to show that $\Gamma_+\cup \Gamma_- \blue{\setminus K}\neq \emptyset$. 

Let $\{q_n\}_{n=1}^\infty\subset \mc{Z}\setminus K$ with $d(q_n,K)\to0$. Such a sequence exists since $K$ is closed and $K$ is not equal to $\mc{Z}$. Let 
$$
C_1>\operatorname{diam}(\operatorname{ch}(\overline{\Omega^-}\cup \supp (g-I))).
$$ 

Without loss of generality, we may assume $|x(q_n)|\leq C_1$. Then, since the $\varphi_t^{\Omega^+}$ is continuous and for all $q'\in K$, $t\in \mathbb{R}$, $|x(\varphi_t^{\Omega^+}(q') )|\leq C_1$, for any $T>0$, there is $n$ large enough such that $|x(\varphi^+_{t}(q_n))|\leq 2C_1$ for $|t|\leq T$. Extracting subsequences, we may assume that 
$$
|x(\varphi_t^{\Omega^+}(q_n))|\leq 2C_1,\quad 0\leq t\leq n.
$$ 
Now, since $q_n\notin K$, there is $ s_n>n$ and $\mu\in\{1,-1\}$ such that $2C_1<|x(\varphi^+_{\mu s_n}(q_n))|<3C_1$.  Then, since $|x(q_n)|\leq C_1$ we have $x(\varphi^+_{\mu t}(q_n))\leq 3C_1$ for $0\leq t\leq s_n$. Taking a subsequences, we may assume that $\varphi^+_{s_n}(q_n)\to q$ with 
$$
2C_1\leq |x(q)|\leq3C_1
$$
Now, fix $T>0$ then, 
$$
d(\varphi^{\Omega^+}_{-\mu T}(q),\varphi^{\Omega^+}_{\mu (s_n-T)}(q_n))\to 0,
$$
and for $n>T$
$$
|x(\varphi^{\Omega^+}_{\mu(s_n-T)}(q_n)))|\leq 3C_1
$$
In particular, 
$$
|x(\varphi^{\Omega^+}_{-\mu T}(q))|\leq 3C_1
$$
for all $T\geq 0$. In particular, $q\in (\Gamma_-\cup \Gamma_+)\setminus K$.

\end{proof}

\begin{lemma}
\label{l:incomingIsNotNormal}
Suppose that $\Omega^-\Subset \mathbb{R}^d$ has connected complement and smooth boundary, $\Omega=\Omega^+:=\mathbb{R}^d\setminus \overline{\Omega}$, $\supp (g-I)$ is compact, and $q\in( \Gamma_-\setminus K)\cap T^*\Gamma$. Then $\xi'(q)\neq 0$. 
\end{lemma}
\begin{proof}
Suppose that $q\in (\Gamma_-\setminus K)\cap T^*\Gamma$ and  $\xi'(q)=0$.  Then $x(\varphi^{\Omega^+}_{t}(q))=x(\varphi^{\blue{\Omega^+}}_{-t}(q))$, and hence, since $q\in \Gamma_-\setminus K$, $|x(\varphi^{\blue{\Omega^+}}_{-t}(q))|\to \infty$. Therefore, $|x(\varphi_t^{\Omega^+}(q))|\to \infty$. In particular $q\notin \Gamma_-$. 
\end{proof}

\paragraph{Distinguished neighborhoods of a GBB segment and perturbations near the boundary}
Our construction of approximate solutions of Helmholtz problems for given data with precise knowledge of the approximate solution's wavefront set uses perturbations of the Hamiltonian $p=|\xi|^2-1$ near the boundary and neighborhoods $U$ of GBB segments such that i) the operator is unchanged in $U$ and ii) every trajectory either escapes backward in time, enters the interior of the domain outside of $U$, or hits a distinguished subset of $T^*\Gamma$.  

For our next lemma, we need some additional notation. First, for $U,V\subset \mc{H}$, $\mu\in \{1,-1\}$, and $q\in\mc{Z}$ define
$$
T^\mu_{U}(q):=\inf \{ t> 0:\varphi_{\mu t}^{\Omega}(q)\in U\},\qquad T^\mu_{U<V}(q):=\inf\{0<t<T^\mu_{V}(q)\,:\, \varphi_{\mu t}^\Omega(q)\in U\}.
$$
Next, we write $a\in S_{\mathsf{T}}^{\comp}(\Omega)$ if in a Fermi coordinate $(x_1,x',\xi_1,\xi')$ neighborhood of the boundary, 
$$
a=\tilde{a}(x_1,x',\xi'),\qquad \tilde{a}\in C_c^\infty(\mathbb{R}\times T^*\Gamma).
$$

\blue{Our next lemma shows that if a GBB starting at $q$ hits the hyperbolic set of the boundary at some time $T_q$ before it hits some specific closed subset $\mc{B}$ of the hyperbolic set, then one can define a billiard map from a neighborhood, $U$, of $q$ to a neighborhood, $V$, of $\varphi_{T_q}^\Omega(q)$ such that i) the length of the every trajectory encoded by the billiard map is close to $T_q$ and ii) each such trajectory stays away from $\mc{B}$ and  $U$. Furthermore, one can modify the operator away from $\Sigma_q$ (the trajectory between $q$ and $\varphi_{T_q}^{\Omega}(q)$) so that every possible GBB does one of the following i) escapes a compact set, ii) enters the interior of $\Omega$ away from $\Sigma_q$, or iii) hits the hyperbolic set in a neighborhood of $\varphi_{T_q}^{\Omega}(q)$.}
\begin{lemma}
\label{l:modifiedDynamics}
Let $\mu\in \{ 1,-1\}$, $\mc{O}\subset \mc{H}$ open and $\mc{B}\subset \mc{H}$ closed. Suppose that $q\in \mc{Z}$ and \blue{there is $T_q\in (0 ,T_{\mc{B}\cup \{q\}}(q))$} such that $\blue{\varphi^{\Omega}_{T_q}(q)}\in\mc{O}$. Then there is $\delta_0>0$, a closed neighborhood, $V\subset  T^*\Gamma\setminus (S^*\Gamma \cup \mc{B}\cup\{q\})$ of $\varphi^{\Omega}_{T_q}(q)$ and for all $0<\delta<\delta_0$ a closed neighborhood $U\subset  \mc{Z}$ of $q$, and a continuous function $T:U\to [T_q-\delta,T_q+\delta]$ such that 
$$
\varphi_{T(q')}^{\Omega}(q')\in V,\quad \inf
\big\{ d(\varphi_t^{\Omega}(q'),U\cup \mc{B}): \delta<t<T(q'), \varphi^{\blue{\Omega}}_t(q')\in T^*\Gamma,q'\in U\big\}>0.
$$

Moreover, for any $R>0$ and any open $V_-\subset V^o$ with $\blue{\varphi_{T_q}^{\Omega}(q)}\in V_-$, there are $T>0$ and $a\in S_{\mathsf{T}}^{\comp}(\Omega)$, $0\leq a\leq \tfrac{1}{2}|\xi'|_{g_{\Gamma}}^2 $, $a|_{\Gamma}=0$ with 
$$
\supp a\cap \Sigma_q=\emptyset,\qquad \Sigma_q:=\bigcup_{0\leq t\leq T_q}\varphi_{\mu t}^{\Omega}(q)
$$
 such that, putting $\mc{Z}_a:={}^b\pi(\{ |\xi|_g^2-a-1=0\})$,  for any GBB, $\gamma$, for $|\xi|_g^2-a-1$, there is $0\leq t\leq T$ such that
\begin{equation}
\label{e:dynamics}
\gamma(-\mu t)\in V_-\cup \Big( \{|x|>R\}\cup\{d(x,\Gamma)>0\}\Big)\setminus \Sigma_q.
\end{equation}
\end{lemma}
\begin{proof}
The existence of $V$ and $U$ as described follows from the facts that 1) $\varphi_t^{\Omega}$ is continuous, 2) since $\mc{O}\subset \mc{H}$, $\partial_t\varphi_t^{\Omega}|_{T_q-0}$ is transverse to $\Gamma$, and 3) $T_q<T_{\mc{B}\cup \{q\}}$ and $\mc{B}\cup\{q\}$ is closed.

For use below, let $t_0,c_0>0$  small enough that 
$$
d(\varphi_{-\mu t_0}^\Omega(q),q)>c_0.
$$
To construct the desired $a$, we work in Fermi normal coordinates $[0,\e)_{x_1}\times \Gamma_{x'}$ near $\Gamma$. In these coordinates, 
$$
|\xi|_g^2-1= \xi_1^2-r(x,\xi'),\qquad r(0,x',\xi')=1-|\xi'|_{g_\Gamma}^2.
$$

Let $\chi\in C_c^\infty((-\frac{1}{2},\frac{1}{2});[0,1])$ with $0\notin \supp (1-\chi)$. Let also $\psi_\e\in C_c^\infty(T^*\Gamma;[0,1])$ with $\supp (1-\psi_\e) \cap \Sigma_q=\emptyset$, $\supp \psi_\e\subset \{d((x',\xi'),\Sigma_q)<\e\}.$  Then, for some $0<\e<c_0/2$ to be determined, define
$$
a:= \e^{-1}x_1|\xi'|_{g_{\Gamma}}^2\chi(\e^{-1} x_1)\chi(\e^{-1}(r(0,x',\xi')))(1-\psi_\e(x',\xi')). 
$$

Without loss of generality, we may assume that $\Gamma\subset \{|x|\leq R\}$. Therefore, to prove~\eqref{e:dynamics}, suppose that there is no such $T$. Then for all $n$ there is $\gamma_n$ such that 
\begin{equation}
\label{e:contradictionTime}
\gamma_n(-\mu[0,n])\subset  [V_-]^c\cap \Big( T^*\Gamma\cup \Sigma_q\Big).
\end{equation}
Suppose first that $\gamma_n(0)\in \Sigma_q\cap [V^o]^c$. Then, since $\supp a\cap \Sigma_q=\emptyset$, there is $T_q+\delta>S>\inf\{0\leq t\,:\, \varphi_{-\mu t}^\Omega(\gamma_n(0))\notin \Sigma_q\}=:S_0$ such that $\gamma_n(-\mu t)=\varphi_{-\mu t}^{\Omega}(\gamma_n(0))$ for $0\leq t\leq S$. In particular, this implies that 
$$
\varphi_{-\mu t}^{\Omega}(\gamma_n(0))\notin \Sigma_q,\qquad S_0<t<S. 
$$
Since $\varphi_{T_q}^{\Omega}(q)\in \mc{H}$, 
$$
\inf\{ t>S_0\,:\, \gamma_n(-\mu t)\in \Sigma_q\}<\inf\{ t>S_0\,:\, \gamma_n(-\mu t)\notin T^*\Gamma\}.
$$
In particular, for $\e>0$ small enough (depending only on $\|\partial_{x_1}r|_{x_1=0=r=0}\|_{L^\infty}$),  $a|_{x_1=0}=0$, this implies
$$
\cup_{S_0<t<n}\gamma_n(-\mu t)\subset \{ r(0,x',\xi')=\partial_{x_n}r+\partial_{x_n}a=0\}\subset \{ d((x',\xi'),\Sigma_q)<\e\},
$$
and $\gamma_n(-\mu t)=\varphi_{-\mu t}^{\Omega}(\gamma_n(0))$ for $0\leq t\leq n$.  Now, let $t=S_0+t_0$. Then,
$$
c_0<d(\varphi_{-\mu t_0}^{\Omega}(q),\Sigma_q)=d(\varphi_{-\mu t}^{\Omega}(\gamma_n(0)),\Sigma_q)<\e<c_0/2,
$$
which is a contradiction for $n$ large enough.

Next, suppose that $\gamma_n(0)\in (T^*\Gamma\cap [V_-]^c)\setminus \Sigma_q$ and there is $0\leq t\leq n$ such that $\gamma_n(-\mu t)\in \Sigma_q$. Then, since $V\subset \mc{H}$, $q\notin V$, and $\Sigma_q\cap \supp a=\emptyset$, there is $0<s<t$ such that $\gamma_n(-\mu s)\notin \Sigma_q\cup T^*\Gamma $, which contradicts~\eqref{e:contradictionTime}.

Finally, we consider the case $\gamma_n(-\mu t)\subset (T^*\Gamma\cap [V_-]^c\cap \Sigma_q^c)$, $0\leq t\leq n$. For $\e>0$ small enough (depending only on $\|\partial_{x_1}r|_{x_1=0=r=0}\|_{L^\infty}$), this implies that $\gamma_n(-\mu t)\in \{ r(0,x',\xi')=\partial_{x_n}r+\partial_{x_n}a=0\}\subset \{ d((x',\xi'),\Sigma_q)<\e\}$ and, since $a|_{x_1=0}=0$, $\gamma_n(-\mu t)=\varphi_{-\mu t}^{\Omega}(\gamma_n(0))$. The continuity of $\varphi_t^{\Omega}$ and the fact that $\Sigma_q$ is a flow-line of length $T_q$  implies that for any $\beta>0$, there are $\e>0$ small enough (depending only on the continuity properties of $\varphi_t^\Omega$) and $0<t<T_q+\beta$ such that $d(\varphi_{-\mu t}^{\Omega}(\gamma_n(0)), q)<\beta$ and hence, choosing $\beta>0$ small enough (again depending only on the continuity properties of $\varphi_t^\Omega$), 
$$
d(\varphi_{-\mu (t+t_0)}^\Omega(\gamma_n(0)),\varphi_{-\mu t_0}^\Omega(q))<c_0/2,
$$
which implies
$$
d(\varphi_{-\mu (t+t_0)}^\Omega(\gamma_n(0)),\Sigma_q)>c_0/2>\e,
$$
a contradiction. Thus, we have proved~\eqref{e:dynamics}.
\end{proof}

\begin{lemma}
\label{l:modifiedDynamics2}
Let $\Omega^-\Subset \mathbb{R}^d$ open with smooth boundary and connected complement $\Omega:=\mathbb{R}^d\setminus \overline{\Omega^-}$ and suppose $q\in\mathcal{Z}\setminus \Gamma_-$. Then for any $R>0$, there are $T>0$  and $a\in S_{\mathsf{T}}^{\comp}(\Omega)$ with $a\geq 0$, $a|_{\Gamma}=0$ such that 
$$
\gamma_q\cap \supp a =\emptyset,\qquad \gamma_q:=\bigcup_{t\geq 0}\varphi_{t}^{\Omega}(q),
$$
and, putting $\mc{Z}_a:={}^b\pi(\{ |\xi|_g^2-a-1=0\})$,  for any GBB, $\gamma$, for $|\xi|_g^2-a-1$, there is $0\leq t\leq T$ such that
\begin{equation}
\label{e:dynamics2}
\gamma(- t)\in  \Big( \{|x|>R\}\cup\{d(x,\Gamma)>0\}\Big)\setminus \gamma_q.
\end{equation}
\end{lemma}
\begin{proof}
The proof follows the same lines as that of Lemma~\ref{l:modifiedDynamics} but is simpler because we need not consider any special subsets of $T^*\Gamma$. \end{proof}

%
%

\subsection{Semiclassical $b$ and tangential-pseudodifferential operators}

\paragraph{$b$-pseudodifferential operators}

The class of $b$-pseudodifferential operators that we work with is the natural class of operators quantizing differential operators that are tangential to the boundary of $\Omega_+$. Away from $\partial\Omega_+$ they are pseudodifferential operators in the sense of \S\ref{sec:SCA}, but near $\partial\Omega_+$ they have a different form. In particular, in coordinates $(x_1,x')$ with $\partial\Omega_+=\{x_1=0\}$, their symbols are functions on the $b$-cotangent bundle, ${}^bT^*\Omega_+$, whose sections are of the form
$$
\sigma \frac{d x_1}{x_1}+\xi' dx';
$$
\blue{i.e,. $\sigma$ and $\xi'$ are coordinates on the $b$-cotangent bundle.}
Notice that ${}^bT^*\Omega_+$ is the dual to sections of $T^*\Omega_+$ that are tangent to $\partial\Omega_+$. We also write ${}^b\overline{T^*\Omega_+}$ for the fiber radially compactified $b$-contangent bundle; i.e., ${}^bT^*\Omega_+$ with the sphere at infinity in $(\sigma, \xi')$ attached.

In coordinates, $b$-pseudodifferential operators are of the form
$$
\Op_{b}(a)(u)(x)=\frac{1}{(2\pi \hsc)^d}\int e^{\frac{i}{\hsc}((x_1-y_1)\xi_1+(x'-y'),\xi')}\phi(x_1/y_1)a(x_1,x',x_1\xi_1,\xi')u(y)dyd\xi,
$$
where $\phi\in C_c^\infty(1/2,2)$ with $\phi\equiv 1$ near $1$  and for some $m$
$$
|D_{x}^\alpha D_\sigma^j D_{\xi'}^\beta a(x_1,x',\sigma,\xi')|\leq C_{j\alpha\beta}\langle (\sigma,\xi')\rangle^{m-j-|\beta|}.
$$
In this case, we write $\Op_{b}(a)\in \Psi_b^m(\Omega_+)$ and $a\in S^m({}^bT^*\Omega_+)$. When $m=0$ we write $S({}^bT^*\Omega_+)$ and $\Psi_b(\Omega_+)$ respectively. We also write $\Psi_b^{-\infty}=\cap_m \Psi_b^m$. 

 The class comes equipped with principal symbol map ${}^b\sigma: {}^b\Psi^m(\Omega_+)\to S^m({}^bT^*\Omega_+)/ \hsc {}^bS^{m-1}(T^*\Omega_+)$ such that if $A\in\Psi_b(\Omega_+)$ and $\sigma(A)=0$ then $A\in \hsc\Psi_b^{m-1}(\Omega_+)$. We now introduce two important sets for $b$-pseudodifferential operators. For $A\in \Psi_b^m(\Omega_+)$ and $q\in {}^b\overline{T^*\Omega_+}$, we say $q\in {}^b\Ell(A)$ if there is a neighbourhood, $U$ of $q$ such that
$$
|\sigma(A)(q')|\langle (\sigma,\xi')\rangle^{-m}>c>0,\qquad q'\in U\cap {}^bT^*\Omega_+.
$$ 
Next, we say $q\notin {}^b\WF(A)$ if there is $E\in {}^b\Psi(\Omega_+)$ with $q\in {}^b\Ell(E)$ such that 
$$
EA\in \hsc^{\infty} \Psi_b^{-\infty}.
$$ 
For a more complete treatment of these operators, we refer the reader to~\cite[Appendix A]{HiVa:18} and the references therein.

\paragraph{Tangential pseudodifferential operators}

In addition to the class of $b$-pseudodifferential operators, it is useful to have a subclass consisting of tangential pseudodifferential operators. For this, let $g$ be a Riemannian metric on $\Omega$ and $(x_1,x')\in \mathbb{R}\times \Gamma$ be Fermi normal coordinates near $\Gamma$. That is,
$$
\Omega=\{x_1>0\},\, x_1=d(x,\partial\Omega).
$$
We then use the canonical coordinates $(x_1,x',\xi_1,\xi')$ on $T^*\Omega$ and define
$$
S_{\mathsf{T}}^{m}:=\{ a(x,\xi')\,:\, |\partial_{x}^\alpha\partial_{\xi'}^\beta a(x,\xi')|\leq C_{\alpha\beta}\langle \xi'\rangle^{m-|\alpha|},
$$
with $S^{\comp}_{\mathsf{T}}$, $S^{\infty}_{\mathsf{T}}$, and $S^{-\infty}_{\mathsf{T}}$ defined as for the usual symbol classes.
We then set
$$
 \Psit^{m}(\Omega):=\{ \Op(a)\,:\, a\in S_{\mathsf{T}}^{m}(\Omega)\},
$$ 
again, with  $\Psi^{\comp}_{\mathsf{T}}$, $\Psi^{\infty}_{\mathsf{T}}$, and $\Psi^{-\infty}_{\mathsf{T}}$. 

\paragraph{Wavefront sets}

We have already defined wavefront sets of pseudodifferential operators, but it is useful, in addition, to have the notion of a wavefront set of certain classes of distributions.
\begin{definition}
Let $s\in \mathbb{R}$ and $M$ be a smooth manifold (possibly with boundary). We say that $u\in \macl{s}(M)$ is $\hsc$-tempered if for all $\chi\in C_c^\infty(\overline{M})$, there are $N>0$ and $C>0$ such that 
$$
\|\chi u\|_{\mac{s}}\leq C\hsc^{-N}.
$$
\end{definition}

We can now define the wavefront set and $b$-wavefront set of a distribution.
\begin{definition}
Let $M$ be a smooth manifold without boundary and $s,t\in\mathbb{R}$ and suppose that $u\in \macl{s}(M)$ is $\hsc$-tempered. Then, we say that $(x_0,\xi_0)\in \overline{T^*M}$ is not in $\WF^{t}_{\mac{s}}(u)$ if there are $C>0$ and $A\in \Psi^0(M)$ with $A$ elliptic at $(x_0,\xi_0)$ such that 
$$
\|Au\|_{\mac{s}}\leq C\hsc^t.
$$
We write $\WF_{\mac{s}}(u)=\cup_t\WF_{\mac{s}}^t(u)$, $\WF^t(u)=\cup_s\WF_{\mac{s}}^t(u)$, and $\WF(u)=\cup_s\WF_{\mac{s}}(u)$. 
\end{definition}

\begin{definition}
Let $\Omega$ be a smooth manifold with boundary and $s,t\in\mathbb{R}$ and suppose that $u\in \macl{s}(M)$ is $\hsc$-tempered. Then, we say that $(x_0,\xi_0)\in \overline{{}^bT^*M}$ is not in ${}^b\WF^{t}_{\mac{s}}(u)$ if there are $C>0$ and $A\in {}^b\Psi^0(M)$ with $A$ elliptic at $(x_0,\xi_0)$ such that 
$$
\|Au\|_{\mac{s}}\leq C\hsc^t.
$$
We write ${}^b\WF_{\mac{s}}(u)=\cup_t{}^b\WF_{\mac{s}}^t(u)$. Since $\mac{1}$ will be our default space on a manifold with boundary, we sometimes abuse notation and write ${}^b\WF(u)$ for ${}^b\WF_{\mac{1}}(u)$. 
\end{definition}

We record the following lemma which is a consequence of~\cite[Theorem 18.3.32]{Ho:07} and allows us to use tangential pseudodifferential operators as test operators rather than the full $b$ calculus.
\begin{lemma}
Let $t,s\in\mathbb{R}$ and suppose that $u\in \macl{s}$ is $h$ tempered and $b\in S_{\mathsf{T}}^0(M)$. Then,
$$
{}^b\WF_{\mac{s}}^t(u)|_{\partial M}\subset {}^b\WF^t_{\mac{s}}(\Op(b)u)\cup \{(x',\xi')\in \overline{T^*\Gamma}\,:\, \sigma(b)(0,x',\xi')=0\}.
$$
\end{lemma}

\paragraph{Defect measures}

In the most general settings below, we prove qualitative pollution bounds; i.e. we show that there is pollution, but do not determine the rate. For this, it is convenient to use the notion of defect measures. 
\begin{definition}
Suppose that $\hsc_n\to 0^+$ and $\{u_{\hsc_n}\}_{n=1}^\infty$ satisfies $\sup_n\|\chi u_{\hsc_n}\|_{L^2(\Omega^+)}<\infty$ for all $\chi\in C_c^\infty(\overline{\Omega}_+)$. Then there is a subsequence $\hsc_{n_k}\to 0$ and a positive radon measure $\mu$ on $T^*\mathbb{R}^d$ such that for any $a\in \Psi^{\comp}(\mathbb{R}^d)$ we have
$$
\langle \Op(a)u_{\hsc_{n_k}}1_{\Omega^+},u_{\hsc_{n_k}}1_{\Omega^+}\rangle_{L^2(\mathbb{R}^d)}\to \int \sigma(a)d\mu. 
$$
\end{definition}
We record the following fact about defect measures for Helmholtz solutions. For its proof, we refer the reader to e.g.~\cite[Lemma 2.12]{GMS},~\cite[Lemma 4.2,4.8]{GaSpWu:20}, and~\cite{Bu:02}.
\begin{lemma}
\label{l:invariance}
Suppose that $\chi\in C_c^\infty(\overline{\Omega}_+)$,  $f\in L^2(\Omega^+)$, with $\|f\|_{L^2}\leq C$,
$$
(-\hsc^2\Delta-1)u=\hsc \chi f\text{ in }\Omega^+,\qquad u|_{\Gamma}=0,\qquad u\text{ is outgoing},
$$
and $u$ has defect measure $\mu$. Then, for any $\psi\in C_c^\infty(\overline{\Omega})$, $\mu(\psi^2)=\lim_{h\to 0^+}\|\psi u\|_{L^2}$, $\supp {}^b\pi_*\mu\subset \mc{Z}$, and   ${}^b\pi_*\mu$ is invariant under $\varphi_t^{\Omega^+}$ away from ${}^b\WF(f)$ and $\mc{G}^\infty$ and, for $R$ large enough,
$$
\mu(\mc{I}_-(R)\setminus \bigcup_{t>0}\varphi_t^{\Omega^+}({}^b\WF(f))=0,
$$
where 
$$
\mc{I}_-(R):=\{ (x,\xi)\in T^*\Omega^+\,:\, |x|>R, \langle \tfrac{x}{|x|},\xi\rangle \leq 0\}.
$$
\end{lemma}

\subsection{Propagation of singularities and solvability with boundary damping}
Let $\Omega\subset \mathbb{R}^d$ be an open  with smooth, compact boundary $\Gamma$,  $g$ be a Riemannian metric on $\Omega$ with $g_{ij}(x)\equiv \delta_{ij}$ for $|x|\gg 1$ and $\nu$ be the outward unit normal to $\Omega$ with respect to $g$. Let $Q\in \Psi^{-1}(\Gamma)$, and $W\in \Psi^{\comp}(\Omega)$ and $A\in \Psit^{\comp}(\Omega)$, 
We study the problem
\begin{equation}
\label{e:damped}
\begin{cases}
\mathsf{P}u:=(-\hsc^2\Delta_g-iW-1-A)u=\hsc f_1&\text{in }\Omega\\
\mathsf{B}u:=Q\hsc D_{\nu}u-u=f_2&\text{on }\Gamma\\
u\text{ is outgoing}.
\end{cases}
\end{equation}
We prove the following propagation of singularities theorem in Appendix~\ref{a:propagate}.
\begin{theorem}[Propagation of singularities]
\label{t:basicPropagate}
Let $Q\in \Psi^{-1}(\Gamma)$ with $\sigma(Q)\geq 0$, $A\in \Psit^{\comp}$ with $A|_{\Gamma}=0$ and $\sigma(A)\in\mathbb{R}$, $W\in \Psi^{\comp}(\Omega)$ with $\sigma(W)\geq 0$, $N>0$, $s\in\mathbb{R}$,  and $u\in \mac{1}$ be $\hsc$-tempered. 
Then, with $S^*_a\Omega:=\{(x,\xi)\,:\, |\xi|_g^2-a=1\}$, $\mc{Z}:=\pi_b(S_a^*\Omega)$,
$$
{}^b\WF_{\mac{2}}^s(u)\subset \big(\mc{Z}\cup {}^b\WF_{L^2}^s(\mathsf{P}u)\cup \WF^{s-\frac{1}{2}}_{\mac{3/2}}(\mathsf{B}u)\big)\setminus\{\sigma(W)>0\}
$$
and for all $\beta>0$, 
$$
{}^b\WF_{\mac{1}}^s(u)\cap (\mc{Z}\setminus \{ \sigma(W)>0\})
$$
is a union of maximally extended backward GBBs for the Hamiltonian $|\xi|_g^2-a-1$ in 
$$
\mc{Z}\setminus\Big( {}^b\WF_{L^2}^{s+1}(\mathsf{P}u)\cup \big(\WF_{L^2}^s(\mathsf{B}u)\cap (\mc{G}_d\cup \mc{H})\big)\cup \big(\WF_{L^2}^{s+\beta}(\mathsf{B}u)\cap \mc{G}_g\big)\Big).
$$
\end{theorem}

\bre\label{r:propagation}
The differences between Theorem \ref{t:basicPropagate} and existing results in the literature are that Theorem \ref{t:basicPropagate}  (i) covers semiclassical boundary damping (via $Q$) and (ii) quantifies the dependence of the microlocalised $H_{k}^{1}$ norm of $u$ on the microlocalised $L^2$ norm of $\mathsf{B}u$ along the GBBs. Surprisingly, even for the Dirichlet problem, (ii) appears not to be in the literature.
The most subtle part of the analysis is in the glancing region, where we follow the proof in~\cite[Chapter 24]{Ho:07}. 
\ere

We also need the following lemma, which is also proved in Appendix~\ref{a:propagate}.
\begin{lemma}
\label{l:DtN}
Let $V\Subset \mc{H}$. Then there is $\Lambda^{-1}\in \Psi^{-1}(\Gamma)$ with $\sigma(\Lambda^{-1})\geq c\langle \xi
\rangle^{-1}>0$ such that for all $q\in V$, $s\in \mathbb{R}$, and all $\hsc$-tempered $u\in \mac{1}$ with
$$
q\notin {}^b\WF^{s+1}_{L^2}(\mathsf{P}u)\cap \WF^{s}(\Lambda^{-1}\hsc D_{\nu}u-u),
$$ 
 there is $\e>0$ such that 
$$
\bigcup_{0< s\leq \e}\varphi_{s}(q)\notin \WF_{\mac{1}}^s(u).
$$
 \end{lemma}

Finally, we require a lemma that gives conditions under which the problem~\eqref{e:damped} is uniquely solvable.
\begin{lemma}
\label{l:dampedResolve}
Let $\Omega\subset \mathbb{R}^d$ open with smooth boundary, $\Gamma\Subset\mathbb{R}^d$ and $g$ a Riemannian metric on $\Omega$ with $\supp (g-I)\Subset \mathbb{R}^d$.
Let  $\hsc_0>0$, $Q\in \Psi^{-1}(\Gamma)$ with $\sigma(Q)\geq 0$, $A\in \Psit^{\comp}$ with $A|_{\Gamma}=0$, $\sigma(A)\in \mathbb{R}$, $W\in\Psi^{\comp}(\Omega)$ with $\sigma(W)\geq 0$. Suppose that for all $R>0$ there are $\delta>0$, $T>0$ such that for all $\rho\in \mc{Z}\cap( B(0,R))$, either there is $0\leq t\leq T$ such that 
\begin{equation}
\label{e:backGood1}
\varphi_{-t}^{\Omega}(\rho)\in \{\sigma(W)>0\}\cup\big( \mc{H}\cap \{\sigma(Q)>0\}\big)\cup T^*\mathbb{R}^d\setminus B(0,R)
\end{equation}
or
\begin{equation}
\label{e:backGood2}
\bigcup_{t-\delta}^{t}\varphi_{s}^{\Omega}(\rho)\subset \mc{G} \cap\{\sigma(Q)>0\}.
\end{equation}
Then for every $R>0$ and $\chi\in C_c^\infty(\overline{\Omega})$, there is $C>0$ such that for all $f_1\in L^2(\Omega)$, $\supp f_1\subset B(0,R)$,  $f_2\in \mac{3/2}(\Gamma)$, the solution, $u$ to~\eqref{e:damped} exists, is unique and for all $0<\hsc<\hsc_0$, 
$$
\|\chi u\|_{\mac{2}(\Omega)}+\|u\|_{\mac{3/2}(\Gamma)}+\|\hsc \partial_\nu u\|_{\mac{1/2}(\Gamma)}\leq C(\|f_1\|_{L^2(\Omega)}+\|f_2\|_{\mac{3/2}(\Gamma)}).
$$
\end{lemma}
\begin{proof}
This lemma follows directly from the analysis in~\cite[Lemma 3.2]{GMS}. 
\end{proof}

With these lemmas in hand, we are able to construct approximate solutions to~\eqref{e:damped} that have prescribed wavefront set properties.
\begin{lemma}[Construction of approximate solutions to \eqref{e:damped}]
\label{l:approxSolve}
Let $\Omega\subset \mathbb{R}^d$ open with smooth boundary, $\Gamma\Subset\mathbb{R}^d$ and $g$ a Riemannian metric on $\Omega$ with $\supp (g-I)\Subset \mathbb{R}^d$.
Let $V\Subset \mc{H}$, $\hsc_0>0$, $Q\in \Psi^{-1}(\Gamma)$ with $\sigma(Q)\geq 0$, $U\subset \mc{Z}$ and suppose that for any $R>0$ and $\rho\in U$ there is $0\leq t<\infty$ such that 
$$
\varphi_{t}^{\Omega}(\rho)\in (\WF(Q-\Lambda^{-1}))^c\cup \{|x|>R\}
$$
Then for all  $R>0$, there is $C>0$ such that for all $\hsc$-tempered,$f_1\in L^2(\Omega)$, $\supp f_1\subset B(0,R)$,  $f_2\in \mac{3/2}(\Gamma)$, with 
$$
({}^b\WF(f_1)\cup \WF(f_2))\cap \mc{Z}\subset U,
$$ 
there is $u\in H_{\hsc ,\loc}^2(\Omega_+)$ such that, $u$ satisfies
\begin{equation}
\label{e:damped2}
\begin{cases}
(-\hsc^2\Delta_g-1)u=\hsc f_1+O(\hsc^{\infty})_{L^2_{\comp}}&\text{in }\Omega\\
Q\hsc D_{\nu}u-u=f_2&\text{on }\Gamma\\
u&\text{ is outgoing},
\end{cases}
\end{equation}
\begin{equation}
\label{e:estU}
\|u\|_{\mac{2}(\Omega)}+\|u\|_{\mac{3/2}(\Gamma)}+\|\hsc \partial_\nu u\|_{\mac{1/2}(\Gamma)}\leq C(\|f_1\|_{L^2(\Omega)}+\|f_2\|_{\mac{3/2}(\Gamma)}),
\end{equation}
and
$$
{}^b\WF(u)\subset {}^b\WF(f_1)\cup \WF(f_2)\cup \bigcup_{q\in {}^b\WF(f_1)\cup \WF(f_2)}\overline{\bigcup_{0\leq t<T_{Q}(q) }\varphi_{t}^{\Omega}(q)},
$$
where 
$$
T_Q(q):=\inf\{t\geq 0\,:\, \varphi_t^{\Omega}(q)\in (\WF(Q-\Lambda_+^{-1}))^c\}.
$$
\end{lemma}
\begin{proof}
Let $q\in {}^b\WF(f_1)\cup \WF(f_2)\cap \mc{Z}$. Then, there is $0\leq t\leq T$ such that 
$$
\varphi_{t}^{\Omega}(\rho)\in (\WF(Q-\Lambda^{-1}))^c\cup \{|x|>R\}.
$$
In particular, either Lemma~\ref{l:modifiedDynamics} or~\ref{l:modifiedDynamics2} applies. Let $a_q$ as constructed there so that either~\eqref{e:dynamics} or~\eqref{e:dynamics2} holds. Then, since $\mc{Z}$ is compact, there is a neighborhood $U_q$ of $q$ and $V_q\Subset V_q^+\Subset T^*\Omega$ such that 
$$
\bigcup_{q'\in U\cap \mc{Z}}\overline{\bigcup_{0\leq t<T_{Q}(q') }\varphi_{t}^{\Omega}(q')}\cap (\supp a\cup V^+_q)=\emptyset.
$$
and such that~\eqref{e:dynamics} or~\eqref{e:dynamics2} holds with $\{d(x,\Gamma)>0\}$ replaced by $V_q$. 

Now, since ${}^b\WF(f_1)\cup \WF(f_2)\cap \mc{Z}$ is compact, there are $q_1,\dots, q_N$ such that
$$
{}^b\WF(f_1)\cup \WF(f_2)\cap \mc{Z}\subset \bigcup_{i=1}^N U_{q_i}. 
$$
Let $\psi_i\in C_c^\infty(U_{q_i})$ be a partition of unity near ${}^b\WF(f_1)$ and $\chi_i\in C_c^\infty(U_i\cap T^*\Gamma)$ a partition of unity near $\WF(f_2)\cap \mc{Z}$ and set $\psi_{0}:=1-\sum_i\psi_i$, $\chi_{0}:=1-\sum_i\chi_i$. 
Let also $W_i\in \Psi^{\comp}(\Omega)$ with $\sigma(W_i)=1$ on $V_{q_i}$ and $\WF(W_i)\Subset V^+_{q_i}$, and $A_i=\Op(a_{q_i})$, $i=1,\dots, N$, and set $W_0=W_1$, $A_0=A_1$.  

Then, by Lemma~\ref{l:dampedResolve}, for $i=0,1,\dots,N$, there is $u_i$ satisfying
\begin{equation*}
\begin{cases}
(-\hsc^2\Delta_g-A_i-iW_i^*W_i-1)u=\hsc{}^b\Op(\psi_i)f_1&\text{in }\Omega\\
Q\hsc D_{\nu}u_i-u_i=\Op(\chi_i)f_2&\text{on }\Gamma\\
u_i&\text{ is outgoing},
\end{cases}
\end{equation*}
with $u_i$ satisfying~\eqref{e:estU}. Moreover, by Theorem~\ref{t:basicPropagate}, Lemma~\ref{l:DtN}, and the outgoing property,
$$
{}^b\WF(u_i)\subset {}^b\WF(f_1)\cup \WF(f_2)\cup \bigcup_{q\in( {}^b\WF(f_1)\cup \WF(f_2))\cap U_{q_i}}\overline{\bigcup_{0\leq t<T_{Q}(q) }\varphi_{t}^{\Omega}(q)}
$$
In particular, ${}^b\WF(u_i)\cap ({}^b\WF(A_i)\cap {}^b\WF(W_i))=\emptyset$, and hence
\begin{equation*}
\begin{cases}
(-\hsc^2\Delta_g-1)u=\hsc{}^b\Op(\psi_i)f_1+O(\hsc^{\infty})_{L^2_{\comp}}&\text{in }\Omega\\
Q\hsc D_{\nu}u_i-u_i=\Op(\chi_i)f_2&\text{on }\Gamma\\
u_i&\text{ is outgoing}.
\end{cases}
\end{equation*}
Setting $u:=\sum_{i=0}^N u_i$ completes the proof.
\end{proof}

\subsection{Review of resolvents and layer potentials}
In this section, we recall some facts about the free, outgoing resolvent and boundary layer potentials; we use the latter to control the boundary traces of solutions to $(-\hsc^2\Delta-1)u=0$. Recall that the outgoing resolvent $R_0(\hsc):L^2_{\comp}\to L^2_{\loc}$ is the unique solution to 
$$
(-\hsc^2\Delta-1)R_0(\hsc)f=f\text{ in }\mathbb{R}^d,\qquad R_0(\hsc)f\text{ is outgoing}.
$$
It is well known (see e.g.~\cite{GaSpWu:20}) that for all  $\chi\in C_c^\infty(\mathbb{R}^d)$, $\hsc_0>0$, and $s\in \mathbb{R}$, there is $C>0$ such that for $0<\hsc<\hsc_0$,
\begin{equation}
\label{e:freeResolveEstimates}
\|\chi R_0(\hsc)\chi\|_{\mac{s}\to \mac{s+2}}\leq C\hsc^{-1}.
\end{equation}

The single and double layer potentials associated to $\Gamma$ are then given, respectively, for $f\in L^2(\Gamma)$ by
$$
\mc{S}\ell f(x):= \hsc^2\int_{\Gamma}R_0(\hsc)(x,y)f(y)dS(y),\qquad \mc{D}\ell f(x):=\hsc^2\int_{\Gamma}\partial_{\nu_y}R_0(\hsc)(x,y)f(y)dS(y),
$$
where $R_0(\hsc)(x,y)$ is the kernel of $R_0(\hsc)$.

\begin{lemma}
Let $\chi \in C_c^\infty(\mathbb{R}^d)$, and $B\in \Psi^{\comp}(\mathbb{R}^d)$ with $\WF(B)\cap \bigcup_{t\geq 0}\varphi_t^{\mathbb{R}^d}(S^*\Gamma)=\emptyset$. Then, for all $\hsc_0>0$ there is $C>0$ such that for $0<\hsc<\hsc_0$,
\begin{equation}
\label{e:layer}
\|B\mc{S}\ell\|_{L^2(\Gamma)\to L^2(\mathbb{R}^d)}\leq C\hsc, \qquad
\|\chi \mc{D}\ell\|_{L^2(\Gamma)\to L^2(\mathbb{R}^d)}\leq C.
\end{equation}
\end{lemma}
\begin{proof}
Let $L$ be a smooth vectorfield with $L|_{\Gamma}=\partial_\nu$. We start by recalling that 
$$
\mc{S}\ell= R_0(\hsc)\hsc^2\delta_{\Gamma},\qquad \mc{D}\ell=R_0(\hsc)L^*\hsc^2\delta_{\Gamma}.
$$
Let $u\in L^2(\Gamma)$ and $v\in L^2(\mathbb{R}^d)$. Then,
\begin{align*}
\langle B\mc{S}\ell u,v\rangle_{L^2(\mathbb{R}^d)}=\langle BR_0(\hsc)\hsc^2\delta_{\Gamma} u,v\rangle_{L^2(\mathbb{R}^d)}
=\langle u, \gamma R_0^*(\hsc)\hsc^2B^*v\rangle_{L^2(\Gamma)}
\end{align*}
Since $R_0^*$ is the \emph{incoming} resolvent,
$$
(-\hsc^2\Delta-1)R_0^*(\hsc)\hsc^2B^*v=\hsc^2B^*v,\qquad \WF(R_0^*(\hsc)\hsc^2B^*v)\subset \WF(B^*v)\cup\bigcup_{t\leq 0}\varphi_t^{\mathbb{R}^d}\big(\WF(B^*)\cap S^*\mathbb{R}^d\big),
$$
In particular, $\WF(R_0^*(\hsc)\hsc^2B^*v)\cap S^*\Gamma=\emptyset$. Thus, for $\chi \in C_c^\infty(\mathbb{R}^d)$ with $\chi \equiv 1$ near $\Gamma$, standard trace estimates (see e.g.~\cite[Proposition 3.1]{GaLe:20})
\begin{align*}
\|\gamma R_0^*(\hsc)\hsc^2B^*v\|_{L^2(\Gamma)}&\leq C(\|\chi R_0^*(\hsc)\hsc^2B^*v\|_{L^2(\mathbb{R}^d)}+\hsc\|B^*v\|_{L^2(\mathbb{R}^d)})\leq C\hsc\|B^*v\|_{L^2(\mathbb{R}^d)}\leq C\hsc\|v\|_{L^2(\mathbb{R}^d)},
\end{align*}
and hence
$$
|\langle B\mc{S}\ell u,v\rangle_{L^2(\mathbb{R}^d)}|\leq C\hsc\|u\|_{\LtG}\|v\|_{L^2(\mathbb{R}^d)},
$$
which implies the first estimate in~\eqref{e:layer} by duality.

For the second, we proceed similarly.
\begin{align*}
\langle \mc{D}\ell u,v\rangle_{L^2(\mathbb{R}^d)}=\langle R_0(\hsc)\hsc^2L^*\delta_{\Gamma} u,v\rangle_{L^2(\mathbb{R}^d)}.
=\langle u, \gamma \hsc LR_0^*(\hsc)\hsc v\rangle_{L^2(\Gamma)}
\end{align*}
Then, we have
$$
(-\hsc^2\Delta-1)R_0^*(\hsc)\hsc v=\hsc v,
$$
and hence, using e.g.~\cite[Corollary 0.6]{Ta:17}, 
$$
\|\hsc\gamma LR_0^*(\hsc)hv\|_{\LtG}\leq C(\|R_0^*(\hsc)\hsc v\|_{L^2(\mathbb{R}^d)}+\|v\|_{L^2(\mathbb{R}^d)})\leq C\|v\|_{L^2(\mathbb{R}^d)},
$$
which, arguing as above, implies the second estimate in~\eqref{e:layer} by duality.
\end{proof}

Next, we consider the Dirichlet resolvent $R_D:L^2_{\comp}(\Omega^+)\to L^2_{\loc}(\Omega^+)$ defined by
$$
(-\hsc^2\Delta-1)R_Df=f,\qquad R_Df|_{\Gamma}=0,\qquad R_D f\text{ is outgoing},
$$
and the Poisson operator $G_D:\mac{1}(\Gamma)\to L^2_{\loc}(\Omega^+)$, defined by
$$
(-\hsc^2\Delta-1)G_Dg=0,\qquad G_Dg|_{\Gamma}=g,\qquad G_Dg\text{ is outgoing}.
$$

We need the following assumption on $R_D$.
\begin{assumption}
\label{ass:poly}
There is $P_{\rm inv}\geq 0$ and $\tilde{\cJ}\subset (0,\infty)$ such that for all $\chi \in C_c^\infty(\Omega^+)$, there is $C>0$ such that 
$$
\|\chi R_D\chi\|_{L^2(\Omega^+)\to L^2(\Omega^+)}\leq C\hsc ^{-1-P_{\rm inv}},\qquad \hsc \notin \tilde{\cJ}
$$
\end{assumption}
Note that by~\cite{LSW1}, for any $\Omega^-\Subset \mathbb{R}^d$ with smooth boundary and connected complement, and any $\delta>0$ there is $\cJ$ with $|\cJ|<\delta$ such that Assumption~\ref{ass:poly} holds. Moreover, for any $N>0$, one can find $\cJ$ and $C>0$ such that $|\cJ\cap (k,\infty)|\leq C k^{-N}$ for all $k>1$.

We now show that, provided the data is located away from $\Gamma_-$, $G_D$ and $R_D$ behave as in the nontrapping case.
\begin{lemma}
\label{l:awayGamma-}
Suppose that Assumption~\eqref{ass:poly} holds. Then for all $B\in \Psi^0(\Gamma)$ with $\WF(B)\cap \Gamma_-=\emptyset$ and all $\chi \in C_c^\infty(\overline{\Omega^+})$, there is $C>0$ such that
$$
\|\chi G_D B\|_{\mac{3/2}(\Gamma)\to \mac{2}(\Omega^+)}\leq C,\qquad \hsc \notin \cJ,
$$
and, if $u\in H^{3/2}_h$ is $\hsc$-tempered
$$
{}^b\WF(G_DBu)\subset \WF(B)\cup \bigcup_{t\geq 0}\varphi_t^{\Omega}(\WF(B)\cap \mc{Z}).
$$
\end{lemma}
\begin{proof}
Since $\WF(B)\cap \Gamma_-=\emptyset$, Lemma~\ref{l:approxSolve} implies that for any $g\in \mac{3/2}$ there is $u\in H_{\hsc ,\loc}^2$ such that for all $\chi\in C_c^\infty(\Omega^+)$, 
$$
\|\chi u\|_{\mac{2}}\leq C\|f\|_{\mac{3/2}}
$$
and 
$$
{}^b\WF(u)\subset \WF(B)\cup \bigcup_{t\geq 0}\varphi_t^{\Omega+}({}^b\WF(B)\cap \mc{Z}),
$$
and 
$$
\begin{cases}(-\hsc^2\Delta-1)u=f=O(\hsc^{\infty}\|g\|_{\mac{3/2}})_{L^2_{\comp}}&\text{ in }\Omega^+\\
u=Bg&\text{ on }\Gamma,\\
u\text{ is outgoing}.
\end{cases}
$$
In particular, by Assumption~\ref{ass:poly}, 
$$
G_DBg=u-R_Df= u+O(\hsc^{\infty})_{H_{\hsc ,\loc}^{2}},
$$
and the lemma follows.
  \end{proof}

 The next lemma is proved using an almost identical argument
  \begin{lemma}
  \label{l:awayGamma-2}
  Suppose that Assumption~\ref{ass:poly} holds. Then for all $B\in \Psi^{\comp}(\Omega^+)$ with $\WF(B)\cap \Gamma_-=\emptyset$ and all $\chi \in C_c^\infty(\overline{\Omega^+})$, there is $C>0$ such that
$$
\|\chi R_D B\|_{L^2\to L^2(\Omega^+)}\leq C\hsc^{-1},\qquad \hsc \notin \tilde{\cJ},
$$
and, if $u\in L^2$ is $\hsc$-tempered,
$$
{}^b\WF(R_DBu)\subset {}^b\WF(B)\cup \bigcup_{t\geq 0}\varphi_t^{\Omega}({}^b\WF(B)\cap \mc{Z}).
$$
  \end{lemma}
  
  We also require the following estimate on the wavefront set properties of $R_D$ which follows directly from Theorem~\ref{t:basicPropagate}, the outgoing condition, and Assumption~\ref{ass:poly}
  \begin{lemma}
  \label{l:outgoing}
  Suppose that Assumption~\ref{ass:poly} holds. Then, for all $B\in {}^b\Psi^{\comp}(\Omega^+)$, $\chi\in C_c^\infty(\overline{\Omega^+})$, and $u\in L^2$ $\hsc$-tempered,
  $$
  {}^b\WF(R_D Bu)\subset {}^b\WF(B)\cup \bigcup_{t\geq 0}\varphi_t^{\Omega^+}({}^b\WF(B)\cap \mc{Z})\cup \Gamma_+.
  $$ 
  \end{lemma}

\subsection{Non-interference of the Impedance-to-Dirichlet maps}

In order to pass from quasimodes in $\Omega^+$ to quasimodes for $A_{k}$ and $A_{k}'$, we need two lemmas that give conditions guaranteeing that
$$
\|P_{\rm ItD}^{-,\eta}P_{\rm DtN}^+f\|_{L^2}\sim \|P_{\rm DtN}^+f\|_{L^2},\qquad \|P_{\rm ItD}^{+,\eta}P_{\rm DtN}^-f\|_{L^2}\sim \|P_{\rm DtN}^-f\|_{L^2}.
$$

\begin{lemma}
\label{l:layerNoCancel}
Let $C>0$ and suppose that $f\in L^2(\Gamma)$, with 
$$
\|f\|_{\LtG}=o(1),\qquad \|P_{\rm DtN}^+f\|_{L^2}=1,
$$
and there is $B\in \Psi^{\comp}(\mathbb{R}^d)$ with $\WF(B)\cap \bigcup_{t>0}\varphi_{t}^{\mathbb{R}^d}(S^*\Gamma)=\emptyset$ such that 
\begin{equation}
\label{e:notGlanceOnly}
\|BG_D f\|_{L^2(\Omega^+)}\geq c>0.
\end{equation}
Then, for all $C^{-1}\leq \hsc \eta\leq C$,
$$
\|P_{\rm ItD}^{-,\eta}P_{\rm DtN}^+f\|_{L^2(\Gamma)}\geq c>0.
$$
\end{lemma}
\begin{proof}
Suppose there is a sequence $\hsc_n\to 0$ with $\|P_{\rm ItD}^{-,\eta}P_{\rm DtN}^+f\|_{L^2(\Gamma)}\to 0$. Then, there are solutions $u_1,u_2$ to
$$
\begin{cases}(-\hsc^2\Delta-1)u_1=0&\text{ in }\Omega\\
(\hsc \partial_\nu -i\hsc \eta )u_1=\hsc \partial_\nu u_2,\end{cases}\qquad \begin{cases}(-\hsc^2\Delta-1)u_2=0&\text{ in }\mathbb{R}^d\setminus \Omega\\
u_2|_{\partial\Omega}=f\\
u_2\text{ outgoing}
\end{cases}
$$
with $\|u_1|_{\partial\Omega}\|_{L^2}=o(1)$ and $\|f\|_{L^2}=o(1)$. Thus,
\begin{align*}
(u_11_{\Omega}+u_21_{\mathbb{R}^d\setminus \Omega})&= R_0(\hsc)(\hsc^2(u_1-u_2)|_{\Gamma}\otimes\partial_\nu \delta_{\Gamma}+hR_0(\hsc)(\hsc \partial_\nu u_1-\hsc \partial_{\nu}u_2)\otimes \delta_{\Gamma}\\
&=\mathcal{D}\ell(u_1-u_2)|_{\Gamma})+\hsc^{-1}\mathcal{S}\ell(\hsc \partial_\nu u_1-\hsc\partial_{\nu}u_2)\\
&=\mathcal{D}\ell (u_1-f)+\hsc^{-1}\mathcal{S}\ell i\hsc \eta u_1\\
&= \mathcal{D}\ell o(1)_{L^2}+ \hsc^{-1}\mathcal{S}\ell o(1)_{L^2}.
\end{align*}
In particular, for any $B\in \Psi^{c}(\mathbb{R}^d)$ with 
$$
\WF(B)\cap \bigcup_{t\geq 0}\varphi_t^{\mathbb{R}^d}(S^*\Gamma)=\emptyset,
$$ 
we have
$$
B((u_11_{\Omega}+u_21_{\mathbb{R}^d\setminus \Omega}))=o(1)_{L^2}.
$$
This contradicts~\eqref{e:notGlanceOnly}.
\end{proof}

\begin{lemma}
\label{l:layerNoCancel2}
Let $C,c>0$ and suppose that $f\in L^2(\Gamma)$, with 
$$
\|f\|_{\LtG}=o(1),\qquad \|P_{\rm DtN}^-f\|_{L^2}=1,
$$
and there is $B\in \Psi^{\comp}(\mathbb{R}^d)$ with $\WF(B)\cap \bigcup_{t\leq 0}\varphi_{t}^{\mathbb{R}^d}(S^*\Gamma\cup N^*\Gamma)=\emptyset$ such that 
\begin{equation}
\label{e:notGlanceOnly2}
\|BG_D^+(h)f\|_{L^2(\Omega^+)}\geq c>0.
\end{equation}
Then, there is $A\in \Psi(\Gamma)$ with $\WF(A)\cap \{\xi'=0\}=\emptyset$ such that for all $C^{-1}\leq \hsc \eta\leq C$, 
$$
\|AP_{\rm ItD}^{+,\eta}P_{\rm DtN}^-f\|_{L^2(\Gamma)}\geq c>0.
$$
\end{lemma}
\begin{proof}
Suppose there is a sequence $\hsc_n\to 0$ with $\|AP_{\rm ItD}^{+,\eta}P_{\rm DtN}^-f\|_{L^2(\Gamma)}\to 0$. Then, there are solutions $u_1,u_2$ to
$$
\begin{cases}(-\hsc^2\Delta-1)u_1=0&\text{ in }\Omega\\
(\hsc \partial_\nu +i\hsc \eta )u_1=\hsc \partial_\nu u_2,\end{cases}\qquad \begin{cases}(-\hsc^2\Delta-1)u_2=0&\text{ in }\mathbb{R}^d\setminus \Omega\\
u_2|_{\partial\Omega}=f\\
u_2\text{ incoming}
\end{cases}
$$
with $\|Au_1|_{\partial\Omega}\|_{L^2}=o(1)$ and $\|f\|_{L^2}=o(1)$. Thus,
\begin{align*}
(u_11_{\Omega}+u_21_{\mathbb{R}^d\setminus \Omega})&= R_0^*(\hsc)(\hsc^2(u_1-u_2)|_{\Gamma}\otimes\partial_\nu \delta_{\Gamma}+hR_0^*(\hsc)(\hsc \partial_\nu u_1-\hsc \partial_{\nu}u_2)\otimes \delta_{\Gamma}\\
&=\widetilde{\mathcal{D}\ell}(u_1-u_2)|_{\Gamma})+\hsc^{-1}\widetilde{\mathcal{S}\ell}(\hsc \partial_\nu u_1-\hsc \partial_{\nu}u_2)\\
&=\widetilde{\mathcal{D}\ell }(u_1-f)+\hsc^{-1}\widetilde{\mathcal{S}\ell} i\hsc\eta u_1\\
&=\widetilde{\mathcal{D}\ell }(u_1)-\widetilde{ \mathcal{D}\ell} o(1)_{L^2}+ \hsc^{-1}\widetilde{\mathcal{S}\ell } i\hsc\eta u_1.
\end{align*}
In particular, for any $B\in \Psi^{c}(\mathbb{R}^d)$ with 
$$
\WF(B)\cap \bigcup_{t\geq 0}\varphi_t^{\mathbb{R}^d}(S^*\Gamma\cup \WF(I-A))=\emptyset,
$$ 
we have
$$
B((u_11_{\Omega}+u_21_{\mathbb{R}^d\setminus \Omega}))=B\widetilde{\mathcal{D}\ell }(u_1)+ \hsc^{-1}\widetilde{\mathcal{S}\ell } i\hsc\eta u_1+o(1)_{L^2}=B\widetilde{\mathcal{D}\ell }(Au_1) \hsc^{-1}\widetilde{\mathcal{S}\ell } i\hsc\eta Au_1+o(1)_{L^2}=o(1)_{L^2}.
$$
This contradicts~\eqref{e:notGlanceOnly2}.
\end{proof}

\subsection{From interior quasimodes to quasimodes for BIEs}
We are now in a position to show that appropriate quasimodes in the bulk produce quasimodes for  the BIE operators that are effective for demonstrating pollution. Before constructing these quasimodes, we need two technical lemmas that reduce~\eqref{e:kOscillate} to a wavefront set statement.
\begin{lemma}
Suppose that Assumption~\ref{ass:parameters} holds \blue{and} Assumption~\ref{ass:poly} holds for some $\tilde{\cJ}$. Then Assumption~\ref{ass:polyboundintro} holds with $\operator=A_{k}$ or $A_{k}'$ with $\cJ:=\{ k\,:\, k^{-1}\in \tilde{\cJ}\}$.\end{lemma}
\begin{proof}
Notice that Lemma~\ref{l:dampedResolve} implies
$$
\|P_{\rm ItD}^{-,\eta}\|_{\mac{1/2}\to \mac{3/2}}\leq C.
$$
Next, we show that that there are $\chi \in C_c^\infty$ and  $N_0>0$ such that 
$$
\|P_{\rm DtN}^+\|_{\mac{3/2}\to \mac{1/2}}\leq C\|\chi R_D\chi\|_{L^2\to L^2}.
$$

For this, observe that there is $E:\mac{3/2}(\Gamma)\to H_{\hsc ,\comp}^{2}(\Omega^+)$, such that $\|E\|_{\mac{3/2}\to \mac{2}}\leq C\hsc^{-1/2}$. Therefore, 
$$
P_{\rm DtN}^+= \hsc \partial_\nu R_D(-\hsc^2\Delta-1)E,
$$
and we use the estimate: for any $\chi \in C_c^\infty(\overline{\Omega}^+)$ with $\chi\equiv 1$ near $\Gamma$,
$$
\|\hsc \partial_\nu u\|_{\mac{1/2}(\Gamma)}\leq C\hsc^{-1/2}\|\chi u\|_{\mac{2}},
$$
to conclude that 
$$
\|P_{\rm DtN}^+\|_{\mac{3/2}\to \mac{1/2}}\leq C\hsc^{-1}\|\chi R_D\chi \|_{L^2\to \mac{2}}\leq C\hsc^{-1}\|\chi R_D \chi\|_{L^2\to L^2}. 
$$

Together with~\eqref{e:BIEinverse}, this implies that for $\Psi\in \Psi^{\comp}(\Gamma)$, 
$$
\|A_{k}^{-1}\Psi\|_{\LtGt}+\|(A_{k}')^{-1}\Psi\|_{\LtGt}\leq C\hsc^{-1}\|\chi R_D\chi \|_{L^2\to L^2}.
$$
Lemma~\ref{lem:HFinverse}, or more precisely its proof, then completes the proof of the lemma.
\end{proof}

\begin{lemma}
\label{l:waveToKOscillate}
Suppose that $f\in L^2(\Gamma)$ is $\hsc$-tempered and  
\begin{equation}
\label{e:goodWave}
\WF(f)\subset \{(x',\xi')\in \cH\,:\, 0<|\xi'|_{g_\Gamma}\}.
\end{equation}
Then for any $\hsc_0>0$, $\e>0$, $\chi \in C_c^\infty ((-(1+2\e)^2,(1+2\e)^2))$ with $\chi \equiv 1$ on $[-1-\e,1+\e]$, there is $\e_0>0$ such that for any $N>0$ there is $C>0$ such that for $0<\hsc<\hsc_0$
$$
\|(1-\chi(\hsc^{2}\Delta_g))f\|_{L^2(\Gamma)}\leq C\hsc^{N},\qquad \|\chi(\e_0^{-2}\hsc^2\Delta_g)f\|_{L^2(\Gamma)}\leq C\hsc^{N}
$$
\end{lemma}
\begin{proof}
The proof follows from the definition of wavefront set and Corollary~\ref{cor:WFcutoff}. 
\end{proof}

We now give a few abstract lemmas that pass from quasimodes in $\Omega^+$ to data, $f$ on $\Gamma$ satisfying~\eqref{e:goodWave} and~\eqref{e:BIEQuasi}.
\begin{lemma}[From Helmholtz quasimode data to growing $P_{\rm DtN}^+$]
\label{l:quasimodeExteriorToDirichlet}
Let $\Omega^-\Subset \mathbb{R}^d$ open with smooth boundary and connected complement $\Omega^+:=\mathbb{R}^d\setminus \overline{\Omega}_-$ such that Assumption~\ref{ass:poly} holds. 
Let $\mu\in \{1,-1\}$, $\mc{O}\subset \mc{H}$ open, $\mc{B}\subset \mc{H}$ closed and $B\in \Psi^{\comp}(\Omega^+)$ such that 
\begin{equation}
\label{e:uIsBig}
\bigcup_{\mu t\geq 0}\varphi_{\mu t}^{\mathbb{R}^d}(S^*\Gamma)\cap \WF(B)=\emptyset.
\end{equation}
 Suppose that there are $u\in \mac{2}(\Omega^+)$, $\tilde{g}\in L^2_{\comp}(\Omega^+)$ such that 

\begin{gather}
\sup_{q\in {}^b\WF(\tilde{g}\cap\mc{Z})}T^\mu_{\mc{O}<\mc{B}\cup\{q\}}(q)<\infty,\qquad \|Bu\|_{L^2(\Omega^+)}\geq c>0
\label{e:WFCond}\\
(-\hsc^2\Delta-1)u= \hsc \tilde{g},\qquad u|_{\Gamma}=0,\qquad \| \tilde{g}\|_{L^2}=o(1),\quad u \text{ is } \begin{cases} \text{outgoing}&\mu=1\\\text{incoming}&\mu=-1\end{cases}.\label{e:uCond}
\end{gather}
Then for any function $T:{}^b\WF(\tilde{g})\to [0,\infty)$ such that $T(q)<T_{\mc{B}\cup\{q\}}(q)$,  
$\varphi_{\mu T(q)}^{\Omega^+}(q)\in \mc{O}\setminus (\mc{B}\cup \{q\}),$ and any $\e>0$ there is $c_1>0$ such that for any $\hsc \notin \cJ$, $q_0\in {}^b\WF(\tilde{g})$ and $g_1\in L^2(\Gamma)$ such that 
\begin{equation}
\label{e:DTNBig}
\|P_{\rm DtN}^{\mu}g_1\|_{L^2}\geq c_1(\|Bu\|_{L^2}/\|\tilde{g}\|_{L^2})\|g_1\|_{L^2},
\end{equation}
\begin{equation}
\label{e:WFDirichlet}
\WF(g_1)\subset B(\varphi_{T(q_0)}^{\Omega_+}(q_0),\e)\subset \mc{O}\setminus (\mc{B}\cup\{q_0\}),
\end{equation}
and
$$
\|BG_Dg_1\|_{L^2}\geq c_1.
$$
\end{lemma}
\begin{proof} 
We consider the case of $\mu=1$, the other being almost identical but with the role of $R_D$ played by $R_D^*$. 

Let $q\in {}^b\WF(\tilde{g})\cap \mc{Z}$ and $0<T_q<T_{\mc{B}\cup\{q\}}$ such that $\varphi_{ T^1_{q}}^{\Omega^+}(q)\in \mc{O}\setminus (\mc{B}\cup \{q\})$. Let $V_q^1\Subset V_q^2\subset \mc{O}\cap (B(\varphi_{T_q}^{\Omega_+}(q),\e)\setminus (\mc{B}\cup\{q\}))$ be a neighborhood of $\varphi_{T_q}^{\Omega^+}(q)$ and $U_q$ a neighborhood of $q$ such that 
$$
\inf_{q'\in U_q}T^1_{V^1_q<(\mc{B}\cup \{q'\}}(q')<\infty.
$$
Such a neighborhood exists because $\varphi_t^{\Omega^+}$ is continuous and $\varphi_{\mu T(q)}^{\Omega^+}(q)\in \mc{H}\setminus (\mc{B}\cup\{q\})$.

Since ${}^b\WF(\tilde{g})\cap \mc{Z}$ is compact, there are $\{q_1,\dots q_N\}$ such that 
$$
{}^b\WF(\tilde{g})\cap \mc{Z}\subset \bigcup_{i=1}^NU_{q_i}. 
$$
Let $\{\psi_j\}_{j=1}^N\in C_c^\infty({}^bT^*\Omega^+)$ with $\sum_{j}\psi_j\equiv 1 $ on ${}^b\WF(\tilde{g})\cap \mc{Z}$ and $\supp \psi_j \subset U_{q_j}$. Then, by Lemma~\ref{l:awayGamma-2}
and hence, using Assumption~\ref{ass:poly}, for any $\chi\in C_c^\infty$, 
$$
\|\chi u-\sum_j\chi R_D {}^b\Op(\psi_j)\hsc \tilde{g}\|_{\mac{2}}\leq C\|\tilde{g}\|_{L^2}.
$$
In particular, there is $i\in\{ 1,\dots, N\}$ such that 
\begin{equation}
\label{e:newAssumption}
\|B R_D {}^b\Op(\psi_{i})\tilde{g}\|_{L^2}\geq \frac{1}{N}\|B R_D \tilde{g}\|_{L^2}\geq \frac{c}{N},\qquad 
\| {}^b\Op(\psi_i)\tilde{g}\|_{L^2}\leq C\|\tilde{g}\|_{L^2}=o(1). 
\end{equation}

Let $g:={}^b\Op(\psi_i)\tilde{g}$. then, by Lemma~\ref{l:approxSolve} there is $u_1$ satisfying
$$
\begin{cases}(-\hsc^2\Delta-1)u_1= \hsc g+O(\hsc^{\infty})_{L^2_{\comp}}&\text{ in }\Omega^+
\\
u_1\text{ outgoing}\\
Q\Lambda^{-1}\hsc D_\nu u_1 -u_1=0&\text{ on }\Gamma,
\end{cases}
$$
where $Q\in \Psi^{\comp}(\Gamma)$ with $\sigma(Q)\geq0$, $\WF(Q)\subset V^2_{q_i}$ and $\WF(Q-\Lambda^{-1})\cap V_{q_i}^1=\emptyset,$ where $\Lambda^{-1}\in \Psi^{-1}$ is from Lemma~\ref{l:approxSolve}. Moreover, for any $\chi\in C_c^\infty(\overline{\Omega^+})$, 
$$
\begin{gathered}
{}^b\WF(u_1)\subset {}^b\WF(g)\cup \bigcup_{q\in {}^b\WF(g)}\overline{\bigcup_{0\leq t<T_{V_{q_i}^1}(q) }\varphi_{t}^{\Omega}(q)},\qquad \WF(u_1|_{\Gamma})\subset V_{q}^2,
\end{gathered}
$$
and
\begin{equation}
\label{e:estimateU1}
\|\chi u_1\|_{\mac{2}}\leq c\|g\|_{L^2}.
\end{equation}

%
%

Put $u_2:= R_Dg-u_1+R_D((-\hsc^2\Delta-1)u_1-g)$. Then 
$$
(-\hsc^2\Delta-1)u_2=0 \text{ in }\Omega^+,\qquad u_2|_{\Gamma}=-u_1|_{\Gamma},\qquad u_2(\hsc)\text{ is outgoing},
$$
and, using Assumption~\ref{ass:poly} and that $(-\hsc^2\Delta-1)u_1-g=O(\hsc^{\infty})_{L^2_{\comp}}$,
$$
u_2=R_Dg-u_1+O(\hsc^{\infty})_{H_{\hsc ,\loc}^2}.
$$
In particular, by~\eqref{e:newAssumption} and~\eqref{e:estimateU1}
\begin{equation}
\label{e:aardvark}
\|B u_2\|_{L^2}\geq \|BR_Dg\|-\|Bu_1\|_{L^2}-O(\hsc^{\infty})\geq c -C\| \tilde{g}\|_{L^2}+O(\hsc^{\infty}).
\end{equation}
Furthermore, since $\WF(u_1|_{\Gamma})\subset\mathcal{H}$, for $\chi\in C_c^\infty(\overline{\Omega^+})$ with $\supp (1-\chi)\cap \Gamma=\emptyset$, 
\begin{equation}
\label{e:squirrel}
\|u_2|_{\Gamma}\|_{L^2}=\|u_1|_{\Gamma}\|_{L^2}\leq \|\chi u_1\|_{L^2}+C\|g\|_{L^2}+O(\hsc^{\infty})\|g\|_{L^2}\leq C\|\tilde{g}\|_{L^2}+O(\hsc^{\infty})\|\tilde{g}\|_{L^2}.
\end{equation}
Now,
$$
u_2=\mc{S}\ell \partial_\nu u_2|_{\Gamma}-\mc{D}\ell u_2|_{\Gamma}.
$$ 
Therefore, using
\begin{equation}
\label{e:lemur}
c-C\|\tilde{g}\|_{L^2}-O(\hsc^{\infty})\leq \|Bu_2\|_{L^2}\leq \|B\mc{S}\ell \partial_\nu u_2|_{\Gamma}\|+\|B\mc{D}\ell u_2|_{\Gamma}\|\leq C\|\hsc \partial_\nu u_2\|_{L^2}+C\| \tilde{g}(h)\|_{L^2}.
\end{equation}
In particular, since $\|\tilde{g}\|_{L^2}=o(1)$, using~\eqref{e:aardvark},~\eqref{e:squirrel}, and~\eqref{e:lemur}, for $h$ small enough
$$
c\frac{\|u_2|_{\Gamma}\|_{L^2}}{\|\tilde{g}\|_{L^2}}\leq c\leq \|\hsc \partial_{\nu}u_2\|_{L^2}=\|P_{\rm DtN}^+u_2|_{\Gamma}\|_{L^2}
$$
\end{proof}

\begin{lemma}[From Helmholtz quasimode data to BIE quasimode data for $A_k$]
\label{l:bWave}
Let $\Omega^-\Subset \mathbb{R}^d$ open with smooth boundary and connected complement $\Omega^+:=\mathbb{R}^d\setminus \overline{\Omega}_-$ and suppose that Assumptions~\ref{ass:parameters} and~\ref{ass:poly} hold. Let $\mu=1$, $B\in \Psi^{\comp}(\Omega^+)$ satisfy~\eqref{e:uIsBig} and suppose there are $u\in H_{\hsc ,\loc}^2(\Omega^+)$, $g\in L^2_{\comp}(\Omega^+)$ satisfying~\eqref{e:WFCond} and~\eqref{e:uCond} with 
$$
\mc{O}=\mc{H}_0:=\{ (x',\xi')\in\mc{H},:\, |\xi'|\neq 0\},\qquad \mc{B}=\emptyset.
$$
Then for all $\hsc \notin \tilde{\cJ}$, there are $v,f\in L^2(\Gamma)$ such that $f$ satisfies~\eqref{e:goodWave} and
$$
A_{k}v=f,\qquad \|f\|_{\LtG}\leq C\|\tilde{g}\|_{L^2(\Omega^+)}\|v\|_{\LtG}.
$$
\end{lemma}
\begin{proof}
By Lemma~\ref{l:quasimodeExteriorToDirichlet}, there is $g_1\in L^2(\Gamma)$ with $\WF(g_1)\subset \mc{H}_0$ such that 
$$
\|P_{\rm DtN}^+g_1\|_{L^2}\geq \|Bu\|_{L^2}/\|\tilde{g}\|_{L^2}\|g_1\|_{L^2},
$$
and 
$$
\|BG_Dg_1\|_{L^2}\geq c>0.
$$
Thus, Lemma~\ref{l:layerNoCancel} implies that 
$$
\|P_{\rm ItD}^{-,\eta}P_{\rm DtN}^+g_1\|_{L^2}\geq c\|\tilde{g}_1\|_{L^2}^{-1}\|g_1\|_{L^2}.
$$

In particular, 
$$
\|A_{k}^{-1}u_2|_{\Gamma}\|_{L^2}=\|(I-P_{\rm ItD}^{-,\eta}(P_{\rm DtN}^+-i\hsc \eta))u_2|_{\Gamma}\|_{L^2}\geq (c\| \tilde{g}\|_{L^2}^{-1}-C)\|g_1\|_{L^2},
$$
which completes the proof.
\end{proof}

We give two versions of Lemma~\ref{l:bWave} for $A_{k,\eta}$. The first is simpler but has more restrictive assumptions.
\begin{lemma}[From Helmholtz quasimode data to BIE quasimode data for $A'_k$:I]
\label{l:bWave2New}
Let $\Omega^-\Subset \mathbb{R}^d$ open with smooth boundary and connected complement $\Omega^+:=\mathbb{R}^d\setminus \overline{\Omega}_-$ and suppose that Assumptios~\ref{ass:parameters} and ~\ref{ass:poly} hold. 
Let $\mu=-1$,  $B\in \Psi^{\comp}(\Omega^+)$ satisfy~\eqref{e:uIsBig} and
\begin{equation}
\label{e:notNormal}
\WF(B)\cap \bigcup_{t\geq 0}\varphi_{-t}^{\mathbb{R}^d}(N^*\Gamma),
\end{equation}
and suppose there are $u\in H_{\hsc ,\loc}^2(\Omega^+)$, $g\in L^2_{\comp}(\Omega^+)$ satisfying~\eqref{e:WFCond} and~\eqref{e:uCond} with $\mc{O}=\mc{H}$
Then for all $\hsc \notin \tilde{\cJ}$, there are $v,f\in L^2(\Gamma)$ such that $f$ satisfies~\eqref{e:goodWave} and
$$
A'_{k}v=f,\qquad \|f\|_{\LtG}\leq C\|\tilde{g}\|_{L^2(\Omega^+)}\|v\|_{\LtG}.
$$
%
%
\end{lemma}
\begin{proof}
By Lemma~\ref{l:quasimodeExteriorToDirichlet} with $\mu=-1$, there is $g_1\in L^2(\Gamma)$ with $\WF(g_1)\subset \mc{H}$ such that 
$$
\|P_{\rm DtN}^{-}g_1\|_{L^2}\geq \|Bu\|_{L^2}/\|\tilde{g}\|_{L^2}\|g_1\|_{L^2},
$$
Thus, Lemma~\ref{l:layerNoCancel2} implies that there is $\tilde{B}\in\Psi(\Gamma)$ with $\WF(\tilde{B})\cap \{\xi'=0\}=\emptyset$ such that 
$$
\|\tilde{B}P_{\rm ItD}^{+,\eta}P_{\rm DtN}^-g_1\|_{L^2}\geq c\|\tilde{g}_1\|_{L^2}^{-1}\|g_1\|_{L^2}.
$$
In particular, since $(P_{\rm ItD}^{+,\eta}P_{\rm DtN}^-)^*=P_{\rm DtN}^+P_{\rm ItD}^{-,\eta}$, we have
$$
\|P_{\rm DtN}^+P_{\rm ItD}^-\tilde{B}\|_{L^2\to L^2}=\|\tilde{B}P_{\rm ItD}^+P_{\rm DtN}^-\|_{L^2\to L^2}\geq \frac{c}{\|\tilde{g}\|_{L^2}}.
$$
Since 
$$
\|(A_{k,\eta}')^{-1}\tilde{B}\|_{\LtGt}\geq  \|(P_{\rm DtN}^+P_{\rm ItD}^{-,\eta})\tilde{B}\|_{\LtGt}-C,
$$
this completes the proof.
\end{proof}

In some settings it is useful to have a more refined estimate than Lemma~\ref{l:bWave2New} provides.  In this case the dynamics in the interior of $\Omega^-$ also play a role.
\begin{lemma}[From Helmholtz quasimode data to BIE quasimode data for $A'_k$:II]
\label{l:bWave2}
Let $\Omega^-\Subset \mathbb{R}^d$ open with smooth boundary and connected complement $\Omega^+:=\mathbb{R}^d\setminus \overline{\Omega}_-$ and suppose that Assumptions~\ref{ass:parameters} and~\ref{ass:poly} hold.  Let $\Gamma_-$ be the incoming set for $\Omega^+$ and define
$$
\mc{O}:=\{(x',\xi')\in \mc{H}\setminus \Gamma_-\,:\, |\xi'|>0\},\qquad \mc{B}:= \Gamma_-\cap \{\xi'=0\}.
$$
Suppose that there is $\tilde{g}\in \mac{3/2}(\Gamma)$ satisfying $\WF(\tilde{g})\subset \mc{H}_0$,~\eqref{e:WFCond} with $\mc{O}$ and $\mc{B}$ as above, and
$$
\|\tilde{g}\|_{\mac{3/2}}=o(1),\qquad \|\hsc P_{\rm DtN}^+\tilde{g}\|_{\LtG}=1.
$$
Then for all $\hsc \notin \cJ$, there are $v,f\in L^2(\Gamma)$ such that $f$ satisfies~\eqref{e:goodWave} and
$$
A_{k}'v=f,\qquad \|f\|_{\LtG}\leq C\frac{\|\tilde{g}\|_{\LtG}}{\|\hsc P_{\rm DtN}^+\tilde{g}\|_{\LtG}},\qquad \|v\|_{\LtG}=1.
$$

\end{lemma}
\begin{proof}
Let $q\in \WF(\tilde{g})$ and $0<T_q<T_{\mc{B}\cup\{q\}}$ such that $\varphi_{T_{q}}^{\Omega^-}(q)\in \mc{O}\setminus (\mc{B}\cup \{q\})$.

Let $V_q^1\Subset V_q^2\subset \mc{O}\cap (B(\varphi_{T_q}^{\Omega_+}(q),\e)\setminus (\mc{B}\cup\{q\}))$ be a neighborhood of $\varphi_{T_q}^{\Omega^+}(q)$ and $U_q$ neighborhoods of $q$ such that 
$$
\inf_{q'\in U_q}T^\mu_{V^1_q<(\mc{B}\cup \{q'\})}(q')<\infty.
$$
Such a neighborhood exists because $\varphi_t^{\Omega^+}$ is continuous and $\varphi_{\mu T(q)}^{\Omega^+}(q)\in \mc{H}\setminus (\mc{B}\cup\{q\})$.

Since $\WF(\tilde{g})\cap \mc{Z}$ is compact, there are $\{q_1,\dots q_N\}$ such that 
$$
\WF(\tilde{g})\subset \bigcup_{i=1}^NU_{q_i}. 
$$
Let $\{\psi_j\}_{j=1}^N\in C_c^\infty(T^*\Gamma)$ with $\sum_{j}\psi_j\equiv 1 $ on $\WF(\tilde{g})\cap \mc{Z}$ and $\supp \psi_j \subset U^1_{q_j}$. Then,
by Lemma~\ref{l:awayGamma-}, 
$$
\|P_{\rm DtN}^+ u-\sum_j P_{\rm DtN}^+Op(\psi_j)\tilde{g}\|_{\mac{1/2}}\leq C\|\tilde{g}\|_{\mac{3/2}}.
$$
In particular, there is $i\in\{ 1,\dots, N\}$ such that 
\begin{equation*}
\|P_{\rm DtN}^+\Op(\psi_{i})\tilde{g}\|_{L^2}\geq \frac{1}{N}\|P_{\rm DtN}^+ \tilde{g}\|_{L^2}\geq \frac{c}{N},\qquad 
\| \Op(\psi_i)\tilde{g}\|_{L^2}\leq C\|\tilde{g}\|_{L^2}=o(1). 
\end{equation*}

Let $Q\in \Psi^{\comp}(\Gamma)$ with $\sigma(Q)\geq 0$ 
$$
\WF(Q)\subset V^1_{q_i},\qquad \WF(Q-I)\cap V^2_{q_i}=\emptyset ,
$$
and $Q_0\in \Psi(\Gamma)$ with $0\leq \sigma(Q_0)\leq 1$, 
\begin{equation}
\label{e:avoidFlowOut}
\WF(Q_0)\cap \bigcup_{q\in U_{q_i}}\bigcup_{0\leq t\leq T_{V_{q_i}^1}(q)}\varphi_{t}^{\Omega_-}(q)\cap \Gamma_-=\emptyset,
\end{equation}
 and $\WF(Q_0-I)\cap \{\xi'=0\}=\emptyset$. 
Then, by Lemma~\ref{l:approxSolve}, there is $u$ satisfying
$$
\begin{cases}
(-\hsc^2\Delta-1)u=O(\hsc^{\infty}\|g\|_{L^2})_{L^2}&\text{in }\Omega^-\\
((Q_0-I)Q\tilde{\Lambda}^{-1}+Q_0\eta^{-1})\hsc D_\nu u-u=g&\text{ on }\tilde{\Gamma},
\end{cases}
$$
where $\tilde{\Lambda}\in \Psi^{-1}$ with $\Lambda^{-1}$ as in Lemma~\ref{l:approxSolve}
$$
\WF(\tilde{\Lambda}^{-1}-\Lambda^{-1})\cap \WF(Q)=\emptyset, \qquad \sigma(\tilde{\Lambda}^{-1})\geq c\langle \xi'\rangle^{-1}>0.
$$ 
Moreover, 
\begin{gather*}
\| u\|_{L^2(\tilde{\Omega}_-)}+\|u|_{\Gamma}\|_{\mac{3/2}}+\|\hsc \partial_{\nu}u|_{\Gamma}\|_{\mac{1/2}}\leq C\| g\|_{\mac{3/2}},\\
{}^b\WF(u)\subset U_{q_i}\cup \bigcup_{q'\in U_{q_i}}\bigcup_{0\leq t\leq T^1_{V_{q_i}^1}(q')}\varphi_t^{\Omega_-}(U_{q_i}\cap \mc{Z}),
\end{gather*}
and using $\WF(Q_0)\cap \WF(g)=\emptyset$ and~\eqref{e:avoidFlowOut}
$$
\WF((\hsc D_\nu -\eta)u|_{\Gamma})\cap \{\xi'=0\}=\emptyset,\qquad \WF(u|_{\tilde{\Gamma}}-g)\cap \Gamma_-=\emptyset
$$
Using Lemma~\ref{l:awayGamma-} this implies
$$
\|P_{\rm DtN}^+P_{\rm ItD}^{-,\eta}(\hsc D_{\nu}-\eta)u|_{\Gamma}\|\geq \|\hsc P_{\rm DtN}^+g\|_{\LtG}-C\|g\|_{\LtG},
$$
and 
$$
\|(\hsc D_\nu -\eta)u\|_{\LtG}\leq C\|g\|_{\LtG}.
$$

Thus, putting $f=(\hsc D_{\nu}-\eta)u$, we have
\begin{align*}
\|(A_{k,\eta}')^{-1}f\|_{\LtG}&\geq \|(P_{\rm DtN}^+P_{\rm ItD}^{-,\eta})f\|_{\LtG}-C\|f\|_{\LtG}\\
&\geq c\|\hsc P_{\rm DtN}^+\tilde{g}\|_{\LtG}/\|\tilde{g}\|_{\LtG}\|f\|_{\LtG}-C\|f\|_{\LtG},
\end{align*}
which proves the claim.
\end{proof}

\subsection{Pollution in concrete situations}

\subsubsection{A good quasimode on the diamond domain: Proof of Theorem~\ref{t:diamond}}
In this section, we construct a quasimode that yields quantitative pollution. Let $0<\e<\pi/2$ and $\Omega_0\subset (-\pi+\e,\pi-\e)\times(-\pi+\e,\pi-\e)$ be convex with smooth boundary and 
$$
\partial\Omega_0\supset \cup_{\pm}\big(\{\pm(\pi-\e)\}\times (-\pi+2\e,\pi-2\e)\cup (-\pi+2\e,\pi-2\e)\times \{\pm(\pi-\e)\}\blue{\big)}
$$
\blue{(i.e., $\Omega_0$ is a ``rounded square")}.
Then, define
\begin{equation}
\label{e:diamondDomain}
\Omega^-:= \Omega_0+\{(-2\pi+\e,0),(0,2\pi -\e), (0,-2\pi+\e), \blue{(2\pi-\e,0)\}}
\end{equation}
\blue{(i.e., $\Omega^-$ consists of the ``rounded square" $\Omega_0$ placed at four points.)}
Observe that $\Omega^-$ has connected complement and smooth boundary, $\Gamma$, such that 
\begin{gather*}
\Omega^-\cap \{ (x,y)\in [-\pi,\pi]\times [-\pi,\pi]\}=\emptyset, \\
 [-\pi+2\e,\pi-2\e]\times\{-\pi,\pi\}\cup  \{-\pi,\pi\}\times [-\pi+2\e,\pi-2\e]\subset\Gamma.
\end{gather*}
Now, let $0<\delta<\e$, $\chi\in C_c^\infty((2\delta,2\pi-2\delta);[0,1])$, and define
$$
u:= \Big[\chi(y-x)e^{in(x+y)} -\chi(2\pi-y-x)e^{in(x-y)} +\chi(x-y)e^{in(-x-y)}  -\chi(2\pi+y+x)e^{in(-x+y)}\Big]1_{[-\pi,\pi]}(x)1_{[-\pi,\pi]}(y)
$$
Notice that 
\begin{gather*}
\supp \chi(y-x)\cap[-\pi,\pi]\times[-\pi,\pi]\subset  [-\pi,\pi-2\delta)\times ( -\pi+2\delta,\pi]\\
\supp \chi(2\pi-y-x)\cap[-\pi,\pi]\times[-\pi,\pi]\subset  (-\pi+2\delta,\pi]\times ( -\pi+2\delta,\pi]\\
\supp \chi(x-y)\cap[-\pi,\pi]\times[-\pi,\pi]\subset  (-\pi+2\delta,\pi]\times [ -\pi,\pi-2\delta)\\
\supp \chi(2\pi+y+x)\cap[-\pi,\pi]\times[-\pi,\pi]\subset  [-\pi,\pi-2\delta)\times [ -\pi,\pi-2\delta),
\end{gather*}
and $u|_{\Gamma}=0$. Hence, $u\in H_0^1(\Omega^+)\cap C_c^\infty(\overline{\Omega^+}).$

Set $\hsc_n=\frac{1}{\sqrt{2}n}$,
\begin{multline*}
\tilde{g}:=\frac{\hsc_n}{2}\Big[\chi''(y-x)e^{in(x+y)} -\chi''(2\pi-y-x)e^{in(x-y)} \\+\chi''(x-y)e^{in(-x-y)}  -\chi''(2\pi+y+x)e^{in(-x+y)}\Big]1_{[-\pi,\pi]}(x)1_{[-\pi,\pi]}(y),
\end{multline*}
so that for $|\e|\leq C$, 
$$
(-\hsc_n^2\Delta  -1+\e \hsc_n^2)u=  \hsc_n^2(\tilde{g}+\e)u.
$$
Observe that 
\begin{align*}
({}^b\WF(u)\cup{}^b\WF(g))\cap T^*\{x=-\pi\} &\subset \{ (y,\tfrac{1}{\sqrt{2}})\,:\, y\in \supp \chi(y+\pi)\}\\
({}^b\WF(u)\cup{}^b\WF(g))\cap T^*\{y=\pi\} &\subset \{ (x,\tfrac{1}{\sqrt{2}})\,:\, x\in \supp \chi(\pi-x)\}\\
({}^b\WF(u)\cup{}^b\WF(g))\cap T^*\{x=\pi\} &\subset \{ (y,-\tfrac{1}{\sqrt{2}})\,:\, y\in \supp \chi(\pi-y)\}\\
({}^b\WF(u)\cup{}^b\WF(g))\cap T^*\{y=-\pi\} &\subset \{ (x,-\tfrac{1}{\sqrt{2}})\,:\, x\in \supp \chi(\pi+x)\}\\
({}^b\WF(u)\cup{}^b\WF(g))\cap T^*(-\pi,\pi)\times (-\pi,\pi)&\subset \bigcup_{\pm}\{ (\xi,\eta)=\pm(\tfrac{1}{\sqrt{2}},\tfrac{1}{\sqrt{2}}), \pm(x-y)\in \supp \chi \}\\
&\qquad \cup \bigcup_{\pm}\{ (\xi,\eta)=\pm(\tfrac{1}{\sqrt{2}},-\tfrac{1}{\sqrt{2}}), 2\pi \pm(-x-y)\in \supp \chi \},
\end{align*}
and hence
\begin{equation}
\label{e:WFgSquare}
\sup_{q\in {}^b\WF(u)}T_{\mc{H}_0<\{q\}}<\infty.
\end{equation}
 Moreover, for $\delta>0$ small enough, since $\Omega_0$ is convex, 
\begin{equation}
\label{e:WFuSquare}
{}^b\WF(u)\cap \bigcup_{t\geq 0}(S^*\Gamma \cup N^*\Gamma)=\emptyset.
\end{equation}
Hence, there is $B$ satisfying~\eqref{e:uIsBig} and~\eqref{e:notNormal} such that~\eqref{e:WFCond} and~\eqref{e:uCond} hold. Thus, $u$ and $g$ satisfy the hypotheses of both Lemma~\ref{l:bWave} and~\ref{l:bWave2New}. This implies that for $\hsc = \hsc_n/\sqrt{1+\e\hsc_n^2}$, there are $f_1,f_2\in L^2(\Gamma)$ with $f_i$ satisfying~\eqref{e:goodWave} such that
$$
\|A_{k}^{-1}g_1\|_{\LtG}\geq c\hsc^{-1}\|g_1\|_{\LtG},\qquad \|(A_{k}')^{-1}g_2\|_{\LtG}\geq c\hsc^{-1}\|g_2\|_{\LtG}.
$$
Hence, Theorems~\ref{t:diamond} and~\ref{t:diamondNew} follow from Theorem~\ref{t:pollutionIntro}.

\subsubsection{Non-quantitative pollution}

To prove the non-quantitative pollution results, we first find appropriate data $g$ with wavefront set near a point in $\Gamma_-\setminus K$ such that $R_D \hsc g$ is ``large". 
Using the fact that $\Gamma_-\setminus K$ avoids the normal bundle (Lemma \ref{l:incomingIsNotNormal}), 
we then apply Lemmas \ref{l:bWave} and \ref{l:bWave2} to produce quasimodes for $A_k$ and $A_k'$, respectively.

\begin{lemma}[Growing Dirichlet resolvents with data near a point in $\Gamma_-$]
\label{l:itGrowsGamma-}
Let $\Omega^-\Subset \mathbb{R}^d$ with smooth boundary, $\Gamma$ and connected complement $\Omega^+:=\mathbb{R}^d\setminus \overline{\Omega^-}$ and suppose that $K\neq \emptyset$ and Assumption~\ref{ass:poly} holds.
Then, for any $q\in \Gamma_-\setminus T^*\Gamma$ and $\delta>0$, there are $g\in L^2_{\comp}(\Omega^+)$ and $\chi\in C_c^\infty(\overline{\Omega^+})$ such that 
\begin{equation}
\label{e:growing1}
\WF(g)\subset \big\{\varphi_t^{\Omega^+}(q)\,:\, -2\delta<t<-\delta\big\},\qquad \|g\|_{L^2}=1,
\end{equation}
and
\begin{equation}
\label{e:growing2}
\lim_{h\to 0^+, \hsc \notin \tilde{\cJ}}\|\chi R_D\hsc g\|_{L^2}=\infty.
\end{equation}
\end{lemma}
\begin{proof}
Let $(x_0,\xi_0)=q\in \Gamma_-\setminus (K\cup T^*\Gamma)$. Without loss of generality, we may assume $x_0=0$ and $\xi_0=(1,0,\dots,0)$.

We start by finding $\delta>0$ and $v\in H^2_{h,\comp}(\Omega^+)$, $g\in L^2_{\comp}$ such that $B(0,\delta)\subset \Omega^+$, and for $\chi\in C_c^\infty(B(0,\delta))$ with $0\notin \supp (1-\chi)$, 
\begin{equation}
\label{e:R0Solve}
\begin{gathered}
(-\hsc^2\Delta-1)v=\hsc g,\qquad \|g\|_{L^2}\leq C,\qquad \|\chi v\|_{L^2}\geq c\\
\WF(g)\subset \{(x_1,0,\dots,0),(1,0,\dots,0)\,:\, \delta<|x_1|<2\delta\},\\ \WF(v)\subset \{(x_1,0,\dots,0),(1,0,\dots,0)\,:\, <|x_1|<2\delta\}.
\end{gathered}
\end{equation}

To do this, recall that by~\cite[Theorem 12.3]{Zw:12}, there is a neighborhood $U$ of $(0,0)$, a symplectomorphism $\kappa:U\to T^*\mathbb{R}^d$  such that $\kappa^*(|\xi|^2-1)=\xi_1$, operators $T_1,T_2$ and  such that for any $u$ with $\WF(u)\subset U$, and $v$ with $\WF(v)\subset \kappa(U)$,
\begin{gather*}
T_1(-\hsc^2\Delta -1)T_2u=2\hsc D_{x_1} u+O(\hsc^{\infty})_{\Psi^{-\infty}}u,\\
T_1T_2u=u+O(\hsc^{\infty})_{\Psi^{-\infty}}u,\qquad T_2T_1v=v+O(\hsc^{\infty})_{\Psi^{-\infty}}v.\\ 
\WF(T_2u)\subset \kappa(\WF(u)).
\end{gather*}
 
 Let $\chi \in C_c^\infty((-2,2))$ with $\supp (1-\chi)\cap [-1,1]=\emptyset$ near $0$. Define 
$$
u:= \hsc^{-\frac{d-1}{4}}e^{-|x'|^2/2\hsc}\chi(\delta^{-1}x_1). 
$$
Then,
$$
2\hsc D_{x_1}u=2\hsc\delta^{-1}\hsc^{-\frac{d-1}{4}}e^{-|x'|^2/2\hsc}\chi'(\delta^{-1}x_1),
$$
so that
$$
\WF(u)\subset \{ |x_1|<2\delta, x'=0,\xi=0\},\qquad \WF(\hsc D_{x_1}u)\subset\{ \delta<|x_1|<2\delta, x'=0,\xi=0\}.
$$
Putting $v:=T_2u$, we have 
\begin{gather*}
\WF(v)= \kappa( \{ |x_1|<2\delta, x'=0,\xi=0\})\subset  \{ |x_1|<2\delta, x'=0,\xi=(1,0,\dots,0)\},\\
\WF((-\hsc^2\Delta-1)v)\subset\kappa( \{ \delta<|x_1|<2\delta, x'=0,\xi=0\})\subset  \{\delta< |x_1|<2\delta, x'=0,\xi=(1,0,\dots,0)\},
\end{gather*}
and for any $\psi\in C_c^\infty(B(0,\delta))$, with $\supp (1-\psi)\cap B(0,\delta/2)=\emptyset$,
$$
\|(-\hsc^2\Delta-1)v\|_{L^2}\leq C\hsc,\qquad \|\psi v\|_{L^2}\geq c>0.
$$
By inserting appropriate cutoffs equal to one near 0, we may assume that $v$ and hence also $g:=\hsc^{-1}(-\hsc^2\Delta-1)v$ have compact support. Hence, we have achieved~\eqref{e:R0Solve}.

Let $\psi\in C_c^\infty (0,1)$ with $\supp (1-\psi)\cap (\delta,2\delta)=\emptyset$. We claim that there is $\chi \in C_C^\infty(\overline{\Omega^+})$ such that 
\begin{equation}
\label{e:itGrows}
\frac{\|g\|_{L^2}}{\|\chi R_D\hsc\psi(x_1)g\|_{L^2}}= o(1).
\end{equation}

Suppose by contradiction that for any $\psi\in C_c^\infty(\overline{\Omega^+})$, there are $C>0$  and $\hsc_n\to 0$ such that 
$$
\|\psi R_D(\hsc_n)\hsc_n\psi(x_1)g(\hsc_n)\|_{L^2}\leq C\|g(\hsc_n)\|_{L^2}.
$$
Then, without loss of generality we may assume $\|g\|_{L^2}=1$ and, setting $w:=v-R_D h\psi g$, up to extracting a further subsequence, we may assume that $w$ has defect measure $\mu$. By Lemma~\ref{l:outgoing}, 
$$
 (0,(1,0,\dots,0))\notin {}^b\WF(\psi R_D(\hsc_n)\hsc_n\psi(x_1)g(\hsc_n))
$$
and hence, by~\eqref{e:R0Solve} and, using Lemma~\ref{l:outgoing} again
$$
\mu(\{ ((x_1,0,\dots, 0),(1,0,\dots,0)\,:\, |x_1|<\delta\})>c>0. 
$$
Now, 
$$
(-\hsc^2\Delta_g-1)w= (1-\psi)\hsc g
$$
and hence by Lemma~\ref{l:invariance}, $\mu$ is invariant under $\varphi_t^{\Omega^+}$ away from $\{ ((x_1,0,\dots,0),(1,0,\dots,0))\,:\, -2\delta<x_1<-\delta\}$. Since $q\in \Gamma_-$, this implies that there is $R>0$ such that 
$$
\mu (|x|<R)=\infty,
$$
a contradiction. Hence~\eqref{e:itGrows} holds, completing the proof.
\end{proof}

\begin{lemma}[Quasimodes for $A_k$ under geometric assumptions]
\label{l:qualitative1}
Let $\Omega^-\Subset \mathbb{R}^d$ with smooth boundary, $\Gamma$, and connected complement $\Omega^+:=\mathbb{R}^d\setminus \overline{\Omega^-}$ and suppose that $K\neq \emptyset$, Assumption~\ref{ass:poly} holds and 
\begin{equation}
\label{e:KnoGlance}
K\cap \bigcup_{t\geq 0} \varphi_t^{\mathbb{R}^d}(S^*\Gamma)=\emptyset.
\end{equation}
Then, if Assumption~\ref{ass:parameters} holds there are $v,f\in L^2(\Gamma)$ such that $f$ satisfies~\eqref{e:goodWave}, 
$$
A_{k}v=f, \qquad \|v\|_{L^2}=1,\quad \lim_{h\to 0^+,h\notin\cJ}\|f\|_{L^2}=0.
$$
Moreover, there is $B\in \Psi^{\comp}(\Omega_+)$ satisfying~\eqref{e:uIsBig}
such that for all $q\in \Gamma_-\setminus (K\cup T^*\Gamma)$, there is $g$ satisfying~\eqref{e:growing1} and
$$
\lim_{\hsc\to 0^+,\hsc \notin \tilde{\cJ}}\|BR_D\hsc g\|_{L^2}=\infty. 
$$
\end{lemma}
\begin{proof}
By Lemma~\ref{l:incoming} there is $q\in \Gamma_-\setminus K$ and, since $|\varphi_{-t}^{\Omega_+}(q)|\to \infty$ as $t\to \infty$, we may assume without loss of generality that $q\notin T^*\Gamma$.  Then, by Lemma~\ref{l:itGrowsGamma-} there is $g$ satisfying~\eqref{e:growing1} and~\eqref{e:growing2}. Since $q\in \Gamma_-$, there is $T>0$ such that $\varphi_t^{\Omega_+}(q)\in T^*\Gamma \cap \Gamma_-\setminus K$ and hence, by Lemma~\ref{l:incomingIsNotNormal} $\varphi_t^{\Omega^+}(q)\notin \{\xi'=0\}$. Using that $K\cap S^*\Gamma=\emptyset$ and increasing $t$ if necessary, there is $t>0$ such that  $\varphi_t^{\Omega^+}(q)\in\mathcal{H}\cap \{\xi'\neq 0\}$. 

In particular,
$$
\sup_{q\in{}^b\WF(g)}T_{\mc{H}_0<\{q\}}^1(q)<\infty.
$$
Now, by~\eqref{e:KnoGlance} for any $R>0$ there are $U\Subset U_1\Subset T^*\Omega^+$ open with 
$$
U_1\cap \bigcup_{t\geq 0}\varphi_t^{\mathbb{R}^d}(S^*\Gamma)=\emptyset
$$  
and $T>0$ such that for all $q\in \mc{Z}\cap B(0,R)$, there is $0\leq s\leq T$ such that $\varphi_{-s}^{\Omega^+}(q)\in U\cup \{|x|>R\}$. In particular, by Theorem~\ref{t:basicPropagate} letting $B\in \Psi^{\comp}(\Omega^+)$ with $\WF(I-B)\cap U=\emptyset$ and $\WF(B)\subset U_1$,
$$
\|\chi R_D\hsc g\|_{L^2}\leq C\|B R_D\hsc g\|_{L^2}+\|g\|_{L^2}.
$$
In particular, using~\eqref{e:growing2}, we have
$$
\lim_{\hsc\to 0^+,\hsc \notin \tilde{\cJ}} \|BR_D\hsc g\|_{L^2}=\infty.
$$
Thus, we have verified the hypotheses of Lemma~\ref{l:bWave} which completes the proof.
\end{proof}

\begin{lemma}[Quasimodes for $A_k'$ under geometric assumptions]
\label{l:qualitative2}
Let $\Omega^-\Subset \mathbb{R}^d$ with smooth boundary, $\Gamma$ and connected complement $\Omega^+:=\mathbb{R}^d\setminus \overline{\Omega^-}$, that $K\neq \emptyset$, and Assumption~\ref{ass:poly} and~\eqref{e:KnoGlance} hold. Suppose in addition that there is  $q_0\in \Gamma_-\cap \mc{H}\setminus K$ such that
\begin{equation}
\label{e:impedanceEscape}
T^1_{\mc{H}\setminus (\Gamma_-\cup \{\xi'=0\})<\Gamma_-\cap \{\xi'=0\}}(q_0)<\infty.
\end{equation}
Then, if Assumption~\ref{ass:parameters} holds, there are $v,f\in L^2(\Gamma)$ such that $f$ satisfies~\eqref{e:goodWave}
$$
A'_{k}v=f,\qquad \|v\|_{L^2}=1,\quad \lim_{\hsc\to 0^+,\hsc \notin \tilde{\cJ}}\|f\|_{L^2}=0.
$$
\end{lemma}
\begin{proof}
Since $q_0\in \mc{H}$, there is $t>0$ such that $\varphi_{-t}^{\Omega_+}(q_0)\in \Gamma_-\setminus (K\cup T^*\Gamma)$. Hence, by Lemma~\ref{l:qualitative1}, there are $B$ satisfying~\eqref{e:uIsBig} and $g$ satisfying~\eqref{e:growing1} with $q=\varphi_{-t}^{\Omega^+}(q_0)$ such that 
$$
\lim_{\hsc\to 0^+,\hsc\notin \tilde{\cJ}}\|BR_D\hsc g\|_{L^2}=\infty.
$$
Now, by construction 
$$
\sup_{q\in{}^b\WF(g)}T_{\mc{H}_0<\{q\}}^1(q)<\infty,
$$
and hence by Lemma~\ref{l:quasimodeExteriorToDirichlet} for any $\e>0$, there is $g_1\in L^2(\Gamma)$ with 
$$
\|P^{+}_{\rm DtN}g_1\|_{L^2}=1,\qquad \|g_1\|_{L^2}=o(1),\qquad \WF(g_1)\subset B(q_0,\e).
$$

By~\eqref{e:impedanceEscape}, and upper semicontinuity of $T^1_{\mc{H}\setminus (\Gamma_-\cup \{\xi'=0\})<\Gamma_-\cap \{\xi'=0\}}$ there is $\e>0$ small enough such that 
$$
\sup_{q'\in \WF(g_1)}T^1_{\mc{H}\setminus (\Gamma_-\cup \{\xi'=0\})<\Gamma_-\cap \{\xi'=0\}}(q')<\infty
$$
and hence the Lemma follows from Lemma~\ref{l:bWave2}.
\end{proof}

The following two Theorems follow directly from Lemmas~\ref{l:qualitative1},~\ref{l:qualitative2}, and ~\ref{l:waveToKOscillate}, and Theorem~\ref{t:pollutionIntro}.

\begin{theorem}
\label{t:qualitativeA}
Let $\Omega^-\Subset \mathbb{R}^d$ with smooth boundary, $\Gamma$, and connected complement $\Omega^+:=\mathbb{R}^d\setminus \overline{\Omega^-}$ and suppose that $K\neq \emptyset$, Assumptions~\ref{ass:parameters} and~\ref{ass:poly}, and~\eqref{e:KnoGlance} hold and 
$$
\text{either }\qquad p=0\qquad \text{or}\qquad P_{\rm inv}\leq p+1.
$$
Then the assumptions of Theorem~\ref{t:pollutionIntro} with $\operator=A_{k}$ hold for some $\e_0>0$, $\Xi_0=1$, all $\hsc \notin \tilde{\cJ}$, and some $\alpha =o(1)$, $\beta=O(\hsc^{\infty})$.
\end{theorem}
Theorem~\ref{t:qualitativeA} implies Theorem~\ref{t:qualitative1}. 
\begin{theorem}
\label{t:qualitativeAp}
Let $\Omega^-\Subset \mathbb{R}^d$ with smooth boundary, $\Gamma$, and connected complement $\Omega^+:=\mathbb{R}^d\setminus \overline{\Omega^-}$, suppose that $K\neq \emptyset$ and Assumptions~\ref{ass:parameters} and~\ref{ass:poly} hold, ~\eqref{e:KnoGlance}, there is  $q_0\in \Gamma_-\cap \mc{H}\setminus K$ such that ~\eqref{e:impedanceEscape} holds, and 
$$
\text{either }\qquad p=0\qquad \text{or}\qquad P_{\rm inv}\leq p+1.
$$
Then the assumptions of Theorem~\ref{t:pollutionIntro} with $\operator=A'_{k,\eta}$ hold for some $\e_0>0$, $\Xi_0=1$, all $\hsc \notin \cJ$, and some $\alpha =o(1)$, $\beta=O(\hsc^{\infty})$.
\end{theorem}

\section{Pollution for the BIEs on the disk}\label{s:disk}

We now turn to the study of pollution for BIEs on the unit disk, $B(0,1)\subset \mathbb{R}^2$
and prove Theorems \ref{t:neumannDisk}, and~\ref{t:dirichletDisk}. 
 We parametrize use $\mathbb{R}/2\pi \mathbb{Z}$ as coordinates on $\partial B(0,1)$  with the coordinate map 
$\gamma(t)=(\cos(t),\sin(t))$. We record the following description of the single layer, double layer, and hypersingular operators as Fourier multipliers~(see, e.g., \cite{Kr:85}).
\begin{lemma}
Let $\Omega^-=B(0,1)$, $\Omega^+:=\mathbb{R}^2\setminus \overline{B(0,1)}$. Then, for any $\tau \neq 0$,
\begin{equation}
\label{e:diskOperators}
\begin{aligned}
S_{\tau}e^{imt}&=\frac{\pi i}{2} H_{|m|}^{(1)}(\tau)J_{|m|}(\tau) e^{imt}\\
\big(\tfrac{1}{2}I+K'_{\tau}\big)e^{imt}=\big(\tfrac{1}{2}I+K_{\tau}\big)e^{imt}&= \frac{\pi i}{2} \tau H_{|m|}^{(1)}(\tau)J_{|m|}'(\tau)e^{imt}\\
H_\tau e^{imt}&=\blue{-\frac{\pi}{2}}i\tau^2(H_{|m|}^{(1)})'(\tau)J_{|m|}'(\tau)e^{imt}.
\end{aligned} 
\end{equation}
In particular, $A_{k,\eta}e^{imt}=A_{k,\eta}'e^{imt}=\lambda_m e^{imt}$ and $B_{k,\rm{reg}}e^{imt}=B_{k,\rm{reg}}'e^{imt}=\mu_m e^{imt}$ with
\begin{align*}
\lambda_m&:= \frac{\pi}{2}\Big(i k H_{|m|}^{(1)}(k)J_{|m|}'(k) +\eta H_{|m|}^{(1)}(k)J_{|m|}(k)  \Big)\\
\mu_m&:= i\eta (1-  \frac{\pi i}{2} k H_{|m|}^{(1)}(k)J_{|m|}'(k)) \blue{+\frac{\pi^2 k^2}{4}} H_{|m|}^{(1)}(ik)J_{|m|}(ik)(H_{|m|}^{(1)})'(k)J_{|m|}'(k).
\end{align*}
\end{lemma}

We require asymptotic expansions for Bessel functions uniformly for large order~\cite[\S 10.20]{Di:25}. Define the decreasing, smooth bijection $\zeta :(0,\infty ) \to (-\infty,\infty)$ by
\begin{equation} \label{e:zeta}
\zeta(z):=\begin{cases}\Big(\frac{3}{2}\int_z^1\frac{\sqrt{1-t^2}}{t}dt\Big)^{2/3}&0<z\leq 1\\
-\Big(\frac{3}{2}\int_1^z\frac{\sqrt{t^2-1}}{t}dt\Big)^{2/3}&1\leq z<\infty,
\end{cases}
\end{equation}
and the Airy function by 
$$
Ai(x):=\frac{1}{\pi}\int_0^\infty \cos \Big(\frac{1}{3}t^3+xt\Big)dt.
$$
Then, for any $I\Subset (0,\infty)$, and all $z\in I$, \blue{with $\zeta=\zeta(z)$ as in~\eqref{e:zeta}}
\begin{align*}
J_{m}(mz)&=\Big(\frac{4\zeta}{1-z^2}\Big)^{\frac{1}{4}}\Big(m^{-1/3}Ai(m^{2/3}\zeta)(1+O_I(m^{-2}))+O_{I}(m^{-5/3}Ai'(m^{2/3}\zeta))\Big)\\
H^{(1)}_{m}(mz)&=2e^{-\pi i/3}\Big(\frac{4\zeta}{1-z^2}\Big)^{\frac{1}{4}}\Big(m^{-1/3}Ai(e^{2\pi i/3}m^{2/3}\zeta)(1+O_I(m^{-2}))+O_{I}(m^{-5/3}Ai'(e^{2\pi i/3}m^{2/3}\zeta))\Big)\\
J'_{m}(mz)&=-\frac{2}{z}\Big(\frac{4\zeta}{1-z^2}\Big)^{-\frac{1}{4}}\Big(m^{-2/3}Ai'(m^{2/3}\zeta)(1+O_I(m^{-2}))+O_{I}(m^{-4/3}Ai(m^{2/3}\zeta))\Big)\\
(H^{(1)}_{m})'(mz)&=\frac{4e^{-2\pi i/3}}{z}\Big(\frac{4\zeta}{1-z^2}\Big)^{-\frac{1}{4}}\Big(m^{-2/3}Ai'(e^{2\pi i/3}m^{2/3}\zeta)(1+O_{I}(m^{-2}))+O_{I}(m^{-4/3}Ai(e^{2\pi i/3}m^{2/3}\zeta))\Big),
\end{align*}
\blue{where $a=O_I(b)$ if there is $C>0$ such that for all $z\in I$ and all $m$
$$
|a(z,m)|\leq C|b(z,m)|.
$$}

We also recall the following estimates for the Airy function~\cite[\S 9.8]{Di:25}
$$
|Ai(x)|\leq C\langle x\rangle^{-1/4},\qquad |Ai'(x)|\leq C\langle x\rangle^{1/4}.
$$

\begin{lemma}
\label{l:neumannBIEDisk}
Let $0<\zeta_1<\zeta_2<\dots$ such that $-\zeta_j$ are the zeros of $Ai$. Then, for all $k_0>0$ and $j\in \mathbb{N}$ there is $C>0$ such that for all $0<\e<1$, $k>k_0$, and $m\in\mathbb{Z}\setminus\{0\}$ satisfying
$$
|\zeta(k/|m|)+\zeta_j|m|^{-2/3}|< \e k^{-2/3},
$$
we have
$$
|\mu_m|\leq C (k^{-1/3} +(\e+k^{-2/3}) |\eta|),\qquad C^{-1}k\leq |m|\leq Ck.
$$
\end{lemma}
\begin{proof}
We first observe that since 
$$
\|k S_{ik}\|_{L^2\to L^2}\leq C, 
$$
we have 
$$
|kH_{|m|}^{(1)}(ik)J_{|m|}(ik)|\leq C. 
$$
Therefore, it is enough to check that 
\begin{equation}
\label{e:toCheckNeumann}
|k(H_{|m|}^{(1)})'(k)J_{|m|}'(k)|\leq Ck^{-1/3},\qquad |(1-  \frac{\pi i}{2} k H_{|m|}^{(1)}(k)J_{|m|}'(k)) |\leq C(\e+k^{-2/3}). 
\end{equation}

To do this, we first note that $\zeta^{-1}$, is smooth,  $|\zeta'|>c>0$, and $\zeta^{-1}(0)=1$. Hence,
\begin{equation}
\label{e:whereIsM}
\bigg|\frac{k}{|m|}-1\bigg|\leq C k^{-2/3},\qquad c<\bigg|\frac{1-(\frac{k}{|m|})^2}{4\zeta(k/|m|)}\bigg|\leq C.
\end{equation}
To obtain the first inequality~\eqref{e:toCheckNeumann}, we observe that 
\begin{align*}
&k(H_{|m|}^{(1)})'(k)J_{|m|}'(k)\\
&=-\frac{8k e^{-2\pi i/3}}{(k/|m|)^2}\Big(\frac{1-(k/|m|)^2}{4\zeta(k/|m|)}\Big)^{1/2}(|m|^{-4/3} Ai'(|m|^{2/3}\zeta(k/|m|))Ai'(e^{2\pi i/3}|m|^{2/3}\zeta(k/|m|))+O(k|m|^{-2})\\
&=O(k^{-1/3})
\end{align*}
and to obtain the second inequality in~\eqref{e:toCheckNeumann}, we use the fact that ~\cite[\S 9.2]{Di:25}
$$
Ai(e^{2\pi i/3}x)Ai'(x)-(Ai(e^{2\pi i/3}x))'Ai(x)=-\frac{e^{-\pi i/6}}{2\pi}
$$
and hence, since $Ai(\zeta_j)=0$ and $Ai$ is smooth,
\begin{align*}
&(1-  \frac{\pi i}{2} k H_{|m|}^{(1)}(k)J_{|m|}'(k))\\
&=(1+  2\pi i e^{-\pi i/3}Ai(e^{2\pi i/3}|m|^{2/3}\zeta(k/|m|))Ai'(|m|^{2/3}\zeta(k/|m|))+O(k^{-2/3}))\\
&=O(\e +k^{-2/3}).
\end{align*}
\end{proof}

\begin{lemma}
\label{l:dirichletBIEDisk}
Let $0<\zeta_1'<\zeta_2'<\dots$ such that $-\zeta_j'$ are the zeros of $Ai'$. Then, for all $k_0>0$, $c>0$ there is $C>0$ such that for all $0<\e<1$, $k_0<k$, $m\in\mathbb{Z}\setminus\{0\}$, and $j\in\mathbb{N}$ satisfying, 
$$
c k^{2/3}\leq |\zeta'_j|\leq 2c^{-1} k^{2/3},\qquad  |\zeta(k/|m|)+\zeta'_j|m|^{-2/3}|< \e k^{-1},
$$
we have
$$
|\lambda_m|\leq C (\e +k^{-1}+k^{-1} |\eta|),\quad C^{-1}k\leq |m|\leq Ck.
$$
\end{lemma}
\begin{proof}
To prove the lemma, we show that
\begin{equation}
\label{e:dirichletIneq}
|kH_{|m|}^{(1)}(k)J_{|m|}'(k)|\leq C(\e+k^{-2/3}),\qquad |H_{|m|}^{(1)}(k)J_{|m|}(k)|\leq Ck^{-1}.
\end{equation}

Since $\zeta:(0,\infty)\to (-\infty,\infty)$ is a smooth  decreasing bijection, with smooth inverse, and $\zeta(0)=1$, there are $c,C>0$ such that 
$$
1+c<\frac{k}{|m|}<C,\qquad -2\delta<\zeta(k/|m|)<-\delta.
$$
To obtain the first inequality in~\eqref{e:dirichletIneq}, we observe that
\begin{align*}
&kH_{|m|}^{(1)}(k)J_{|m|}'(k)\\
&=-4e^{-\pi i/3}Ai(e^{2\pi i/3}|m|^{2/3}\zeta(k/|m|))Ai'(|m|^{2/3}\zeta(k/|m|))+O(k|m|^{-2})\\
&=O(\e +k^{-1}).
\end{align*}
For the second inequality, 
\begin{align*}
&H_{|m|}^{(1)}(k)J_{|m|}(k)\\
&=2e^{-\pi i/3}\Big(\frac{4\zeta(k/|m|)}{1-(k/|m|)^2}\Big)^{\frac{1}{2}}|m|^{-2/3}Ai(|m|^{2/3}\zeta(k/|m|))Ai(e^{2\pi i/3}|m|^{2/3}\zeta(k/|m|))+O(m^{-2})\\
&\leq C|m|^{-1}\leq Ck^{-1}.
\end{align*}
\end{proof}

As an easy Corollary of Lemmas~\ref{l:neumannBIEDisk} and~\ref{l:dirichletBIEDisk} we obtain Theorems~\ref{t:neumannDisk} and~\ref{t:dirichletDisk}.

\begin{proof}[Proof of Theorem~\ref{t:neumannDisk}]
By~\cite[Theorem 2.3]{GaMaSp:21N}, for $|\eta|\sim 1$, 
$$
\|B_{k,\rm{reg}}^{-1}\|_{\LtGt}\leq Ck^{1/3}. 
$$
Hence, $B_{k,\rm{reg}}$ satisfies Assumption~\ref{ass:polyboundintro} with $\cJ=\emptyset$.  Now, given $k>k_0$, we must find $m$ such that the hypotheses of Lemma~\ref{l:neumannBIEDisk} hold. Let $|j|\leq C$, and $z_j:=\zeta^{-1}(-k^{-2/3}\zeta_j)$. Then, $z_j=1+O(k^{-2/3})$ and let $m\in\mathbb{Z}$  such that $| |m|-k/z_j|\leq \frac{1}{2}$. Then, since $ \frac{1}{4}k\leq |m|\leq \frac{3}{2}k$, 
$$
\Big|\frac{k}{|m|} -z_j\Big|\leq \frac{|z_j|}{2|m|}=\frac{1}{2|m|}(1+O(k^{-5/3}))= O(k^{-1}),
$$
and hence
\begin{equation}
\label{e:zeta1}
\zeta(k/|m|)+k^{-2/3}\zeta_j=O(k^{-1}). 
\end{equation}
Next, 
\begin{equation}
\label{e:zeta2}
||m|^{-2/3}\zeta_j-k^{-2/3}\zeta_j|=|\zeta_j||(k+O(1))^{-2/3}-k^{-2/3})|= O(k^{-5/3}).
\end{equation}
Combining~\eqref{e:zeta1} and~\eqref{e:zeta2} yields
$$
|\zeta(k/|m|)+m^{-2/3}\zeta_j|\leq Ck^{-1}.
$$
In particular, by Lemma~\ref{l:neumannBIEDisk}, 
$$
|\mu_{m}|\leq Ck^{-\frac{1}{3}}. 
$$
Choosing $\e_0=\frac{1}{8}$, $\Xi_0=2$, and using that $ \frac{1}{4}k\leq |m|\leq \frac{3}{2}k$,  and $-\Delta_{\Gamma}e^{imt}= m^2e^{imt}$, we have 
$$
\chi(-\e_0^{-2}k^{-2}\Delta_{\Gamma})e^{imt}=0,\qquad (1-\chi(-\Xi_0^{-2}k^{-2}\Delta_{\Gamma}))e^{imt}=0,
$$
and 
$$
\chi(-\e_0^{-2}k^{-2}\Delta_\Gamma) (B_{k,\rm reg}^*)^{-1}\mu_m e^{imt}=\chi(-\e_0^{-2}k^{-2}\Delta_\Gamma) \overline{\mu_m}^{-1}\mu_m e^{imt}=0.
$$
Hence, there is $C>0$ such that for any seqence $k_n\to \infty$,  the hypotheses of Theorem~\ref{t:pollutionIntro} with $\alpha_n \leq Ck_n^{-1/3}$ and $\beta_n=0$. This completes the proof.
\end{proof}

\begin{proof}[Proof of Theorem~\ref{t:dirichletDisk}]
By the bound on $P_{\rm ItD}^{-,\eta}$ from ~\cite[Theorem 4.3]{ChMo:08} and  the bound on $P_{\rm DtN}^+$ coming from the fact that $B(0,1)$ is nontrapping, for any $k_0$ there is $C>0$ such that for $k>k_0$,
$$
\|A_{k}^{-1}\|_{\LtGt}\leq C|\eta|^{-1}. 
$$
Hence, $A_{k}$ satisfies Assumption~\ref{ass:polyboundintro} with $\cJ=\emptyset$.  

Now, given $m\in\mathbb{Z}\setminus \{0\}$ large enough, we must find $k>0$ such that the hypotheses of Lemma~\ref{l:dirichletBIEDisk} hold. Let $\delta>0$ and $C_0>0$ such that 
\begin{equation}
\label{e:zetaSmall}
\zeta(x)\leq -3\delta x^{2/3},\qquad x>C_0.
\end{equation}
Such a $\delta>0$ exists since~\cite[\S 10.20]{Di:25}
$$
\zeta(x)=\Big(\frac{3}{2}(\sqrt{z^2-1}-\operatorname{arcsec}z\Big)^{2/3},\quad z>1.
$$
Now fix $j$ such that 
$$
\delta |m|^{2/3}\leq \zeta_j'\leq 2\delta |m|^{2/3},
$$
and define $f:(0,\infty)\to (-\infty,\infty)$ by $f(x):=|m|^{2/3}\zeta(x/|m|)$. Then, $f$ is smooth, $f(|m|)=0$ and, by~\eqref{e:zetaSmall}
$$
f(x)\leq -3\delta k^{2/3},\qquad x>C_0|m|
$$
and hence for there is $|m|<k_j\leq C_0|m|$ such that $f(k_j)=\zeta_j'$. 
In particular, by Lemma~\ref{l:neumannBIEDisk}, 
$$
|\lambda_{m}(k_j)|\leq C(1+|\eta|)k^{-1}. 
$$
Since $-\Delta_{\Gamma}e^{imt}= m^2e^{imt}$, and there is $C_0$ such that $C_0^{-1} k_j \leq |m|\leq C_0 k$, choosing $\e_0<\frac{1}{2}C_0^{-1}$ and $\Xi_0\geq 2C$ this implies the there are $k_m\to \infty$ and $\alpha_m\leq C|\eta_m|k_m^{-1}$ such that the hypotheses of Theorem~\ref{t:pollutionIntro} hold with $\beta_m=0$. This completes the proof.
\end{proof}

\section{Details of the numerical experiments  in \S\ref{sec:main}}
\label{sec:numerical}

We now describe the parametrizations of the trapping and nontrapping domains considered in \S\ref{sec:main} and provide a brief description of the numerical solver used for the Galerkin examples (the Nystr\"om method is discussed in Section~\ref{sec:Nystrom}).

\paragraph{Geometry parametrization:}
We solve the scattering problem for two nontrapping domains:~the unit disk and a star-shaped domain whose parametrization is given by
\begin{equation}
\begin{bmatrix}
x(t) \\
y(t)
\end{bmatrix}
= (1 + 0.3\cos{(t)}) 
\begin{bmatrix}
\cos{(t)} \\
\sin{(t)}
\end{bmatrix} \,, \qquad t\in[0,2\pi) \,. 
\end{equation} 
We consider two trapping domains. The first consists of $4$ rounded and tilted squares, which we refer to as the ``four diamonds'' geometry. Let $\Gamma_{0}$ be the boundary of the square with vertices $(\pm 4\sqrt{2}\pi/5, 0), (0, \pm 4\sqrt{2} \pi/5)$, and let $\widetilde{\Gamma}_{0}$ be $\Gamma_{0}$ but whose vertices are rounded using a Gaussian filter as described in~\cite{epstein2016smoothed}. Furthermore, let $\Gamma_{j} = (\pm\sqrt{2}\pi, \pm \sqrt{2}\pi)+ \tilde{\Gamma}_{0}$, $j=1,2,3,4$. Then the boundary of the four diamonds geometry is given by $\Gamma = \cup_{j=1}^{4} \Gamma_{j}$. The second trapping domain is a crescent shaped boundary that we refer to as the ``cavity'' geometry. Let $s \in [-\pi/2, \pi/2]$, $a=0.2$, and $b=\pi/12$. Let 
\begin{equation}
\begin{aligned}
\theta(s) &= b-a + 2\left(1-\frac{(b-a)}{\pi}\right) \left(\frac{a}{\sqrt{\pi}} e^{-(s/a)^2}+s \cdot \textrm{erf}(s/a) \right) \,, \\
r(s) &= 1 - a \cdot \textrm{erf}(s/a)\, .
\end{aligned}
\end{equation}
Finally, let $\Gamma$ be the union of the curve $(r(s) \sin{(\theta{(s)})}, r(s) \cos{(\theta(s))})$ and it's reflection about the $y$ axis. This results in a cavity whose opening is approximately a sector of $b$ radians, and whose width is approximately $2a$. 
Since this curve is not smooth for the specific choices of $a$ and $b$,
we sample the curve at $M=400$ equispaced points in the following manner. Let $s_{j} = -\frac{\pi}{2} + (j-1/2) \pi/M$, $j=1,2,\ldots M/2$, and consider
\begin{equation}
\begin{pmatrix}
x_{j}\\
y_{j} 
\end{pmatrix}
= 
\begin{cases}
\begin{pmatrix}
r(s_{j}) \sin(\theta{(s_{j})}) \\
r(s_{j}) \cos(\theta{(s_{j})}) \\
\end{pmatrix} \, ,& j =1,2,\ldots M/2 \vspace*{1ex}\\
\begin{pmatrix}
-r(s_{M-j+1}) \sin(\theta{(s_{M-j+1})}) \\
r(s_{M-j+1}) \cos(\theta{(s_{M-j+1})}) \\
\end{pmatrix} \, ,& j =M/2+1,2,\ldots M
\end{cases}\,.
\end{equation}
The boundary of the cavity domain is then defined to be the curve corresponding to the discrete Fourier series of the sampled curve above.

\paragraph{Galerkin Discretization:}
The geometries are discretized with $\nch$ equispaced panels in parameter space, denoted by $\Gamma_{j}$, $j=1,2,\ldots \nch$, sampled at $16$th order Gauss-Legendre nodes on each panel. Let $\gamma_{j}(t):[-1,1] \to \Gamma_{j}$ denote the parametrization of panel $\Gamma_{j}$.  
For an integral operator with kernel $K$, the Galerkin discretization requires accurate evaluation of the integrals
\begin{equation}
I_{i,j,k,\ell} = \int_{-1}^{1} \int_{-1}^{1} K(\gamma_{i}(t), \gamma_{j}(s)) P_{k}(t) P_{\ell}(s) |\gamma'_{i}(t)| |\gamma'_{j}(s)| ds\, dt \,,
\end{equation}
$i,j=1,2,\ldots \nch$, $k,\ell=0,\ldots, \ngk$, where $\ngk$ is the order of the Galerkin discretization, and $P_{\ell}(t)$ is the Legendre polynomial of degree $\ell$.
$I_{i,j,k,\ell}$ corresponds to the contribution from basis function $\ell$ on $\Gamma_{j}$ to basis function $k$ on $\Gamma_{i}$. In all the integral representations considered, the kernel $K(x,y)$ has at most a log-singularity as $|x-y| \to 0$. 

For the geometries considered and the equispaced discretization above, when panels $\Gamma_{i}$, and $\Gamma_{j}$ do not share a vertex, the integrand is smooth and a high-order Gauss-Legendre quadrature rule suffices to approximate $I_{i,j,k,\ell}$. In particular, we use a $24$th order Gauss-Legendre rule in both $t$ and $s$ to compute it. 

Suppose now that $\Gamma_{i}$ is adjacent to $\Gamma_{j}$ with $\gamma_{i}(1) = \gamma_{j}(-1)$. The integrand in $s$ has a near singularity close to $s=-1$. After computing the integral in $s$, the integrand in $t$ has a log-singularity at $t=1$.  To handle the log-singularity at the end point, we use a custom quadrature rule that accurately computes all integrals of the form
\begin{equation}
\int_{-1}^{1} \Big(\log{|1-s|} P_{\ell}(s) + \log{|1+s|} P_{m}(s) + P_{n}(s)\Big) \, ds ,
\end{equation}
with $0\leq \ell,m,n \leq q-1$. Let $Q_{\textrm{log}} = \{t_{j}, w_{j} \}$, $j=1,2,\ldots, N_{q}$, denote such a rule computed using Generalized Gaussian quadratures, see, e.g.,~\cite{bremer}. We use a rule with $q=20$ which results in a $N_{q}=24$ point rule. To handle the integrand in $s$, we subdivide the $[-1,1]$ into three equispaced panels and the panels at the ends are dyadically subdivided $4$ times. A $32$ point Gauss-Legendre quadrature rule is used on each of these panels resulting in a total $352$ quadrature nodes. 

Finally, when $\Gamma_{i} = \Gamma_{j}$, after computing the integral in $s$, the integrand in $t$ now has a log-singularity at $t=\pm1$. We use the quadrature rule $Q_{\textrm{log}}$ above to compute the integral in $t$. For any given $t_{j}$, in order to compute the 
integral in $s$, we use a mapped version of $Q_{\textrm{log}}$ on $[-1,t_{j}]$, and $[t_{j},1]$ to handle the log-singularity in the kernel.

\begin{remark}
Far fewer nodes would suffice to compute the integral in $s$, but the quadrature rule above guarantees that error in the solution computed using the Galerkin discretization above will not be dominated by the quadrature error of computing $I_{i,j,k,\ell}$. 
\end{remark}

\paragraph{Fast solver:}
The quadrature method described above can be easily coupled to fast multipole methods. 
The discretized linear system is solved using GMRES until the relative residual drops below $5 \times 10^{-8}$. The matrix vector product in each GMRES iteration is computed using \texttt{fmm2d}~\cite{fmm2d}, 
a wideband fast-multipole method which uses far-field signatures to accelerate the translation operators at high frequencies~\cite{crutchfield2006remarks}. The computational complexity of applying an $N\times N$ matrix using a wideband FMM for high-frequency problems is $O(N\log{N})$. 

\appendix

\section{Definition of the scattering problems and the standard boundary-integral operators}\label{app:A}

In this section we show how scattering of a plane-wave by an obstacle with zero Dirichlet or Neumann boundary conditions can be reformulated as a BIE involving the operators $A_k, A'_k$ \eqref{e:DBIEs} (for the Dirichlet problem) and $\Breg, \Bregp$ \eqref{e:NBIEs} (for the Neumann problem).
The proof that the general Dirichlet and Neumann problems (i.e., with arbitrary boundary data) can be reformulated via these BIEs is very similar; see, e.g., \cite[\S2.6]{ChGrLaSp:12}.


Let $\Oi \subset \Rea^d$, $d\geq 2$ be a bounded open set such that its open complement $\Oe :=\Rea^d \setminus \overline{\Oi }$ is connected.
Let $\Gamma:= \partial \Oi $. The results in the main body of the paper require that $\Gamma$ is $C^\infty$, but the results in this appendix hold when $\Gamma$ is Lipschitz.
Let $\nu$ be the outward-pointing unit normal vector to $\Oi $, and let $\gamma^\pm$ and $\partial^{\pm}_\nu$ denote the Dirichlet and Neumann traces, respectively, on $\Gamma$ from $\Omega^{\pm}$.

\begin{definition}[Plane-wave sound-soft/-hard scattering problems]\label{def:scat}
Given $k>0$ and the incident plane wave $u^I(\bx):= \exp(\ri k \bx\cdot \ba)$ for $\ba\in\Rea^d$ with $|\ba|_2=1$
find the total field $u\in H^1_{\rm loc}(\Oe)$ satisfying
$\Delta u + k^2 u =0$ in $\Oe$, 
\beqs
\text{ either }\gamma^+ u = 0 \text{ (sound-soft) } \,\,\text{ or }\,\, \partial_\nu^+u =0  \text{ (sound-hard) }\,\,\ton\Gamma,
\eeqs
and, where \(u^S := u - u^I\) is the scattered field,
\beq\label{e:src}
\dfrac{\partial u^S }{\partial r} -\ri ku^S = o \left(\frac{1}{r^{(d-1)/2}}\right)  \text{ as }r:=|x|\rightarrow \infty, \text{ uniformly in $x/r$}.
\eeq
\end{definition}

The solutions of the sound-soft and sound-hard plane-wave scattering problems exist and are unique; see, e.g., \cite[Theorem 2.12 and Corollary 2.13]{ChGrLaSp:12}.

Let $\Phi_k(x,y)$ be the fundamental solution of the Helmholtz equation defined by
\beq\label{e:fund}
\Phi_k(x,y):= \frac{\ri}{4}\left(\frac{k}{2\pi |x-y|}\right)^{(d-2)/2}H_{(d-2)/2}^{(1)}\big(k|x-y|\big)= \left\{\begin{array}{cc}
                                                                                                            \displaystyle{\frac{\ri}{4}H_0^{(1)}\big(k|x-y|\big)}, & d=2, \\
                                                                                                            \displaystyle{\frac{\re^{\ri k |x-y|}}{4\pi |x-y|}}, & d=3,
                                                                                                          \end{array}\right.
\eeq
where $H^{(1)}_\nu$ denotes the Hankel function of the first kind of order $\nu$.
The single- and double-layer potentials, $\mc{S}\ell$ and $\mc{D}\ell$ respectively, are defined for $k\in \Com$, $\phi\in L^1(\Gamma)$, and $\bx \in \Rea^d\setminus \Gamma$ by
\begin{align}\label{e:SLP}
    \mc{S}\ell \varphi (\bx) = \int_{\Gamma} \Phi_k (\bx,\by) \varphi (\by) \dif s (\by) \quad \tand\quad
    \mc{D}\ell \varphi (\bx) = \int_{\Gamma} \dfrac{\partial \Phi_k (\bx,\by)}{\partial \nu(\by)} \varphi (\by) \dif s (\by).
\end{align}
The standard single-layer, adjoint-double-layer, double-layer, and hypersingular operators are defined for $k\in \mathbb{C}$, $\phi\in \LtG$, $\psi\in H^1(\Gamma)$, and $x\in \Gamma$ by
\begin{align}\label{e:SD'}
&S_k \phi(\bx) := \int_\bound \Phi_k(\bx,\by) \phi(\by)\,\rd s(\by), \qquad
\DL_k' \phi(\bx) := \int_\bound \frac{\partial \Phi_k(\bx,\by)}{\partial \nu(\bx)}  \phi(\by)\,\rd s(\by),\\
&\DL_k \phi(\bx) := \int_\bound \frac{\partial \Phi_k(\bx,\by)}{\partial \nu(\by)}  \phi(\by)\,\rd s(\by),
\quad
H_k \psi(\bx) := \pdiff{}{\nu(\bx)} \int_\bound \frac{\partial \Phi_k(\bx,\by)}{\partial \nu(\by)}  \psi(\by)\,\rd s(\by).
\label{e:DH}
\end{align}
(We use the notation $\DL_k$, $\DL_k'$ for the double-layer and its adjoint, instead of $D_k$, $D_k'$, to avoid a notational clash with the operator $D:= -\ri \partial$ used in the rest of the paper.)

\begin{theorem}\label{thm:BIEs}

(i) If $u$ is solution of the sound-soft scattering problem of Definition \ref{def:scat}, then
\beq\label{e:Ddirect}
A_k' \partial_\nu^+ u = \partial_\nu^+ u^I - \ri \eta_D u^I \quad\text{ and } \quad u=u^I-\mc{S}\ell(\partial_\nu^+ u).
\eeq

(ii) If $v\in\LtG$ is the solution to
\beq\label{e:Dindirect}
A_k v = -\gamma^+ u^I,
\quad\text{ then } \quad
u:= u^I + (\mc{D}\ell -\ri \eta_D \mc{S}\ell)v
\eeq
is the solution of the sound-soft scattering problem of Definition \ref{def:scat}.

(iii) If $u$ is solution of the sound-hard scattering problem of Definition \ref{def:scat}, then
\beq\label{e:Ndirect}
\Breg \gamma^+ u = \ri \eta_N \gamma^+ u^I - S_{\ri k} \partial^+_\nu u^I  \quad\text{ and } \quad u=u^I+\mc{D}\ell(\gamma^+ u).
\eeq

(iv) If $v\in\LtG$ is the solution to
\beqs
\Bregp v = -\partial_\nu^+ u^I,
\quad\text{ then } \quad
u:= u^I + (\mc{D}\ell S_{\ri k} -\ri \eta_N \mc{S}\ell)v
\eeqs
is the solution of the sound-hard scattering problem of Definition \ref{def:scat}.
\end{theorem}

\bpf[References for the proof]
Part (i) is proved in, e.g., \cite[Theorem 2.46]{ChGrLaSp:12}. Part (ii) is proved in, e.g., \cite[Equations 2.70-2.72]{ChGrLaSp:12}. Part (iii) is proved in, e.g., \cite[Equation 1.6]{GaMaSp:21N}. Part (iv) is proved in, e.g., \cite[Equation 1.8]{GaMaSp:21N}.
\epf

\ble\label{lem:inversebound}
If $\Gamma$ is $C^1$ then
\begin{align*}
\N{A_k'}_\LtGt \geq 1/2,\qquad &\N{(A_k')^{-1}}_\LtGt \geq 2,\\
\N{\Breg}_\LtGt \geq \left(\frac{\eta_N}{2}+\frac{1}{4}\right),
\qquad&\N{(\Breg)^{-1}}_\LtGt \geq \left(\frac{\eta_N}{2}+\frac{1}{4}\right)^{-1}.
\end{align*}
\ele

\bpf
The results for $A_k'$ are proved in \cite[Lemma 4.1]{ChGrLaLi:09} using that $A_k' - (1/2)I$ is compact when $\Gamma$ is $C^1$. The results for $\Breg$ are proved in an analogous way, using that $\Breg - ( \ri\eta_N/2 -1/4)I$ is compact by, e.g., \cite[Proof of Theorem 2.2]{GaMaSp:21N}.
\epf

\section{Propagation of singularities}
\label{a:propagate}

The goal of this section is to prove Theorem~\ref{t:basicPropagate}. Since propagation of singularities away from $\Gamma$ follow from standard results (see e.g.~\cite[Appendix E]{DyZw:19}), we may work in a small neighborhood of a point on $\Gamma$.  In particular, we work in Fermi normal coordinates $(x_1,x')$ with $\Gamma=\{x_1=0\}$. Then, since $W\in\Psi^{\comp}(\Omega)$ it is, in particular, supported away from $\Gamma$,  in these coordinates the operator $-\hsc^2\Delta-1-iW$ is given by
$$
\hsc^2D_{x_1}^2 +\hsc a(x)\hsc D_{x_1}-\tilde{R}(x,k^{-1} D_{x'})
$$
for some $a\in C^\infty$, where $\tilde{R}\in \operatorname{Diff}_{\hsc}^2$ i.e. is a semiclassical differential operator of order $2$. Moreover, $r(x,\xi):=\sigma(\tilde{R})=1-|\xi'|_{g}.$ 
Now, setting $\psi(x_1)=\frac{1}{2}\int_0^{x_1}a(s,x')ds$, we have
$$
(-\hsc^2\Delta-1)e^{-i\psi} = \hsc^2D_{x_1}^2-\tilde{R}(x,k^{-1} D_{x'})+\hsc^2|\partial_{x_1}\psi|^2+i\hsc^2\partial_{x_1}^2\psi.
$$

Throughout this section, we use the notation $\Psit^m:= C^\infty((-\infty,\infty)_{x_1};\Psi^m(\Gamma))$ and $\St^m:C^\infty((-\infty,\infty);S^m(T^*\Gamma))$ and study operators of the form
\begin{equation}
\label{e:Pdef}
\begin{aligned}
P&:=\hsc^2D_{x_1}^2-R(x,k^{-1} D_{x'}),
\end{aligned}
\end{equation}
with $R\in \Psi_{\mathsf{T}}^2$ having real principal symbol, $r$ and $\{r=0\}\cap \{ \partial_{(x',\xi')}r=0\}=\emptyset$. 

For the boundary conditions, observe that since $\psi|_{x_1=0}=1$, the boundary conditions are then changed to 
$$
 (-Q\hsc D_{x_1}u-u)=e^{i\psi}(-Q\hsc D_{x_1})e^{-i\psi}e^{i\psi}u -e^{i\psi}u= (-Q\hsc D_{x_1}+Q\hsc \partial_{x_1}\psi u-u)
$$
Notice that, for $\hsc$ small enough, $E:=(1+Q\hsc \partial_{x_1}\psi)^{-1}\in \Psi^0$ exists with $\sigma(E)=1$. Therefore, $\tilde{Q}:=EQ\in \Psi^{-1}$ with $\sigma(EQ)=\sigma(Q)$ and it is enough to  study propagation of singularities for 
\begin{equation}
\label{e:prop}
Pu=f\text{ in }\{x_1>0\},\qquad Q\hsc D_{x_1}u+u=g\text{ on }x_1=0,
\end{equation}
where $Q\in \Psi^{-1}(\Gamma)$.

  We divide our analysis into three regions
\begin{equation*}
\mc{E}:=\{ (x',\xi')\,:\, r(0,x',\xi')<0\},\qquad \mc{H}:=\{ (x',\xi')\,:\, r(0,x',\xi')>0\},\qquad \mc{G}:=\{ (x',\xi')\,:\, r(0,x',\xi')=0\},
\end{equation*}
respectively the elliptic, hyperbolic, and glancing regions.
\subsection{Elliptic Estimates}
We start by recalling a standard factorization together with elliptic estimates.
\begin{lemma}[Lemma 4.2 \cite{FaGa:25}]
\label{l:factor}
Let $P$ as in~\eqref{e:Pdef}. Then for all $\chi \in \St^{0}$ with $\supp \chi \subset \{ |r|>0\}$, there are $\Lambda^{j}_{\pm}\in \Psit^{1}$, $j=0,1$ such that, on $\supp \chi$, $\sigma(\Lambda_{\pm})= \pm\sqrt{r}$ on $\{r>0\}$ and $\sigma(\Lambda_{\pm})= i\sqrt{|r|}$ on $\{r<0\}$ and 
\begin{equation}
\label{e:factor}
\begin{aligned}
\chi(x,\hsc D_{x'})P&=\chi(x,\hsc D_{x'})(\hsc D_{x_1}+\Lambda_-)(\hsc D_{x_1}-\Lambda_-) +O(k^{-\infty})_{\Psit^{-\infty}}\\
&=\chi(x,\hsc D_{x'})(\hsc D_{x_1}+\Lambda_+)(\hsc D_{x_1}-\Lambda_+) +O(k^{-\infty})_{\Psit^{-\infty}}
\end{aligned}
\end{equation}
\end{lemma}
Using energy estimates, one then obtains the following elliptic estimates.
\begin{lemma}
\label{l:elliptic}
Let $Q\in \Psi^{-1}(\Gamma)$, $k\in \mathbb{N}$, $k\geq 2$, $s,N\in \mathbb{R}$, $0<t_1<t_2$, $X,\tilde{X}\in \Psit^0$ such that $\WF(X)\subset \Ell(\tilde{X})$ and $\WF(\tilde{X})\subset \{r<0\}$ and there is $c>0$ such that, 
$$
|\sigma(Q)i\sqrt{|r|}+1|>c>0,\quad\text{ on }\WF(X).
$$ 
Then there is $C>0$ such that 
\begin{align*}
\sum_{j=0}^2\|Xu\|_{\mac{k-j}((0,t_1);\mac{s+j})}&\leq C\sum_{j=0}^{k-2}\|\tilde{X}Pu\|_{\mac{k-2-j}((0,t_2);\mac{s+j})}+C\hsc^{\frac{1}{2}}\|\tilde{X}(Q\hsc D_{x_1}u-u)(0)\|_{\mac{s+k-\frac{1}{2}}}\\
&\qquad +C\hsc^N(\|u\|_{\mac{1}((0,t_2);\mac{-N})}+\|Pu\|_{L^2((0,t_2);\mac{-N})}+\|u(0)\|_{\mac{-N}}).
\end{align*}
\end{lemma}
\begin{proof}
Let $X'\in \Psi^0(\Gamma)$ with $\WF(X)\subset \Ell(X')$ and $\WF(X')\subset \Ell(\tilde{X})\cap\{ |\sigma(Q)i\sqrt{|r|}+1|>c/2\}$. Then, since, on $\WF(X')$,  $|\sigma(Q\Lambda_{-}+1)|=|\sigma(Q)i\sqrt{|r|}+1|\geq c>0$, 
\begin{align*}
\|X'u(0)\|_{\mac{s+k-\frac{1}{2}}}&\leq C\|\tilde{X}(-Q\Lambda_{-}-1)u(0)\|_{\mac{s+k-\frac{1}{2}}}+C\hsc^N\|u(0)\|_{\mac{-N}} \\
&\leq C\|\tilde{X}(-Q\hsc D_{x_1}u-u)(0)\|_{\mac{s+k-\frac{1}{2}}} +C\|\tilde{X}Q(\hsc D_{x_1}-\Lambda_{-}^0)u(0)\|_{\mac{s+k-\frac{1}{2}}}+C\hsc^N\|u(0)\|_{\mac{-N}}.
\end{align*}
Therefore, the lemma follows from~\cite[(4.28), Lemma 4.8]{FaGa:25}], where $iE_-$ is replaced by $\Lambda_-^0$ in our notation.
\end{proof}

\subsection{Propagation in the hyperbolic region}
We next proceed to the hyperbolic region. In this region, singularities follow broken bicharacteristics. 
\begin{lemma}
\label{l:hyperbolic}
Let $N>0$, $Q\in \Psi^{-1}(\Gamma)$, $A\in \Psit$ with $\WF(A)\cap T^*\Gamma\subset \{r>0\}\cap\{|\sigma(Q)\sqrt{r}+1|>0\}$ and $B_1\in {}^b\Psi^{0}$ with $\WF(A)\subset {}^b\Ell(B_1)$. Then, there is $\e_0>0$ small enough such that for all $B\in {}^b\Psi^{\comp}$, $B'\in \Psi^{\comp}(\Gamma)$ and $0<T<\e_0$ such that  $\WF(A)\cap T^*\Gamma\subset \Ell(B')$, for all $\rho\in \WF(A)$, there is $0\leq t\leq T$ such that 
$$
\varphi_{-t}(\rho)\in \Ell(B),\qquad \bigcup_{0\leq t\leq T}\varphi_{-t}(\rho)\subset {}^b\Ell(B_1),
$$
there is $C>0$ such that 
\begin{align*}
&\|A\hsc D_{x_1}u(0)\|_{L^2}+\|Au(0)\|_{L^2}+\|A\hsc D_{x_1}u\|_{L^2}+\|Au\|_{L^2}\\
 &\leq C(\hsc^{-1}\|B_1Pu\|_{L^2}+\|B'(-Q\hsc D_{x_1}-u)(0)\|_{L^2}+\|Bu\|_{L^2}) \\
&\qquad+C\hsc^N(\|(-Q\hsc D_{x_1}u-u)(0)\|_{\mac{-N}}+\|\hsc D_{x_1}u(0)\|_{\mac{-N}}+\|u\|_{L^2(0,2\e)}+\|\hsc D_{x_1}u\|_{L^2(0,2\e)}+\|Pu\|_{L^2(0,2\e)})
\end{align*}
\end{lemma}

In order to prove Lemma~\ref{l:hyperbolic}, we need a pseudodifferential cutoff that nearly commutes with one factor in~\eqref{e:factor}.
\begin{lemma}
\label{l:commuter}
Let $a\in S^{\comp}(T^*\Gamma)$ with $\supp a\subset \{r>0\}$. Then, there are $\e>0$ and $E_{\pm}\in \Psit^{\comp}$ such that  $E_{\pm}|_{x_1=0}=a(x',\hsc D_{x'})$, for any $\theta\in C_c^\infty(-\e,\e)$, 
$$
\theta [\hsc D_{x_1}+\Lambda_{\pm},E_{\pm}]=O(k^{-\infty})_{\Psit^{-\infty}},
$$
and 
$$
\WF(E_{\pm})\subset \bigcup_{0\leq x_1\leq 2\e} \{(x_1,\exp(x_1H_{\mp\sqrt{r}})(\supp (a)))\}.
$$
\end{lemma}
\begin{proof}
We start by solving
$$
\sigma([(\hsc D_{x_1}+\Lambda_{\pm}),\tilde{e}^\pm_0(x,\hsc D_{x'})])=-i\hsc \partial_{x_1}\tilde{e}^\pm_0-i\hsc \{ \pm \sqrt{r},\tilde{e}^\pm_0\}=0\quad 0<x_1<2\e,\qquad \tilde{e}^\pm_0|_{x_1=0}=a.
$$
This is possible for $\e$ small enough since $\supp a\subset \{r>0\}$ and hence $\{x_1=0\}$ is non-characteristic for the vector-field $\partial_{x_1}\pm H_{\sqrt{r}}.$  Let $\psi, \psi_0,\psi_1\dots,\in C_c^\infty((-2\e,2\e))$ with $x_1\psi'(x_1)\leq 0$, $\supp (1-\psi)\cap [-\e,\e]=\emptyset$, $\supp \psi\cap \supp(1-\psi_j)=\emptyset$, $j\geq 0$ and $\supp \psi_{j+1}\cap \supp (1-\psi_j)=\emptyset$, $j\geq 0$.  Then, set $e^\pm_0:=\psi\tilde{e}^\pm_0$ and $E^\pm_0:=e^\pm_0(x,\hsc D_{x'})$. Then,
$$
[(\hsc D_{x_1}+\Lambda_{\pm}),E^\pm_0(x,\hsc D_{x'})]=\hsc^2 b_0(x,\hsc D_{x'}), \text{ for } x_1\notin \supp (1-\psi_0)
$$
for some $b_0\in \St^{\comp}$ with $\supp b_0\subset \supp \psi e^\pm_0\subset \{r>0\}$. 

Suppose by induction that form some $N\geq 1$ we have found $e^\pm_0,\dots,e^\pm_{N-1}\in \St^{\comp}$ with $\supp e^\pm_j\subset \supp \psi_0\tilde{e}^\pm_0\subset \{r>0\}$ such that, with $E^\pm_{N-1}:=\sum_{j=0}^{N-1}\hsc^j \psi_{N-1} e^\pm_j(x,\hsc D_{x'})$, 
\begin{equation}
\label{e:factorInduct}
[(\hsc D_{x_1}+\Lambda_{\pm}),E^\pm_{N-1}(x,\hsc D_{x'})]=\hsc^{N+1} b_N(x,\hsc D_{x'}),\qquad \text{ for }x_1\notin \supp (1-\psi_{N-1}). 
\end{equation}
for some $b_N\in \St^{\comp}$ with $\supp b_N\subset \supp \psi_0 \tilde{e}^\pm_0\subset\{r>0\}$. 
Then, let $e^\pm_{N}\in \St^{\comp}$ satisfy
$$
\partial_{x_1}e^\pm_N-\{\mp \sqrt{r},e^\pm_N\}=-b_N,\qquad e^\pm_N|_{x_1=0}=0,\qquad x_1\notin \supp (1-\psi_{N-1}).
$$
Defining $E^\pm_N:=\sum_{j=0}^N\hsc^j\psi_N e^\pm_j(x,\hsc D_{x'})$, we have~\eqref{e:factorInduct} with $N$ replaced by $N+1$. Putting $E_\pm\sim \sum_{j=0}^\infty \hsc^j\psi e^\pm_j(x,\hsc D_{x'})$, completes the proof of the lemma.
\end{proof}

We can now prove Lemma~\ref{l:hyperbolic}.

\begin{proof}[Proof of Lemma~\ref{l:hyperbolic}]
By a partition of unity argument, we may assume that $\varphi_{-T}(\rho)\in \Ell(B)$ for all $\rho\in \WF(A)$. In addition, since $\WF(A)\subset\{r>0\}$, for any $\delta>0$, may further assume that $\WF(B)\subset \{ \xi_1<0,\, \e/2<x_1<\e\}.$

Let $a_i\in C_c^\infty(T^*\Gamma)$, $i=1,2$  with $\supp a_i\subset \{ r>0\}$, $\supp a_1\cap\supp (1-a_2)=\emptyset$,
$$
\WF(A)\cap \bigcup_{-2\e\leq x_1\leq 2\e}\{ (x_1,  \exp(x_1 H_{\sqrt{r}}\supp (1-a_1)\}=\emptyset,
$$
and
$$
\bigcup_{-2\e\leq x_1\leq 2\e}\{ (x_1,\exp(x_1H_{\sqrt{r}})(\supp a_2))\}\subset {}^b\Ell(B_1).
$$ 

Let $E^i_{\pm}$ from Lemma~\ref{l:commuter} with $a$ replaced by $a_i$.  Then, using the factorizations~\eqref{e:factor}, we have
\begin{equation}
\label{e:factored}
(\hsc D_{x_1}+\Lambda_\pm)E^i_{\pm}(\hsc D_{x_1}-\Lambda_{\pm} )=E_\pm^i P +O(\hsc^{\infty})_{\Psit^{-\infty}}.
\end{equation}

Since $\supp a_i\subset {}^b\Ell(B_1)$, there is $\e>0$ small enough such that for $X\in \Psi(\{x_1>0\})$ with $\WF(X)\subset \{\pm \xi_1\geq -\e, 0<x_1<\e\}$,
$$
\|XE^i_{\pm}(\hsc D_{x_1}-\Lambda_{\pm})u\|_{L^2}\leq C\|B_1Pu\|_{L^2}+C\hsc ^N\|u\|_{L^2},
$$
In particular, there is $\chi \in C_c^\infty((\e/2,\e))$ with $\chi\equiv 1 $  near $\frac{2\e}{3}$ and $X$ as above such that 
$$
\WF(E^i_{\pm}\chi)\subset (\{\pm \xi_1>0\}\cup\{ p^2+b^2>c>0\})\cap {}^b\Ell(B_1).
$$
Therefore,
\begin{equation}
\label{e:control}
\|\chi E^i_{\pm}(\hsc D_{x_1}-\Lambda_{\pm})u\|_{L^2}\leq C(\|B_1 Pu\|_{L^2} +\|Bu\|_{L^2}+\hsc ^N\|u\|_{L^2}+\hsc ^N\|Pu\|_{L^2}).
\end{equation}

Basic energy estimates together with the fact that on $\supp e_0$, $\Im\sigma(\Lambda)=0$  (see e.g.~\cite[Lemma 4.4]{FaGa:25}) imply that for $0\leq t_1\leq  \frac{\e}{2}$ and $0\leq t\leq \e$,
\begin{equation}
\label{e:factoredEnergy}
\begin{aligned}
\|E^i_{\pm}(t_1)(\hsc D_{x_1}-\Lambda_{\pm} )u(t_1)\|_{L^2}&\leq C\hsc^{-1}(\|E^i_{\pm} Pu\|_{L^2(0,\e)}+\|\chi E^i_{\pm}(\hsc D_{x_1}-\Lambda_{\pm})u\|_{L^2}),\\
\|E^i_{\pm}(\hsc D_{x_1}-\Lambda_{\pm} )u(t)\|_{L^2}&\leq C\hsc^{-1}(\|E^i_{\pm} Pu\|_{L^2(0,\e)}+\|a_i(x',\hsc D_{x'})(\hsc D_{x_1}-\Lambda_{\pm})u(0)\|_{L^2}).
\end{aligned}
\end{equation}

In particular, combining the first estimate with $t_1=0$,~\eqref{e:control} and the fact that $\WF(E^i_{\pm})\cap \{x_1<\e\}\subset {}^b\Ell(B_1)$, we obtain
\begin{align*}
&\|a_2(x',\hsc{D}_{x'})(I+\Lambda_{+}Q)\hsc D_{x_1}u(0)\|_{L^2}+\|E^2_{\pm}(t_1)(\hsc D_{x_1}-\Lambda_{\pm} )u(t_1)\|_{L^2}\\
&\leq C\|a_2(x',\hsc D_{x'})\Lambda_{+}(-Q\hsc D_{x_1}u-u)(0)\|_{L^2}+C\hsc^{-1}(\|B_1 Pu\|_{L^2}+\|Bu\|_{L^2})+C\hsc^N(\|u\|_{L^2}+\|Pu\|_{L^2})\\
&\leq C(\hsc^{-1}\|B_1Pu\|_{L^2}+\|B'(-Q\hsc D_{x_1}-u)(0)\|_{L^2}+\|Bu\|_{L^2}) +C\hsc^N(\|(-Q\hsc D_{x_1}u-u)(0)\|_{\mac{-N}}+\|u\|_{L^2}+\|Pu\|_{L^2})
\end{align*}
It remains to observe that $\sigma(I+\Lambda_+Q)>0$ on $\supp a_2$ and hence
\begin{align*}
&\|a_2(x',\hsc D_{x'})\hsc D_{x_1}u(0)\|+\|a_2(x',\hsc D_{x'})u(0)\|_{L^2}\\
&\leq \|a_2(x',\hsc D_{x'})\hsc D_{x_1}u(0)\|_{L^2}+\|a_2(x,\hsc D_{x'}) (-Q\hsc D_{x_1}u-u)(0)\|_{L^2}\\
&\leq  C(\hsc^{-1}\|B_1Pu\|_{L^2}+\|B'(-Q\hsc D_{x_1}-u)(0)\|_{L^2}+\|Bu\|_{L^2}) \\
&\qquad+C\hsc^N(\|(-Q\hsc D_{x_1}u-u)(0)\|_{\mac{-N}}+\|\hsc D_{x_1}u(0)\|_{\mac{-N}}+\|u\|_{L^2}+\|Pu\|_{L^2})
\end{align*} 

Now that we have control of the boundary traces, we can use the second inequality~\eqref{e:factoredEnergy} to finish the proof. Indeed, since $\supp a_1\cap \supp (1-a_2)=\emptyset$,  
\begin{align*}
&\|E^1_{\pm}(\hsc D_{x_1}-\Lambda_{-} )u(t_1)\|_{L^2}+\|E^2_{\pm}(t_1)(\hsc D_{x_1}-\Lambda_{+} )u(t_1)\|_{L^2}\\
&\leq C(\hsc^{-1}\|B_1Pu\|_{L^2}+\|B'(-Q\hsc D_{x_1}-u)(0)\|_{L^2}+\|Bu\|_{L^2})+ C\|a_1(x',\hsc D_{x'})(\hsc D_{x_1}-\Lambda_{\pm})u(0)\|_{L^2}) \\
&\qquad+C\hsc^N(\|(-Q\hsc D_{x_1}u-u)(0)\|_{\mac{-N}}+\|\hsc D_{x_1}u(0)\|_{\mac{-N}}+\|u\|_{L^2}+\|Pu\|_{L^2})\\
&\leq C(\hsc^{-1}\|B_1Pu\|_{L^2}+\|B'(-Q\hsc D_{x_1}-u)(0)\|_{L^2}+\|Bu\|_{L^2}) \\
&\qquad+C\hsc^N(\|(-Q\hsc D_{x_1}u-u)(0)\|_{\mac{-N}}+\|\hsc D_{x_1}u(0)\|_{\mac{-N}}+\|u\|_{L^2}+\|Pu\|_{L^2})
\end{align*}
and hence, using that $\WF(A)\subset \Ell(E_{\pm}^i)\cap \{ x_1<\e/2\}$ and $|\sigma(\Lambda_+-\Lambda_-)|>c>0$ on $\WF(A)$, we have 
\begin{align*}
&\|A \hsc D_{x_1}u\|_{L^2}+\|Au\|_{L^2}\\
&\leq C(\|E^1_{\pm}(\hsc D_{x_1}-\Lambda_{-} )u\|_{L^2}+\|E^2_{\pm}(\hsc D_{x_1}-\Lambda_{+} )u\|_{L^2}+ \hsc^N(\|\hsc D_{x_1}u\|_{L^2}+\|u\|_{L^2})\\
&\leq C(\hsc^{-1}\|B_1Pu\|_{L^2}+\|B'(-Q\hsc D_{x_1}-u)(0)\|_{L^2}+\|Bu\|_{L^2}) \\
&\qquad+C\hsc^N(\|(-Q\hsc D_{x_1}u-u)(0)\|_{\mac{-N}}+\|\hsc D_{x_1}u(0)\|_{\mac{-N}}+\|u\|_{L^2}+\|\hsc D_{x_1}u\|_{L^2}+\|Pu\|_{L^2})
\end{align*}

\end{proof}

As a corollary of~\eqref{e:factoredEnergy}, and Lemma~\ref{l:factor}, we obtain Lemma~\ref{l:DtN}

\begin{proof}[Proof of Lemma \ref{l:DtN}]
Let $\chi \in C_c^\infty(\mathbb{R}^n\times \mathbb{R}^{n-1})$ with $\chi(q)=1$ such that 
$$
\|\chi(x,\hsc D_{x'})Pu\|_{L^2}\leq C\hsc^{s+1},\qquad \|\chi(0,x',\hsc D_{x'})(\hsc D_{x_1}-\Lambda_-)u\|_{L^2}\leq Ch^{s}.
$$
Let $a\in C_c^\infty(T^*\Gamma)$ with $\supp a\subset \{\chi>0\}$ and $E_{\pm}$ as constructed in Lemma~\ref{l:commuter}, and $\e>0$ small enough that $\supp E_-\cap\{ x_1<\e\}\subset \{\chi>0\}$.  
Then, by~\eqref{e:factoredEnergy}, 
$$
\|E_{-}(\hsc D_{x_1}-\Lambda_-)u(t)\|_{L^2}\leq C\hsc^{-1}\|E_{+}Pu\|_{L^2(0,\e)}+\|a(x,\hsc D_{x;})(\hsc D_{x_1}-\Lambda_-)u\|_{L^2}\leq Ch^{s}.
$$
In particular,
$$
{}^b\WF^s_{\mac{1}}(u)\cap \{x_1>0,\xi_1>0\}\cap \{\sigma(E_-)>0\}=\emptyset, 
$$
and, since 
$$
\bigcup_{0< s< \e}\varphi_{s}(q)\subset \{\sigma(E_-)>0\},
$$
the lemma follows.
\end{proof}

\subsection{Propagation in the glancing region}

The most subtle part of the analysis is in the glancing region, where we follow the proof in~\cite[Chapter 24]{Ho:07}. Throughout this section, $B_i\in\Psit^0$, $i=0,1,$ satisfies  
\begin{equation}
\label{e:Bdef}
B(x,hD):=B_1(x,hD')\hsc D_{x_1}+B_0(x,\hsc D_{x'}),\qquad B_1^*=B_1,\qquad B_0^*=B_0-[\hsc D_{x_1},B_1].
\end{equation}
Then, integration by parts yields the following lemma.
\begin{lemma}[ Lemma 24.4.2 \cite{Ho:07}]
\label{l:intByParts}
Let $P$ as in~\eqref{e:Pdef} and $B$ as in~\eqref{e:Bdef}. Then,
$$
2\hsc^{-1}\Im \langle Pu,Bu\rangle_\Omega=\sum_{j,k=0}^1 \Re \langle E_{jk}(x',\hsc D_{x'})u_k,u_j\rangle_{\Gamma}+\sum_{j,k=0}^1\Re\langle C_{jk}u_k,u_j\rangle_\Omega,
$$
where $u_j=h^jD_{x_1}^j$ and $E_{11}=B_1$, $E_{01}^*=E_{10}=B_0$, $E_{00}=B_1(R+R^*)/2$, and the symbol, $c_{jk}$ of $C_{jk}$ is real and satisfies $c_{01}=c_{10}$,
$$
\sum_{jk}c_{jk}(x,\xi')\xi_1^{j+k}=\{p,b\}+2b\Im p_{-1},
$$
where $p_{-1}$ is the subprincipal symbol of $P$. 
\end{lemma}

We use Lemma~\ref{l:intByParts} repeatedly in positive commutator type arguments.  To do this, we first record a technical lemma similar to~\cite[Lemma 24.4.5]{Ho:07}.
\begin{lemma}
\label{l:systemToEstimate}
Let $A_{jk}\in \Psi^{\comp}$ have real principal symbols $a_{jk}$ compactly supported in $x$, $a_{01}=a_{10}$ and such that 
$$
\sum_{j,k=0}^1a_{jk}(x,\xi')\xi_1^{j+k}=-\psi(x,\xi)^2,\quad \text{when }\xi_1^2=r(x,\xi'),
$$
where $\psi\in C_c^\infty$. Furthermore, suppose that $\partial_{x_1}r>0$ or $\partial_{(x',\xi')}r\neq 0$ on $\cup \supp a_{jk}$. Then, there are $\Psi_i$ $i=0,1$ with principal symbols $\psi$ and $\partial_{\xi_1}\psi$ when $\xi_1=r(x,\xi')=0$ such that for any $N$
\begin{align*}
\Re \sum \langle A_{jk}&(hD_1)^ku,(hD_1)^ju\rangle_{\Omega}+\|\Psi_0(x,\hsc D_{x'})u+\Psi_1(x,\hsc D_{x'})\hsc D_{x_1}u\|_{L^2(\Omega)}^2\\
&\leq C\hsc^{-1}\|Pu\|_{L^2}^2+C\hsc(\|u\|^2_{L^2}+\|\hsc D_{x_1}u\|_{L^2}^2+\|\hsc D_{x_1}u(0)\|_{L^2}^2+\|(Q\hsc D_{x_1}+u)(0)\|_{L^2}^2).
\end{align*}
\end{lemma}
\begin{proof}
Let $\psi_0$, $\psi_1\in C_c^\infty$ such that $\psi_j=\partial_{\xi_1}^j\psi$ when $\xi_1=r(x,\xi')=0$. and such that 	
$$
\sum a_{jk}(x,\xi')\xi_1^{j+k}+(\psi_0(x,\xi')+\psi_1(x,\xi')\xi_1)^2\leq g(x,\xi')(\xi_1^2-r(x,\xi')),
$$
for some $g\in C_c^\infty$. This is possible by~\cite[Lemma 24.4.3]{Ho:07}.

Let $\Psi_j$ and $G$ with principal symbols $\psi_j$ and $g$ respectively. Then, by the Sharp G\aa rding inequality for systems,
\begin{align*}
&\Re \Big(\sum \langle A_{jk}(x,\hsc D_{x'})u_k,u_j\rangle_{\Omega} +\sum \langle \Psi_j^*(x,\hsc D_{x'})\Psi(x,\hsc D_{x'})u_k,u_j\rangle_{\Omega}\\
&\qquad\qquad -\langle G(x,\hsc D_{x'})u_1,u_1\rangle_{\Omega} +\langle G(x,\hsc D_{x'})R(x,\hsc D_{x'})u_0,u_0\rangle_{\Omega}\Big)\\
&\leq C\hsc(\|u_0\|^2+\|\hsc D_{x_1}u\|^2)
\end{align*}
Next, 
\begin{align*}
&\Re\Big(\langle G(x,\hsc D_{x'})u_1,u_1\rangle_{\Omega} -\langle G(x,\hsc D_{x'})R(x,\hsc D_{x'})u_0,u_0\rangle_{\Omega}\Big)\\
&=\Re\Big(\langle GPu,u\rangle_{\Omega} -i\hsc\langle G(0,x',\hsc D_{x'})\hsc D_{x_1}u,u\rangle|_{\Gamma}+\langle [\hsc D_{x_1},G(x,\hsc D_{x'})]\hsc D_{x_1}u,u\rangle_{\Omega}\Big)\\
&\leq C\hsc^{-1}\|Pu\|_{L^2}^2 +C\hsc\|u\|_{L^2} +\hsc\Im\langle G(0,x',\hsc D_{x'})\hsc D_{x_1}u,u\rangle_{\Gamma}+C\hsc\|\hsc D_{x_1}u\|_{L^2(\Omega)}\|u\|_{L^2(\Omega)}.
\end{align*}
\end{proof}

We also need a microlocal estimate on the normal derivative.
We start with a microlocal estimate on the normal derivative.
\begin{lemma}
\label{l:microNormalDerivative}
There is $\delta_0>0$ such that for all $a\in S^{\comp}(T^*\Gamma)$, $\supp a\subset \{|r| \leq \delta\}$, $A:=a(x',\hsc D_{x'})$, $\psi\in C_c^\infty(-1,1)$ with $0\notin \supp (1-\psi)$ and $\chi ,\chi_1\in C_c^\infty(\mathbb{R}^\times \mathbb{R}^{n-1})$ with $\supp a(x',\xi')\psi(x_1)\cap \supp(1-\chi)=\emptyset$, $\supp (1-\chi_1)\cap\supp \chi=\emptyset$. Then, 
\begin{align*}
\|A\hsc D_{x_1}u(0)\|_{L^2}&\leq C\hsc^{-1}\|XPu\|_{L^2}+C\|Xu\|_{L^2} +C\|X(Q\hsc D_{x_1}u+u)(0)\|_{L^2}\\
&\qquad +C\hsc^N(\|(Q\hsc D_{x_1}u+u)(0)\|_{L^2}+\|\hsc D_{x_1}u(0)\|_{L^2}).\\
\|X \hsc D_{x_1}u\|_{L^2}&\leq C(\hsc^{-1}\|X_1Pu\|_{L^2}+\|X_1u\|_{L^2}+C\|X_1(Q\hsc D_{x_1}u+u)(0)\|_{L^2}\\
&\qquad +C\hsc^N(\|(Q\hsc D_{x_1}u+u)(0)\|_{L^2}+\|\hsc D_{x_1}u(0)\|_{L^2})
\end{align*}
\end{lemma}
\begin{proof}
Let $\psi\in C_c^\infty((-1,1);[0,1])$ with $0\notin \supp (1-\psi)$ and define
$$
\begin{gathered}
B_1:=A^*A\psi(x_1),\qquad B_0:=\frac{-i\hsc}{2}B_1\psi'(x_1),
\end{gathered}
$$
Then, Lemma~\ref{l:intByParts} yields, with $B=B_1\hsc D_{x_1}+B_0$,
$$
2\hsc^{-1}\Im \langle PXu,BXu\rangle_{\Omega}=\Re\sum_{j,k=0}^1\langle  E_{jk}(\hsc D_{x_1})^kXu,(\hsc D_{x_1})^jXu\rangle_{\Gamma}+\Re\sum_{j,k=0}^1\langle  C_{jk}(\hsc D_{x_1})^kXu,(\hsc D_{x_1})^jXu_j\rangle_{\Omega},
$$
where 
\begin{gather*}
E_{01}^*=E_{10}=0,\qquad E_{11}=A^*A,\qquad E_{00}=A^*A(R+R^*)/2.
\end{gather*}
Hence,
\begin{align*}
&\| A\hsc D_{x_1}Xu(0)\|_{L^2}^2\\
&=2\hsc^{-1}\Im \langle PXu,BXu\rangle_{\Omega}-\Re\sum_{j,k=0}^1\langle  C_{jk}(\hsc D_{x_1})^kXu,(\hsc D_{x_1})^jXu\rangle_{\Omega}-\tfrac{1}{2}\langle A(R+R^*)Xu,AXu\rangle_{\Gamma}\\
&=2\hsc^{-1}\Im \langle PXu,BXu\rangle_{\Omega}-\Re\sum_{j,k=0}^1\langle  C_{jk}(\hsc D_{x_1})^kXu,(\hsc D_{x_1})^jXu\rangle_{\Omega}\\
&\qquad-\tfrac{1}{2}\langle A(R+R^*)(X(Q\hsc D_{x_1}u+u),AX(Q\hsc D_{x_1}u+u)\rangle_{\Gamma}\\
&\qquad+\tfrac{1}{2}\langle A(R+R^*)X(Q\hsc D_{x_1}u+u),AXQ\hsc D_{x_1}u\rangle_{\Gamma}+\tfrac{1}{2}\langle A(R+R^*)XQ\hsc D_{x_1}u,AX(Q\hsc D_{x_1}u+u)\rangle_{\Gamma}\\
&\qquad-\tfrac{1}{2}\langle A(R+R^*)XQ\hsc D_{x_1}u,AXQ\hsc D_{x_1}u\rangle_{\Gamma}\\
&\leq 2\hsc^{-1}\Im \langle PXu,BXu\rangle_{\Omega}+C\|X\hsc D_{x_1}u\|_{L^2}+C\|Xu\|_{L^2}\\
&\qquad +C\|X(Q\hsc D_{x_1}u+u)(0)\|_{L^2}^2+C\delta\|A\hsc D_{x_1}u(0)\|_{L^2}^2+C\hsc\|X\hsc D_{x_1}u(0)\|_{L^2}^2\\
&\qquad +O(\hsc^{\infty})(\|(Q\hsc D_{x_1}u+u)(0)\|_{L^2}^2+\|\hsc D_{x_1}u(0)\|_{L^2}^2)\\
&\leq C\hsc^{-2}\|XPu\|_{L^2}^2+C\|X\hsc D_{x_1}u\|_{L^2}+C\|Xu\|_{L^2}\\
&\qquad +C\|X(Q\hsc D_{x_1}u+u)(0)\|_{L^2}^2+C\delta\|A\hsc D_{x_1}u(0)\|_{L^2}^2+C\hsc\|X\hsc D_{x_1}u(0)\|_{L^2}^2\\
&\qquad +O(\hsc^{\infty})(\|(Q\hsc D_{x_1}u+u)(0)\|_{L^2}^2+\|\hsc D_{x_1}u(0)\|_{L^2}^2).
\end{align*}
In particular, for $\delta>0$ small, 
\begin{align*}
\|A\hsc D_{x_1}u(0)\|_{L^2}&\leq C\hsc^{-1}\|XPu\|_{L^2}+C\|X\hsc D_{x_1}u\|_{L^2}+C\|Xu\|_{L^2}\\
&\qquad +C\|X(Q\hsc D_{x_1}u+u)(0)\|_{L^2}+C\hsc^{1/2}\|X\hsc D_{x_1}u(0)\|_{L^2}\\
&\qquad +O(\hsc^{\infty})(\|(Q\hsc D_{x_1}u+u)(0)\|_{L^2}+\|\hsc D_{x_1}u(0)\|_{L^2}).
\end{align*}

Now, we need to estimate $\|X\hsc D_{x_1}u\|_{L^2}$ by $Pu$ and $u$. For this, self-adjoint.
\begin{align*}
 \langle Pu,X^*X u\rangle &= \langle ((\hsc D_{x_1})^2-r(x,\hsc D_{x'}))u,X^*Xu\rangle \\
&= -\langle r(x,\hsc D_{x'})u,X^*Xu\rangle +\|X \hsc D_{x_1}u\|_{L^2}^2\\
&\qquad \langle (\hsc D_{x_1})u,[\hsc D_{x_1},X^*X]u\rangle +i\hsc\langle \hsc D_{x_1}u,X^*Xu\rangle_{\Gamma} . \end{align*}
Hence,
\begin{align*}
\|X\hsc D_{x_1}u\|_{L^2}^2&\leq  \langle Pu,X^*Xu\rangle +\langle r(x,\hsc D_{x'})Xu, Xu\rangle \\
&\qquad +\langle [X,r]u,Xu\rangle - \langle (\hsc D_{x_1})u,[\hsc D_{x_1},X^*]X+X^*[\hsc D_{x_1},X]u\rangle\\
&\qquad -\Re hi\langle X\hsc D_{x_1}u,X(Q\hsc D_{x_1}u+u)\rangle_{\Gamma}+\Re hi\langle X\hsc D_{x_1}u,XQ\hsc D_{x_1}u\rangle_{\Gamma}\\
&\leq C\|XPu\|^2+C\|Xu\|^2+ C\hsc(\|\tilde{X}u\|_{L^2}^2+\|\tilde{X}\hsc D_{x_1}u\|_{L^2}^2)+C \hsc\|\tilde{X}\hsc D_{x_1}u(0)\|_{L^2}^2\\
&\qquad+C\hsc\|\tilde{X}(Q\hsc D_{x_1}+u)(0)\|_{L^2}^2 \\
&\qquad+O(\hsc^{\infty})(\|u\|^2_{L^2}+\|\hsc D_{x_1}u\|^2_{L^2}+\|(Q\hsc D_{x_1}u+u)(0)\|^2_{L^2}+\|\hsc D_{x_1}u(0)\|_{L^2}^2).
\end{align*}
The proof is completed by iteration.
\end{proof}

\subsection{Diffractive points}

We start by considering diffractive points. The main technical estimate is provided in the next lemma.
\begin{lemma}
\label{l:ivrii}
Let $Q\in \Psi^{-1}(\Gamma)$ with $\Re\sigma(Q)\geq 0$ near $r=0$. Then there are $\delta_0>0$ and $M_0>0$ such that for all $M>M_0$, $f\in C_c^\infty( (0,\infty)\times \mathbb{R}^{n-1}\times \mathbb{R}^{n-1})$,  $b_1,b_0,v,\chi,\chi_+\rho\in C_c^\infty(\mathbb{R}^n\times \mathbb{R}^{n-1};\mathbb{R})$, $t,t_0\in C^\infty(\mathbb{R}^{n-1}\times \mathbb{R}^{n-1};\mathbb{R})$, and $\psi\in C_c^\infty(\mathbb{R}^n\times \mathbb{R}^n;\mathbb{R})$ satisfying
\begin{gather*}
b_0|_{x_1=0}\geq 0,\qquad \supp b_i\cap \supp (1-\chi)=\emptyset,\qquad \supp \chi\cap\supp(1-\chi_+)=\emptyset \\
\supp \chi_+\subset \{|r|\leq \delta_0\}\cap \{\partial_{x_1}r>0\}, \\
\supp \rho\subset \{r>0\},\qquad\supp \rho\cap \supp (1-\chi)=\emptyset, \qquad b_1|_{x_1=0}=-t^2,\qquad  b_0|_{x_1=0}=tt_0.
\end{gather*}
 setting $b:=b_1\xi_1+b_0$,
\begin{gather*}
\{p,b\}+bM|\xi'|=-\psi^2+\rho(\xi_1-r^{1/2}),\quad b=v^2,\quad \text{when }p=0,
\end{gather*}
and for all $(x',\xi')\in \supp \chi$ either $r(x,\xi')<0$ or there are $0\leq s_{\pm}(x,\xi')\leq 1$ such that 
$$
 \varphi_{-s_{\pm}}(x,\pm \sqrt{r(x,\xi')},\xi')\in \{f>0\},\text{ and } \Big(\bigcup_{0\leq s\leq s_{\pm}}\varphi_{-s}(x,\pm\sqrt{r(x,\xi')},\xi')\Big)\cap \supp (1-\chi)=\emptyset.
$$
Then there are $C>0$, $\Psi_j\in C^\infty(\mathbb{R}_{x_1};\Psi^{\comp}(\mathbb{R}^{n-1}))$, $j=0,1$ such that for all $0<h<1$, defining $X:=\chi(x,\hsc D_{x'})$, $B_i:=b_i(x,\hsc D_{x'})$, $T:=t(x',\hsc D_{x'})$, 
\begin{equation}
\label{e:diffractive1}
\begin{aligned}
&\|\Psi_0Xu+\Psi_1\hsc D_{x_1}Xu\|_{L^2(\Omega)}+\|T\hsc D_{x_1}Xu(0)\|_{L^2}+\|BXu\|_{L^2}\\
&\leq  C\hsc^{-1}\|XPu\|+C(1+\e^{-1})\| X(Q\hsc D_{x_1}u+u)(0)\|_{L^2}+\|Fu\|_{L^2}\\
&\qquad +C\hsc^{\frac{1}{2}}(\|X_+u\|_{L^2}+\|X_+\hsc D_{x_1}u\|_{L^2}+\|\hsc D_{x_1}Xu(0)\|_{L^2}+\|(Xu+Q\hsc D_{x_1}Xu)(0)\|_{L^2})\\
&\qquad +C\hsc^N(\|u\|_{L^2}+\|\hsc D_{x_1}u\|_{L^2}+\|Pu\|_{L^2}+\|u(0)\|_{L^2}+\|\hsc D_{x_1}u(0)\|_{L^2}),
\end{aligned}
\end{equation}
and $\sigma(\Psi_j)|_{r=0}(0,x',0,\xi')=\partial_{\xi_1}^j\psi(0,x',0,\xi')|_{r=0}.$
\end{lemma}
\begin{proof}
By Lemma~\ref{l:intByParts}
\begin{equation}
\label{e:weird0}
\begin{aligned}
2\hsc^{-1}\Im \langle PXu,BXu\rangle &=\sum_{j,k=0}^1  \langle E_{jk}(x',\hsc D_{x'})(\hsc D_{x_1})^kXu,(\hsc D_{x_1})^j(Xu)\rangle_{\Gamma}\\
&\quad+\sum_{j,k=0}^1\langle C_{jk}(\hsc D_{x_1})^kXu,(\hsc D_{x_1})^j(Xu)\rangle_\Omega,
\end{aligned}
\end{equation}
where
$$
\sum c_{jk}\xi_1^{j+k}=\{p,b\}+2b\Im p_{-1}=-bM-\psi^2,\qquad \text{when }\xi_1^2=r.
$$
Hence, by Lemma~\ref{l:systemToEstimate}, 
\begin{equation}
\label{e:weird1}
\begin{aligned}
&\Re \sum_{j,k=0}^1\langle C_{jk}(\hsc D_{x_1})^kXu,(\hsc D_{x_1})^j(Xu)\rangle_\Omega\\
&\qquad +\Re \langle ( M-2\Im p_{-1}(x,\hsc D_{x'}))Xu,BXu\rangle -\Re\langle (\hsc D_{x_1}-\Lambda(x,\hsc D_{x'}))Xu,\rho(x,\hsc D_{x'})Xu\rangle\\
&\leq -\|\Psi_0Xu+\Psi_1\hsc D_{x_1}Xu\|^2 +C\hsc^{-1}\|PXu\|_{L^2}^2\\
&\qquad+C\hsc(\|Xu\|^2_{L^2}+\|\hsc D_{x_1}Xu\|_{L^2}^2+\|(Q\hsc D_{x_1}Xu+Xu)(0)\|_{L^2}^2+\|\hsc D_{x_1}Xu(0)\|_{L^2}^2).
\end{aligned}
\end{equation}
Next, observe that for $M$ large enough,
$$
-Mb+2\Im p_{-1}b+\e |b|^2=-(M-2\Im p_{-1} - v^2)v^2,\text{ when }\xi_1^2=r.
$$
Therefore, apply Lemma~\ref{l:systemToEstimate} again we obtain
\begin{equation}
\label{e:weird2}
\begin{aligned}
&-\Re \langle ( M-2\Im p_{-1}(x,\hsc D_{x'}))Xu,BXu\rangle +\|BXu\|_{L^2}^2\\
&\leq C\hsc^{-1}\|PXu\|_{L^2}^2\\
&\qquad+C\hsc(\|Xu\|^2_{L^2}+\|\hsc D_{x_1}Xu\|_{L^2}^2+\|(Q\hsc D_{x_1}Xu+Xu)(0)\|_{L^2}^2+\|\hsc D_{x_1}Xu(0)\|_{L^2}^2)
\end{aligned}
\end{equation}

Since $\supp \rho\cap \supp (1-\chi)=\emptyset$, there is $\tilde{\chi}\in C_c^\infty(\mathbb{R}^n\times \mathbb{R}^{n-1})$ such that $\supp \rho\cap \supp (1-\tilde{\chi})=\emptyset$, $\supp \tilde{\chi}\cap \supp(1-\chi)=\emptyset$. Hence, using Lemma~\ref{l:hyperbolic} with $A=\tilde{\chi}(x,\hsc D_{x'})X$, $B'=X|_{x_1=0}$, $B_1=X$, and $B=F$,   by Lemma~\ref{l:hyperbolic},
\begin{equation}
\label{e:weird3}
\begin{aligned}
&\langle (\hsc D_{x_1}-\Lambda(x,\hsc D_{x'}))Xu,\rho(x,\hsc D_{x'})Xu\rangle\\
&\leq \|(\hsc D_{x_1}-\Lambda(x,\hsc D_{x'}))\tilde{\chi}(x,\hsc D_{x'})Xu\|^2_{L^2}+\|\rho(x,\hsc D_{x'})\tilde{\chi}(x,\hsc D_{x'})Xu\|^2_{L^2}\\
&\qquad+O(\hsc^{\infty})(\|\hsc D_{x_1}u\|^2_{L^2}+\|u\|^2_{L^2})\\
&\leq  C\hsc^{-2}\|XP u\|^2_{L^2}+\|X(Q\hsc D_{x_1}u+u)(0)\|^2_{L^2}+\|Fu\|^2_{L^2} \\
&\qquad +O(\hsc^{\infty})(\|\hsc D_{x_1}u\|^2_{L^2}+\|u\|^2_{L^2}+\|Pu\|^2_{L^2}+\|(Q\hsc D_{x_1}u+u)(0)\|^2_{L^2}+\|\hsc D_{x_1}u(0)\|^2_{L^2}).
\end{aligned}
\end{equation}

Next, we consider 
$$
\sum_{j,k=0}^1  \langle E_{jk}(x',\hsc D_{x'})(\hsc D_{x_1})^kXu,(\hsc D_{x_1})^j(Xu)\rangle_{\Gamma}.
$$
First, notice that, since $\sigma(E_{11})=-t^2$,
\begin{equation}
\label{e:weird3b}
\langle E_{11}\hsc D_{x_1}Xu,\hsc D_{x_1}Xu\rangle_{\Gamma} +\|T\hsc D_{x_1}Xu(0)\|^2\leq C\hsc\|\hsc D_{x_1}Xu(0)\|_{L^2}^2,
\end{equation}
Next, since $\Re Q\geq 0$ on $\WF(B_0|_{x_1=0})$ and $\sigma(B_0)\geq 0$,
\begin{equation}
\label{e:weird4}
\begin{aligned}
&\Re (\langle E_{10}\hsc D_{x_1}Xu,Xu\rangle_{\Gamma}+\langle E_{01}Xu,\hsc D_{x_1}Xu\rangle_{\Gamma}) \\
&=2\Re\langle B_0\hsc D_{x_1}Xu,(Xu+Q\hsc D_{x_1}Xu)\rangle_{\Gamma}-2\Re\langle B_0 \hsc D_{x_1}Xu,Q\hsc D_{x_1}Xu\rangle_{\Gamma}\\
&\leq 2\Re\langle t_0(x',\hsc D_{x'})T\hsc D_{x_1}Xu,(Xu+Q\hsc D_{x_1}Xu)\rangle_{\Gamma}-2\Re\langle B_0 \hsc D_{x_1}Xu,Q\hsc D_{x_1}Xu\rangle_{\Gamma}\\
&\qquad +C\hsc\|\hsc D_{x_1}Xu(0)\|_{L^2}^2+C\hsc\|X(Q\hsc D_{x_1}u+u)(0)\|_{L^2}^2\\
&\qquad+O(\hsc^{\infty})(\|u(0)\|_{L^2}^2+\|\hsc D_{x_1}u(0)\|_{L^2}^2).\\
&\leq \e\|T\hsc D_{x_1}Xu\|^2_{L^2}+C\hsc\|\hsc D_{x_1}Xu(0)\|_{L^2}^2+C\e^{-1}\|X(u(0)+Q\hsc D_{x_1}u(0))\|_{L^2}^2\\
&\quad +O(\hsc^{\infty})(\|u(0)\|_{L^2}^2+\|\hsc D_{x_1}u(0)\|_{L^2}^2).
\end{aligned}
\end{equation}
Finally,
\begin{equation}
\label{e:weird5}
\begin{aligned}
&2\langle E_{00}Xu,Xu\rangle_{\Gamma}=\langle B_1(R+R^*)Xu,Xu\rangle_{\Gamma}\\
&=\langle B_1(R+R^*)Q\hsc D_{x_1}Xu,Q\hsc D_{x_1}Xu\rangle_{\Gamma} -\langle B_1(R+R^*)Q\hsc D_{x_1}Xu,Xu+Q\hsc D_{x_1}Xu\rangle_{\Gamma}\\
&\qquad +\langle B_1(R+R^*)(Xu+Q\hsc D_{x_1}Xu),(Xu+Q\hsc D_{x_1}Xu)\rangle_{\Gamma}\\
&\qquad - \langle B_1(R+R^*)(Xu+Q\hsc D_{x_1}Xu),Q\hsc D_{x_1}Xu\rangle_{\Gamma}\\
&\leq \langle T^*T(R+R^*)Q\hsc D_{x_1}Xu,Q\hsc D_{x_1}Xu\rangle  -\langle T^*T(R+R^*)Q\hsc D_{x_1}Xu,Xu+Q\hsc D_{x_1}Xu\rangle\\
&\qquad +\langle T^*T(R+R^*)(Xu+Q\hsc D_{x_1}Xu),(Xu+Q\hsc D_{x_1}Xu)\rangle\\
&\qquad- \langle T^*T(R+R^*)(Xu+Q\hsc D_{x_1}Xu),Q\hsc D_{x_1}Xu\rangle\\
&\qquad C\hsc(\|\hsc D_{x_1}Xu\|_{L^2}^2+\|X(u+Q\hsc D_{x_1}u)\|_{L^2}^2 +O(\hsc^{\infty})\|\hsc D_{x_1}u(0)\|_{L^2}+\|u(0)\|_{L^2})\\
&\leq C\delta \|T\hsc D_{x_1}Xu\|_{L^2}^2+ C\|X(Q\hsc D_{x_1}u+u)(0)\|_{L^2}^2\\
&\qquad  C\hsc(\|\hsc D_{x_1}Xu\|_{L^2}^2+\|X(u+Q\hsc D_{x_1}u)\|_{L^2}^2 +O(\hsc^{\infty})\|\hsc D_{x_1}u(0)\|_{L^2}+\|u(0)\|_{L^2})\\
\end{aligned}
\end{equation}

Combining~\eqref{e:weird0}, \eqref{e:weird1} \eqref{e:weird2}, \eqref{e:weird3}, \eqref{e:weird3b}, \eqref{e:weird4}, and~\eqref{e:weird5}, and choosing $\e>0$ small enough,
\begin{align*}
&2\hsc^{-1}\Im \langle PXu,BXu\rangle +\|\Psi_0Xu+\Psi_1\hsc D_{x_1}Xu\|_{L^2}^2+\|BXu\|_{L^2}^2+\|T\hsc D_{x_1}Xu(0)\|_{L^2}^2\\
&\leq C\hsc^{-1}\|PXu\|_{L^2}^2+C\hsc^{-2}\|XPu\|^2_{L^2}+C(1+\e^{-1})\|X(Q\hsc D_{x_1}u+u)(0)\|_{L^2}^2\\
&\qquad +C\hsc(\|Xu\|_{L^2}^2+\|\hsc D_{x_1}Xu\|_{L^2}^2+\|(Q\hsc D_{x_1}Xu+Xu)(0)\|_{L^2}^2+\|\hsc D_{x_1}Xu\|_{L^2}^2)\\
&\qquad +C\hsc^N(\|\hsc D_{x_1}u\|_{L^2}^2+\|u\|_{L^2}^2+\|Pu\|_{L^2}^2+\|(Q\hsc D_{x_1}u+u)(0)\|_{L^2}^2+\|\hsc D_{x_1}u(0)\|_{L^2}^2)
\end{align*}
Finally, 
\begin{align*}
 2\hsc^{-1}\Im \langle PXu,BXu\rangle&= 2\hsc^{-1}\Im \langle XP+[P,X]u,BXu\rangle\\
 &\geq 2\hsc^{-1}\langle XPu,BXu\rangle +O(\hsc^{\infty})(\|u\|_{L^2}^2+\|\hsc D_{x_1}u\|_{L^2}^2)\\
 &\geq 
-C\hsc^{-2}\|PXu\|_{L^2}^2-\frac{1}{2}\|BXu\|_{L^2}^2+O(\hsc^{\infty})(\|u\|_{L^2}^2+\|\hsc D_{x_1}u\|_{L^2}^2)
\end{align*}
and hence the estimate follows after we notice that 
\begin{align*}
\|PXu\|_{L^2}&= \|XPu\|_{L^2}+\|[P,X]u\|_{L^2}\\
&\leq \|XP\|_{L^2}+C\hsc(\|X_+u\|_{L^2}+\|X_+\hsc D_{x_1}u\|_{L^2})+C\hsc^N(\|u\|_{L^2}+\|\hsc D_{x_1}u\|_{L^2})
\end{align*}
\end{proof}

We now construct the functions required to make use of Lemma~\ref{l:ivrii}.
\begin{proposition}
Suppose that $r(0,y',\eta')=0$ and $\partial_{x_1}r(0,y',\eta')>0$ and let $\gamma(t)=\exp(tH_p)(0,y',\eta')$. Then for all $N>0$, $\chi\in C_c^\infty(\mathbb{R}^{n}\times \mathbb{R}^{n-1})$ with $\chi(0,y',\eta')=1$ there is $0<s<1$ such that for all $f\in C_c^\infty(\mathbb{R}_+\times \mathbb{R}^{n-1}\times \mathbb{R}^n)$ with $f(\gamma(-s))=1$ there are $C>0$, $a\in C_c^\infty(\mathbb{R}^n\times \mathbb{R}^{n-1})$ with $a(0,y',\eta')=1$  such that for all $0<h<1$
\begin{align*}
&\|A(x,\hsc D_{x'})u\|_{L^2} +\|A(0,x',\hsc D_{x'})\hsc D_{x_1}u\|_{L^2}\\
&\leq C(\hsc^{-1}\|XPu\|_{L^2}+\|Fu\|_{L^2}+\|X(Q\hsc D_{x_1}u+u)(0)\|_{L^2})\\
&\qquad+C\hsc^N(\|Pu\|_{L^2}+\|u\|_{L^2}+\|(Q\hsc D_{x_1}u+u)(0)\|_{L^2}).
\end{align*}
\end{proposition}

\begin{proof}
To prove the proposition, we show that given a neighborhood, $V$ of $(0,y',\eta')$ there are $a,a_1\in C_c^\infty(\mathbb{R}^{n}\times \mathbb{R}^{n-1})$ with $a(0,y',\eta')=1$ and $\supp a_1\subset V$ such that 
\begin{equation}
\label{e:realGoal}
\begin{aligned}
&\|a(x,\hsc D_{x'})u\|_{L^2}+\|a(x,\hsc D_{x'})\hsc D_{x_1}u\|_{L^2}+\|a(0,x',\hsc D_{x'})\hsc D_{x_1}u(0)\|_{L^2}\\
&\leq C(\hsc^{-1}\|XPu\|_{L^2}+\|Fu\|_{L^2}+\|X(Q\hsc D_{x_1}u+u)(0)\|_{L^2})\\
&\qquad+C\hsc^{1/2}(\|A_1u\|_{L^2}+\|A_1\hsc D_{x_1}u\|_{L^2}+\|A_1\hsc D_{x_1}u(0)\|_{L^2})\\
&\qquad+C\hsc^N(\|Pu\|_{L^2}+\|u\|_{L^2}+\|(Q\hsc D_{x_1}u+u)(0)\|_{L^2}).
\end{aligned}
\end{equation}
Then, iterating this estimate and using propagation of singularities away from the boundary implies the proposition.

Define
$$
\phi:=\xi_1+\phi_0(x,\xi'),\qquad \phi_0=x_1^2+|x'-y'|^2+|\xi'-\eta'|^2.
$$
Then,
$$
H_p\phi=\{\xi_1^2-r,\phi\}=\partial_{x_1}r +2\xi_1\partial_{x_1}\phi-H_r\phi.
$$
There is a neighborhood, $U$, of $(y',\eta')$ and $c>0$ such that 
$$
\partial_{x_1}r>4c>0,\quad |H_r\phi_0|+ |\partial_{x_1}\phi|\leq c.
$$
Hence, 
$$
H_p\phi>c>0,\qquad \text{ on }\{(x,\xi_1,\xi')\,:\, |\xi_1|\leq 1,\,(x,\xi')\in U\}.
$$
Let $\delta>0$ small enough such that 
$$
\{\phi_0(x',\xi')\leq 3\delta\}\subset U\cap V\cap \{ \chi>0\}.
$$

We start by finding a function $\tilde{b}$ that has all the required properties except that it is not linear in $\xi_1$.
Set 
$$
\tilde{b}:=\chi_2(\phi_0/\delta)^2\chi_0(1-\phi/\delta),
$$
where 
$$
\chi_0(t):=\begin{cases} \exp(-1/t)&t>0\\0&t\leq 0\end{cases},\quad  \chi_2\in C_c^\infty(-3,3;[0,1]),\, \supp (1-\chi_2)\cap [-2,2]=\emptyset.
$$
For later use, we also let $\chi_1\in C_c^\infty((-2,2);[0,1])$ with $\supp (1-\chi_1)\cap [-1,1]=\emptyset$.

We claim that 
\begin{equation}
\label{e:diffractiveClaim}
\begin{gathered}
H_p\tilde{b}+M\tilde{b}=-\psi^2+\rho(\xi_1-r^{1/2}),\quad \text{when }p=0\\
\psi:=\chi_1(\delta^{-1}\xi_1)N^{1/2},\\ -2r
^{1/2}\rho:=\begin{cases}-(1-\chi_1(\delta^{-1}\xi_1)^2)N+\chi_0(1-\delta^{-1}\phi)H_p\chi_2(\delta^{-1}\phi_0)^2|_{\xi_1=-\sqrt{r}}&r\geq0,\\
0&r<0,
\end{cases}\\
N:=\chi_2(\delta^{-1}\phi_0)^2(\chi_0'(1-\delta^{-1}\phi)\delta^{-1}H_p\phi-\chi_0(1-\delta^{-1}\phi)M).
\end{gathered}
\end{equation}
To see this, we observe that  
\begin{gather*}
\supp \tilde{b}\subset \{\phi\leq \delta\}\subset\{\xi_1\leq \delta\},\\
\supp\tilde{b}\cap\supp (\partial(\chi_2(\delta^{-1}\phi_0))\subset \{\phi_0+\xi_1\leq \delta,\,\phi_0\geq 2\delta\}\subset \{ \xi_1\leq -\delta\}.
\end{gather*}
In particular, 
$$
\supp\tilde{b}\cap\supp (\partial(\chi_2(\delta^{-1}\phi_0))\cap \{p=0\}\subset \{r\geq \delta^2\},
$$
and~\eqref{e:diffractiveClaim} follows by a direct calculation.

Now, we need to check that $\psi$ and $\rho$ are smooth function. To see that $\psi\in C^\infty$, observe that on $\supp \psi$, $|\xi_1|\leq 2\delta$ and $|\phi_0|\leq 3\delta$. Therefore $(1-\delta^{-1}\phi)\leq 3$. To use this, observe that 
\begin{align*}
N^{1/2}&=\chi_2(\delta^{-1}\phi_0)[\chi_0'(1-\delta^{-1}\phi_0)\delta^{-1}H_p\phi]^{1/2}(1-M\delta \chi_0(t)/\chi_0'(t)|_{t=1-\delta^{-1}\phi})^{1/2}\\
&=\chi_2(\delta^{-1}\phi_0)[\chi_0'(1-\delta^{-1}\phi_0)\delta^{-1}H_p\phi]^{1/2}(1-M\delta (1-\delta^{-1}\phi)^2)^{1/2}.
\end{align*}
Therefore, on $\supp \psi$, $1-M\delta (1-\delta^{-1}\phi)^2\geq 1-9M\delta$ and for $\delta<\frac{1}{10M}$, using that $(\chi_0')^{1/2}$ is smooth, we see that $N^{1/2}$ is smooth.

Next, to see that $\rho\in C^\infty$, observe that 
$$
\supp (1-\chi_1^2)N\cup \supp \chi_0(1-\delta^{-1}\phi)H_p\chi_2(\delta^{-1}\phi_0)^2\subset \{ \xi_1\leq -\delta\}.
$$
Hence, $\supp\rho\subset r\geq \delta^2$ which implies that $\rho\in C^\infty$. 

Finally, we observe that $\tilde{b}|_{p=0}=v^2$ with $v=\chi_1(\xi_1)\chi_2(\delta^{-1}\phi_0)[\chi_0(1-\delta^{-1}\phi)]^{1/2}$, provided that $\delta\leq 1$. 

With $\tilde{b}$ in hand, we construct the required function $b$. By the Malgrange preparation theorem, there are $\tilde{b}_1,b_0\in C^\infty(\mathbb{R}^n\times \mathbb{R}^{n-1})$ such that 
$$
\tilde{b}(x,\xi)=p(x,\xi)g(x,\xi)+\tilde{b}_1(x,\xi')\xi_1+\bar{b}_0(x,\xi')
$$
Define $\bar{b}:=\bar{b}_1(x,\xi')\xi_1+\bar{b}_0(x,\xi')$.

Observe that on $p=0$, $H_p\bar{b}=H_p\tilde{b}$ and $b=\bar{b}$.  Hence,
\begin{equation}
\label{e:target}
H_p\bar{b}+M\bar{b}=-\psi^2+\rho(\xi_1-r^{1/2}),\qquad \bar{b}=v^2\qquad\text{ on } p=0.
\end{equation}
In addition, 
\begin{gather*}
\bar{b}_1(x,\xi')=\frac{\tilde{q}(x,\sqrt{r(x,\xi')},\xi')-\tilde{q}(x,-\sqrt{r(x,\xi')},\xi')}{2\sqrt{r(x,\xi')}},\text{ on } r(x,\xi')>0,\\
\bar{b}_0(x,\xi')=\frac{\tilde{q}(x,\sqrt{r(x,\xi')},\xi')+\tilde{q}(x,-\sqrt{r(x,\xi')})}{2},\text{ on }r(x,\xi')\geq 0.
\end{gather*}
We first modify $\bar{b}_1$ and find $t$. Notice that $-\partial_{\xi_1}\tilde{b}=\chi_2(\delta^{-1}\phi_0)^2\chi_0'(1-\phi/\delta)\delta^{-1}\phi_1$. Hence by~\cite[Lemma 24.4.9]{Ho:07} there is $W_1\in C^\infty(\mathbb{R}\times \mathbb{R}^n\times \mathbb{R}^{n-1})$ such that $\tilde{b}_1=-W_1^2(r,x,\xi')$ when $r>0$. Setting $b_1:=-W_1^2(r(x,\xi'),x,\xi')$ modifies $\bar{b}_1$ on $r<0$ and hence does not affect~\eqref{e:target}. Therefore, we set $t(x,\xi'):=W(r(x,\xi'),0,x,\xi')$.   

Now, observe that  there is $\tilde{t}\in C^\infty$ with $\tilde{t}(0,x,\xi')> 0$ such that 
$$
t=\chi_2(\phi_0/\delta)\tilde{t}(r(x',\xi'),x,\xi').
$$
On the other hand, by \cite[Theorem C.4.4]{Ho:07}, there is $W_2\in C^\infty(\mathbb{R}\times \mathbb{R}^{n}\times \mathbb{R}^{n-1})$ such that 
$$
\bar{b}_0(x,\xi')=\chi_2^2(\phi_0/\delta)W_2(r(x,\xi'),x,\xi'),\qquad r>0.
$$
Setting $b_0(x,\xi')=\chi_2^2(\phi_0/\delta)W_2(r(x,\xi'),x,\xi')$ modifies $b_0$ only on $r<0$ and hence does not affect~\eqref{e:target}. Moreover, we have $b_0=tt_0$ with $t_0:=\chi_2(\phi_0/\delta)W_2(r(x,\xi'),x,\xi')/\tilde{t}(r(x',\xi'),x,\xi')\in C^\infty$ since $\tilde{t}$ is non-vanishing near $r=0$.

Defining $b:=b_1\xi_1+b_0$ and letting $\tilde{\chi},\tilde{\chi}_+\in C_c^\infty(\mathbb{R}^n\times \mathbb{R}^{n-1})$ with $\supp \tilde{\chi}_+ \subset U\cap \{\chi>0\}$, $\supp (1-\chi_+)\cap \supp \chi=\emptyset$, and $\supp (1-\chi)\cap \supp b_i=\emptyset$. we have verified all the hypotheses of Lemma~\ref{l:ivrii} and hence the estimate~\eqref{e:diffractive1} holds for any $\delta>0$ small enough. Now, observe that 
$$
\begin{gathered}
t(y',\eta')=(\chi_0'(1)\delta^{-1})^{1/2},\qquad \psi_0^2(0,y',\eta')=\chi_0'(1)\partial_{x_1}r(0,y',\eta')\delta^{-1}-M\chi_0(1)\\
2\psi_0\psi_1(0,y',\eta')=-\chi_0''(1)\partial_{x_1}r(0,y',\eta')\delta^{-2}+\chi_0'(1)M\delta^{-1}.
\end{gathered}
$$
In particular $t(y',\eta')>c$ and $\psi_1/\psi_0=-c\delta^{-1}+O(1).$ hence, choosing, for instance $\delta$ and $\delta/2$, there is $a\in C_c^\infty(\mathbb{R}^n\times \mathbb{R}^{n-1})$ such that $a(0,y',\eta')=1$, 
\begin{align*}
\|a(x,\hsc D_{x'})u\|_{L^2}+\|a(x,\hsc D_{x'})\hsc D_{x_1}u\|_{L^2}&\leq C (\|\Psi_0^{\delta}u+\Psi_1^{\delta}\hsc D_{x_1}u\|_{L^2}+\|\Psi_0^{\delta/2}u+\Psi_1^{\delta/2}\hsc D_{x_1}u\|_{L^2})\\
&\qquad+C\hsc^N(\|u\|_{L^2}+\|\hsc D_{x_1}u\|_{L^2})\\
\|a(0,x',\hsc D_{x'})\hsc D_{x_1}u(0)\|_{L^2}&\leq C\|T^\delta \hsc D_{x_1}u(0)\|_{L^2}+C\hsc^N\|\hsc D_{x_1}u\|_{L^2}.
\end{align*}
Hence, using that $\supp \tilde{\chi}_+\subset V\cap \{\chi>0\}$ and letting $s$ small enough that $\exp(-tH_p)(0,y',\eta')\subset \{\chi>0\}$, for $0\leq t\leq s$, we have proved~\eqref{e:realGoal} and hence the proposition.
\end{proof}

\subsection{Non-diffractive points}

Finally, we consider rays tangent to the boundary that are non-diffractive. Once again, we proceed by using a positive commutator type estimate followed by a careful construction of an escape function.
\begin{lemma}
\label{l:gbb}
Let $Q\in \Psi^{-1}(\Gamma)$ with $\Re\sigma(Q)|_{r=0}\geq 0$. Then there are $\delta_0,M_0>0$ such that for all  $N>0$, $M\geq M_0$, $b,v,a,f,\chi,\chi_+\in C_c^\infty(\mathbb{R}^n\times \mathbb{R}^{n-1};\mathbb{R})$, satisfying 
\begin{gather*}
b|_{x_1=0}\geq 0,\qquad \supp b\cap \supp (1-\chi)=\emptyset,\qquad\supp \chi\cap\supp(1-\chi_+)=\emptyset,\qquad\supp \chi_+\subset \{|r|\leq \delta_0\},\\
\{p,b\}+bM=-\psi^2-a^2+f^2,\quad b=v^2,\qquad \text{when }\xi_1^2=r,
 \end{gather*}
with $\psi \in C_c^\infty(\mathbb{R}^n\times \mathbb{R}^n;\mathbb{R})$ therehere is $C>0$ such that, defining $A:=a(x,\hsc D_{x'})$ $B:=b(x,\hsc D_{x'})$, $F=f(x,\hsc D_{x'})$, $X:=\chi(x,\hsc D_{x'})$, for all and $0<\mu<1$, $0<h<1$,
\begin{align*}
&\|\Psi_0u+\Psi_1\hsc D_{x_1}u\|_{L^2}+\|AXu\|_{L^2}\\
&\leq  C\hsc^{-1}\|X_+Pu\|+C\mu^{-1}\|X_+(u+Q\hsc D_{x_1}u)(0)\|_{L^2}+\|FXu\|_{L^2}+(\mu+\hsc^{1/2})\|X_+u\|_{L^2}\\
&+C\hsc^N(\|u(0)\|_{L^2}+\|\hsc D_{x_1}u(0)\|_{L^2}),
\end{align*}
where $\sigma(\Psi_j)(0,x',0,\xi')|_{r=0}=\partial_{\xi_1}^j\psi_j(0,x',0,\xi')|_{r=0}.$
\end{lemma}
\begin{proof}
The proof is similar to that of Lemma~\ref{l:ivrii}. The changes are that we estimate
\begin{align*}
&\Re \sum \langle C_{jk}(x,\hsc D_{x'})(\hsc D_{x_1})^kXu,(\hsc D_{x_1})^jXu\rangle +\Re \langle (M-2\Im p_{-1}(x,\hsc D_{x'}))Xu,BXu\rangle \\
&\qquad +\|AXu\|_{L^2}^2-\|FXu\|_{L^2}^2\\
&\leq -\|\Psi_0Xu+\Psi_1\hsc D_{x_1}Xu\|^2 +C\hsc^{-1}\|PXu\|_{L^2}^2+C\hsc(\|Xu\|^2_{L^2}+\|\hsc D_{x_1}Xu\|_{L^2}^2)\\
&\qquad+C\hsc\|\hsc D_{x_1}Xu\|^2+C\hsc\|(uX+Q\hsc D_{x_1}Xu)\|_{L^2}^2,
\end{align*}
 that $E_{00}=E_{11}=0$, and instead of~\eqref{e:weird4} we use that $\Re Q\geq 0$ on $\WF(B|_{x_1=0})$ and $\sigma(B)\geq 0$, to obtain
\begin{equation}
\label{e:weird42}
\begin{aligned}
&\Re (\langle E_{10}\hsc D_{x_1}Xu,Xu\rangle_{\Gamma}+\langle E_{01}Xu,\hsc D_{x_1}Xu\rangle_{\Gamma}) \\
&=2\Re\langle B\hsc D_{x_1}Xu,(Xu+Q\hsc D_{x_1}Xu)\rangle_{\Gamma}-2\Re\langle B \hsc D_{x_1}Xu,Q\hsc D_{x_1}Xu\rangle_{\Gamma}\\
&\leq \mu^2\|B\hsc D_{x_1}Xu(0)\|^2_{L^2}+C\hsc\|\hsc D_{x_1}Xu(0)\|_{L^2}^2+C\mu^{-2}\|X(u(0)+Q\hsc D_{x_1}u(0))\|_{L^2}^2\\
&\quad +O(\hsc^{\infty})(\|u(0)\|_{L^2}^2+\|\hsc D_{x_1}u(0)\|_{L^2}^2).
\end{aligned}
\end{equation}
We then use Lemma~\ref{l:microNormalDerivative} to estimate $\|B\hsc D_{x_1}Xu(0)\|$ and $\|\hsc D_{x_1}Xu\|_{L^2}$ completing the proof of the lemma.
\end{proof}

\begin{lemma}
Let $N>0$ $(y',\eta')\in \mathbb{R}^{2n-2}$, $r_0(y',\eta')=0$ and $\chi\in C_c^\infty(\mathbb{R}^n\times \mathbb{R}^{n-1})$ with $\chi(0,y',\eta')=1$. Then there are $\e_0>0, C_0>0$ such that for all $0<\e<\e_0$, there is $\delta_\e>0$ such that for all $\delta \in (0,\delta_\e)$ and $f\in C_c^\infty(\mathbb{R}^{n}\times \mathbb{R}^{n-1})$ with 
$$
\{(x_1,x',\xi')\,:\, x_1<C_0\e\delta, \, (x',\xi')-(y',\eta')-\delta H_{r_0}(y',\eta')|<C_0\e \delta\}\subset \{f>0\},
$$
there is $a\in C_c^\infty(\mathbb{R}^{n}\times \mathbb{R}^{n-1})$ with $a(0,y',\eta')=1$ such that for any $\beta>0$,  there is $C>0$ such that for all $0<\hsc<1$, 
\begin{align*}
\|Au\|_{L^2}&\leq C(\hsc^{-1}\|XPu\|_{L^2}+ \|Fu\|_{L^2}+ \hsc^{-\beta}\|X(Q\hsc D_{x_1}u+u)(0)\|_{L^2})\\
&\qquad +C\hsc^N(\|u\|_{L^2}+\|Pu\|_{L^2} +\|(Q\hsc D_{x_1}u+u)(0)\|_{L^2}).
\end{align*}
\end{lemma}
\begin{proof}
In order to start the proof, we need to define a hypersurface transversal to the flow of $H_{r_0}$. For this, we let
$N\in C^\infty(\mathbb{R}^{n-1}\times \mathbb{R}^{n-1})$ with $H_{r_0}N(x',\xi')=1$ in a neighborhoood of $(y', \eta')$ and $N(y',\eta')=0$. We then define 
$$
\Sigma_{(y',\eta')}:=\{(x',\xi')\,:\, N(x',\xi')=0\}.
$$ 
Let $\omega\in C^\infty(\mathbb{R}^{n-1}\times \mathbb{R}^{n-1})$ solve
\begin{equation}
H_{r_0}\omega=0 \text{ in neighborhood of }(y',\eta'),\qquad \omega=|x'-y'|^2+|\xi'-\eta'|^2,\qquad (x',\xi')\in \Sigma_{(y',\eta')}.
\end{equation}
Since $H_{r_0}N(y',\eta')=1$, this uniquely defines $\omega$ in a neighborhood of $(y',\eta')$. 

Let $\gamma:=\{ \exp(tH_{r_0})(y',\eta)\,:\,|t|\leq 1\}$. Then, $\omega|_{\gamma}=0$, $\partial\omega|_{\gamma}=0$, and  there is $c>0$ such that $\omega(x',\xi')\geq cd( (x',\xi'),\gamma)^2$. In particular, since $r_0|_{\gamma=0}$ and $\partial \omega$ vanish on $\gamma$, 
$$
|r_0|+|\partial \omega|\leq C\omega^{1/2}.
$$

Now, set 
$$
\phi(x,\xi'):= -N(x',\xi')+\frac{x_1}{\e}+\frac{\omega}{\e^2\delta}.
$$
Then,
$$
H_{p}\phi= -H_rN+\frac{2\xi_1}{\e}-\frac{H_{r}\omega}{\delta \e^2}.
$$
Notice that 
$$
H_r\omega= H_{r_0}\omega +(H_r-H_{r_0})\omega= O(|\partial (r(x_1,x',\xi')-r(0,x',\xi'))| |\partial\omega|)=O(x_1\omega^{1/2}),
$$
and when $\xi_1^2=r$, 
$$
\xi_1^2=r\leq r_0+O(x_1)=O(\omega^{1/2}+x_1). 
$$
Next, when $N(x',\xi')\leq 2\delta $ and $\phi(x,\xi')\leq 2\delta$, we have
$$
4\delta\e \geq N\e +2\delta\e \geq \e \phi +N\e \geq x_1+\frac{\omega}{\delta\e}.
$$
Hence, since $x_1,\omega\geq 0$, 
$$
x_1\leq 4\delta \e,\qquad \omega \leq 4\delta^2\e^2.
$$
Finally, observe that 
$$
H_rN=1+(H_r-H_{r_0})N=1+O(x_1)
$$
Therefore, for $\delta>0$ small enough and $\e\leq 1$,
\begin{equation*}
\begin{gathered}
H_p\phi =H_rN+O(\e^{-1}(\omega^{1/2}+x_1)^{1/2})+O(\delta^{-1}\e^{-2}x_1\omega^{1/2})=1+O(\delta^{1/2}\e^{-1/2})\\
\text{ on }\{p=0\}\cap \{ N\leq 2\delta\}\cap \{\phi\leq 2\delta\}.
\end{gathered}
\end{equation*}

Now, define
$$
\chi_0(t):=\begin{cases} \exp(-1/t)&t>0\\ 0&t\leq 0,\end{cases}
$$
and let $\chi_1(t)\in C^\infty(\mathbb{R})$ with $\supp \chi_1\subset (0,\infty)$, $\supp (1-\chi_1)\subset (-\infty,1)$ and such that for all $\beta <1$, there is $C_{\alpha \beta}$ such that 
$$
|D^\alpha \chi_1'|\leq C_{\alpha\beta}(\chi_1')^\beta. 
$$

Let $0\leq t\leq 1$ and set
$$
b_t(x',\xi'):=\chi_0(1+t-\delta^{-1}\phi)\chi_1(\e^{-1}\delta^{-1}(-N(x',\xi')+\delta) +t).
$$
Then,
$$
H_pb_t= \chi_0'\delta^{-1}H_p\phi \chi_1 +\chi_0\chi_1'\e^{-1}\delta^{-1}H_rN
$$
Observe that for $t_1<t_2$, 
\begin{gather*}
 \supp b_{t}\subset \{ \phi\leq \delta (1+t)\}\cap \{ N\leq \delta (1+\e t) \}\subset \{x_1\leq 4\delta \e,\ \omega \leq 4\delta^2\e^2\}\cap \{N\leq 2\delta\}\\
 \supp b_{t_1}\subset \{b_{t_2}>0\},\qquad
 \end{gather*}
In particular, for $\e,\delta$ small enough, $\supp b_t\subset \{\chi>0\}$. 

Now, define 
$$
f_t:=(\chi_0(1+t-\delta^{-1}\phi)\chi_1'(\e^{-1}\delta^{-1}(-N+\delta)+t)\e^{-1}\delta^{-1}H_rN)^{1/2}.
$$
We claim that $f_t$ is smooth and,  setting $(y'(\delta),\eta'(\delta)):=\exp(\delta H_{r_0})(y',\eta')$, 
\begin{equation}
\label{e:eSupp}
\supp f_t\subset \{x_1^2+|x'-y'(\delta)|^2+|\xi'-\eta'(\delta)|^2\leq C(\e \delta)^2\}\subset\{ f>0\}.
\end{equation}

To prove the claim first observe that 
$$
\supp f_t\subset \{|\delta -N|\leq \e \delta\}\cap \{ \phi\leq 2\delta\}\subset \{ x_1\leq 4\e \delta, \omega \leq 4\e^2\delta^2\}.
$$
Therefore, since $(\chi_1')^{1/2}$, $\chi_0^{1/2}$ are smooth functions, and $H_rN >c>0$ on $\supp f_t$, $f_t$ is smooth.   Moreover, on $\supp f_t$, 
$$
|N(y'(\delta),\eta'(\delta))-N(x',\xi')|\leq \e \delta +O(\delta^2).
$$
Now,  since $\omega\sim d( (x',\xi'),\gamma)^2$ and $H_{r_0}N=1$, 
$$
|N(x',\xi')-N(y'(\delta),\eta'(\delta))^2+\omega\sim  |x'-y'(\delta)|^2+|\xi'-\eta'(\delta)|^2
$$
and hence,  for $\delta>0$ small enough, we have proved~\eqref{e:eSupp}.

Next, we find $a_t$ and $\psi_t$ such that 
$$
\chi_1(\chi_0'(1+t-\delta^{-1}\phi)\delta^{-1}H_p\phi-M\chi_0(1+t-\delta^{-1})\phi)=\psi_t^2+a_t^2.
$$
Define
\begin{align*}
a_t&:=(\chi_1\chi_0'(1+t-\delta^{-1}\phi)/(2\delta))^{1/2},\\
\psi_t&:=\chi_1^{1/2}(\chi_0'(1+t-\delta^{-1}\phi)(H_p\phi-1/2)\delta^{-1}-\chi_0(1+t-\delta^{-1}\phi)M)^{1/2}.
\end{align*}
Now, $a_t\in C^\infty$ since $\chi_1^{1/2}\in C^\infty$ and $(\chi_0')^{1/2}\in C^\infty$. Furthermore, for $0\leq t\leq 1$ and $\e<1$,
$$
a_t(0,y',\eta')=(\chi_0'(1+t)\chi_1(\e^{-1}+t))^{1/2}(2\delta)^{-1/2}>c\delta^{-1/2}
$$ 

To see that $\psi_t$ is smooth in a neighborhood of $p=0$, observe that $\chi_0(s)/\chi_0'(s)=s^2$ and hence, 
\begin{align*}
\psi_t&=\chi_1^{1/2}(\chi_0')^{1/2}((H_p\phi-1/2)\delta^{-1}-(1+t-\delta^{-1}\phi)^2M)^{1/2}.
\end{align*}
Now, $H_p\phi\geq \frac{3}{4}$ on $\supp \chi_1\cap \supp \chi_0\cap \{p=0\}$ and $(1+t-\delta^{-1}\phi)\leq 4$. Hence
$$
(H_p\phi-1/2)\delta^{-1}-(1+t-\delta^{-1}\phi)^2M\geq \frac{1}{4\delta}-16M\geq 1
$$
for $\delta>0$ small enough. Cutting off to an appropriate neighborhood of the characteristic set, we see that $\psi_t\in C^\infty$. 

Now, let  $0\leq t_{<}<\dots<t_1<t_{0}\leq 1$, $i=1,\dots, N$, $\supp b_{t_i}\subset \{a_{t_{i-1}}>0\}$ and hence, letting $\chi_{t_i},\chi_{+,t_i}\in C_c^\infty(\mathbb{R}^{n}\times \mathbb{R}^{n-1})$ with $\supp(1-\chi_{t_i})\cap \supp b_{t_i}=\emptyset$, $\supp \chi_{t_i}\cap \supp(1-\chi_{+,t_i})=\emptyset$, and $\supp \chi_{+,t_i}\subset \{a_{t_{i-1}}>0\}$, we obtain from Lemma~\ref{l:gbb} that for any $\mu>0$, 
\begin{align*}
&\|A_{t_i}Xu\|_{L^2}\\
&\leq  C\hsc^{-1}\|X_{t_i}Pu\|+C\mu^{-1}\|X_{t_i}(u+Q\hsc D_{x_1}u)(0)\|_{L^2}+C\|F_{t_i}Xu\|_{L^2}+(\mu+\hsc^{1/2})\|X_{+,t_i}u\|_{L^2}\\
&\qquad+C\hsc^N(\|u(0)\|_{L^2}+\|\hsc D_{x_1}u(0)\|_{L^2}+\|u\|_{L^2}+\|Pu\|_{L^2})\\
&\leq  C\hsc^{-1}\|XPu\|+C\mu^{-1}\|X(u+Q\hsc D_{x_1}u)(0)\|_{L^2}+C\|Fu\|_{L^2}+(\mu+\hsc^{1/2})\|X_{+,t_i}u(0)\|_{L^2}\\
&\qquad+C\hsc^N(\|u(0)\|_{L^2}+\|\hsc D_{x_1}u(0)\|_{L^2}+\|u\|_{L^2}+\|Pu\|_{L^2})\\
\end{align*}
Using  this for $t_M<t_{M-1}<\dots <t_0$, we have and letting 
\begin{align*}
&\|A_{t_M}Xu\|_{L^2}\\
&\leq  C\hsc^{-1}\|X_Pu\|+C\mu^{-1}\|X(u+Q\hsc D_{x_1}u)\|_{L^2}+\|Fu\|_{L^2}+(\mu+\hsc^{1/2})^M\|X_{+,t_0}\hsc D_{x_1}Xu(0)\|_{L^2}\\
&\qquad+C\hsc^N(\|u(0)\|_{L^2}+\|\hsc D_{x_1}u(0)\|_{L^2}+\|u\|_{L^2}+\|Pu\|_{L^2}).
\end{align*}
Taking $\mu=\hsc^\beta$ and $M\geq N/\beta$ completes the proof.
\end{proof}


\section*{Acknowledgements}
The authors thank Ivan Graham (University of Bath) and Pierre Marchand (INRIA Paris) for useful discussions, Alex Barnett (Flatiron Institute) for the parametrization of the cavity domain, Jeremy Hoskins (University of Chicago) for the generalized Gaussian quadrature rule used in the Galerkin method, and Andras Vasy (Stanford) for a helpful discussion about the proof of propagation of singularities for boundary value problems.

JG was supported by EPSRC grants EP/V001760/1 and EP/V051636/1, Leverhulme Research Project Grant RPG-2023-325, and ERC Synergy Grant PSINumScat - 101167139, and EAS was supported by EPSRC grant EP/R005591/1 and ERC Synergy Grant PSINumScat - 101167139.

\footnotesize{
\bibliographystyle{acm}
\bibliography{biblio_combined_sncwadditions.bib}
}

\end{document}